\documentclass[12pt]{book}

\usepackage{a4wide}
\usepackage{fancyhdr}
\usepackage{graphicx}
\usepackage[english]{babel}
\usepackage[bookmarks]{hyperref}
\usepackage{parskip}
\usepackage[nottoc]{tocbibind}

\usepackage{amsmath}
\usepackage{amssymb}
\usepackage{amscd}
\usepackage{amsfonts}
\usepackage{mathrsfs}
\usepackage{amsthm}

\usepackage{braket}

\begingroup
    \makeatletter
    \@for\theoremstyle:=definition,remark,plain\do{%
        \expandafter\g@addto@macro\csname th@\theoremstyle\endcsname{%
            \addtolength\thm@preskip\parskip
            }%
        }
\endgroup

\newcommand{\E}{\mathcal{E}}
\newcommand{\Ee}{\mathcal{E}_{\rm e}}
\newcommand{\DEl}{D(E)_{\rm loc}}
\newcommand{\DEln}{D(E)_{{\rm loc},\, \mathcal{N}}}

\newcommand{\Ep}{E_{\varphi}}
\newcommand{\Eph}{\widehat{E}_{\varphi}}

\newcommand{\Eg}{ E^{(M)}}
\newcommand{\Egm}{E^{(M)}}
\newcommand{\Ekm}{E^{(k)}}
\newcommand{\Ek}{\widehat{E}^{(k)}}

\newcommand{\Ew}{ E^{(w)}}
\newcommand{\sE}{ {\rm supp}\,E}
\newcommand{\Er}{E^{\,{\rm ref}}}

\newcommand{\Bf}{\mathcal{B}_{\rm fin}}
\newcommand{\Lpf}{L^p_{\rm fin}(m)}
\newcommand{\Ltf}{L^2_{\rm fin}(m)}
\newcommand{\Lpl}{L^p_{\rm loc}(m)}
\newcommand{\Ltl}{L^2_{\rm loc}(m)}
\newcommand{\Ltlm}{L^2_{\rm loc}({\rm vol}_g)}

\newcommand{\IR}{\mathbb{R}}

\newcommand{\IN}{\mathbb{N}}

\newcommand{\cB}{\mathcal{B}}
\newcommand{\cN}{\mathcal{N}}

\newcommand{\ceh}{{\rm cap}_{E,h}}

\newcommand{\Ge}{G^{\,r}}

\newcommand{\qedc}{\hfill $\bigtriangleup$}

\newcommand{\D}{\, {\rm d}}

\newcommand{\as}[1]{\langle #1\rangle}

\newcommand{\ow}[1]{\widetilde{#1}}

\theoremstyle{plain}
\newtheorem{lemma}{Lemma}[chapter]
\newtheorem{proposition}[lemma]{Proposition}
\newtheorem{corollary}[lemma]{Corollary}
\newtheorem{theorem}[lemma]{Theorem}

\newtheorem{conjecture}[lemma]{Conjecture}

\theoremstyle{definition}
\newtheorem{definition}[lemma]{Definition}
\newtheorem{remark}[lemma]{Remark}
\newtheorem{example}[lemma]{Example}

 \newtheorem{question}[]{Question}
  \newtheorem{looseend}[]{Loose end}

\newcommand{\Hmm}[1]{*\leavevmode{\marginpar{\tiny%
$\hbox to 0mm{\hspace*{-0.5mm}$\leftarrow$\hss}%
\vcenter{\vrule depth 0.1mm height 0.1mm width \the\marginparwidth}%
\hbox to 0mm{\hss$\rightarrow$\hspace*{-0.5mm}}$\\\relax\raggedright #1}}}

\newenvironment{rcases}{%
  \left|\renewcommand*\lbrace.%
  \begin{cases}}%
{\end{cases}\right.}

\renewcommand{\chaptermark}[1]%
         {\markboth{\thechapter.\ #1}{}}
\renewcommand{\sectionmark}[1]%
         {\markright{\thesection\ #1}}
\newcommand{\LMUTitle}[9]{
  \thispagestyle{empty}
  \vspace*{\stretch{1}}
  {\parindent0cm
   \rule{\linewidth}{.7ex}}
  \begin{flushright}

    \vspace*{\stretch{1}} 
    \sffamily\bfseries\Huge
    #1\\
    \vspace*{\stretch{1}}
    \sffamily\bfseries\large
    #2
    \vspace*{\stretch{1}}
  \end{flushright}
  \rule{\linewidth}{.7ex}
  \vspace*{\stretch{5}}
  \begin{center}
    \includegraphics[width=2in]{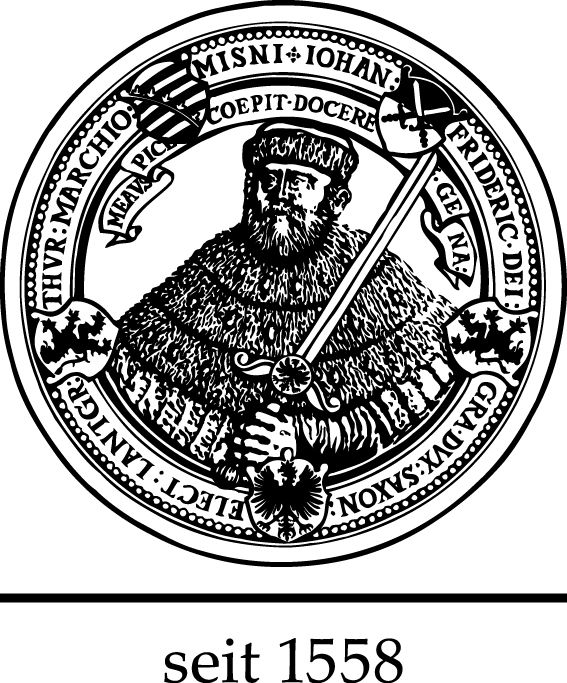}
  \end{center}
  \vspace*{\stretch{1}}
  \begin{center}\sffamily\LARGE{Friedrich-Schiller-Universit\"at Jena 2016}\end{center}
  \newpage
  \thispagestyle{empty}

  \cleardoublepage
  \thispagestyle{empty}

  \vspace*{\stretch{2}}
  {\parindent0cm
  \rule{\linewidth}{.7ex}}
  \begin{flushright}
    \vspace*{\stretch{1}}
    \sffamily\bfseries\Huge
    #1\\
    \vspace*{\stretch{2}}
  \end{flushright}
  \rule{\linewidth}{.7ex}

  \vspace*{\stretch{3}}
  \begin{center}
   \begin{center}
    \includegraphics[width=2in]{Logo_FSU_Bild_jpg}
  \end{center}
  
    \Large {\bf Dissertation} \\
    zur Erlangung des akademischen Grades\\
    {\bf doctor rerum naturalium}\\
    \Large vorgelegt dem Rat der #4\\
    \Large der Friedrich-Schiller-Universit\"at Jena\\
    \vspace*{\stretch{1}}
    \Large   von\\
    \Large {\bf Dipl.-Math. Marcel Schmidt}\\
    \Large geboren am 20.11.1987 in #3\\
    \vspace*{\stretch{2}}
  \end{center}

  \newpage
  \thispagestyle{empty}

  \vspace*{\stretch{1}}

  \begin{flushleft}
    \large 1. Gutachter und Betreuer:  #7 \\[1mm]
    \large 2. Gutachter: #8 \\[1mm]
    \large 3. Gutachter: #9\\[1mm]
    \large Tag der Abgabe: 20.12.2016 \\[1mm]
    \large Tag der \"offentlichen Verteidigung:  \\[1mm]
  \end{flushleft}

  \cleardoublepage
}

\begin{document}

  \frontmatter

  \LMUTitle
      {Energy forms}               
      {Marcel Schmidt}                       
      {Lichtenstein}                             
      {Fakult\"at f\"ur Mathematik und Informatik}                         
      {Jena 2016}                          
      {Abgabedatum}                            
      {Prof. Daniel Lenz (Jena)}                          
      {Prof. Alexander Grigor'yan (Bielefeld)}                         
      {Prof. Peter Stollmann (Chemnitz)}

  \tableofcontents
  



\chapter{Zusammenfassung in deutscher Sprache}

Die vorliegende Arbeit besch\"aftigt sich mit Energieformen (energy forms). Eine Energie\-form ist eine unterhalbstetige quadratische Form auf dem Raum der reellwertigen, $m$-fast \"uberall definierten Funktionen
$$E:L^0(m) \to [0,\infty],$$
welche einer messbaren Funktion $f$ ihre Energie $E(f)$ zuordnet und eine Kontraktionseigenschaft besitzt. Die Kontraktionseigenschaft besagt, dass f\"ur jede normale Kontraktion $C:\IR \to \IR$ die Energie einer Funktion $f$ die Ungleichung 
$$E(C \circ f) \leq E(f)$$
erf\"ullt. Dies ist eine abstrakte Formulierung des Postulats, dass das Abschneiden von Fluktuationen einer Funktion (welche eine physikalische Gr\"o\ss e beschreibt) ihre Energie verringert. 

Energieformen stellen eine Verallgemeinerung von Dirichletformen, erweiterten Dirichletformen und Widerstandsformen (resistance forms) dar. Diese Arbeit untersucht ihre strukturellen Eigenschaften und globalen Eigenschaften. Zu letzteren z\"ahlen Rekurrenz und Transienz, die Eindeutigkeit von Silversteinerweiterungen und stochastische Voll\-st\"andig\-keit. Wir konstruieren maximale Silversteinerweiterungen und entwickeln den Begriff der schwachen L\"osung, einer zu einer Energieform assoziierten Laplacegleichung, um die globalen Eigenschaften zu  charakterisieren. Anders als bisher verwenden wir dabei aus\-schlie\ss lich die algebraische Struktur der Formen und keine Darstellungstheorie oder extrinsische Informationen. 


\chapter{Acknowledgments}

First of all I express my gratitude to my advisor Daniel Lenz. His continuous guidance over the past years has led me to interesting mathematical problems; his generous support always provided me with the means and the opportunities to travel and to communicate my research;  his trust gave me freedom in my scientific work, allowed me  to give lectures to students and to supervise theses of Bachelor and Master students. I thank him for these experiences and challenges, and the skills that I learned from them.

Next, I want to thank my co-authors Batu G\"uneysu, Sebastian Haeseler, Matthias Keller, Daniel Lenz, Jun Masamune,  Florentin M\"unch, Felix Pogorzelski,  Andras Telcs,  Melchior Wirth  and Radek Wojciechowski. Working with them has broadened my mathematical knowledge and strongly influenced my mathematical thinking.

I am indebted to Siegfried Beckus, Daniel Lenz and Mirjam Schmidt for proofreading  this thesis. Thank you for helping me improve the language and for saving the reader from too long and complicated sentences. 

Since the second year of my PhD studies I am part of the Research Training Group 'Quantum and Gravitational Fields' at the Friedrich-Schiller-Universit\"at Jena. I am grateful for the financial  support it provided on multiple occasions. 

During the past years I enjoyed the hospitality of various institutions.  In particular, I want to thank Bobo Hua (Fudan University Shanghai), Matthias Keller (Universit\"at Potsdam), Jun Masamune (Tohoku University Sendai at that time), Radek  Wojciechowski (CUNY New York), Rene Schilling (TU Dresden) and Toshihiro Uemura (Kansai University Osaka) for the invitations and the opportunities to present my research outside of Jena.  I gratefully acknowledge the financial travel support for two visits to Japan by the  European Science Foundation (ESF).

Special thanks go to my colleagues in Jena. Without Daniel L., Daniel S., Felix, Markus, Matthias, Melchior, Ren\'e, Sebastian, Siegfried, Therese and Xueping writing this thesis and working in Jena would have been less fun. In particular, I thank Siegfried for organizing many non-mathematical activities in our group, which led to the excellent atmosphere that we enjoy today.

\newpage

I am very grateful to my dear parents and grandparents who always supported me strongly. Their belief in my abilities gave me the courage to study mathematics and to write this thesis.   

Finally, I want to express my gratitude and my love to my wife Mirjam and my daughter Henrike. With their constant encouragement and their waiving of time spent together they helped me a lot in finishing this thesis. DaDa!


\chapter{Introduction}

In this thesis we study energy forms. These are quadratic forms on the space of real-valued measurable $m$-a.e. determined functions
$$E:L^0(m) \to [0,\infty],$$
which assign to a measurable function $f$ its energy $E(f)$. Their two defining characteristics are a contraction property and some form of continuity. The contraction property  demands that for each normal contraction $C:\IR \to \IR$ the energy of a function $f$ satisfies
$$E(C \circ f) \leq E(f).$$
This is an abstract formulation of the postulate that cutting off fluctuations of a function (which is thought to describe some physical quantity) decreases its energy. The continuity assumption that we impose on energy forms is lower semicontinuity with respect to local convergence in measure. 

 A typical  energy form is Dirichlet's energy integral on bounded subsets $ \Omega \subseteq \IR^n$, which assigns to $f:\Omega \to \IR$ the quantity
 $$\int_{\Omega} |\nabla f|^2 \D x.$$
 It was introduced by Dirichlet in the middle of the 19th century to study boundary value problems for the Laplace operator. This example shows that energy forms have been around for quite a while, but a first systematic approach that emphasizes the importance of the contraction property was given around 100 years later by Beurling and Deny in the seminal papers \cite{BD,BD1}, where they introduce Dirichlet spaces. These works created a lot of momentum in the fields of potential theory, probability theory and operator theory and inspired a vast amount of research. Here, we do not try to give a comprehensive historical overview over this development (we would most certainly fail at this task); we limit ourselves to mentioning some cornerstones in order to formulate the goals of this thesis and to discuss how its contents fit into the picture. 
 
 About a decade after the pioneering work of Beurling and Deny, Fukushima introduced Dirichlet forms as $L^2$-versions of Dirichlet spaces to classify different Brownian motions on open subsets of Euclidean space, see \cite{Fuk69}. In the following years, the efforts of Fukushima \cite{Fuk71,Fuk76,Fuk80} and Silverstein \cite{Sil,Sil2}  revealed the deep connection between regular Dirichlet forms and Hunt processes on locally compact metric spaces. Nowadays, a summary of their results can be found in the books \cite{FOT,CF}. The original Dirichlet spaces of Beurling and Deny appear in this theory as special cases of transient extended Dirichlet spaces. In \cite{AH-K0,AH-K1,AH-K2} Albeverio  and H{\o}egh-Krohn presented examples of Dirichlet forms and associated Markov processes on infinite dimensional state spaces, a case that is not covered by the theory of Fukushima and Silverstein. During the 80's several people tried to extend the probabilistic side of Dirichlet form theory to Dirichlet forms on separable metric spaces.  In full generality, this was realized in the book  of Ma and R\"ockner \cite{MR}, where they introduce quasi-regular Dirichlet forms and show that these are exactly the forms that are associated with (sufficiently good) Markov processes.  
 
  Another popular class of quadratic forms in the spirit of Beurling and Deny are resistance forms,  which emerged from the study of Brownian motion on fractals.  They were defined by Kigami as measure free versions of Dirichlet forms, see \cite{Kig1,Kig2}. Resistance forms satisfy the mentioned contraction property, but they do not live on $L^2$-spaces and, in general, they are not extended Dirichlet forms. Their importance stems from the fact that on fractals there are multiple, mutually singular Borel measures that could serve as reference measure. With the help of resistance forms, the different choices of measures can be treated simultaneously.

 The {\bf first goal} of this thesis is to show that our notion of energy forms provides a general framework that encompasses Dirichlet forms, extended Dirichlet forms and resistance forms.  The three  classes  have in common that they satisfy  the  contraction property, while at a first glance  their continuity properties are rather different. Dirichlet forms are closed forms on $L^2(m)$, i.e., they are lower semicontinuous with respect to $L^2(m)$-convergence; extended Dirichlet forms are introduced in an ad-hoc manner, their definition has similarities with the closure of a form; resistance forms are degenerate complete inner products  with the additional property that their domain embeds continuously into a quotient of the underlying function space. In the first chapter of this thesis, we introduce closed quadratic forms on general topological vector spaces and provide several characterizations for closedness and closability.  It turns out that the  continuity properties of Dirichlet forms, extended Dirichlet forms and resistance forms are different aspects of closedness. In this language, the extended Dirichlet form is the closure of the Dirichlet form on the topological vector space $L^0(m)$ equipped with the topology of local convergence in measure and a resistance form is a closed form on the topological vector space of all functions on a set equipped with the topology of pointwise convergence. The latter space can also be interpreted as $L^0(\mu)$, where $\mu$ is the counting measure. In this case, the topology of pointwise convergence and the topology of local convergence in measure coincide.
 
 As sketched above, the development of Dirichlet form theory was mainly driven by the close connection to probability theory.  This relation is certainly important and powerful; however, too much structure hides intrinsic structure and can lead to complicated proofs. We take the viewpoint that all theorems on energy forms that make claims about the form itself should be proven by form methods only. The {\bf second goal} of this thesis is to demonstrate that a restriction to form methods can open the way to shorter, more elegant proofs and to some new structural insights; it is the leitmotif in our work.   
 
 The local behavior of a Markov process does not determine the associated Dirichlet form uniquely. Silverstein realized that a Dirichlet form corresponds to a process with prescribed local behavior iff it is an extension of the Dirichlet form of the minimal process with the given local behavior and the domain of the Dirichlet form of the minimal process is an ideal in the larger domain, see \cite{Sil}. Today, such extensions are called Silverstein extensions. Since their definition only involves simple algebraic properties of the form domains, it is possible to introduce Silverstein extensions for all energy forms. It is the {\bf third goal} of this thesis to 'properly understand' their role in energy form theory. To this end, we provide several characterizations of Silverstein extensions, discuss where they naturally appear and where the limits of Silverstein extensions lay, prove uniqueness theorems and algebraically construct the maximal Silverstein extension when possible.  These results generalize known results for (quasi-)regular Dirichlet forms in various directions and shed some new light on them.
  
 Inspired by the works of Aronson, Davies, De Giorgi, Grigor'yan, Karp, Li, Nash, Moser, Yau and others on manifolds, Sturm showed in a series of papers \cite{Stu1,Stu2,Stu4,Stu3}   that local regular Dirichlet forms provide a convenient framework to study the interplay between the geometry of the underlying space and properties of solutions to the Laplace equation and to the heat equation. His results on recurrence, Liouville theorems, stochastic completeness, heat kernel estimates and Harnack inequalities have since then been extended to many other situations. The fundamental concept that is required to even formulate these theorems is the notion of a weak solution to an equation involving 'the Laplacian' of the form. There are two approaches for defining such weak solutions in the literature. One is powerful but extrinsic, it utilizes notions of weak solutions available in the concrete model, e.g. distributional Laplacians. The second approach makes use of representation theory for Dirichlet forms viz the Beurling-Deny formula. This is technically involved and limited to quasi-regular Dirichlet forms. The {\bf fourth goal} of this thesis is to develop a theory of weak solutions for energy forms that is more algebraic in nature. While a geometric analysis in the spirit of Sturm's work is beyond its scope, we try to convince the reader that the concept of weak solutions that we introduce in the fourth chapter can be useful to study the mentioned problems in great generality.  To this end, we show  that our notion of weak solutions extends classical notions and we prove certain meta theorems of Dirichlet form theory that are known in some concrete situations, e.g. for Dirichlet forms associated with manifolds and graphs, but whose assertions could not even be formulated for general Dirichlet forms. More precisely, we characterize excessive functions as weakly superharmonic functions, we prove that recurrence is equivalent to certain weakly superharmonic functions being constant and we show that stochastic completeness is equivalent to uniqueness of bounded weak solutions to an eigenvalue problem; on manifolds and graphs these are classical results. 
 
 All four of our goals can be summarized as follows: we try to provide a coherent framework for investigating all closed forms with the contraction property such that the developed theory is algebraic in nature and independent of additional extrinsic information.  Apart from the strive for the greatest possible generality, there are several reasons why we work in this universal setup.
 
 Motivated by physics and other applications, the contraction property is inherent in the theory of energy forms.  In contrast,  lower semicontinuity (closedness) is a technical assumption that is used to ensure that the forms do not show the pathological behavior of general functions on infinite dimensional vector spaces. There is no divine reason why one should work with lower semicontinuity with respect to local convergence in measure. Our choice is based on two observations: All of the  examples are lower semicontinuous with respect to local convergence in measure, the weakest topology on Lebesgue spaces; this allows us to apply lower semicontinuity for the largest possible class of nets. The topology of local convergence in measure is  rather independent of the concrete measure $m$; it is determined by the order structure of $L^0(m)$. In spirit, a similar approach is taken by Hinz in \cite{Hin}, where he uses the topology of uniform convergence to obtain independence of the underlying measure. For a discussion on why this  property is important we refer the reader to his text. 
  
 On our quest to reaching the first goal, in the first chapter of this thesis we develop the theory of closed quadratic forms on topological vector spaces to be able to deal with forms on $L^0(m)$. The concepts that are used there do not only apply to quadratic forms but also to other nonnegative convex functionals. For example, energy functionals of $p$-Laplacians on various $L^q$-spaces and, more generally, nonlinear Dirichlet forms in the sense of \cite{CG} can be treated with our theory. In the later chapters, the assumption that we deal with quadratic forms becomes more and more essential, but still some of the proofs, which we develop in view of our second goal, apply to the nonlinear situation. More specifically, we believe that our discussion of recurrence, transience, of excessive functions and parts of the theory of Silverstein extensions could be useful in this context.

 The reasons for our attempts on the third and the fourth goal in the presented generality are twofold. As mentioned above, the notion of weak solutions plays a major role in the geometric analysis of Dirichlet forms and it is fairly well-developed for strongly local regular Dirichlet forms. However, recent years have also seen a strong interest in geometric analysis of nonlocal Dirichlet forms, see e.g. the survey \cite{Kel} and references therein. In the nonlocal world a coherent theory of weak solutions is available for discrete Dirichlet forms, which derive from graphs, but for general jump forms it is missing. In Chapter~\ref{chapter:weak solutions} we fill this gap once and for all, for all Dirichlet forms.  
 
 Dirichlet form theory can be seen as the study of the interplay of quadratic forms with the algebraic and the order structure of the underlying function space. Originally, it concerns forms on commutative measure spaces, but Gross realized in \cite{Gro1,Gro2} that there is an extension to the noncommutative situation. He used it in order to construct and investigate Hamiltonians for interacting boson and fermion systems in quantum field theory. After this pioneering work the theory of noncommutative Dirichlet forms was  developed and many applications in various areas of mathematics have surfaced, see the recent surveys \cite{Rie} and \cite{Cip}. Since our approach to studying Silverstein extensions and weak solutions is purely algebraic, it extends to forms on noncommutative spaces (in a more or less straightforward manner). Therefore, it  can be seen as a first step towards geometric analysis in the spirit of Sturm's work in the noncommutative world.

Many of the theorems in this thesis are  far-reaching generalizations of results that are known for particular examples or for (quasi-)regular Dirichlet forms. Besides larger generality, the novelty lies in their proofs, which are more structural and have a more functional analytic flavor than the original ones. Below, we elaborate on the structure of the thesis  and mention our main theorems. How they relate to known theorems is discussed in detail in the main text.

The thesis is organized as follows. In the first section of Chapter~\ref{chapter:quadratic forms}, we start with a brief introduction on topological vector spaces and then discuss various Lebesgue spaces of measurable functions and their properties.  The subsequent Section~\ref{section:quadratic forms} is devoted to closed quadratic forms on topological vector spaces. Closedness is a property that makes quadratic forms accessible to functional analytic investigations. We characterize  it in terms of completeness of the form domain with respect to the form topology, lower semicontinuity of the form, and the form domain being a Hilbert space, see Proposition~\ref{prop:lower semincontinuity implies closedness}, Theorem~\ref{theorem:characterization closedness} and Theorem~\ref{theorem:characterization closedness hilbert space}. This is followed by a short  discussion of Dirichlet form theory in Section~\ref{section:quadratic forms on lebesgue spaces}. More specifically, we recall the relation of Dirichlet forms and their associated Markovian objects, discuss their contraction properties, see Theorem~\ref{theorem:cutoff properties Dirichlet form}, and give a new and short proof for the existence of the extended Dirichlet form in terms of closable forms on $L^0(m)$, see Theorem~\ref{theorem: existence extended Dirichlet space}. We finish Chapter~\ref{chapter:quadratic forms} with an approximation for closed quadratic forms on $\Ltf$ by continuous ones, see Lemma~\ref{lemma: approximation of closed forms}.

Chapter~\ref{chaper:energy forms} introduces  energy forms, the main objects of our studies, and is devoted to their basic properties. In Section~\ref{section:the definition and main examples} we give the precise definition of energy forms, see Definition~\ref{definition:energy form}, and discuss important examples. In particular, we show that Dirichlet forms, extended Dirichlet forms and resistance forms all provide energy forms, see Subsection~\ref{subsection:resistance forms}. The subsequent Section~\ref{section:contraction properties} treats contraction properties of energy forms. Its main result is Theorem~\ref{theorem:cutoff properties energy forms}, an extension of Theorem~\ref{theorem:cutoff properties Dirichlet form} to all energy forms. After that we discuss structure properties of energy forms in Section~\ref{section:structure properties}. We prove that the form domain of an energy form is a lattice and that bounded functions in the domain  form an algebra, see Theorem~\ref{theorem:algebraic and order properties}. It turns out that the kernel of an energy form is quite important for its properties; its triviality implies that the form domain equipped with the form norm is a Hilbert space, see Theorem~\ref{theorem:continuous embedding of energy forms}. We introduce the concepts of recurrence and transience, see Definition~\ref{definition:recurrence}, and of invariant sets and irreducibility, see Definition~\ref{definition:invariant sets},  to study  the kernel and compute it when it is nontrivial, see Theorem~\ref{thm:kernel of an energy form} and Corollary~\ref{corollary:kernel recurrent forms}. Section~\ref{section:superharmonic functions} deals with superharmonic and excessive functions. We characterize them in terms of additional contraction properties that are not induced by normal contractions, see Theorem~\ref{theorem:characterization of superharmonic functions} and Theorem~\ref{theorem:characterization of excessive functions}, and show that the existence of nonconstant excessive functions is related to recurrence and transience, see Theorem~\ref{theorem:recurrence in terms of constant excessive functions}. In the last two sections of Chapter~\ref{chaper:energy forms}, we introduce capacities and local spaces of energy forms by adapting the corresponding notions from Dirichlet form theory to our setting.

Chapter~\ref{chapter:silverstein extensions} is devoted to Silverstein extensions of energy forms. These are defined via an ideal property of the form domains. In Section~\ref{section:ideals} we characterize this ideal property in terms of the algebraic structure, the order structure and in terms of the local form domain, see Theorem~\ref{theorem:Ideal properties of energy forms}, and discuss where Silverstein extensions naturally appear. Section~\ref{section:uniqueness of silverstein extensions} deals with Silverstein uniqueness. We prove that recurrence implies Silverstein uniqueness, see Theorem~\ref{theorem:recurrent forms are Silverstein unique}, and give another abstract criterion for Silverstein uniqueness in terms of capacities, see Theorem~\ref{theorem:capacity characterization of uniqueness of Silverstein extensions}. The latter is then used to relate Silverstein uniqueness with vanishing of the topological capacity of  boundaries, see Theorem~\ref{theorem:silverstein uniqueness boundary inside space} and Theorem~\ref{theorem:silverstein uniqueness boundary outside space}. The last section of  Chapter~\ref{chapter:silverstein extensions} deals with the existence of the maximal Silverstein extension. We give two examples of energy forms that do not have a maximal Silverstein extension, see Example~\ref{example:counterexample maximal Silverstein extension energy form} and Example~\ref{example:counterexample maximal Silverstein extension regular dirichlet form}. We decompose energy forms into a main part, see Definition~\ref{definition:main part}, and a monotone killing part, see Definition~\ref{definition:killing part}, and provide an algebraic construction of the maximal Silverstein extension when the killing vanishes, see Theorem~\ref{theorem:maximality er}. At the end of the section, we discuss how this can be employed to obtain maximal Markovian extensions of hypoelliptic operators, see Theorem~\ref{theorem:maximality er operator} and Corollary~\ref{corollary:maximal markovian extension}.

In Chapter~\ref{chapter:weak solutions} we study properties of weak (super-)solutions to the Laplace equation. In the first section we introduce the weak form extension, see Definition~\ref{definition:weak form}, that is an off-diagonal extension of the given form, see Theorem~\ref{theorem:ew extends e}. We use the weak form to define weak solutions to the Laplace equation, see Definition~\ref{definition:weak solution}, and show for several examples that this notion coincides with the 'classical' one.  The second section deals with excessive functions, which are characterized as weakly superharmonic functions, see Theorem~\ref{theorem:characterization excessive functions as superharmonic functions}. The following section is devoted to the existence of minimal weak solutions to the Laplace equation. More precisely, we define the extended potential operator, see Definition~\ref{definition:extended potential operator}, and show that it provides minimal weak solutions when they exist, see Theorem~\ref{theorem:maximal principle potential operator}. We conclude Chapter~\ref{chapter:weak solutions} with several applications. We prove that uniqueness of bounded weak solutions to the Laplace equation is related to a conservation property for the extended potential operator, see Theorem~\ref{theorem:uniqueness of bounded weak solutions}, show that $L^p$-resolvents of Dirichlet forms provide weak $L^p$-solutions to an Eigenvalue problem, see Theorem~\ref{theorem:lp resolvents}, and characterize recurrence in terms of certain  weakly superharmonic functions being constant, see Theorem~\ref{theorem:finish}.

  \mainmatter\setcounter{page}{1}
  \chapter{Quadratic forms on topological vector spaces} \label{chapter:quadratic forms}

\section{Topological vector spaces}

In this section we briefly review parts of the theory of topological vector spaces and fix some notation. The notions and theorems are standard and can be found in (almost) any textbook on this subject. For details we refer the reader to the classics \cite{Bou} and \cite{Schae}. As main examples for later applications we discuss Lebesgue spaces of measurable functions. For a background on them we refer the reader to \cite{Fre1,Fre2,Fre3}.

\subsection{Generalities}

Let $X$ be a nonempty set. A {\em net} in $X$ is a mapping $\Lambda: I \to X$, where $I$ is a {\em directed set}, i.e., it carries a preorder such that each pair of elements of $I$ has an upper bound. Given a net $\Lambda:I \to X$, for $i\in I$ we set $x_i := \Lambda(i)$ and write $(x_i)_{i \in I}$ or simply $(x_i)$ instead of $\Lambda$. We abuse notation and always denote directed preorders by $\prec$.  A {\em subnet of the net $(x_i)_{i\in I}$} is a pair $((y_j)_{j \in J},\Delta)$, where $(y_j)_{j \in J}$ is a net and $\Delta:J \to I$ is a monotone mapping such that $y_j = x_{\Delta(j)}$ for all $j \in J$ and $\Delta(J)$ is {\em cofinal} in $I$, i.e., for each $i \in I$ there exists $j \in J$ with $i \prec\Delta(j)$. For a subnet $((y_j)_{j \in J},\Delta)$ of $(x_i)_{i \in I}$, we set $x_{i_j}:= x_{\Delta(j)} = y_j$ and  write $(x_{i_j})_{j \in J}$ or $(x_{i_j})$ instead of $((y_j)_{j \in J},\Delta)$. A net with index set $\IN$, the natural numbers, is called {\em sequence}. Whenever $X$ carries a topology $\tau$, an element $x \in X$ is a {\em limit of the net} $(x_i)_{i\in I}$ if for each $O \in \tau$ there exists $i_O \in I$ such that $x_i \in O$ for all $i\in I$ with $i_O \prec i$. In this case, we write
$$x = \lim_{i\in I} x_i = \lim_i x_i. $$
If $(x_i)$ is a net in $\IR$, the real numbers, we let
$$\liminf_{i} x_i := \lim_{i} \inf \{x_j \mid i \prec j\} \text{ and } \limsup_{i} x_i := \lim_{i} \sup \{x_j \mid i \prec j\}, $$
which exist in the two point compactification $\IR \cup\{-\infty, \infty\}$. The space of real-valued continuous functions on $(X,\tau)$ is denoted by $C(X)$ and $C_c(X)$ is the subspace of continuous functions of compact support.  

Let $V$ be a vector space over the real numbers. In what follows we assume that $V$ is equipped with a {\em vector space topology} $\tau$, i.e., a topology under which the operations 
$$V \times V \to V, \quad (u,v) \mapsto  u + v \quad \text{ and } \quad  \IR \times V \to V, \quad (\lambda,u) \mapsto \lambda u$$
are continuous. Here, the product spaces carry the corresponding product topology. If, additionally, $\tau$ is Hausdorff, the pair  $(V,\tau)$ is called {\em topological vector space}.  A subset $B\subseteq V$ is called {\em balanced} if for all $u \in B$ and all $\lambda \in \IR$ with $|\lambda| \leq 1$ we have $\lambda u \in B$. The required continuity properties imply that any vector space topology has a basis that is translation invariant and consists of balanced sets.

Every topological vector space $(V,\tau)$ carries a canonical translation invariant uniform structure, namely the system  of subsets of $V \times V$ given by 
$$\mathcal{U}(\tau) := \left\{ \{(u,v) \mid u - v \in U\} \mid  U \text{ is a neighborhood of }0\right\}.$$
When referring to uniform concepts such as Cauchyness or uniform continuity, we always assume that the underlying uniform structure is the above. In particular, we call $(V,\tau)$ {\em complete} if it is complete with respect to $\mathcal{U}(\tau)$, i.e., if every $\mathcal{U}(\tau)$-Cauchy net converges. When the topology $\tau$ is induced by a translation invariant metric, uniform concepts with respect to the metric and with respect to the uniform structure $\mathcal{U}(\tau)$ coincide. In this case, complete spaces are known under the following name.

\begin{definition}[F-space]
A topological vector space $(V,\tau)$ is called {\em $F$-space} if it is complete and $\tau$ is metrizable.
\end{definition}

Due to Baire's category theorem, linear operators between $F$-spaces obey the fundamental theorems of functional analysis. Indeed, we have the following, see e.g. \cite[Theorem 3.8]{Hus}.

\begin{theorem}[Open mapping and closed graph theorem] \label{theorem:open mapping and closed graph theorem}
Let $(V_i,\tau_i),$ $i=1,2$, be two $F$-spaces. Then the following holds.
\begin{itemize}
 \item[(a)] Every linear and continuous operator from $V_1$ onto $V_2$ is open.
 \item[(b)] Every linear operator from $V_1$ into $V_2$ whose graph is closed in $V_1 \times V_2$ is continuous.  
\end{itemize}
\end{theorem}

Quotients of topological vector spaces can be equipped with a vector space topology. The following lemma is contained in \cite[Section 2.3]{Schae}.

\begin{lemma}[Quotients of topological vector spaces]\label{lemma:quotient topology}
 Let $V$ be a vector space and $\tau$ be a vector space topology. Let $S \subseteq V$ be a subspace and let $\pi:V \to V/S$ be  the canonical projection to the quotient. Then
 $$\tau/S := \{U \subseteq V/S \mid \pi^{-1}(U) \in \tau\}$$
 is a vector space topology on $V/S$. It is Hausdorff if and only if $S$ is closed. Moreover, the map $\pi$ is open and for any basis $\mathcal{O}$ of $\tau$ the collection $\{\pi(O)\mid O \in \mathcal{O}\}$ is a basis for $\tau / S$.  
\end{lemma}

Metrizability and completeness pass to the quotient, see e.g. \cite[Theorem 6.3]{Schae}.

\begin{lemma}[Completeness of quotients]\label{lemma:completeness of quotients}
 Let $(V,\tau)$ be an $F$-space and let $S \subseteq V$ be a closed subspace. Then $(V/S,\tau/S)$ is an $F$-space. 
\end{lemma}


A topological vector space $(V,\tau)$ is called {\em locally convex} if there exists a family of seminorms $(p_i)_{i\in I}$ such that the collection of sets
$$\left\{\{u\in V \mid p_{i}(u-v) < \varepsilon\}\mid i \in I,\, v \in V,\, \varepsilon >0\right\}$$
forms a subbasis for the topology $\tau$. In this case, we say that {\em $(p_i)_{i\in I}$ generates $\tau$}. For locally convex spaces, the Hahn-Banach theorem   implies that their continuous dual 
$$V' := \{\varphi:V \to \IR \mid \varphi \text{ is linear and continuous}\}$$
separates the points of $V$, see e.g. \cite[Corollary~1 of Theorem~4.2]{Schae}.  As a consequence, the {\em weak topology} $\sigma(V,V')$ that is generated by the family of seminorms
$$p_\varphi:V \to [0,\infty), \,  u \mapsto p_\varphi (u) := |\varphi(u)|, \,   \varphi \in V',$$
is Hausdorff, i.e., the pair $(V,\sigma(V,V'))$ is a topological vector space. The fundamental theorem relating $\tau$ and $\sigma(V,V')$ on convex sets is the following, see e.g. \cite[Corollary 2 of Theorem 9.2]{Schae}.

\begin{theorem}[Weak closure of convex sets] \label{theorem:weak closure of convex sets}
 Let $(V,\tau)$ be a locally convex topological vector space and let $K \subseteq V$ be convex. The closure of $K$ with respect to $\tau$ and the closure of $K$ with respect to $\sigma(V,V')$ coincide.
\end{theorem}

\subsection{Lebesgue spaces as topological vector spaces} \label{section:lebesgue spaces}

In later sections $V$ is some Lebesgue space of real-valued measurable functions. We give a short overview over their topological and order properties and provide some examples. 

Let $(X,\mathcal{B},m)$ be a measure space. If not further specified we do not make any assumptions on the measure $m$. The collection of measurable sets of finite measure is
$$\Bf:= \{U \in \cB \mid m(U) < \infty\}.$$
We equip the reals with the Borel $\sigma$-algebra and denote the space of real-valued measurable function by $\mathcal{M}(X,\cB)$. The space of $m$-a.e. defined functions is
$$L^0(X,\cB, m) := \mathcal{M}(X,\cB) / \{f \in \mathcal{M}(X,\cB) \mid f = 0\, m\text{-a.e.}\}.$$
We often omit  the $\sigma$-algebra or the underlying space and simply write $L^0(X,m)$ or $L^0(m)$ instead of $L^0(X,\cB, m)$. Most of the times we do not distinguish between functions in $\mathcal{M}(X,\cB)$ and elements of $L^0(m)$. In this sense, we treat addition, multiplication, pointwise inequalities etc.  Likewise, we do not distinguish between sets and up to $m$-measure zero defined sets. We warn the reader that this viewpoint leads to subtleties in some proofs that we only address implicitly. The space $L^0(m)$ is equipped with the natural order relation on functions, i.e., for $f,g \in L^0(m)$ we say $f \leq g$ if and only if $f(x) \leq g(x)$ for $m$-a.e. $x\in X$.  The functions  
$$f \wedge g := \min \{f,g\} \text{ and } f \vee g := \max\{f,g\}$$
are the biggest lower bound of $f,g$ and the smallest upper bound of $f,g$, respectively. A function $f \in L^0(m)$ is nonnegative if $f \geq 0$ and strictly positive if $f(x) > 0$ for $m$-a.e. $x \in X$. By $f_+ := f \wedge 0$ and $f_- := (-f) \wedge 0$ we denote the positive and the negative part of $f$, respectively. For $f,g \in L^0(m)$, we use the notation 
$$\{f \geq g\} := \{x \in X \mid f(x) \geq g(x) \},$$
which is a set that is defined up to a set of $m$-measure zero. Likewise, we define the $m$-a.e. determined sets $\{f>g\}, \{f \leq g\}$ and $\{f < g\}$.
%
%

The {\em indicator function} of a set $U \subseteq X$ is given by
$$1_U:X \to \IR,\, x \mapsto 1_U(x) := \begin{cases}
                            0 &\text{if } x \in X \setminus U\\
                            1 &\text{if } x \in U
                           \end{cases}.
$$ 

For two measurable sets $A,B \subseteq X$, we say that $A\prec B$ if $m(A \setminus B) =0$, or equivalently, if $1_A \leq 1_B$ holds in $L^0(m)$. This introduces an order relation on the collection of all $m$-a.e. defined measurable sets  (and a preorder on all measurable sets). Whenever a collection of measurable sets is the index set of a net it will always be equipped with this order relation.

We topologize the space $L^0(m)$ as follows. For every $U \in \Bf$, we let the pseudometric $d_U$  be given by
$$d_U:L^0(m) \times L^0(m) \to [0,\infty), \, (f,g) \mapsto  d_U(f,g) := \int_U |f-g|\wedge 1\, {\rm d}m. $$
The so constructed family $\{d_U \mid U \in \Bf\}$  induces a topology $\tau(m)$ on $L^0(m)$ (the smallest topology that contains all balls with respect to the $d_U$, $U\in \Bf$), which we call the {\em topology of local convergence in measure}. A net $(f_i)$ converges towards $f$ with respect to $\tau(m)$ if and only if $\lim_i d_U(f_i,f) = 0$ for each $U \in \Bf$. In this case, we write $f_i \overset{m}{\to}f$ for short. An alternative subbasis of zero neighborhoods for $\tau(m)$ is the family of sets
$$\{g \in L^0(m) \mid m(U\cap\{|g|\geq \delta\}) < \varepsilon\}, \text{   } \varepsilon,\delta > 0 \text{ and }U\in \Bf.$$
The following lemma shows that $\tau(m)$ is compatible with the vector and lattice operations on $L^0(m)$, see \cite[Proposition~245D]{Fre2}.

\begin{lemma}
 The topology $\tau(m)$ is a vector space topology on $L^0(m)$. Moreover, multiplication and taking maxima and minima are continuous  operations from $L^0(m) \times L^0(m)$ to $L^0(m)$.  
\end{lemma}

 Further properties of the topology $\tau(m)$ depend on $m$. It may or may not be complete and, if too many sets have infinite measure, it may happen that $\tau(m)$ is not Hausdorff. In order to clarify these phenomena we recall the following concepts of measure theory, cf. \cite[Chapter 21]{Fre2}.

\begin{definition}[Taxonomy of measures] \label{definition:taxonomy of measures} Let $(X,\mathcal{B},m)$ be a measure space. 
 \begin{itemize}
  \item[(a)] $m$ is  {\em finite} if $m(X) < \infty$.  It is {\em $\sigma$-finite} if there exists a sequence of measurable sets of finite measure $(X_n)_{n \in \IN}$ such that $X = \bigcup_n X_n$.
  \item[(b)] $m$ is  {\em strictly localizable} if there exists a partition $(X_i)_{i\in I}$ of $X$ into measurable sets of finite measure such that $\mathcal{B} = \{A \subseteq X \mid A \cap X_i \in \mathcal{B} \text{ for all } i \in I\}$ and
$$m(A) = \sum_{i \in I} m(A \cap X_i) \text{ for every } A \in \mathcal{B}.$$
\item[(c)] $m$ is {\em semi-finite} if for each $A \in \mathcal{B}$ with $m(A) = \infty$ there exists some $A' \subseteq A$ with $A' \in \mathcal{B}$ and $0 < m(A')<\infty$.
\item[(d)] $m$ is {\em localizable} if it is semi-finite and for each $\mathcal{A} \subseteq \mathcal{B}$ there exists a set $S \in \mathcal{B}$ such that 
\begin{itemize}
 \item $m(A\setminus S) = 0$ for each $A \in \mathcal{A},$ 
 \item if $G \in \mathcal{B}$ and $m(A\setminus G) = 0$ holds for all $A \in \mathcal{A}$, then $m(S\setminus G) = 0$.
\end{itemize}
In this case, $S$ is called an {\em essential supremum} of $\mathcal{A}$.
 \end{itemize}
\end{definition}

\begin{remark}
 The definition of localizable measures is a bit clumsy. Indeed, it is a concept related to the measure algebra of $m$. Localizability is equivalent to the Dedekind completeness of the measure algebra with respect to the order relation on sets that we introduced previously.  
\end{remark}

Some relations among the mentioned concepts are trivial; finite measures and $\sigma$-finite measures are strictly localizable. The following is less obvious, see \cite[Theorem~211L]{Fre2}.

\begin{theorem}\label{theorem:strictly localizable implies localizable}
 Every strictly localizable measure is localizable. In particular, finite and $\sigma$-finite measures are localizable. 
\end{theorem}

With the these concepts at hand, we can state the theorem that characterizes  the topological properties of $\tau (m)$, see \cite[Theorem~245E]{Fre2}.

\begin{theorem}[Topological properties of $L^0(m)$] \label{theorem:properties of L0} Let $(X,\mathcal{B},m)$ be a measure space.
\begin{itemize}
 \item[(a)] $m$ is semi-finite if and only if $(L^0(m),\tau(m))$ is Hausdorff.
 \item[(b)] $m$ is $\sigma$-finite if and only if $(L^0(m),\tau(m))$ is metrizable.
 \item[(c)] $m$ is localizable if and only if $(L^0(m),\tau(m))$ is Hausdorff and complete. 
 \end{itemize}
\end{theorem}

\begin{remark}
The previous two theorems show that $L^0(m)$ is an $F$-space if and only if $m$ is $\sigma$-finite. In this case, convergence of sequences with respect to $\tau(m)$ can be characterized in terms of $m$-a.e. convergent subsequences. Namely, if $m$ is $\sigma$-finite a sequence $(f_n)$ converges towards $f$ with respect to $\tau(m)$ if and only if any of its subsequences has a subsequence that converges to $f$ $m$-a.e., see \cite[Proposition~245K]{Fre2}.
\end{remark}

Whether or not $L^0(m)$ is locally convex depends on the atoms of $m$. We do not give a full characterization but rather illustrate this fact with two examples.

\begin{example}
 \begin{itemize}
  \item[(a)]  Let $X = [0,1]$ and let $m =\lambda,$ the Lebesgue measure on Borel subsets of $[0,1]$. It can be shown that all continuous functionals on $(L^0(\lambda),\tau(\lambda))$ vanish. Therefore, it can not be locally convex. 
  \item[(b)] Let $X \neq \emptyset$ be arbitrary and let $m = \mu$, the counting measure on all subsets of $X$. In this case, $L^0(\mu)$ is the space of all real-valued functions on $X$ and the topology $\tau(\mu)$ is the topology of pointwise convergence.  It is obviously generated by a family of seminorms, and so $(L^0(\mu),\tau (\mu))$ is a locally convex topological vector space.
 \end{itemize}
\end{example}


If $m$ is localizable, the space $L^0(m)$ is Dedekind complete, i.e., any subset of $L^0(m)$ that is bounded from above has a supremum in $L^0(m)$, see \cite[Theorem~241G]{Fre2}. A consequence is the following characterization of convergence for monotone nets in $L^0(m)$. Recall that a net $(f_i)$ in $L^0(m)$ is called monotone increasing if $i \prec j$ implies $f_i \leq f_j$. 

\begin{proposition}\label{proposition:convergence of monotone nets}
 Let $m$ be localizable. A monotone increasing net $(f_i)$ in $L^0(m)$ is convergent with respect to $\tau(m)$ if and only if there exists $g \in L^0(m)$ such that for all $i$ the inequality $f_i \leq g$ holds. In this case, we have
 $$\lim_i f_i = \sup_i f_i.$$
\end{proposition}
\begin{proof}
 Let $(f_i)$ be a monotone increasing net and let $g \in L^0(m)$ with $f_i \leq g$ for all $i$. Due to the Dedekind completeness of $L^0(m)$, the supremum $f := \sup_i  f_i$ exists. We show $f_i \overset{m}{\to} f$. Assume  this were not the case. The monotonicity of $(f_i)$ implies that there exist  $\varepsilon,\delta> 0$ and $U \in \Bf$ such that for all $i$ we have
 $$m(U\cap \{|f-f_i|\geq \delta\}) \geq \varepsilon.$$
 We let $U_i:= U\cap \{|f-f_i|\geq \delta\}$ and choose an increasing sequence of indices $i_n, n \in \IN,$ with
 $$ \lim_{n\to \infty} m(U_{i_n}) = \inf_i m(U_i) \geq \varepsilon. $$
 We set $\ow{U}:= \cap_{n\geq 1} U_{i_n}$. The $U_{i_n}$ are decreasing and $U$ has finite measure. We obtain
 $$m(\ow{U}) =  \lim_{n\to \infty} m(U_{i_n}) \geq \varepsilon$$
 and use this observation to show that $f - \delta 1_{\ow{U}}$ is an upper bound for the $f_i$, which contradicts the definition of $f$. According to the definition of $U_i$, it suffices to prove $m(\ow{U} \setminus U_i) = 0$ for all $i$. Assume the contrary. Then there exists an $i_0$ with
 $$0 < m(\ow{U} \setminus U_{i_0}) = m(\ow{U}) - m(U_{i_0} \cap \ow{U}) = m(\ow{U}) - \lim_{n\to \infty}m(U_{i_0} \cap U_{i_n}).$$
 Consequently, there exists an $n \in \IN$ such that $m(\ow{U}) > m(U_{i_0} \cap U_{i_n})$. We choose an index $j$ with $i_0,i_n \prec j$.  By the monotonicity of $(f_i)$, we obtain $m(\ow{U}) > m(U_{i_0} \cap U_{i_n}) \geq m(U_j)$. This is a contradiction to the definition of $\ow{U}$ and the convergence of $(f_i)$ to $f$ is proven.  
 
 Now, assume that the net $(f_i)$ is convergent. The continuity of the lattice operations with respect to local convergence in measure shows that $g:= \lim_i f_i$ is an upper bound for the $f_i$. This finishes the proof.
 \end{proof}

For $1\leq p <\infty$, we let  $L^p(m)$ be the space of $p$-integrable functions in $L^0(m)$ with corresponding norm
$$\|f\|_p := \left(\int_X |f|^p\, {\rm d}m\right)^{\frac{1}{p}}.$$
 The space of essentially bounded $L^0(m)$-functions is denoted by $L^\infty(m)$ and equipped with the norm 
 $$\|f\|_\infty = {\rm ess\, sup} |f|.$$
When $p=2$, the norm $\|\cdot\|_2$ is induced by the  inner product 
$$\as{f,g}:= \int_X fg\, {\rm d}m.$$
The space of functions that are $p$-integrable on sets of finite measure is denoted by
$$\Lpf := \{f \in L^0(m) \mid f 1 _U \in L^p(m) \text{ for each } U \in \Bf\}.$$
We equip it with the locally convex topology generated by the family of seminorms
$$\|\cdot\|_{p,U}:\Lpf \to [0,\infty),\, f \mapsto \|f\|_{p,U} := \|f1_U\|_p, \, U \in \Bf.$$
When $X$ is a Hausdorff topological space and $m$ is a Radon measure on the Borel $\sigma$-algebra of $X$, i.e., $m$ is locally finite and inner regular, then 
$$\Lpl := \{f \in L^0(m) \mid f 1_K \in L^p(m) \text{ for all compact } K \subseteq X \}$$
is the local $L^p$-space.  It is equipped with the locally convex topology generated by the family of seminorms
$$\|\cdot\|_{p,K},\, K \subseteq X \text{ compact}.$$

H\"older's inequality implies that $L^p(m)$  and  $\Lpf$ are continuously embedded in $L^0(m)$. For $1\leq p < \infty$, a converse of this statement holds on order bounded sets. Indeed, this is a consequence of the following version Lebesgue's theorem for nets. Since it is usually formulated for sequences only, we include a short proof. 
\begin{lemma}[Lebesgue's dominated convergence theorem]\label{lemma:Lebesgue's theorem}
 Let $(f_i)$ be a convergent net in $L^0(m)$ with $f_i \overset{m}{\to}f$. Assume that there exists a function $g \in L^1(m)$ such that for all $i$
 $$|f_i| \leq g.$$
 Then  $f \in L^1(m)$ and
 $$\lim_i \int_X |f-f_i|\, {\rm d}m = 0.$$
\end{lemma}
\begin{proof}  For $k\in \IN$ we let $A_k:= \{x \in X \mid 1/k \leq g(x) \leq k\}$. Then $m(A_k) < \infty$ and the sequence $(g 1_{A_k})$ converges in $L^1(m)$ to $g$, as $k \to \infty$. Since the lattice operations are continuous on $L^0(m)$, we have $|f| \leq g$. This implies 
$$\int_{X\setminus A_k} |f-f_i|\, {\rm d}m  \leq 2\int_{X\setminus A_k} g\, {\rm d}m \to 0, \text{ as } k\to \infty, $$
where the convergence is uniform in $i$.  Thus, for each $\varepsilon > 0$ we may choose some $k$ independent of $i$ such that
$$\int_X |f-f_i|\, {\rm d}m \leq \int_{A_k} |f-f_i|\, {\rm d}m + \varepsilon = \int_{A_k} |f-f_i|\wedge 2k\, {\rm d}m + \varepsilon.$$
Since $A_k$ has finite measure and $(f_i)$ converges locally in measure to $f$, we obtain the statement. 
\end{proof}
\begin{remark}
In contrast to the previous lemma, it is well known that Lebesgue's theorem does not hold for almost everywhere convergent nets. Indeed, this is one of the main technical reasons why we work with convergence in measure instead of almost everywhere convergence. 
\end{remark}
\begin{lemma}
 Let $g \in L^p(m)$ (or $g \in \Lpf$). The $L^p(m)$-topology (or the $\Lpf$-topology) coincides with the subspace topology of  $\tau(m)$ on
$$\{f \in L^0(m) \mid |f| \leq g\}.$$
\end{lemma}
\begin{proof}
 We only show the statement for $g \in L^p(m),$ the other case works similarly. Let $(f_i)$ be a net in $\{f \in L^0(m) \mid |f| \leq g\}$ with $f_i \overset{m}{\to} f$. We obtain $|f_i - f|^p \overset{m}{\to} 0$. Since $|f_i - f|^p \leq 2^p g^p$ and $g^p \in L^1(m)$, the claim follows from Lebesgue's theorem, Lemma~\ref{lemma:Lebesgue's theorem}.
\end{proof}
The $L^p(m)$ spaces are always Hausdorff and complete. In contrast we have seen that this need not be true for $L^0(m)$. Similar phenomena occur for $\Lpf$. Indeed, it can be shown that Theorem~\ref{theorem:properties of L0} remains true when $L^0(m)$ is replaced by $\Lpf$. We only need the following special case.

\begin{proposition}\label{proposition:completeness local Lp} Let $m$ be localizable. For $1\leq p < \infty$, the space $\Lpf$ is a complete locally convex topological vector space.
\end{proposition}
\begin{proof}
   The space $\Lpf$ continuously embeds into the Hausdorff space $L^0(m)$ and so it is Hausdorff itself. It remains to prove completeness. Let $(f_i)$ be a Cauchy net in $\Lpf$. Since $\Lpf$ is continuously embedded in $L^0(m)$, the net $(f_i)$ is also Cauchy with respect to $\tau(m)$. The localizability of $m$ and Theorem~\ref{theorem:properties of L0} show that it has a $\tau(m)$-limit $f$. Let $U\in \Bf$ be given. We use the monotone convergence theorem and Lemma~\ref{lemma:Lebesgue's theorem} to obtain 
 $$\int_U |f - f_i|^p\, {\rm d} m = \lim_{n \to \infty} \int_U |f - f_i|^p \wedge n \, {\rm d} m = \lim_{n \to \infty} \lim_j\int_U |f_j - f_i|^p \wedge n\, {\rm d} m. $$
 The above expression is monotone increasing in $n$. Therefore, the previous computation and Lemma~\ref{lemma:monotone nets} yield
 $$\int_U |f - f_i|^p\, {\rm d} m \leq  \liminf_j \lim_{n \to \infty}\int_U |f_j - f_i|^p \wedge n\, {\rm d} m = \liminf_j \int_U |f_j - f_i|^p \, {\rm d} m.$$ 
 Since $(f_i)$ was assumed to be Cauchy in $\Lpf$, this finishes the proof.
\end{proof}

Next, we discuss the dual spaces of $\Lpf$ and their properties.  To keep notation simple we restrict ourselves to the case when $p=2$, but the arguments work for other $1< p  < \infty$ as well.

%

\begin{proposition}\label{proposition:characterization duals of Lpf}
 The map
 $$\{\varphi \in L^2(m) \mid \text{there exists } U \in \Bf \text{ s.t. } \varphi 1_{X\setminus U} = 0\} \to \Ltf'$$
 $$\varphi \mapsto \ell_\varphi,$$
 where
  $$\ell_\varphi:\Ltf \to \IR,\, f \mapsto \ell_\varphi(f) := \int_X f \varphi\, {\rm d}m$$
  is a vector space isomorphism.
\end{proposition}
\begin{proof}
 It suffices to show surjectivity, the rest is clear. Let $\varphi \in \Ltf'$ be given. Its continuity yields the existence of a set $U \in \Bf$ and a constant $C> 0$ such that for all  $f \in \Ltf$ we have
$$|\varphi(f)| \leq C \|f\|_{2,U}.$$
Therefore, $\varphi(f)$ only depends on the restriction of $f$ to the set $U$ and so it can be interpreted as a continuous functional on the Hilbert space $L^2(U,m)$. Consequently, it is  induced by some $\ow{\varphi} \in L^2(U,m)$ via the identity

$$ \varphi(f) = \int_U \ow{\varphi} f|_U \, {\rm d}m.$$
This finishes the proof. 
\end{proof}

\begin{proposition} \label{prop:weak compactness in Ltf}
 Let $m$ be localizable. For any $g \in \Ltf$ the set
 $$\{f \in \Ltf \mid |f|\leq g\}$$
 is compact with respect to the weak topology $\sigma(\Ltf,\Ltf')$.
\end{proposition}
\begin{proof}
We mimic the standard proof of the Banach-Alaoglu-Bourbaki theorem. Since we do not work in the dual space, we have to prove an additional closedness statement, for which we use the localizability of $m$. 

For $U\in \Bf$ we equip $L^2(U,m)$ with the weak topology and consider the space
$$ \mathcal{P} := \prod_{U \in \Bf} L^2(U,m),$$
which we equip with the corresponding product topology. The semi-finiteness of $m$ implies the injectivity of the map 
$$\iota : \Ltf \to  \mathcal{P} , \quad f \mapsto (f|_U)_{U \in \Bf}.$$
 The characterization of $\Ltf'$ in Proposition~\ref{proposition:characterization duals of Lpf} yields that  $\iota$ is a homeomorphism onto its image when $\Ltf$ is equipped with $\sigma(\Ltf,\Ltf')$.

Order bounded sets are weakly compact in $L^2(U,m)$. Accordingly,  Tychonoff's theorem yields that $\{(f_U) \in  \mathcal{P}  \mid |f_U|\leq g|_U\})$ is compact in $\mathcal{P}$. Since
$$\iota \left(\{f \in \Ltf \mid |f|\leq g\}\right) \subseteq \{(f_U) \in  \mathcal{P}  \mid |f_U|\leq g|_U\}$$
and $\iota$ is a homeomorphism onto its image, it suffices to show that $\iota  \left(\{f \in \Ltf \mid |f|\leq g\}\right)$ is closed in $ \mathcal{P}$. 

To this end, let a net $(f_i)$ in $\{f \in \Ltf \mid |f|\leq g\}$ with $\iota f_i \to (h_U)_{U \in \Bf} \in  \mathcal{P}$ be given. We interpret $h_U \in L^2(U,m)$ as an element of $L^0(m)$ by letting $h_U := 0$ on $X \setminus U$. As limits of globally defined functions, the $(h_U)$ satisfy the compatibility properties

 $$h_U 1_{U\cap V} = h_{U\cap V} \text { and } h_U = h_{U\cap V} + h_{U \cap X\setminus V}, \text{ for all } U,V \in \Bf.$$

 If the $(h_U)$ were functions and not equivalence classes of functions, these compatibility properties would be sufficient to obtain a function on the whole space that locally agrees with them. Since they are not defined everywhere and we care about measurability, we need to argue as follows. 
 
The localizability of $m$ yields that the space $L^0(m)$ is Dedekind complete, see \cite[Theorem~241G]{Fre2}. Since for each $U\in \Bf$ the inequality $|h_U| \leq g$ holds, we obtain  that the suprema 
$$ h_1 := \sup\{(h_U)_+ \mid U \in \Bf\}  \text{ and } h_2:=  \sup\{(h_U)_- \mid U \in \Bf\}$$
exist in $L^0(m)$. We finish the proof by showing $h_1 1_U = (h_U)_+$ and $h_2 1_U = (h_U)_-$, which implies $(h_U) =  \iota (h_1 - h_2)$.

The definition of $h_1$ yields $h_1 1_U \geq (h_U)_+$  for all  $U \in \Bf$. Suppose now $h_1 1_U \neq (h_U)_+$, i.e., there exists some $\varepsilon > 0$ such that 
$$F:= \{ h_1 1_U  - (h_U)_+ \geq \varepsilon\} $$
satisfies $m(F) > 0$. 
For an arbitrary $V \in \Bf$, the compatibility conditions of $(h_U)$ and the definition of $F$ imply
$$(h_1 - \varepsilon 1_F)1_V   =   h_1 1_{V \cap X \setminus F} +  (h_1 - \varepsilon )1_{V \cap F} \geq  (h_{V \cap X \setminus F})_+ +  (h_U)_+|_{V \cap F} = (h_V)_+.$$
Therefore, $h_1 - \varepsilon 1_F$ is an upper bound for the $(h_V)_+$. This  contradicts to the definition of $h_1$ and we obtain $h_1 1_U = (h_U)_+$. A similar reasoning yields the statement for $h_2$.   This finishes the proof.
\end{proof}

\begin{remark}
The previous theorem can also be inferred from the Banach-Alaoglu-Bourbaki theorem after noting that the bidual of $\Ltf$ in an appropriate sense is $\Ltf$ itself.  As the proof of the latter fact is not shorter and needs some further functional analytic concepts, we rather mimicked the standard proof the Banach-Alaoglu-Bourbaki theorem. The real work was to show that the set $\iota \{f \in \Ltf \mid |f|\leq g\}$ is closed in the product space $\mathcal{P}$.
\end{remark}

We finish this section by discussing linear functionals on $L^0(m)$ and its subspaces. A {\em linear functional on $L^0(m)$} is a linear real-valued function $\ell:D(\ell) \to \IR$ whose {\em domain} $D(\ell) \subseteq L^0(m)$ is a subspace.  If, additionally, $D(\ell)$ is a sublattice  of $L^0(m)$ (i.e. $f,g \in D(\ell)$ implies $f \wedge g, f\vee g \in D(\ell)$) and for all nonnegative $\psi \in D(\ell)$ the inequality $\ell(\psi) \geq 0$ holds, it is called {\em positive functional}.  A linear functional $\ell$  is  {\em regular} if there exist positive linear functionals $\ell_1,\ell_2:D(\ell)\to \IR$ such that $\ell = \ell_1 - \ell_2$. In particular, the domain of a regular functional is a sublattice of $L^0(m)$. 

For a subspace $V \subseteq L^0(m)$, we let $L(V)$ be the vector space of linear functionals with domain $V$. If, additionally, $V$ carries a vector space topology $\tau$, then $L_c(V) = L^\tau_c(V) \subseteq L(V)$ denotes  the subspace of $\tau$-continuous linear functionals. When $V$ is a sublattice of $L^0(m)$, we let $L_+(V)$ denote the space of positive functionals with domain $V$ and $L_r(V)$ denote the space of regular functionals with domain $V$. In this case, $L_+(V)$ is a positive cone in $L(V)$ and we order this space with respect to this cone, i.e., we say that $\ell_1,\ell_2 \in L(V)$ satisfy $\ell_1 \geq \ell_2$ if $\ell_1 - \ell_2 \in L_+(V)$. With respect to this order we understand the following lemma, which is a special case of \cite[Theorem~1.3.2]{M-N}. 
 
\begin{lemma}
  Let $V\subseteq L^0(m)$ be a sublattice and let $\ell \in L_r(V)$. Then 
  $$\ell_+:= \sup\{\ell,0\} \text{ and } \ell_- := \sup\{-\ell,0\}$$
  exist, are positive and satisfy $\ell = \ell_+ - \ell_-$. 
\end{lemma}

 For a regular functional $\ell$ we let $|\ell|:= \ell_+ + \ell_-$, where $\ell_+$ and $\ell_-$ are taken in the space $L_r(D(\ell))$.

\section{Quadratic forms}  \label{section:quadratic forms}

In the following section we provide the framework for later chapters, namely we introduce positive quadratic forms on topological vector spaces and study their fundamental properties. First, we treat  basic properties of quadratic forms on general vector spaces. More precisely, we give a criterion on how to check whether a given functional is a quadratic form and discuss two versions of the Banach-Saks theorem.  The second subsection is devoted to studying closedness of quadratic forms, a property that makes them accessible to functional analytic investigations. We discuss how closedness can be characterized in terms of lower semicontinuity and in terms of completeness of the form domain. This also leads to characterizations of the existence of closed extensions. In the last subsection we prove several technical lemmas, which provide convergence statements for nets in the form domain.   In view of later applications, we only treat the case when the underlying vector space is real. However, all the results of this section are  true for complex vector spaces as well.

\subsection{Basics} \label{section:basics on quadratic forms}

This subsection is devoted to the very basics properties of quadratic forms on real vector spaces. In what follows we always use the conventions   $x + \infty = \infty$ for all $x \in [0,\infty]$, $0 \cdot \infty = 0$ and $x \cdot \infty = \infty$ for all $x \in  (0,\infty]$.

\begin{definition}[Quadratic form] 
 Let $V$ be a vector space over $\IR$. A (positive) {\em quadratic form} is a functional $q:V \to [0,\infty]$ that satisfies the following conditions.
 \begin{itemize}
  \item $2q(u) + 2q(v) = q(u-v) + q(u+v)$ for all $u,v \in V$. \hfill(parallelogram identity)
  \item $q(\lambda u) = |\lambda|^2 q (u)$ for all $u \in V$ and all $\lambda \in \IR$. \hfill (homogeneity)
 \end{itemize}
The set  $$D(q) := \{u\in V \mid q(u) < \infty\}$$ is called the {\em domain of $q$} and the set $$\ker q := \{u\in V\mid q(u) = 0\}$$ is the {\em kernel of $q$}.
\end{definition}

For checking whether some functional is a quadratic form or not, it suffices to verify somewhat weaker conditions than the parallelogram identity and homogeneity. This is discussed in the following lemma. 

\begin{lemma} \label{lemma:inequality characterization of quadratic forms}
 Let $V$ be a vector space and let $q: V \to [0,\infty]$. The following assertions are equivalent.
 \begin{itemize}
  \item[(i)] $q$ is a quadratic form on $V$. 
  \item[(ii)] For all $u,v \in D(q)$ and all $\lambda \in \IR$ the inequalities
  $$q(\lambda u) \leq |\lambda|^2 q(u) \text{ and } q(u + v) + q(u-v) \leq 2q(u) + 2q(v)$$
  hold. 
  \item[(iii)] For all $u,v \in D(q)$ and all $\lambda \in \IR$ the inequalities
  $$q(\lambda u) \leq |\lambda|^2 q(u) \text{ and } 2q(u ) + 2q( v) \leq q(u+v) + q(u-v)$$
  hold. 
 \end{itemize}
\end{lemma}
\begin{proof}
 (ii) $\Rightarrow$ (i): Assertion (ii) implies that $D(q)$ is a vector space. Hence, it suffices to prove homogeneity and the parallelogram identity for elements in $D(q)$. 
 
 For $0\neq \lambda \in \IR$ and $u \in D(q)$ assertion (ii) yields
 $$q(u) = q(\lambda^{-1}\lambda u) \leq |\lambda|^{-2} q(\lambda u) \leq q(u)$$
and
$$0 \,  q(u) = 0 \leq q(0\, u ) \leq 0\, q(u).$$
 Therefore, $q$ is homogeneous on $D(q)$. Since $D(q)$ is a vector space, assertion (ii) and the homogeneity of $q$ imply for $u,v \in D(q)$ the inequality
 $$4q(u) + 4q(v) = q(2u) + q(2v) = q((u + v) + (u-v)) + q((u+v) - (u-v)) \leq 2 q(u+v) + 2 q(u-v).$$
 Hence, the parallelogram identity holds for elements in $D(q)$. 
 
 The implication (iii) $\Rightarrow$ (i) can be proven along the same lines.  This finishes the proof. 
\end{proof}

It is an immediate consequence of the definition that the domain of a quadratic form and its kernel are vector spaces. By a theorem of Jordan and von Neumann \cite[Theorem~I]{JvN} any quadratic form induces a bilinear form on its domain via polarization, i.e., the mapping
$$q:D(q) \times D(q) \to \IR, \quad (u,v) \mapsto q(u,v) := \frac{1}{4}\left(q(u+v) - q(u-v)\right)$$
is bilinear.  We abuse notation and use the same letter for a quadratic form and the induced bilinear form on its domain. In this sense, we have $q(u) = q(u,u)$ for all $u \in D(q)$. The bilinearity of $q$ implies that the mapping 
$$\|\cdot\|_q:D(q) \to [0,\infty), \quad u \mapsto \sqrt{q(u)}$$
is a seminorm that we call the {\em form seminorm} associated with $q$.

The fact that form seminorms come from degenerate inner products implies that bounded subsets with respect to such norms are almost weakly compact, i.e., their completion is weakly compact. Together with Theorem~\ref{theorem:weak closure of convex sets} this is the reason behind the following well-known lemma, see e.g. \cite[Theorem~A.4.1]{CF}. 
\begin{lemma}[Banach-Saks theorem for sequences]\label{lemma:banach saks sequences}
 Let $q$ be a quadratic form on $V$ and let $(u_n)$ be a $\|\cdot\|_q$-bounded sequence in $D(q)$. It possesses a subsequence $(u_{n_k})$ whose sequence of C\'esaro means 
 $$\left(\frac{1}{N} \sum_{k = 1}^N u_{n_k}\right)_{N \geq 1}$$
 is $\|\cdot\|_q$-Cauchy. 
\end{lemma}
 
We say that a form $\tilde{q}$ is an {\em extension of $q$} if its domain $D(\tilde{q})$ satisfies $D(q) \subseteq D(\tilde{q})$ and the equation $q(u) = \tilde{q}(u)$ holds for all $u \in D(q)$. For any subspace $D\subseteq V$, we let $q|_D$, the {\em restriction of $q$ to $D$}, be given by
$$q|_D(u):= \begin{cases}
             q(u) &\text{ if } u\in D\\
             \infty &\text{ else}
            \end{cases}.
$$
A form $\tilde{q}$ is then called a {\em restriction} of $q$ if $\tilde{q} = q|_{D(\tilde{q})}.$ 

There is a natural partial order on the set of quadratic forms on a given vector space $V$. Namely, we say that  $q \leq q'$ if  for all $u \in V$ the inequality $q'(u) \leq q(u)$ is satisfied. Note that this order is opposite from what one would expect when comparing form values; we rather compare the size of form domains than their values. Indeed, the following characterization holds.
$$q\leq q' \text{ if and only if } D(q) \subseteq D(q') \text{ and } q'(u) \leq q(u) \text{ for each } u \in D(q). $$

\subsection{Closed forms} \label{sect:closed forms}

In this section we study properties of quadratic forms that make them accessible to functional analytic investigations. Given a quadratic form on a topological vector space, we first introduce a natural vector space topology on its domain and then show that its completeness is related to lower semicontinuity of the form. 

\begin{definition}[Form topology]
 Let $(V,\tau)$ be a topological vector space and let $q$ be a quadratic form on $V$. The {\em form topology} $\tau_q$ is is the smallest topology on $D(q)$ containing both the subspace topology of $\tau$ and the topology generated by the seminorm $\|\cdot\|_q$.
\end{definition}

Obviously, the form topology is a vector space topology and a net converges with respect to $\tau_q$ if and only if it converges with respect to both $\tau$ and $\|\cdot\|_q$. 

\begin{definition}(Closedness and closability) 
 Let $(V,\tau)$ be a topological vector space.  A a quadratic form $q$ on $V$ is called {\em closed} (with respect to $\tau$) if the topological vector space $(D(q),\tau_q)$ is complete. It is called {\em closable} if it possesses a closed extension. 
\end{definition}

We start by proving that closedness is preserved under restrictions and that closed forms on a topological vector space form a cone.

\begin{lemma}[Form restriction] \label{lemma:form restriction}
 Let $(V,\tau)$ be a topological vector space and let $q$ be a closed form on $V$. Let $(\ow{V},\ow{\tau})$ be a complete topological vector space and $\iota:\ow{V} \to V$ a continuous linear map. Then  $q\circ \iota$ is a closed quadratic form on $\ow{V}$.  
\end{lemma}
\begin{proof}
We set $\ow{q} := q\circ \iota$. Obviously, $\ow{q}$ is a quadratic form. We show that $(D(\ow{q}),\ow{\tau}_{\ow{q}})$ is complete. To this end, let $(u_i)$ be a  $\ow{\tau}_{\ow{q}}$-Cauchy net in $D(\ow{q})$. The completeness of $(\ow{V},\ow{\tau})$ implies that it has a $\ow{\tau}$-limit $u$.  By the continuity and the linearity of $\iota$, and the definition of $\ow{q}$, the net $\iota(u_i)$ is $\tau_q$-Cauchy. Since $q$ is closed it has a $\tau_q$-limit $v \in D(q)$. From the continuity of $\iota$ we infer $\iota (u) = v$. This shows $u_i \to u$ with respect to $\ow{\tau}_{\ow{q}}$.
\end{proof}
\begin{remark}
 The completeness assumption on $(\ow{V},\ow{\tau})$  in the previous lemma is necessary.  For example, consider $V = \IR$ with the Euclidian topology and $\ow{V} = \mathbb{Q}$ with the induced subspace topology. The quadratic form $q:\IR \to [0,\infty)$, $x \mapsto x^2$ is closed as its form topology is the Euclidean topology. However, its restriction to $\mathbb{Q}$ is not closed since $\mathbb{Q}$ with the Euclidean topology is not complete. 
\end{remark}
\begin{lemma}[Form sum]
 Let $(V,\tau)$ be a topological vector space and let $q$ and $\ow{q}$ be closed quadratic forms on $V$. 
 $$q+\ow{q}:V \to[0,\infty], \quad u \mapsto q(u) + \ow{q}(u)$$
 is a closed quadratic form. Its domain satisfies $D(q + \ow{q}) = D(q) \cap D(\ow{q})$. 
\end{lemma}
\begin{proof}
 This is straightforward.
\end{proof}
\begin{remark}
In order to prove that the form sum of $q$ and $\ow{q}$ is closed, it actually suffices to assume that $q$ is closed on $(V,\tau)$ and that $\ow{q}$ is closed on $(D(q),\tau_q)$.
\end{remark}

We now turn to proving characterizations of closedness and closability of forms in terms of their continuity properties. The first continuity property that we discuss is lower semicontinuity.
\begin{definition}[Lower semicontinuity]
Let $(V,\tau)$ be a topological vector space. A function $f:V \to \IR \cup \{\infty\} \cup \{ -\infty\}$  is called {\em lower semicontinuous}  (with respect to $\tau$) if for all nets $(u_i)$ in $V$ and all $u\in V$ the  convergence $u_i \to u$ with respect to $\tau$  implies 
 $$f(u) \leq \liminf_{i} f(u_i).$$
\end{definition}
Lower semicontinuity of  a function $f$ can be characterized in terms of closedness of its {\em epigraph}
  $${\rm epi} f := \{(u,\lambda) \in V \times \IR \mid f(u) \leq \lambda\}.$$
\begin{lemma}\label{lemma:lower semicontinuity via epigraph}
 Let $(V,\tau)$ be a topological vector space. A function $f:V \to \IR \cup \{\infty\} \cup \{ -\infty\}$ is lower semicontinuous if and only if ${\rm epi} f$ is closed in $V \times \IR$. In particular, if $(V,\tau)$ is metrizable, then $f$ is lower semicontinuous if and only if for all sequences $(u_n)$ in $V$ and all $u \in V$ the convergence $u_n \to u$ with respect to $\tau$ implies
 $$f(u) \leq \liminf_{n\to \infty} f(u_n).$$
\end{lemma}
\begin{proof}
 The statements follow immediately from the definitions and the fact that in metrizable spaces closed sets can be characterized by sequences. 
\end{proof}

\begin{proposition}[Lower semicontinuous forms are closed]\label{prop:lower semincontinuity implies closedness}
 Let $(V,\tau)$ be a complete topological vector space. A lower semicontinuous quadratic form on $V$  is closed.
\end{proposition}
\begin{proof}
 Let $q$ be a lower semicontinuous quadratic form on $V$ and let $(u_i)$ be a Cauchy net in $(D(q),\tau_q)$. We need to show that it has a limit with respect to $\tau_q$. 
 
The completeness of $(V,\tau)$ implies that $(u_i)$ has a $\tau$-limit $u \in V$. From the lower semicontinuity of $q$ we infer
$$q(u - u_i) \leq \liminf_{j} q(u_j - u_i). $$
Since $(u_i)$ is Cauchy with respect to $q$, this inequality implies $u \in D(q)$ and $u_i \to u$ with respect to $q$. We obtain $u_i \to u$ with respect to  $\tau_q$. 
\end{proof}

\begin{remark}
  The completeness assumption on $(V,\tau)$ in the previous proposition is necessary. For example, consider the the rational numbers $\mathbb{Q}$ equipped with the standard Euclidian topology and the quadratic form $q:\mathbb{Q} \to [0,\infty)$, $x \mapsto x^2$.  Obviously, $q$ is continuous but not closed. The same is true for any non-complete pre-Hilbert space and the quadratic form coming from its inner product. 
\end{remark}

For the converse of the previous proposition, we do not need completeness of the underlying space. 

\begin{theorem}[Closed forms are lower semicontinuous] \label{theorem:characterization closedness}
 Let $(V,\tau)$ be a topological vector space. A closed quadratic form on $V$ is lower semicontinuous.
\end{theorem}
\begin{proof}
 Let $q$ be a closed quadratic form on $V$. Let $(u_i)_{i \in I}$ be a net in $V$ that $\tau$-converges to $u$. Without loss of generality we assume that $(q(u_i))$ is bounded and
 $$\lim_i q(u_i) = \liminf_i q(u_i) < \infty.$$
 The idea of the proof is as follows. We use the boundedness of $(u_i)$ with respect to $q$ to construct a $q$-weakly convergent subnet. We then use certain finite convex combinations of this subnet to make it $q$-strongly convergent and identify its $q$-limit as $u$. The fact that $q^{1/2}$ is convex then yields the claim. Since $(V,\tau)$ need not be locally convex, taking convex combinations of $(u_i)$ need not yield a $\tau$-convergent net. This is why some care is necessary when choosing the convex combinations. 
 
 Let $\mathcal{O}$ denote the set of all balanced zero neighborhoods of $(V,\tau)$. It is a neighborhood base of zero for the topology $\tau$.
 
 {\em Claim 1: } For each $i \in I$ and $U \in \mathcal{O}$, there exists a function $\varphi(i,U,\cdot):\IN \to I$ with the following properties. 
 \begin{itemize}
  \item[(a)] $i \prec \varphi(i,U,k)$ for  all $k\in \IN$.
  \item[(b)] If $F \subseteq \IN$ finite, then 
  $$\sum_{k \in F} (u_{\varphi(i,U,k)}-u) \in U.$$
  \item[(c)] The sequence of C\'esaro means
  $$ \left( \frac{1}{N}\sum_{k = 1}^N u_{\varphi(i,U,k)} \right)_{N \geq 1}$$
  is $\|\cdot\|_q$~-~Cauchy. 
 \end{itemize}
 
 {\em Proof of Claim~1.} Since $(V,\tau)$ is a topological vector space and $u_i \to u$ with respect to $\tau$, we can inductively define a function $\ow{\varphi}(i,U,\cdot):\IN \to I$ that satisfies properties (a) and (b). By our assumptions the sequence $(u_{\ow{\varphi}(i,U,k)})_k$ is $\|\cdot\|_q$~-bounded. Hence, we can apply the Banach-Saks theorem, Lemma~\ref{lemma:banach saks sequences}, to obtain a subsequence  whose C\'esaro means are $\|\cdot\|_q$~-Cauchy. This subsequence still satisfies (a) and (b) and Claim~1 follows. \qedc
 
 Using the axiom of choice, for each $i \in I$ and $U \in \mathcal{O}$ we choose a function $\varphi(i,U,\cdot)$ as in Claim~1 and define
  $$v_{(i,U,N)} := \frac{1}{N}\sum_{k = 1}^N u_{\varphi(i,U,k)}.$$
  We introduce a preorder on the set $J := I \times \mathcal{O} \times \IN$ by letting $(i,U,M) \prec (j,W,N)$ if and only if $i \prec j$,  $W \subseteq U$ and $M \leq N$. Clearly, $(J,\prec)$ is upwards directed. 
  
  {\em Claim~2:} The net $(v_j)_{j \in J}$ converges to $u$ with respect to $\tau$. 
  
  {\em Proof of Claim~2.} Let $U\in \mathcal{O}$ be given. For $W \subseteq U$ we use property (b) of $\varphi$ and that $W$ is balanced to obtain
  $$v_{(i,W,N)} - u = \frac{1}{N} \sum_{k = 1}^N (u_{\varphi(i,W,k)}-u) \in \frac{1}{N}W \subseteq W \subseteq U.  $$
  Since $\mathcal{O}$ is a neighborhood base around zero, this proves Claim~2. \qedc
  
  {\em Claim~3:} The inequality
  $$\lim_{j\in J } q(v_j) \leq \lim_{i\in I} q(u_i). $$
  holds. 
  
  {\em Proof of Claim~3.} 
  According to property (c) of $\varphi$, the limit $\lim_j q(v_j)$ exists. Now, let $\varepsilon > 0$ and choose $i_0 \in I$ such that for each $i' \succ i_0$ we have 
  $$q (u_{i'})^{1/2} \leq \lim_i q(u_i)^{1/2} + \varepsilon.$$
 Property (a) of $\varphi$ states that $i' \prec \varphi(i',U,k)$. Therefore,  $i' \succ i_0$ implies
  $$q (v_{(i',U,N)})^{1/2} \leq \frac{1}{N}\sum_{k = 1}^N q(u_{\varphi(i',U,k)})^{1/2} \leq \lim_i q(u_i)^{1/2} + \varepsilon, $$
  which  shows Claim~3. \qedc
  
 Combining Claim~1 and Claim~2 we obtain that $(v_j)$ is $\tau_q$-Cauchy. Thus, the closedness of $q$ yields that $(v_j)$ has a $\tau_q$-limit $v \in D(q)$. The uniqueness of $\tau$-limits and Claim~2 then implies $u = v$. We use this observation and Claim~3 to obtain
  $$q(u) = \lim_{j\in J} q(v_j) \leq \lim_i q(u_i).$$
  This finishes the proof of the theorem.
\end{proof}

For locally convex spaces, the previous two observations can be somewhat strengthened by equipping $V$ with the weak topology. This is discussed next.

\begin{theorem}
\label{theorem:characterization closedness lc}
 Let $(V,\tau)$ be a complete locally convex topological vector space and let $q$ be a quadratic form on $V$. The following assertions are equivalent.
 \begin{itemize}
  \item[(i)] $q$ is closed. 
  \item[(ii)] $q$ is lower semicontinuous with respect to $\tau$. 
  \item[(iii)] $q$ is lower semicontinuous with respect to $\sigma(V,V')$. 
 \end{itemize}
\end{theorem}
\begin{proof}
(ii) $\Leftrightarrow$ (iii): By Lemma~\ref{lemma:lower semicontinuity via epigraph} it suffices to show the following.

{\em Claim:} ${\rm epi}\|\cdot\|_q$ is closed in $V\times \IR$ when $V$ is equipped with $\tau$ if and only if  ${\rm epi}\|\cdot\|_q$ is closed in $V\times \IR$ when $V$ is equipped with $\sigma(V,V').$ 

{\em Proof of the claim:} This is a consequence of the following observations.
\begin{itemize}
 \item The convexity of $\|\cdot\|_q$ implies that its epigraph ${\rm epi}\|\cdot\|_q$ is a convex subset of $V \times \IR$.
 \item The space $V \times \IR$ equipped with the product topology of $\tau$ and the standard Euclidean topology on $\IR$ is a locally convex vector space. 
 \item The topology $\sigma(V\times \IR,(V\times \IR)')$ on $V \times \IR$ is the product topology of $\sigma(V,V')$ and the standard metric topology on $\IR$. 
 \item For convex sets in locally convex vector spaces, the weak closure and the closure with respect to the original topology coincide, see Theorem~\ref{theorem:weak closure of convex sets}.\qedc
\end{itemize} 

(i) $\Leftrightarrow$ (ii): This was proven in Proposition~\ref{prop:lower semincontinuity implies closedness} and Theorem~\ref{theorem:characterization closedness}.
\end{proof}

\begin{remark}
\begin{itemize}
\item The observation of Theorem~\ref{theorem:characterization closedness} that closed forms on complete topological vector spaces are lower semicontinuous seems to be new in this generality. In the literature we could only find proofs for the case when the underlying space is a Hilbert space, cf. Proposition~\ref{proposition: closedness in L2}. Some proofs for the Hilbert space case can be adapted to the situation when the underlying space is locally convex. In this sense, the novelty of our result lies in removing this assumption. In later chapters we apply the theory to the topological vector space $(L^0(m),\tau(m))$, which need not be locally convex, cf. Section~\ref{section:lebesgue spaces}.

 \item The equivalence of (ii) and (iii) in the previous theorem is well known, see e.g. \cite[Proposition~2.10]{BP}. Indeed, it holds true  for all convex functions. 
 
 \item We have shown that closedness and lower semicontinuity of the form norm coincide. The only tool we used in the course of the proof was the Banach-Saks theorem. Hence, we could treat any convex function,  for which some form of the Banach-Saks theorem holds, with the same methods. For example, if $(V,\tau)$ is a topological vector space and $\|\cdot\|$ is a norm on some subspace $W \subseteq V$ such that the completion of $(W,\|\cdot\|)$ is reflexive, then our theorem holds for the function
 $$f:V \to [0,\infty],\, u \mapsto \begin{cases}
                                    \|u\| &\text{ if } u \in W\\
                                    \infty &\text{ else}
                                   \end{cases}.
 $$
\end{itemize}

\end{remark}

Naturally, the space $(D(q)/\ker q,q)$ is an inner product space and one may wonder whether closedness is related to its completeness. This is indeed true under one additional assumption.  

\begin{theorem} \label{theorem:characterization closedness hilbert space}
 Let $(V,\tau)$ be a complete topological vector space and let $q$ be a quadratic form on $V$. Assume that $(D(q),\tau_q)$ is metrizable. Then each two of the following assertions imply the third.
 \begin{itemize}
  \item[(i)] $q$ is closed.
  \item[(ii)] The kernel $\ker q$ is $\tau$-closed and the canonical embedding 
  $$(D(q)/\ker q, q) \to (V/\ker q, \tau/\ker q), \quad u + \ker q \mapsto u + \ker q$$ 
  is continuous.
  \item[(iii)] $(D(q)/\ker q,q)$ is a Hilbert space. 
 \end{itemize}
\end{theorem}

Before proving the theorem we need the following characterization of the quotient topology $\tau_q /\ker q$. 

\begin{lemma}
 The topology $\tau_q /\ker q$ on $D(q)/ \ker q$ is the smallest topology containing both $\tau / \ker q$ and the topology generated by the form norm $\|\cdot\|_q$. 
\end{lemma}
\begin{proof} Let $\ow{\tau}$ denote the smallest topology on $D(q) / \ker q$ containing both $\tau / \ker q$ and the topology generated by the form norm $\|\cdot\|_q$. Furthermore, let 
$$\pi: D(q) \to D(q)/ \ker q, \quad u \mapsto u + \ker q$$
be the canonical projection. For $\varepsilon > 0$ and $v \in D(q)$, we set 
$$B_\varepsilon(v) := \{u \in D(q) \mid \|u-v\|_q < \varepsilon\} \text{ and } \ow{B}_\varepsilon(\pi(v)) := \{w \in D(q) / \ker q\mid \|w - \pi(v)\|_q < \varepsilon\}.$$   
 Lemma~\ref{lemma:quotient topology} implies that the collection
$$\{\pi(U)  \cap \ow{B}_\varepsilon(\pi(v)) \mid U \in \tau,\, \varepsilon > 0,\, v \in D(q)\}$$
is a basis for the topology $\ow{\tau}$ and that 
$$\{\pi(U \cap B_\varepsilon(v)) \mid U \in \tau,\,  \varepsilon > 0,\,  v \in D(q)\}$$
is a basis for the topology $\tau_q / \ker q$. Since $u \in B_\varepsilon(v)$ if and only if $\pi(u) \in \ow{B}_\varepsilon(\pi(v))$, we obtain
$$\pi(U \cap B_\varepsilon(v)) = \pi(U)  \cap \ow{B}_\varepsilon(\pi(v)).$$
This finishes the proof of the lemma. 
\end{proof}

\begin{proof}[Proof of Theorem~\ref{theorem:characterization closedness hilbert space}]
 (i) \& (ii) $\Rightarrow$ (iii): Let $(u_n)$ be a Cauchy sequence in $(D(q)/\ker q,q)$. By assertion (ii) it is also $\tau/\ker q$~-~Cauchy. Therefore, the characterization of $\tau_q/\ker q$ of the previous lemma shows that $(u_n)$ is Cauchy  with respect to $\tau_q/\ker q$. 
 
 Since $\ker q$ is $\tau_q$-closed and $(D(q),\tau_q)$ is complete, Lemma~\ref{lemma:completeness of quotients} implies the completeness of $(D(q)/\ker q,\tau_q/\ker q)$.   Hence, $(u_n)$ converges to some limit $u \in D(q)/\ker q$ with respect to $\tau_q/\ker q$.  By the characterization of $\tau_q/\ker q$ of the previous lemma, we obtain $u_n \to u$ with respect to the form norm. 
 
 (i) \& (iii) $\Rightarrow$ (ii):  The lower semicontinuity of $q$ implies that $\ker q$ is $\tau$-closed. It remains to show the continuity of the embedding. Since $\ker q$ is a $\tau_q$-closed subspace of $D(q)$, (i) and Lemma~\ref{lemma:completeness of quotients} imply that $(D(q)/\ker q,\tau_q/\ker q)$ is an $F$-space. Furthermore, (iii) yields the completeness of $(D(q)/\ker q,q)$. With this at hand, (ii) is a consequence of the open mapping theorem, Theorem~\ref{theorem:open mapping and closed graph theorem}, applied to 
 $$(D(q)/\ker q,\tau_q/\ker q) \to (D(q)/\ker q,q), \quad u + \ker q \mapsto u + \ker q$$
 and the characterization of $\tau_q / \ker q$ of the previous lemma. 
 
 (ii) \& (iii) $\Rightarrow$ (i): Let $(u_n)$ be $\tau_q$~-~Cauchy. Assertion (iii) implies that $(u_n)$ has a $\|\cdot\|_q$~-~limit $u\in D(q).$ By (ii) we obtain $u_n + \ker q \to u + \ker q$ with respect to $\tau/ \ker q$. Furthermore, the completeness of $(V,\tau)$ implies that $(u_n)$ has a $\tau$-limit $v$. The definition of $\tau/\ker q$ yields $u_n + \ker q \to v + \ker q$ with respect to $\tau/\ker q$. Since $\ker q$ is $\tau$-closed, the topology $\tau / \ker q$ is Hausdorff and we obtain $u + \ker q = v + \ker q$. This finishes the proof.  
 \end{proof}

\begin{remark}
\begin{itemize}
 \item The previous theorem seems to be new. However, we note that the implication (i) \& (ii) $\Rightarrow$ (iii) is implicitly used in the literature to prove that extended Dirichlet spaces are Hilbert spaces, see Theorem~\ref{theorem:continuous embedding of energy forms} and Remark~\ref{remark:continuous embedding of energy forms}.
 
 \item We  used the metrizability condition in the previous theorem to guarantee that the open mapping theorem holds. One can weaken this assumption to asking that the open mapping theorem is valid for mappings from $(D(q)/\ker q,\tau_q/\ker q)$ to a Hilbert space. 
 \item The form topology is metrizable if $(V, \tau)$ is metrizable but this condition is by no means necessary. If $\ker q = \{0\}$ and (ii) holds, then $\tau_q$ is automatically metrizable as in this situation the form norm $\|\cdot\|_q$ induces the form topology. Further examples where $\tau_q$ is automatically metrizable are provided by resistance forms, see Section~\ref{section:the definition and main examples} or, more generally, by irreducible energy forms, see Corollary~\ref{corollary:irreducible forms are metrizable}.
\end{itemize}
\end{remark}

Based on the previous observations we can also give characterizations for closability of quadratic forms.

\begin{proposition}
  Let $(V,\tau)$ be a complete topological vector space and let $q$ be a quadratic form on $V$.  The following assertions are equivalent.
 \begin{itemize}\label{prop:characterization closability}
  \item[(i)] $q$ is closable. 
  \item[(ii)] Every $q$-Cauchy net $(u_i)$ in $D(q)$ that $\tau$-converges to $0$ satisfies 
  $$\lim_{i} q(u_i) = 0.$$
  \item[(iii)] $q$ is lower semicontinuous on $(D(q),\tau)$, i.e., for all nets $(u_i)$ in $D(q)$ and all $u \in D(q)$  the convergence $u_i \to u$ with respect to $\tau$ implies 
  $$q(u) \leq \liminf_{i} q(u_i).$$
 \end{itemize}
If one of the above conditions is satisfied, the form  $\bar{q}$ given by
$$\bar{q}:V \to [0,\infty],\, u \mapsto \bar{q}(u) := \begin{cases}
                \lim\limits_{i} q(u_i) &\text{if }(u_i) \text{ in } D(q) \text{ is } q\text{-Cauchy and satisfies } u_i \to u\\
                \infty &\text{else}
               \end{cases}
$$
is a closed extension of $q$. Furthermore, $\bar{q}$ is the smallest closed extension of $q$, i.e., any closed extension of $q$ is an extension of $\bar{q}$.  
\end{proposition}
\begin{proof}
 (i) $\Rightarrow$ (iii): Let $\tilde{q}$ be a closed extension of $q$ and let  $(u_i)$ be a net in $D(q)$ that $\tau$-converges to $u\in D(q)$. Since $\tilde{q}$ is closed, Theorem~\ref{theorem:characterization closedness} yields
$$q(u) = \tilde{q}(u) \leq \liminf_{i} \tilde{q}(u_i) = \liminf_{i} q(u_i).$$ 
This shows (iii).

(iii) $\Rightarrow$ (ii): Let  $(u_i)$ be a $q$-Cauchy net that $\tau$-converges to $0$. From (iii) we infer
$$q(u_i) \leq \liminf_{j} q(u_i - u_j).$$
Since $(u_i)$ was assumed to be $q$-Cauchy, this implies (ii).

 (ii) $\Rightarrow$ (i): It suffices to show that $\bar{q}$ is a closed quadratic form. By assertion (ii) the functional $\bar{q}$ is well defined and coincides with $q$ on $D(q)$. It follows immediately from its definition that $\bar{q}$ is a quadratic form.  It remains to show its closedness. 
 
 To this end, let $(u_i)$ be a $\tau_{\bar{q}}$-Cauchy net in $D(\bar{q})$. Since the space $(V,\tau)$ is complete, it $\tau$-converges to some limit $u \in V$. We have to prove $u \in D(\bar{q})$ and that $(u_i)$ $\bar{q}$-converges to $u$. The definition of $\bar{q}$ implies that for each $\varepsilon > 0$, each zero neighborhood $U\in \tau$ and each $i$, there exists an element $u_{(i,\varepsilon,U)} \in D(q)$ such that  
$$\|u_i - u_{(i,\varepsilon,U)}\|_{\bar{q}} < \varepsilon \quad \text{ and } \quad u_i - u_{(i,\varepsilon,U)} \in U. $$
We introduce a preorder on the triplets $(i,\varepsilon,U)$ by letting $(i,\varepsilon,U) \prec  (i',\varepsilon',U')$ if and only if $i \prec i'$, $\varepsilon \geq \varepsilon'$ and $U \supseteq U'$. The so constructed net $(u_{(i,\varepsilon,U)})$ is $q$-Cauchy and $\tau$-converges to $u$.  This implies $u \in D(\bar{q})$. Using the definition of $\bar{q}$, we obtain
$$\|u - u_j\|_{\bar{q}}  \leq \|u - u_{(j,\delta,V)}\|_{\bar{q}} + \|u_j - u_{(j,\delta,V)}\|_{\bar{q}} =\lim_{(i,\varepsilon,U)}  \|u_{(i,\varepsilon,U)} - u_{(j,\delta,V)}\|_{q} + \|u_j - u_{(j,\delta,V)}\|_{\bar{q}}.$$
Hence, the properties of $(u_{i,\varepsilon,U})$ yield $u_j \to u$ with respect to $\bar{q}$. This concludes the proof of the implication (ii) $\Rightarrow$ (i).

The minimality statement about $\bar{q}$ is obvious and the rest was proven along the way. 
\end{proof}

 \begin{definition}[Closure of a quadratic form] Let $(V,\tau)$ be a complete topological vector space and let $q$ be a closable quadratic form on $V$. The form $\bar{q}$ that was introduced in the previous proposition is called the {\em closure} of $q$.  
 \end{definition}

\begin{remark}
\begin{itemize}
\item  The previous proposition is well-known for $L^2$-spaces. However, in this generality, a characterization of closability seems to be new.

 \item The previous proposition shows that if $q$ is closable, the domain of its closure $\bar{q}$ is given by
$$D(\bar{q}) = \{u \in V \mid  \text{ there exists a } q\text{-Cauchy net } (u_i) \text{ in } D(q) \text{ with } u_i \to u\}.$$
%
\end{itemize}
\end{remark}

\subsection{Some technical lemmas on quadratic forms}

In this subsection we discuss several technical lemmas for quadratic forms on topological vector spaces. They provide various convergence results, which will be useful later on.

\begin{lemma} \label{lemma:characterization convergence in form topology lsc forms}
 Let $(V,\tau)$ be a topological vector space and let $q$ be a lower semicontinuous quadratic form on $V$. Let $(u_i)$ be a net in $D(q)$ and let $u \in D(q)$. The following assertions are equivalent.
 \begin{itemize}
  \item[(i)] $u_i \to u$ with respect to $\tau_q$. 
  \item[(ii)] $u_i \to u$ with respect to $\tau$ and 
  $$\limsup_i q(u_i) \leq q(u).$$
 \end{itemize}
\end{lemma}
\begin{proof}
  The implication (i) $\Rightarrow$ (ii) is clear as convergence with respect to $\|\cdot\|_q$ implies $q(u_i) \to q(u)$. 
  
  (ii) $\Rightarrow$ (i): We use that $q$ is a quadratic form and obtain
 $$\limsup_i q(u-u_i) = \limsup_i \left(2q(u) + 2q(u_i) - q(u + u_i)\right). $$
 The subadditivity of $\limsup$ and the assumption on $q(u_i)$ yields
 $$\limsup_i q(u-u_i) \leq 4q(u) - \liminf_i q(u + u_i).$$
 From the lower semicontinuity of $q$ we infer
 $$4q(u) = q(2u) \leq \liminf_i q(u + u_i).$$
 This finishes the proof.
\end{proof}

\begin{lemma} \label{lemma:existence of a weakly convergent subnet}
 Let $(V,\tau)$ be a topological vector space and let $q$ be a lower semicontinuous form on $V$. Let $(u_i)$ be a $\|\cdot\|_q$-bounded net in $D(q)$ that $\tau$-converges to $u$. Then  it $q$-weakly converges to $u$. 
\end{lemma}
\begin{proof}
 The lower semicontinuity of $q$ and the $\|\cdot\|_q$-boundedness of $(u_i)$ imply $u \in D(q)$. Hence, we can assume $u = 0$. 
 
  The boundedness of $(u_i)$ implies that for each $v \in D(q)$ we have 
 $$- \infty < \liminf_i q(u_i,v) \leq \limsup_i q(u_i,v) < \infty. $$
 Let $M \geq 0$ such that $\|u_i\|_q \leq M$ for each $i$. Since $\tau$ is a vector space topology, for  $\alpha > 0$ and $v \in D(q)$, we obtain $v-\alpha u_i \to v$ with respect to $\tau$.  The lower semicontinuity of $q$   yields
 \begin{align*}
  q(v) &\leq \liminf_i q(v - \alpha u_i)\\
  &= \liminf_i \left(q(v) - 2\alpha q(u_i,v) + \alpha^2 q(u_i) \right) \\
  &\leq \liminf_i \left(q(v) - 2\alpha q(u_i,v) + \alpha^2 M^2 \right) \\
  &=q(v) -2\alpha \limsup_i q(u_i,v) + \alpha^2 M^2.
 \end{align*}
Hence, for all $\alpha > 0$ we obtain $ 2\limsup_i q(u_i,v) \leq \alpha M^2$, which implies $\limsup_i q(u_i,v) \leq 0$.  Since $v$ was arbitrary, we also have $\limsup_i q(u_i,-v) \leq 0$ and conclude 
$$0 \leq  \liminf_i q(u_i,v) \leq \limsup_i q(u_i,v) \leq 0.$$
This finishes the proof. 
\end{proof}


An immediate consequence of the previous lemma is the following. 

\begin{corollary}
 Let $(V,\tau)$ be a topological vector space and let $q$ be a lower semicontinuous quadratic form on $V$. Let $(u_i)$ be a $\|\cdot\|_q$-bounded net that $\tau$-converges to $u$ and $q$-weakly converges to $v$. Then $u-v \in \ker q$. 
\end{corollary}

\begin{remark}
 \begin{itemize}
  \item Lemma~\ref{lemma:characterization convergence in form topology lsc forms} is certainly well-known in some special cases. It can be found in the literature when the underlying space is a Hilbert space. A related result when the underlying space is the space of all functions on a discrete set equipped with the topology of pointwise convergence is \cite[Theorem~1.9]{Schmi}.  Lemma~\ref{lemma:existence of a weakly convergent subnet} is an extension of \cite[Lemma~I.2.12]{MR} to arbitrary topological vector spaces.
  
  \item Both lemmas do not require the topological vector space to be complete nor the form to be closed.  The given proofs are new and make direct use of lower semicontinuity instead of using some form of the Banach-Saks theorem.
 \end{itemize}
\end{remark}

\begin{lemma}\label{lemma:form weak convergence implies lower limit}
 Let $V$ be a vector space and let $q$ be a quadratic form on $V$. Let $(u_i)$ be a $\|\cdot\|_q$-bounded net that $q$-weakly converges to $u \in D(q)$. Then 
 $$q(u) \leq \liminf_i q(u_i).$$
\end{lemma}
\begin{proof}
Let $(H,\bar{q})$ be the Hilbert space completion of $(D(q)/\ker q, q)$. We equip $H$ with the vector space topology of $\bar{q}$-weak convergence, which we denote by $\mathcal{W}$. Since $(H,\bar{q})$ continuously embeds into $(H,\mathcal{W})$, the form topology $\mathcal{W}_{\bar{q}}$ coincides with the Hilbert space topology of $(H,\bar{q})$. Therefore, $\bar{q}$ is a closed form on the topological vector space $(H,\mathcal{W})$. By Theorem~\ref{theorem:characterization closedness} it is lower semicontinuous. The $\|\cdot\|_q$-boundedness of $(u_i)$ and the $q$-weak convergence $u_i \to u$  imply the $\bar{q}$-weak convergence $u_i \to u$.  Hence, the statement follows from the lower semicontinuity of $\bar{q}$.
\end{proof}

 \begin{lemma}\label{lemma:weak convergence implies weak convergence}
  Let $V$ be a vector space and let $q,q'$ be two quadratic forms on $V$. Assume that for all $u \in D(q)$ the inequality $q'(u) \leq q(u)$ holds. If $(u_i)$ is a $\|\cdot\|_q$-bounded net in $D(q)$ that $q$-weakly converges to $u \in D(q)$,  then $(u_i)$ converges $q'$-weakly to $u$. 
 \end{lemma}
 \begin{proof}
  Let $(H,\bar{q})$ be the Hilbert space completion of $(D(q)/\ker q, q)$. The inequality between $q'$ and $q$ allows us to uniquely extend $q'$ to a quadratic form $\bar{q}'$ on $H$, which is continuous with respect to $\bar{q}$.  In particular, this extension $\bar{q}'$ is a lower semicontinuous quadratic form on the Hilbert space $(H,\bar{q})$. Since Hilbert spaces are locally convex,  we can apply Theorem~\ref{theorem:characterization closedness lc} to obtain that $\bar{q}'$ is lower semicontinuous with respect to $\bar{q}$-weak convergence. 
  
  The $\|\cdot\|_q$-boundedness of $(u_i)$ and the $q$-weak convergence $u_i \to u$ imply the $\bar{q}$ weak convergence $u_i \to u$. Furthermore, the inequality between $q$ and $q'$ yields that $(u_i)$ is $\|\cdot\|_{\bar{q}'}$-bounded. Now, the statement follows from Lemma~\ref{lemma:existence of a weakly convergent subnet} and the fact that $\bar{q}'$ is lower semicontinuous with respect to $\bar{q}$-weak convergence. 
 \end{proof}

 \begin{remark}
  \begin{itemize}
   \item The previous two lemmas are well-known. Lemma~\ref{lemma:form weak convergence implies lower limit} is a standard result in Hilbert space theory while an $L^2$-version of Lemma~\ref{lemma:weak convergence implies weak convergence} is shown in \cite[Lemma~3.2.2]{FOT}. We included the (short) proofs to emphasize that both can be seen as applications of Theorem~\ref{theorem:characterization closedness} and of Theorem~\ref{theorem:characterization closedness lc}, respectively. 
   \item If  $(D(q)/\ker q,q)$ is a Hilbert space, the boundedness assumption on $(u_i)$ can be dropped in the previous two lemmas. In this case, the Banach-Steinhaus theorem implies that boundedness automatically follows from weak convergence.  
  \end{itemize}
 \end{remark}

\section{Quadratic forms on Lebesgue spaces} \label{section:quadratic forms on lebesgue spaces}

For the rest of this thesis we only deal with forms  on Lebesgue spaces of measurable functions. These spaces are not only vector spaces, but have an order structure and allow algebraic manipulations. Therefore, their form theory is richer than in the abstract setting. In this section  we give a glimpse at this observation by discussing the very basics of Dirichlet form theory. To this end, we basically follow \cite{FOT} and \cite{MR} with some slight modifications. Furthermore, we introduce extended Dirichlet spaces in the light of the theory that we developed in the previous section and provide a new and short proof for their existence.  At the end of this section, we discuss how closed forms on $\Ltf$ can be approximated by continuous ones.

\subsection{Closed forms on $L^2(m)$ and associated objects}

Let $q$ be a quadratic form on the Hilbert space $L^2(m)$.  For $\alpha > 0$ and $f \in L^2(m)$ we set 
$$q_\alpha(f) := q(f) + \alpha \|f\|_2^2.$$
The associated form norm $\|\cdot\|_{q_\alpha}$ induces the form topology of $q$ on $D(q)$. Since $L^2(m)$ is a subspace of $L^0(m)$, we can extend $q$ to a quadratic form on $L^0(m)$ by letting 
$$q(f): = \infty \text{ for } f \in L^0(m)\setminus L^2(m).$$
With this convention the domain of $q$ when considered as a form on $L^2(m)$ and when considered as a form on $L^0(m)$ coincide. In this case, the characterizations of closedness of the previous section read as follows.

\begin{proposition} \label{proposition: closedness in L2}
Let $q$ be a quadratic form on $L^2(m)$. Then the following assertions are equivalent.
 \begin{itemize}
  \item[(i)]  $q$ is closed on $L^2(m)$. 
  \item[(ii)] $q$ is lower semicontinuous with respect to strong convergence on $L^2(m)$.
  \item[(iii)] $q$ is lower semicontinuous with respect to weak convergence on $L^2(m)$.
  \item[(iv)] For one/any $\alpha > 0$, the space $(D(q),q_\alpha)$ is a Hilbert space.  
  \item[(v)] For one/any $\alpha >0$, the form $q_\alpha$ is closed on $L^0(m)$. 
  \end{itemize}
  If, additionally, $m$ is localizable, then all of these assertions are equivalent to the following.
  \begin{itemize}
   \item[(vi)]  For one/any $\alpha > 0$, the form $q_\alpha$ is lower semicontinuous on $L^0(m)$. 
  \end{itemize}

\end{proposition}
\begin{proof}
The equivalence of (i), (ii), (iii) and (iv) is an immediate consequence of the discussion in Section~\ref{sect:closed forms}. Since  $D(q) = D(q_\alpha)$, the equivalence of assertions (iv) and (v) follows from the definition of closedness and the observation that $(D(q_\alpha),q_\alpha)$ is continuously embedded in $L^0(m)$. If $m$ is localizable, Theorem~\ref{theorem:properties of L0} shows that the space $L^0(m)$ is Hausdorff and complete. Hence, the equivalence of (v) and (vi) is guaranteed by Proposition~\ref{prop:lower semincontinuity implies closedness} and Theorem~\ref{theorem:characterization closedness}.
\end{proof}

\begin{remark}
 The equivalence of the assertions (i) - (iv) in the previous lemma is well known and holds even if $L^2(m)$ is replaced by an arbitrary Hilbert space. However, the last conditions (v) and (vi) stand out as they pay tribute to the fact that $L^2(m)$ is a vector space of functions that continuously embeds into $L^0(m)$. In particular, they show that closed quadratic forms on $L^2(m)$ can be interpreted  as special examples of certain closed quadratic forms on $L^0(m)$.   
\end{remark}

%
%

Every closed quadratic form on $L^2(m)$ comes with a whole zoo of associated objects. We do not assume that the domains of the forms are dense in $L^2(m)$. Therefore, we comment on some necessary modifications.

Let $q$ be a closed quadratic form on $L^2(m)$. We denote by $\overline{D(q)}$ the closure of $D(q)$ in $L^2(m)$. The domain of the {\em associated operator $L$} is given by
$$D(L) := \{f\in D(q) \mid \text{ there exists } h \in \overline{D(q)} \text{ s.t. } q(f,g) = \as{h,g} \text{ for all } g \in D(q)\},$$
on which it acts by $f \mapsto Lf := h.$ Note that $L$ can be viewed as an operator on $\overline{D(q)}$ and as an operator on $L^2(m)$, where it may not be densely defined. The important observation is that $L$ is densely defined, positive and self-adjoint on $\overline{D(q)}$. Its square root $L^{1/2}$ exists in $\overline{D(q)}$ and satisfies
$$q(f) = \begin{cases}
           \as{L^{1/2}f,L^{1/2}f} &\text{if } f \in D(L^{1/2})\\
           \infty &\text{else}
          \end{cases}.
 $$ 
In particular, the domain of $q$ is given by $D(q) = D(L^{1/2})$. 
 On $\overline{D(q)}$ the form $q$ has an {\em associated semigroup} $(T_t)_{t>0}$ and an {\em associated resolvent} $(G_\alpha)_{\alpha > 0}$, which are defined by
$$T_t := e^{-tL} \text{ and } G_\alpha := (L+\alpha)^{-1}.$$
Both objects are strongly continuous and self-adjoint. We extend them to $L^2(m)$ by setting $T_t f  :=  G_\alpha f :=  0$ for $f \in \overline{D(q)}^\perp$. These extended versions still satisfy the semigroup and resolvent equations but need not be strongly continuous. For $f \in L^2(m)$  the resolvent $G_\alpha f$ is the unique element in $D(q)$ that satisfies
$$q_\alpha(G_\alpha f,g) = \as{f,g} \text{ for all } g \in D(q).$$


\subsection{Dirichlet forms}

In this subsection we briefly recall the notion of (regular) Dirichlet forms and discuss properties of the associated objects. 

\begin{definition}[Normal contraction]\label{definition:normal contraction}
 A {\em normal contraction} is a function $C:\IR^n \to \IR$ that satisfies
 $$C(0) = 0 \text{ and } |C(x) - C(y)| \leq \sum_{i = 1}^n |x_i-y_i| \text{ for every } x,y \in \IR^n.$$ 
 We call a normal contraction $C:\IR \to \IR$ an {\em $\varepsilon$-cutoff} if $C$ is monotone increasing, maps $\IR$ to the bounded interval $[-\varepsilon,1+\varepsilon]$ and satisfies $C(x) = x$ for every $x \in [0,1]$. 
\end{definition}

\begin{definition}[Markovian form / Dirichlet form] \label{definition:Markovian form}
 A quadratic form $\E$ on $L^2(m)$ is called {\em Markovian} if for every $\varepsilon > 0$ there exists an $\varepsilon$-cutoff $C_\varepsilon$ such that for all $u \in L^2(m)$ the inequality 
 $$\E(C_\varepsilon \circ u) \leq \E(u)$$
 holds. A closed Markovian form is called {\em Dirichlet form}. If, additionally, $X$ is a separable locally compact metric space, $m$ is a Radon measure of full support and the space
 $$C_c(X) \cap D(\E)$$
 is dense in the Hilbert space $(D(\E),\E_1)$ and in $(C_c(X),\|\cdot\|_\infty)$, then $\E$ is called {\em regular}.  
\end{definition}
\begin{remark}
\begin{itemize}
\item Some authors reserve the term normal contraction exclusively for contractions in one dimension. As the distinction between one and more dimensions is somewhat artificial we ignore it and use the term as indicated above. The concept of what we have called $\varepsilon$-cutoff is widely used in Dirichlet form theory without a special name. 

 \item  Usually Dirichlet forms are required to have a dense domain in $L^2(m)$. Dirichlet forms without dense domain are sometimes called Dirichlet forms in the wide sense, see \cite[Section~1.3]{FOT}. Since all the theorems from Dirichlet form theory that we use in this thesis remain true in the non densely-defined setting, we do not make this distinction.  
\end{itemize}
\end{remark}
 Dirichlet forms can be characterized by somewhat stronger contraction properties than Markovian forms, see \cite[Theorem I.4.12]{MR}.
\begin{theorem}\label{theorem:cutoff properties Dirichlet form}
 Let $\E$ be a closed quadratic form on $L^2(m)$. Then the following assertions are equivalent.
 \begin{itemize}
  \item[(i)] $\E$ is a Dirichlet form. 
  \item[(ii)] For every normal contraction $C:\IR \to \IR$ and every $f \in L^2(m)$, the inequality
  $$\E(C\circ f) \leq \E(f)$$
  holds.
  \item[(iii)] For every normal contraction $C:\IR^n \to \IR$ and all $f_1,\ldots,f_n \in L^2(m)$, the inequality

  $$\E(C(f_1,\ldots,f_n))^{1/2} \leq \sum_{k = 1}^n \E(f_k)^{1/2}$$
  holds.
    \item[(iv)] For all $f,f_1,\ldots,f_n \in L^2(m)$, the  inequalities
$$|f(x)| \leq \sum_{k=1}^n|f_k(x)| \text{ and } |f(x)-f(y)| \leq \sum_{k = 1}^n  |f_k(x)-f_k(y)| \text{ for }  m\text{-a.e. } x,y \in X $$
 imply
$$\E(f)^{1/2} \leq \sum_{k = 1}^n \E(f_k)^{1/2}.$$
 \end{itemize}
\end{theorem}

\begin{remark} 
Philosophically, a form $\E$ being Markovian means that for each $\varepsilon$ one can find sufficiently smooth $\varepsilon$-cutoffs that decrease the value of the form. If, additionally, the form is closed, the previous theorem shows that smoothness of the cutoff can be replaced by Lipschitz continuity.
\end{remark}
  The Markov property of Dirichlet forms can be characterized in terms of properties of the associated semigroup and the associated resolvent. This is discussed next.  A linear operator $T:L^2(m) \to L^2(m)$ is {\em Markovian} if for each $f \in L^2(m)$ with $0 \leq f \leq 1$ the inequality 
$$0 \leq Tu \leq 1$$
holds. For later purposes we note that self-adjoint Markovian operators can be extended to continuous operators on $L^1(m)$. More precisely, if $T$ is Markovian, then  $f \in L^1(m) \cap L^2(m)$ implies $Tf \in L^1(m)\cap L^2(m)$ and
$$\|Tf\|_1 \leq \|f\|_1,$$
see e.g. \cite[Section~1.5]{FOT}. Since $L^1(m) \cap L^2(m)$ is dense in $L^1(m)$, this  implies that $T$ uniquely extends to a contraction $T^1$ on $L^1(m)$. Dirichlet forms can be characterized by the Markov property for the associated semigroup and the associated resolvent, see e.g. \cite[Theorem~1.4.1]{FOT}.
 
\begin{lemma}
Let $\E$ be a closed quadratic form on $L^2(m)$. The following assertions are equivalent.  
\begin{itemize}
 \item[(i)] $\E$ is a Dirichlet form.
 \item[(ii)] For each $\alpha > 0$, the operator $\alpha G_\alpha$ is Markovian.
 \item[(iii)] For each $t > 0$, the  operator $T_t$ is Markovian.
\end{itemize}
\end{lemma}

\subsection{Extended Dirichlet spaces}

In this subsection  we show that Dirichlet forms have a closed  extension to $L^0(m)$ when the underlying measure is localizable. This $L^0$-closure is nothing more but the extended Dirichlet space, a well known and important object in Dirichlet form theory. In later chapters it will serve as one of the main examples for energy forms.  

Recall that for a form $\E$ on $L^2(m)$ we use the convention $\E(f) = \infty$ if $f \in L^0(m) \setminus L^2(m)$. 

\begin{theorem}[Closability of Dirichlet forms on $L^0(m)$]\label{theorem: existence extended Dirichlet space}
Let $m$ be localizable. Every Dirichlet form on $L^2(m)$ is closable on $L^0(m)$.
\end{theorem}
\begin{proof}
 Let $\E$ be a Dirichlet form. Since $m$ is localizable, the space $L^0(m)$ is complete, see Theorem~\ref{theorem:properties of L0}. By Proposition~\ref{prop:characterization closability} it suffices to show the lower semicontinuity of $\E$ on its domain $D(\E)$. 
 
 To this end, let $(f_i)$  in $D(\E)$ and let $f \in D(\E)$ with $f_i \overset{m}{\to} f$ be given. Without loss of generality we may assume
 $$\lim_{i}\E(f_i) = \liminf_{i} \E(f_i) < \infty.$$
 {\em Claim:} The inequality $\E(f) \leq \liminf \E(f_i)$ holds under the additional assumption that there exists some constant $M >0$ such that for all $i$ the inequality $|f_i| \leq M$ holds. 
 
{\em Proof of the Claim.} Let  $\varphi \in L^1(m) \cap L^2(m)$ and let $\alpha >0$. Since $\E$ is a Dirichlet form, the associated resolvent $G_\alpha$ is a contraction on $L^1(m)$. In particular, it satisfies $G_\alpha \varphi \in L^1(m) \cap L^2(m)$. Using this observation, the form characterization of the resolvent and Lebesgue's dominated convergence theorem, Lemma~\ref{lemma:Lebesgue's theorem}, we infer
 $$\E(f_i,G_\alpha \varphi) = \as{f_i,\varphi} - \alpha \as{f_i,G_\alpha \varphi} \overset{i}{\longrightarrow} \as{f,\varphi} - \alpha \as{f,G_\alpha \varphi} = \E(f,G_\alpha \varphi).$$
 Since elements of the form $G_\alpha \varphi$ with $\varphi \in L^1(m) \cap L^2(m)$ are $\E$-dense in $D(\E)$ and the net $(\E(f_i))$ is bounded, we obtain 
 $$\lim_{i } \E(f_i,g) = \E(f,g) \text{ for each }g \in D(\E). $$
This implies the claim as in inner product spaces the corresponding norm is lower semicontinuous with respect to weak convergence, see Lemma~\ref{lemma:form weak convergence implies lower limit}. \qedc

For $M  > 0$, we let the normal contraction $C_M:\IR \to \IR$ be given by
$$C_M(x) = (x \wedge M) \vee ({-M}).$$
Obviously, the convergence $C_M \circ f\to f$ as $M \to \infty$ holds in $L^2(m)$. From the lower semicontinuity of $\E$ on $L^2(m)$ we derive
$$\E(f) \leq \liminf_{M\to \infty}\E(C_M \circ f). $$
Furthermore, for each $M>0$ we have $C_M \circ f_i \overset{m}{\to} C_M \circ f$. An application of the previously proven claim and the contraction properties of Theorem~\ref{theorem:cutoff properties Dirichlet form} (ii) yield
$$\E(C_M \circ f) \leq \liminf_i \E(C_M \circ f_i) \leq \liminf_i \E(f_i). $$
This finishes the proof.
\end{proof}

\begin{definition}[Extended Dirichlet space]
 Let $\E$ be a Dirichlet form on $L^2(m)$. Its closure on $L^0(m)$ is called its {\em extended form} and is denoted by $\Ee$. The domain $D(\Ee)$ is the {\em extended Dirichlet space} of $\E$. 
\end{definition}

\begin{remark}
 \begin{itemize}
  \item In the literature the extended Dirichlet space is introduced in a slightly different manner.  It is defined to be the space
 $$\{f \in L^0(m) \mid  \text{ there exists an } \E\text{-Cauchy sequence } (f_n) \text{ with } f_n \to f \, m\text{-a.e.}\}$$
 to which the form extends. This bears a close resemblance to the characterization of the closure of a form given in Proposition~\ref{prop:characterization closability} but with convergence in measure replaced by almost everywhere convergence.  When $m$ is $\sigma$-finite, convergence in measure can be characterized by almost everywhere convergence. In this case, a sequence $(f_n)$ satisfies $f_n \overset{m}{\to} f$ if and only if any subsequence of $(f_n)$ has a subsequence that converges to $f$ $m$-a.e., see Subsection~\ref{section:lebesgue spaces}. With this characterization at hand, it is not hard to see that in the $\sigma$-finite situation our definition coincides with the classical one. However, we believe that our definition is somewhat more natural from a functional analytic viewpoint.
 
 \item The extended Dirichlet space and the extension of the form to it were introduced in \cite{Sil} under some topological assumptions on the underlying space $X$ and the measure $m$. That it can also be defined for $\sigma$-finite $m$ without further assumptions was first realized  in \cite{Schmu1,Schmu2}. The main advantages of our approach are its generality, it works for localizable $m$, and its brevity. Combining our result and Theorem~\ref{theorem:characterization closedness}, we obtain another main result of \cite{Schmu1,Schmu2}, namely  the lower semicontinuity of extended Dirichlet forms on $L^0(m)$.
 \end{itemize}
\end{remark}

We finish this subsection by discussing the relation of the domain of $\E$ and $\Ee$.

\begin{proposition} \label{propostiont:domain extended dirichlet form}
 Let $m$ be localizable and let $\E$ be a Dirichlet form on $L^2(m)$. Then $$D(\E) = D(\Ee)\cap L^2(m).$$ 
\end{proposition}
\begin{proof}
The inclusion $D(\E) \subseteq D(\Ee)\cap L^2(m)$ is obvious; we only prove the opposite inclusion. The characterization of closures of quadratic forms in  Proposition~\ref{prop:characterization closability} yields that  for  $f \in D(\Ee)\cap L^2(m)$  there exists an $\E$-Cauchy net $(f_i)$ in $D(\E)$ with $f_i \overset{m}{\to} f$. Lebesgue's dominated convergence theorem shows that for each fixed $j$ the functions
 $$\ow{f}_{i,j} := (f_i \wedge |f_j|) \vee (-|f_j|)$$
 converge in $L^2(m)$ towards $g_j := (f \wedge |f_j|) \vee (-|f_j|)$ . Applying Lebesgue's theorem again, we obtain $g_j \to f$ in $L^2(m)$.  The $L^2$-lower semicontinuity of $\E$ yields
 $$\E(f)^{1/2}\leq  \liminf_{j} \liminf_{i} \E(\ow{f}_{i,j})^{1/2} \leq \liminf_{j} \liminf_{i} (\E(f_i)^{1/2} + 2 \E(f_j)^{1/2}) < \infty.$$
Here, we used that by  Theorem~\ref{theorem:cutoff properties Dirichlet form}~(iii) the inequalities
$$\E(f \wedge g)^{1/2} \leq \E(f)^{1/2} + \E(g)^{1/2} \text{ and } \E(f \vee g)^{1/2} \leq \E(f)^{1/2} + \E(g)^{1/2} $$
hold for all $f,g \in D(\E)$ (cf. also the proof of Theorem~\ref{theorem:algebraic and order properties}). This finishes the proof.
\end{proof}
\begin{remark}
 The previous proposition is well known when $m$ is $\sigma$-finite, see e.g. \cite[Theorem~1.1.5.]{CF}. Again, we would like to point the brevity of our proof. 
\end{remark}

\subsection{Closed forms on $\Ltf$} \label{section:form on L0}

This subsection is devoted to proving that each closed quadratic form on $\Ltf$ can be approximated by continuous quadratic forms on $\Ltf$. To this end, we first introduce the approximating forms and then show that they converge pointwise to the given form.

\begin{definition}[Approximating forms]
 Let $q$ be a quadratic form on $\Ltf$. For $\alpha > 0$ and  $U\in \Bf$,  the {\em approximating form} $q^{(\alpha,U)}:\Ltf \to [0,\infty]$ is defined by
$$f \mapsto  q^{(\alpha,U)}(f) := \inf\left\{q(g) +  \alpha \int_U (g-f)^2\, {\rm d}m \, \middle| \, g \in \Ltf \right\}.$$

\end{definition}
\begin{remark}
   If $m$ is finite, the space $\Ltf$ coincides with $L^2(m)$. In this case, it can be shown that $q^{(\alpha,X)}$ is given by  
  $$q^{(\alpha,X)}(f) = \alpha \as{(I - \alpha G_\alpha)f,f},$$
  where $G_\alpha$ is the $L^2$-resolvent of $q$. These are forms that are usually used to approximate closed forms on $L^2(m)$, cf. \cite[Section~1.3]{FOT}. 
\end{remark}
The following lemma shows that the approximating forms are indeed continuous quadratic forms on $\Ltf$.
\begin{lemma}\label{lemma: properties of approximating forms}
Let $q$ be a quadratic form on $\Ltf$. For each $\alpha > 0$ and each $U \in \Bf$, the functional $q^{(\alpha,U)}$ is a continuous quadratic form with domain  $D(q^{(\alpha,U)}) = \Ltf$. 
\end{lemma}
\begin{proof}
 Obviously, $q^{(\alpha,U)}$ is finite on the whole space $\Ltf$. We employ Lemma~\ref{lemma:inequality characterization of quadratic forms} to prove that $q^{(\alpha,U)}$ is a quadratic form, i.e., we show that for $\lambda \in \IR$ and $f,f' \in \Ltf$, the inequalities
 $$q^{(\alpha,U)}(\lambda f) \leq |\lambda|^2 q^{(\alpha,U)}(f)$$
 and 
  $$q^{(\alpha,U)} (f+f') + q^{(\alpha,U)} (f-f') \leq 2q^{(\alpha,U)} (f) + 2q^{(\alpha,U)} (f')$$
  hold.
 
 For $\lambda \in \IR$ and  $g \in \Ltf$, we have
  $$q^{(\alpha,U)}(\lambda f) \leq q(\lambda g) +  \alpha \int_U (\lambda g - \lambda f)^2\, {\rm d}m. $$
 Taking the infimum and using that $q$ is a quadratic form yields 
 $$q^{(\alpha,U)}(\lambda f) \leq |\lambda|^2 q^{(\alpha,U)}(f).$$
 Let $g,g'\in \Ltf$. The definition of $q^{(\alpha,U)}$ and that $q$ is a quadratic form imply
 \begin{align*}
 q^{(\alpha,U)} (f+f') + q^{(\alpha,U)} (f-f') &\leq q(g+g') + \alpha \int_U (g + g '- (f  +  f') )^2 {\rm d}m \\& \quad\quad\quad\quad\quad + q(g-g') + \alpha\int_U (g - g '- (f  -  f') )^2 {\rm d}m 
 \\&= 2 q(g) + 2\alpha \int_U (g - f)^2 {\rm d}m +  2 q(g') + 2\alpha \int_U (g '- f' )^2 {\rm d}m
 \end{align*}
After taking the infimum over all such $g,g'$  we arrive at
 $$q^{(\alpha,U)} (f+f') + q^{(\alpha,U)} (f-f') \leq 2q^{(\alpha,U)} (f) + 2q^{(\alpha,U)} (f').$$
 As already mentioned, these inequalities show that $q^{(\alpha,U)} $ is a quadratic form. It remains to prove its continuity. For $f,f',g \in \Ltf$ we have
 $$q(g) +  \alpha \int_U (g - (f - f') )^2\, {\rm d}m \leq  q(g) +   2 \alpha \int_U g^2\, {\rm d}m  + 2\alpha \int_U (f - f')^2\, {\rm d}m.$$
 Taking the infimum over $g$ yields
 $$q^{(\alpha,U)} (f-f') \leq 2\alpha \int_U (f - f')^2\, {\rm d}m.$$
Since $q^{(\alpha,U)} $  is a quadratic form, this inequality implies its continuity and finishes the proof.
\end{proof}

The following lemma is the main result of this section.

\begin{lemma}[Approximation of closed forms] \label{lemma: approximation of closed forms}
Let $m$ be localizable and let $q$ be a closed quadratic form on $\Ltf$. For each $f\in \Ltf$  we have
$$q(f) = \sup \{q^{(\alpha,U)} (f)\mid \alpha > 0,\,  U \in \Bf\}.$$
\end{lemma}
\begin{proof}
The definition of the approximating forms yields
$$q(f) \geq \sup \{q^{(\alpha,U)} (f)\mid \alpha > 0,\,  U \in\Bf \}.$$
We show the opposite inequality. Let $f \in \Ltf$ be given and let $\varepsilon >0$ arbitrary. For each $\alpha >0$ and $U \in \Bf$, we choose  $g_{\alpha,U}\in \Ltf$ such that 
$$q^{(\alpha,U)} (f) + \varepsilon \geq q(g_{\alpha,U}) + \alpha \int_U (f - g_{\alpha,U})^2 {\rm d}m.$$
We turn $(g_{\alpha, U})$ into a net by letting $(\alpha,U) \prec (\alpha',U')$ if and only if $\alpha \leq \alpha'$ and $U \subseteq U'$. 

{\em Case 1:} $q(f) < \infty$. For $V \in \Bf$ with  $V \supseteq U$, we have
$$\int_U (f - g_{\alpha,V})^2 {\rm d}m \leq \int_V (f - g_{\alpha,V})^2 {\rm d}m \leq \alpha^{-1} (q_{\alpha,V}(f) + \varepsilon) \leq \alpha^{-1} (q(f) + \varepsilon).$$
This shows that the net $(g_{\alpha,U})$ converges to $f$ in $\Ltf$. Since $q$ is closed on $\Ltf$, Theorem~\ref{theorem:characterization closedness} implies that $q$ is lower semicontinuous on $\Ltf$ and we infer
$$q(f) \leq \liminf_{(\alpha,U)} q(g_{\alpha,U}).$$
This inequality and the choice of the $g_{\alpha,U}$ yield
$$q(f) \leq \sup \{q^{(\alpha,U)} (f)\mid \alpha > 0,\,  U \in \Bf\} + \varepsilon.$$
{\em Case 2:} $q(f) = \infty$. If the net $(g_{\alpha,U})$ converges to $f$, we can use the lower semicontinuity of $q$ and argue as in Case 1. If $(g_{\alpha,U})$ does not converge to $f$, then there exist $U \in \Bf$, some $\delta >0$ and a subnet $(g_{\alpha_i,U_i})$ such that for all $i$ we have
$$\int_U (f - g_{\alpha_i,U_i})^2 {\rm d}m  \geq \delta \text{ and } U \subseteq U_i.$$
The definition of the  $(g_{\alpha,U})$ and the choice of the subnet imply
$$q_{\alpha_i,U_i}(f) + \varepsilon \geq \alpha_i \int_{U_i} (f - g_{\alpha_i,U_i})^2 {\rm d}m \geq \delta \alpha_i.$$
Since the mapping $i \mapsto \alpha_i$ is cofinal, the net $(\alpha_i)$ satisfies $\lim_i \alpha_i = \infty$. This finishes the proof. 
\end{proof}

\begin{remark}
  In many topological spaces all lower semicontinuous functions are pointwise limits of continuous functions, see e.g. \cite{Nag}. The remarkable fact about the previous lemma is that the continuous functions that approximate a quadratic form can be chosen to be quadratic forms as well. 
\end{remark}
  
  In later section we will apply the previous lemma to restrictions of closed forms on $L^0(m)$ to $\Ltf$. This is why we note the following.
  
\begin{lemma}\label{lemma:restriction to Ltf}
 Let $m$ be localizable and let $q$ be a closed quadratic form on $L^0(m)$. The restriction of $q$ to $\Ltf$ is a closed form on $\Ltf$. 
\end{lemma}
\begin{proof}
Since $m$ is localizable Proposition~\ref{proposition:completeness local Lp} shows that $\Ltf$ is complete. Furthermore, the natural inclusion of $\Ltf$ into $L^0(m)$ is continuous. Thus, we can apply  Lemma~\ref{lemma:form restriction} and obtain that the restriction of $q$ to $\Ltf$ is a closed form on $\Ltf$. 
\end{proof}

  \chapter{Energy forms} \label{chaper:energy forms}

In this chapter we introduce energy forms, the main objects of our studies.  Our goals are to show that energy forms are a common generalization of Dirichlet forms, extended Dirichlet forms and resistance forms and to extend important concepts and theorems from Dirichlet forms to energy forms. In this sense, the contents of this chapter could be seen as a mere copy of the contents of the first few chapters of any textbook on Dirichlet form theory that in some way overcomes the technical  difficulty that energy forms do not live on $L^2$-spaces but on $L^0$-spaces. However, the lack of structure of closed forms on $L^0$ (there is no associated semigroup, no associated resolvent and no associated self-adjoint operator), which forces the consequent use of lower semicontinuity, leads to new and short proofs for known results and to some new structural insights. 

The first section of this chapter provides the precise definition of energy forms and some important examples, among which are Dirichlet forms, extended Dirichlet forms and resistance forms, see Subsection~\ref{subsection:resistance forms}.  The second section is devoted to contraction properties; its main result is Theorem~\ref{theorem:cutoff properties energy forms}, an extension of Theorem~\ref{theorem:cutoff properties Dirichlet form} to energy forms. In the third section we demonstrate how the contraction properties of an energy form influence the algebraic properties of its form domain, the concrete shape of its kernel and its continuity properties. More specifically, we prove that the form domain is a lattice and that bounded functions in the form domain are an algebra, see Theorem~\ref{theorem:algebraic and order properties}, that the triviality of the kernel implies that the form domain equipped with the form norm is a Hilbert space, see Theorem~\ref{theorem:continuous embedding of energy forms}, and we explicitly compute the form kernel, see Theorem~\ref{thm:kernel of an energy form} and Corollary~\ref{corollary:kernel recurrent forms}. The fourth section introduces superharmonic and excessive functions and characterizes them in terms of some additional contraction properties, which are not induced by normal contractions, see Theorem~\ref{theorem:characterization of superharmonic functions} and Theorem~\ref{theorem:characterization of excessive functions}. It is followed by a section on capacities and a section on local spaces, where we basically adapt the existing notions to our situation.

In the course of this chapter, it turns out that the kernel of an energy form determines many of its properties. More precisely, the question whether it contains the constant functions, i.e., whether the energy form is recurrent, is of importance.  This is why we give several characterizations of recurrence in the mentioned sections, see Corollary~\ref{corollary:kernel recurrent forms}, Theorem~\ref{theorem:recurrence in terms of constant excessive functions} and Theorem~\ref{theorem:capacities for recurrent forms}.

How our results for energy forms compare with the existing results for Dirichlet forms and resistance forms is discussed individually after each theorem.

\section{The definition and main examples} \label{section:the definition and main examples}

In this section we introduce energy forms and give some important examples. Moreover, we discuss how energy forms can be seen as a common generalization of Dirichlet forms, extended Dirichlet forms and resistance forms. For the following definition recall the definition of normal contractions and $\varepsilon$-cutoffs, see Definition~\ref{definition:normal contraction}.

\begin{definition}[Energy form] \label{definition:energy form}
Let $m$ be localizable. A quadratic form $E$ on $L^0(m)$ is called {\em Markovian} if for each $\varepsilon >0$ there exists an $\varepsilon$-cutoff $C_\varepsilon$ such that for all $f \in L^0(m)$ the inequality 
$$E(C_\varepsilon \circ f ) \leq E(f)$$
holds. A closed Markovian form on $L^0(m)$ is called {\em energy form}. 
\end{definition}

Since it is part of the definition of an energy form, from now on we make the {\bf standing assumption} that the measure $m$ is localizable.  

It may happen that $D(E) = \{0\}$. This situation is not so interesting and some of the desired theorems do not hold. We say that $E$ is {\em nontrivial} if its domain satisfies $D(E) \neq \{0\}.$

The following theorem is important for providing concrete examples of energy forms. It shows that the Markov property of a form passes to its closure.  

\begin{theorem}\label{theorem:closure is markovian}
 The closure of a closable Markovian form on $L^0(m)$ is an energy form. 
\end{theorem}
\begin{proof}
Let $E$ be a closable Markovian form on $L^0(m)$. We need to show that its closure $\overline{E}$ is Markovian. Let $\varepsilon > 0$ and let $C_\varepsilon$ be an $\varepsilon$-cutoff such that $E(C_\varepsilon \circ f) \leq E(f)$ for all $f \in L^0(m)$.  Proposition~\ref{prop:characterization closability} shows that  $\overline{E}:L^0(m) \to [0,\infty],$ the closure of $E$, is given by 
$$
 f \mapsto \overline{E}(f) = \begin{cases}
                \lim\limits_{i} E(f_i) &\text{if }(f_i)\text{ in } D(E) \text{ is } E\text{-Cauchy and satisfies } f_i \overset{m}{\to} f\\
                \infty &\text{else}
               \end{cases}.
$$ 
 In particular, for $f \in D(\overline{E})$ there exists an $E$-Cauchy net $(f_i)$ with $f_i \overset{m}{\to} f$. Since $C_\varepsilon$ is a normal contraction, we have $C_\varepsilon \circ f_i \overset{m}{\to} C_\varepsilon \circ f$. The lower semicontinuity of $\overline{E}$ and the Markov property of $E$ imply
$$\overline{E}(C_\varepsilon \circ f) \leq  \liminf_i \overline{E}(C_\varepsilon \circ f_i) = \liminf_i E(C_\varepsilon \circ f_i) \leq  \liminf_i E(f_i) = \overline{E}(f). $$
This finishes the proof.
\end{proof}

\begin{remark}
Energy forms can be seen as generalizations of Dirichlet forms, see Subsection~\ref{subsection:resistance forms}. In this sense, the previous theorem is an extension of \cite[Theorem~3.1.1]{FOT}. Our proof relies on the characterization of the closure of $E$ and its lower semicontinuity. A similar approach  was used in the proof of \cite[Theorem~3.2]{ABR}.
\end{remark}

In some cases, the underlying space $X$ is equipped with a topology that derives from its geometry.  A way of saying that an energy form is compatible with the geometry of $X$ is the following. 

\begin{definition}[Regular energy form]
 Let $X$ be a locally compact separable metric space and let $m$ be a Radon measure on $X$ with full support.  An energy form $E$ on $L^0(m)$ is called {\em regular} if the space
 $$D(E) \cap C_c(X)$$
 is dense in $D(E)$ with respect to the form topology and in $C_c(X)$ with respect to uniform convergence. 
\end{definition}

Before developing the theory further, we give several examples. The reader is strongly encouraged to pick his favorite one and keep it in mind.

\subsection{Energy forms associated with Riemannian manifolds} \label{subsection:manifolds}  Let $(M,g)$ be a smooth Riemannian manifold with associated volume measure ${\rm vol}_g$. We define the quadratic form $E_{(M,g)}:L^0({\rm vol}_g) \to [0,\infty]$  by 
 $$f \mapsto E_{(M,g)}(f) := \begin{cases} \int_M |\nabla f|^2 \D {\rm vol}_g & \text{if } f \in L^2_{\rm loc}({\rm vol}_g) \text{ and } |\nabla f| \in L^2({\rm vol}_g) \\ \infty & \text{else}\end{cases}.$$
 Here,  $\nabla$ is the distributional  gradient and $|\cdot|$ denotes the length of a tangent vector. 
 
 \begin{proposition}\label{proposition: example manifolds}
  $E_{(M,g)}$ is an energy form on $L^0({\rm vol}_g).$
 \end{proposition}

 \begin{proof}
  The Markov property of $E_{(M,g)}$ is a consequence of Lemma~\ref{lemma:contraction manifold} and the fact that there exist smooth $\varepsilon$-cutoffs. It remains to prove closedness. Since ${\rm vol}_g$ is $\sigma$-finite, the space $L^0({\rm vol}_g)$ is metrizable and we can work with sequences. 
  
  Let $(f_n)$ be a sequence in $D(E_{(M,g)})$ that is  Cauchy with respect to the form topology. The completeness of $L^0(m)$ implies that it has a $\tau(m)$-limit $f$. We show $f \in D(E_{(M,g)})$ and the convergence $E_{(M,g)}(f-f_n) \to 0$. 
  
  We first prove $f \in L^2_{\rm loc}({\rm vol}_g)$. According to Lemma~\ref{lemma:local poincare}, for each $x \in M$, there exists a relatively compact open neighborhood $U_x$ of $x$ and a constant $C_x > 0$ such that for all  $h \in D(E_{(M,g)})$ the inequality
  $$\int_{U_x} |h - \bar{h}|^2 {\rm d vol}_g \leq C_x \int_M |\nabla h|^2 {\rm d vol}_g  $$
   holds, where $\bar{h} = {\rm vol}_g(U_x)^{-1} \int_{U_x} h \D {\rm  vol}_g$. Therefore, the sequence $(f_n - \bar{f_n})$ is Cauchy in $L^2(U_x,{\rm vol}_g)$. Since the latter space is complete, there exists $h \in L^2(U_x,{\rm vol}_g)$ such that $f_n - \bar{f_n} \to h$ in $L^2(U_x,{\rm vol}_g)$. Combining this with $f_n \to f$ in  $L^0({\rm vol}_g)$ yields $\bar{f_n} \to f - h$ in $L^0(U_x,{\rm vol}_g)$. This shows that $f - h$ is constant on $U_x$ and so $f\in L^2(U_x,{\rm vol}_g)$. Since any compact set can be covered by finitely many of the $U_x$, we obtain $f \in L^2_{\rm loc} ({\rm vol}_g).$
  
  It remains to prove $E_{(M,g)}(f-f_n) \to 0$, as $n \to \infty$. The space of square integrable vector fields $\vec{L}^2({\rm vol}_g)$ is complete. Therefore, the sequence $(\nabla f_n)$ has a limit $\mathcal{X} \in \vec{L}^2({\rm vol}_g)$. We need to show $\nabla f = \mathcal{X}$. To this end, for each $N \in \IN$, we choose a bounded function $\psi_N \in C^\infty(\IR)$ with $\|\psi'_N \|_\infty \leq 1$ and $\psi(x) = x$ for $x \in [-N,N]$. The chain rule, Lemma~\ref{lemma:contraction manifold}, and the continuity of $\psi'_N$ imply the $\vec{L}^2({\rm vol}_g)$-convergence
  $$\nabla (\psi_N \circ f_n) = (\psi_N' \circ f_n) \nabla f_n \to (\psi_N' \circ f) \mathcal{X}, \text{ as } n \to \infty.$$
  For any compactly supported smooth vector field $\mathcal{Y}$ we obtain
  \begin{align*}
  \int_M f \, {\rm div}\,\mathcal{Y} \, {\rm d vol}_g &= \lim_{N \to \infty} \int_M (\psi_N \circ f)\, {\rm div}\, \mathcal{Y} \, {\rm d vol}_g \\
  &= \lim_{N \to \infty} \lim_{n\to\infty} \int_M (\psi_N \circ f_n)\,  {\rm div}\, \mathcal{Y} \, {\rm d vol}_g\\
  &= \lim_{N \to \infty} \lim_{n\to\infty} \int_M g(\nabla (\psi_N \circ f_n), \mathcal{Y} )\, {\rm d vol}_g\\
  &=\lim_{N \to \infty} \int_M g((\psi'_N \circ f) \mathcal{X} , \mathcal{Y})\, {\rm d vol}_g\\
  &=  \int_M g(\mathcal{X}, \mathcal{Y})\, {\rm d vol}_g.\\
  \end{align*}
   This shows $\nabla f = \mathcal{X}$ and finishes the proof. 
 \end{proof}
\begin{remark}
The idea that a local version of the Poincar\'e inequality implies the closedness of $E_{(M,g)}$ on $L^0({\rm vol}_g)$ is taken from \cite{HKLMS}. It can be shown that the kernel of $E_{(M,g)}$ consists of functions that are constant on connected components of $M$. Therefore, the local Poincar\'e inequality implies the continuity of the embedding 
$$D(E_{(M,g)})/\ker E_{(M,g)}  \to L^2_{\rm loc}({\rm vol}_g)/\ker E_{(M,g)}, \, f + \ker E_{(M,g)} \mapsto  f + \ker E_{(M,g)}. $$
This embedding and the completeness of $\vec{L}^2({\rm vol}_g)$ yielded the closedness of $E_{(M,g)}$. It is almost the same argument as in the proof of Theorem~\ref{theorem:characterization closedness hilbert space}.
\end{remark}
There is another example for an energy form associated with $(M,g)$. We let $E^0_{(M,g)}$ be the closure of the restriction of $E_{(M,g)}$ to the subspace $C_c^\infty(M)$, the {\em smooth function of compact support on $M$}. 
\begin{proposition}
 $E^0_{(M,g)}$ is a regular energy form on $L^0({\rm vol}_g).$
\end{proposition}
\begin{proof}
 The closedness of $E^0_{(M,g)}$ follows from its definition and the regularity is a consequence of the Stone-Weierstra\ss\, theorem. By  the chain rule, Lemma~\ref{lemma:contraction manifold}, and the existence of smooth $\varepsilon$-cutoffs, the restriction of $E_{(M,g)}$ to $C_c^\infty(M)$ is Markovian. According to Theorem~\ref{theorem:closure is markovian}, its closure $E^0_{(M,g)}$ is Markovian as well.
\end{proof}

\begin{remark}
 The $L^2$-restrictions  $\E_{(M,g)}:= E_{(M,g)}|_{L^2({\rm vol}_g)}$ and $\E^0_{(M,g)}: = E^0_{(M,g)}|_{L^2({\rm vol}_g)}$ are Dirichlet forms on $L^2({\rm vol}_g)$ and well studied, see e.g. \cite{Gri}. The operator associated with $\E_{(M,g)}$ is the $L^2$-Laplacian with Neumann boundary conditions (at infinity) and the operator associated with $\E^0_{(M,g)}$ is the $L^2$-Laplacian with Dirichlet boundary conditions (at infinity). It is known that $E^0_{(M,g)}$ is the extended Dirichlet form of $\E^0_{(M,g)}$, see e.g. \cite[Example~1.5.3]{FOT} for the case of open subsets of $\IR^n$. A similar statement for $\E_{(M,g)}$ and $E_{(M,g)}$ need not be true. It can happen that $\E^0_{(M,g)} = \E_{(M,g)}$ while $E^0_{(M,g)} \neq E_{(M,g)}$.
\end{remark}

\subsection{Jump-type forms} \label{subsection:jump forms} Let $(X,\mathcal{B},m)$ be a $\sigma$-finite measure space. Let $J$ be a measure on $\mathcal{B}\otimes \mathcal{B}$ with the property that $m(A) =  0$ implies $J(A\times X) = J(X \times A) = 0$ and let $V:X \to [0,\infty]$ be measurable. We define the functional $E_{J,V}: L^0(m) \to [0,\infty]$ by 
$$ f \mapsto E_{J,V}(f) := \int_{X\times X} (f(x) - f(y))^2 \D J(x,y) + \int_X f(x)^2 \, V(x) \D m(x).$$
\begin{proposition}
  $E_{J,V}$ is an energy form on $L^0(m)$.
\end{proposition}
\begin{proof}
 The condition on $J$ ensures that the form is well-defined. From the properties of the integral with respect to $J$ and $V {\rm d} m$, it easily follows that $E_{J,V}$ is a  Markovian quadratic form. It remains to prove its closedness. The $\sigma$-finiteness of $m$ yields that the space $L^0(m)$ is metrizable. Hence, it suffices to show lower semicontinuity for sequences. 
 
 To this end, let $(f_n)$ in $L^0(m)$ and $f \in L^0(m)$ with $f_n \overset{m}{\to} f$ be given. By the characterization of convergence in measure in the $\sigma$-finite situation, there exists a subsequence $(f_{n_k})$ with  $f_{n_k} \to f$ $m$-a.e. and
 $$\liminf_{n\to \infty} E_{J,V}(f_n) = \liminf_{k\to \infty} E_{J,V}(f_{n_k}).$$
 The assumption on $J$ implies $(f_n(x) - f_n(y))^2 \to (f(x) - f(y))^2$ for $J$-a.e. pair $(x,y)$. An application of Fatou's lemma yields
 $$E_{J,V}(f) \leq \liminf_{k \to \infty}E_{J,V}(f_{n_k}) = \liminf_{n \to \infty} E_{J,V}(f_n).$$
 This calculation finishes the proof.
\end{proof}

We would like to warn the reader that the condition on the measure $J$ does not imply its absolute continuity with respect to $m \otimes m$ nor does it imply any symmetry. 

One important special case of jump forms arises when $X$ is countable and equipped with the  $\sigma$-algebra of all subsets of $X$. In this case, the measure $m$ can be identified with a function $m:X \to [0,\infty]$ via the identity
$$m(A) = \sum_{x \in A} m(x) \text{ for all } A \subseteq X.$$
 Furthermore, there exists a function $b: X\times X \to [0,\infty]$ such that for all $B \subseteq X \times X$ the measure $J$ satisfies
$$J(B) = \sum_{(x,y) \in B} b(x,y),$$
and the energy form $E_{J,V}$ takes the form
$$E_{J,V}(f) = \sum_{x,y \in X} b(x,y)(f(x)-f(y))^2 + \sum_{x \in X} f(x)^2 V(x) m(x).$$
The function $b$ can be interpreted as an edge weight of a graph whose vertices $x,y \in X$ are connected if and only if $b(x,y) >0$.  Forms of this type (with various additional conditions on $b$ and $V$) have been studied quite extensively. We refer the reader to \cite{GHKLW,Schmi,Soa,Woe} and references therein. Indeed, the seminal paper \cite{BD}, which introduces the notion of Dirichlet spaces, treats exactly this case under the additional assumptions that $X$ is a finite set and that the functions $b$ and $V$ take finite values only.  

\subsection{Resistance forms, Dirichlet forms and extended Dirichlet forms} \label{subsection:resistance forms}

The previous two examples have been rather explicit. We finish this section by discussing two more abstract situations. The next one is the main reason why we do not restrict ourselves to the $\sigma$-finite situation but allow localizable measures.

{\bf Resistance forms.}  Let $X \neq \emptyset$ arbitrary and let $\mu$  be the counting measure on all subsets of $X$. In this case, $L^0(\mu)$ is the space of all real valued functions on $X$ and its topology $\tau(\mu)$ is the topology of pointwise convergence. The measure $\mu$ is strictly localizable and hence localizable by Theorem~\ref{theorem:strictly localizable implies localizable}. 
 
 Following \cite{Kig1}, we call a functional $\E:L^0(\mu)\to [0,\infty]$ a {\em resistance form }if it satisfies the subsequent five conditions.

(RF1)  $\E$ is a quadratic form on $L^0(\mu)$ with $\ker \E = \IR 1$.

(RF2) $(D(\E)/\IR 1, \E)$ is a Hilbert space.

(RF3) For every finite $V \subseteq X$ and each $g:V \to \IR$, there exists an $f \in D(E)$ with $f|_V = g$. 

(RF4) For all $x,y \in X$, we have
$$R_\E (x,y):= \sup \left\{ \frac{|f(x) - f(y)|^2}{\E(f)}\, \middle|\, 0 < \E(f) < \infty  \right\} < \infty. $$
(RF5)  For all $f \in L^0(\mu)$, we have 
$$ \E( (f \vee 0) \wedge 1) \leq \E(f).$$
\begin{proposition}
 Every resistance form is an energy form. 
\end{proposition}
\begin{proof}
 Let $\E$ be a resistance form. It is Markovian by (RF5).  To show its closedness we use Theorem~\ref{theorem:characterization closedness hilbert space}. In order to apply it, we need to prove that $D(\E)$ equipped with the form topology is metrizable and that the canonical inclusion of $(D(\E)/ \IR 1,\E)$ into $(L^0(\mu)/ \IR 1, \tau(\mu)/\IR 1)$ is continuous. Once we have shown these two facts, (RF2) and Theorem~\ref{theorem:characterization closedness hilbert space} yield the statement. 
 
 We fix $o \in X$ and define the norm $\|\cdot\|_o$ by
 $$\|\cdot\|_o:D(\E) \to [0,\infty),\, f \mapsto \|f\|_o := \sqrt{\E(f) + |f(o)|^2}.$$
 {\em Claim 1: } The norm $\|\cdot\|_o$ induces the form topology. In particular, the form topology is metrizable.
 
 {\em Proof of Claim 1.} Obviously, convergence in the form topology implies convergence with respect to $\|\cdot\|_o$ and convergence with respect to $\|\cdot\|_o$ implies convergence with respect to $\E$. Furthermore, the definition of $R_\E$ yields that for $x\in X$ we have
 $$|f(x)|^2 \leq 2 |f(x) - f(o)|^2 + 2 |f(o)|^2 \leq 2 R(x,o) \E(f) + 2 |f(o)|^2 \leq 2 \max \{R(x,o),1\} \|f\|^2_o. $$
 Since $R_\E(x,o) < \infty$, the above inequality shows that convergence with respect to $\|\cdot\|_o$ implies pointwise convergence. This finishes the proof of Claim 1. \qedc
 
 {\em Claim 2: } The canonical inclusion of $(D(\E)/ \IR 1,\E)$ into $(L^0(\mu)/ \IR 1, \tau(\mu)/\IR 1)$ is continuous.
 
 {\em Proof of Claim 2.}  Let $(f_i)$ be a net in $D(\E)$ and  let $f \in D(\E)$ with $\E(f-f_i) \to 0$. We fix $o\in X$. The statement follows once we show that $\ow{f}_i := f_i - f_i(o)$ converges pointwise to $\ow{f} := f - f(o)$. Since $\ker \E  = \IR 1$ and $\ow{f}_i(o) = \ow{f}(o) = 0$, we obtain
 $$|\ow{f}(x) - \ow{f}_i(x)|^2 = |\ow{f}(x) - \ow{f}_i(x) - (\ow{f}(o) - \ow{f}_i(o))|^2 \leq R_\E(x,o) \E(\ow{f} - \ow{f}_i) = R_\E(x,o) \E(f - f_i).$$
 Condition (RF3) yields  $R_\E(x,o) > 0$ and so we infer the pointwise convergence $\ow{f_i} \to \ow{f}$. This proves Claim~2. 
\end{proof}

\begin{remark}
 \begin{itemize}
  \item For proving that resistance forms are energy forms  not all of the properties (RF1) - (RF5) are needed.  The condition $\ker \E = \IR 1$ can be removed. In this case, (RF2) has to be replaced by demanding that $(D(\E)/\ker \E , \E)$ is a Hilbert space. Indeed, this latter assumption is not too complicated to check as (RF4) automatically implies $\ker \E \subseteq \IR 1$ and so either $\ker \E = \IR 1$ or $\ker \E = \{0\}$.
  
  An example for these 'generalized' resistance forms is the form $E^0_{(M,g)}$ when $(M,g)$ is a one-dimensional connected Riemannian manifold. We leave the details to the reader but mention that both $\ker E^0_{(M,g)} = \{0\}$ and $\ker E^0_{(M,g)}  =  \IR 1$ can happen for suitable choices of the metric $g$.

  \item The previous proposition together with the discussion of Section~\ref{sect:closed forms} shows that resistance forms are lower semicontinuous with respect to pointwise convergence. We believe that this is an important feature, which has not been widely used in the literature. 
  \end{itemize}
\end{remark}

{\bf Dirichlet forms and extended Dirichlet spaces.} Let $m$ be localizable and let $\E$ be a Dirichlet form on $L^2(m)$. Recall our convention $\E(f) = \infty$ for $f \in L^0(m) \setminus L^2(m)$ and that for $\alpha > 0$ and $f \in L^0(m)$ we set
$$\E_\alpha(f) = \E(f) + \alpha \int_X |f|^2 \D m.$$
\begin{proposition}
 Let $m$ be localizable and let $\E$ be a Dirichlet form on $L^2(m)$. 
 \begin{itemize}
  \item[(a)] For each $\alpha > 0$, the form $\E_\alpha$ is an energy form on $L^0(m)$.
  \item[(b)] The extended Dirichlet form $\Ee$ is an energy form on $L^0(m)$.
 \end{itemize}
\end{proposition}
\begin{proof}
 (a): Since a Dirichlet form is closed on $L^2(m)$, we can use Proposition~\ref{proposition: closedness in L2} to obtain that $\E_\alpha$ is closed on $L^0(m)$. The Markov property of $\E_\alpha$ follows from the fact that both $\E$ and the norm $\|\cdot\|^2_2$ are Markovian. 
 
 (b): $\Ee$ is the closure of $\E$ in $L^0(m)$ and, therefore, closed by definition. As the closure of the Markovian form $\E$ it is Markovian by Theorem~\ref{theorem:closure is markovian}.
\end{proof}

\begin{remark}
 In general, a Dirichlet form is not an energy form and so we can not directly apply the theory we develop below to a Dirichlet form itself. Nevertheless, this does not mean any loss of information. If one is interested in the $L^2$-theory of Dirichlet forms, then everything is encoded in the forms $\E_\alpha$, $\alpha> 0$. For other properties of Dirichlet forms that are not captured by the $L^2$-theory one can instead use the extended Dirichlet form.  
\end{remark}

\section{Contraction properties}\label{section:contraction properties}

This section is devoted to contraction properties of energy forms. We prove several consequences of the Markov property, which then accumulate in a version of   Theorem~\ref{theorem:cutoff properties Dirichlet form} for energy forms. In the $L^2$-setting this theorem is usually proven in two steps. It is first established for continuous forms and then extended to closed forms by an approximation procedure. We pursue a similar strategy and use the approximating forms of Section~\ref{section:form on L0}. 

As a preparation we start with elementary contraction properties, which follow immediately from the Markov property and lower semicontinuity.

\begin{proposition}[Elementary contraction properties] \label{proposition:elementary cut-off properties}
  Let $E$ be an energy form. For all $f \in L^0(m)$ and all $n \geq 0$, the following inequalities hold.
 \begin{itemize}
  \item[(a)] $E(f_+ \wedge 1) \leq E(f)$.
  \item[(b)] $E(f_+) \leq E(f)$.
  \item[(c)] $E(|f|) \leq E(f)$. 
  \item[(d)] $E(f  \wedge n) \leq E(f)$.  
 \end{itemize}
\end{proposition}
\begin{proof}
(a): For  $n>0$ we choose an $n^{-1}$-cutoff $C_{n^{-1}}$ with $E(C_{n^{-1}}\circ f)\leq E(f)$. By definition we have $C_{n^{-1}} \circ f \overset{m}{\to} f_+ \wedge 1,$   as $n \to \infty$. From the lower semicontinuity of $E$ we infer
$$E(f_+\wedge 1) \leq \liminf_{n\to \infty} E(C_{n^{-1}}\circ f) \leq E(f).$$
(b): We have  $n  ((n^{-1}f)_+ \wedge 1) \overset{m}{\to} f_+,$ as $n\to \infty$. The lower semicontinuity of $E$  and (a) imply
$$E(f_+) \leq  \liminf_{n\to \infty} E(n  ((n^{-1}f)_+ \wedge 1)) = \liminf_{n\to \infty} n^2 E((n^{-1}f)_+ \wedge 1) \leq E(f).$$
(c): Without loss of generality, we may assume $f \in D(E)$. By (b) we already know that $f_+,f_- \in D(E)$. We compute
$$E(|f|) = E(f_+) + 2 E(f_+,f_-) + E(f_-) \text{ and } E(f) = E(f_+) - 2 E(f_+,f_-) + E(f_-). $$
Hence, it suffices to show $E(f_+,f_-) \leq 0$. For $\varepsilon > 0$, assertion (b) implies
$$E(f_+) = E( (f_+ - \varepsilon f_-)_+) \leq E( f_+ - \varepsilon f_-) = E(f_+) - 2\varepsilon E(f_+,f_-) + \varepsilon ^2E(f_-). $$
Letting $\varepsilon \to 0+$ yields $E(f_+,f_-) \leq 0$.

(d): We have $f_+ \wedge n = n  ((n^{-1}f)_+ \wedge 1)$. By (a) we obtain
$$E(f \wedge n) = E(f_+ \wedge n) - 2 E(f_+ \wedge n,f_-) + E(f_-) \leq E(f_+) - 2 E(f_+ \wedge n,f_-) + E(f_-).$$
Therefore, it suffices to prove $E(f_+ - f_+ \wedge n,f_-) \leq 0$. The latter inequality can be shown with the same arguments which were used to prove $E(f_+,f_-)\leq 0$. This finishes the proof.
\end{proof}
We deduced the inequality $E(|f|) \leq E(f)$ from $E(f_+,f_-) \leq 0$. The following lemma extends this observation. 
\begin{lemma}\label{lemma:cutoff for functions with disjoint support}
 Let $E$ be an energy form and let $f,g \in D(E)$ be nonnegative. If $f \wedge g = 0$, then
 $$E(f,g) \leq 0.$$
 If, additionally, there are nonnegative $\ow{f},\ow{g}\in D(E)$ with $f \geq \ow{f}$ and $g \geq \ow{g}$, then
 $$E(f,g) \leq E(\ow{f},\ow{g}).$$
\end{lemma}
\begin{proof}
 Let $f,g \in D(E)$ with $f \wedge g = 0$ be given. For any $\varepsilon >0$,  Proposition~\ref{proposition:elementary cut-off properties} shows
 $$E(f) = E((f - \varepsilon g)_+) \leq E(f - \varepsilon g) = E(f) - 2 \varepsilon E(f,g) + \varepsilon ^2 E(g).$$
 Consequently, we obtain $E(f,g) \leq 0$. For the 'furthermore' statement, we compute
 $$E(f,g) -  E(\ow{f},\ow{g}) = E(f-\ow{f},g) + E(\ow{f}, g -\ow{g}).$$
 Since $(f-\ow{f}) \wedge g = (g -\ow{g}) \wedge \ow{f} = 0$,  what we have shown so far yields that the right-hand side of the above equation is negative. This finishes the proof.
\end{proof}

Another technical lemma that uses the same proof technique  is the following.

\begin{lemma}\label{lemma:cutoff for functions beeing one on support}
 Let $E$ be an energy form and let $ f,g \in D(E)$ be nonnegative. If the inequality $1_{\{g > 0\}} \leq f \leq 1$ holds, then 
 $$E(f,g) \geq 0.$$
\end{lemma}
\begin{proof}
  Let nonnegative $f,g \in D(E)$ with $1_{\{g > 0\}} \leq f \leq 1 $ be given. For any $\varepsilon >0$,  Proposition~\ref{proposition:elementary cut-off properties} yields
  $$E(f) = E((f + \varepsilon g)\wedge 1) \leq E(f + \varepsilon g) = E(f) + 2 \varepsilon E(f,g) + \varepsilon ^2 E(g).$$
 Letting $\varepsilon \to 0+$ implies $E(f,g) \geq 0.$ 
\end{proof}

\begin{remark}\label{remark:cutoff remark}
\begin{itemize}
 \item For later purposes, it is important to note  that the statements of Proposition~\ref{proposition:elementary cut-off properties} and, therefore, Lemma~\ref{lemma:cutoff for functions with disjoint support} and Lemma~\ref{lemma:cutoff for functions beeing one on support} remain true for closed Markovian forms on $\Ltf$. This is due to the fact that in the proof of Proposition~\ref{proposition:elementary cut-off properties} (a) and (b) we used lower semicontinuity with respect to convergence in measure, which can be replaced by lower semicontinuity with respect to $\Ltf$ convergence. 
 \item The presented proofs for the previous two lemmas are well-known, see e.g. \cite{All}. 
\end{itemize}
\end{remark}

Before proving the main result of this section, we show that the approximation procedure of Section~\ref{section:form on L0} leads to an approximation by Markovian forms if the form we start with is Markovian. Recall that for $\alpha > 0$ and $U \in \Bf$, the corresponding approximating form  is given by 
$$E^{(\alpha,U)}:\Ltf \to [0,\infty),\, f \mapsto  E^{(\alpha,U)}(f) = \inf\left\{E(g) +  \alpha \int_U (g-f)^2\, {\rm d}m\, \middle| \, g \in \Ltf \right\}.$$

\begin{lemma}\label{lemma:Markovian approximation}
Let $E$ be a Markovian quadratic form on $\Ltf$. For all $\alpha > 0$ and all $U \in \Bf$, the form  $E^{(\alpha,U)}$ is Markovian.
\end{lemma}
\begin{proof}
 Let $f \in \Ltf$ and let $\varepsilon > 0$. We choose an  $\varepsilon$-cutoff $C_\varepsilon$ such that $E(C_\varepsilon \circ f) \leq E(f)$. For arbitrary $g \in \Ltf$, we obtain
 $$E^{(\alpha,U)}(C_\varepsilon \circ f) \leq  E(C_\varepsilon \circ g) +  \alpha \int_U (C_\varepsilon \circ f - C_\varepsilon \circ g)^2\, {\rm d}m \leq E( g) +  \alpha \int_U ( g -  f)^2\, {\rm d}m.$$
 Taking the infimum over all such $g$ finishes the proof. 
\end{proof}

\begin{proposition}[Approximation by Markovian forms]\label{proposition:Markovian approximation}
 Let $E$ be a closed Markovian form on $\Ltf$. There exists a net of continuous Markovian forms $(E_i)$ on $\Ltf$ with $D(E_i) = \Ltf$ such that for all $f \in \Ltf$ we have
 $$E(f) = \lim_i E_i(f) = \sup_i E_i(f).$$
\end{proposition}
\begin{proof}
This follows from Lemma~\ref{lemma: approximation of closed forms} and Lemma~\ref{lemma:Markovian approximation}.
\end{proof}

\begin{theorem} \label{theorem:cutoff properties energy forms}
Let $E$ be a closed quadratic form on $L^0(m)$. The following assertions are equivalent.
  \begin{itemize}
  \item[(i)] $E$ is an energy form. 
  \item[(ii)] For every normal contraction $C:\IR \to \IR$ and every $f \in L^0(m)$, the inequality
  $$E(C\circ f) \leq E(f)$$
  holds.
  \item[(iii)] For every normal contraction $C:\IR^n \to \IR$ and all $f_1,\ldots,f_n \in L^0(m)$, the inequality

  $$E(C(f_1,\ldots,f_n))^{1/2} \leq \sum_{k = 1}^n E(f_k)^{1/2}$$
  holds.
  \item[(iv)] For all $f,f_1,\ldots,f_n \in L^0(m)$, the  inequalities
$$|f(x)| \leq \sum_{k=1}^n|f_k(x)| \text{ and } |f(x)-f(y)| \leq \sum_{k = 1}^n  |f_k(x)-f_k(y)| \text{ for }  m\text{-a.e. } x,y \in X $$
 imply
$$E(f)^{1/2} \leq \sum_{k = 1}^n E(f_k)^{1/2}.$$
 \end{itemize}
\end{theorem}

\begin{proof}
 The implications (iv) $\Rightarrow$ (iii) $\Rightarrow$ (ii) $\Rightarrow$ (i) are obviously satisfied. 
 
 (iii) $\Rightarrow$ (iv): Let $f,f_1,\ldots,f_n \in L^0(m)$  with
$$|f(x)| \leq \sum_{k=1}^n|f_k(x)| \text{ and } |f(x)-f(y)| \leq \sum_{k = 1}^n  |f_k(x)-f_k(y)| \text{ for } m\text{-a.e. } x,y \in X $$
 be given. We fix some measurable versions $\tilde{f}, \tilde{f}_1,\ldots, \tilde{f}_n$ for which the above inequalities hold for all points and define  
 $$C: (\tilde{f}_1,\ldots, \tilde{f}_n)(X) \to \IR, \quad (\tilde{f}_1(x),\ldots, \tilde{f}_n(x)) \mapsto \tilde{f}(x). $$
 By the choice of the $\ow{f}_i$ the map $C$ is well defined and a normal contraction on its domain.  According to Lemma~\ref{lemma:extension of normal contractions}, it can be extended to a normal contraction $\ow{C}$ on the whole space $\IR^n.$  Assertion (iii) implies 
 $$E^{1/2}(f) = E(\widetilde{C}(f_1,\ldots,f_n))^{1/2} \leq \sum_{i=1}^n E(f_i)^{1/2}. $$
 This shows (iv).

 (i) $\Rightarrow$ (iii): We make several reductions and then show the statement for continuous Markovian forms on $\Ltf$.
 
 {\em Claim 1:} If (iii) holds for functions in $\Ltf$, then (iii) holds for functions in $L^0(m)$. 
 
 {\em Proof of Claim 1.} Let $f_1,\ldots,f_n \in L^0(m)$. For $k\in \IN$ we set $f_{i,k} := (f_i \vee (-k)) \wedge k$. Obviously,  $C(f_{1,k},\ldots,f_{n,k}) \overset{m}{\to}  C(f_{1},\ldots,f_{n})$, as $k \to \infty$. Furthermore, the $f_{i,k}$ and $C(f_{1,k},\ldots,f_{n,k})$ are bounded and,  therefore,  belong to $\Ltf$. The lower semicontinuity of $E$ and  (iii) for $\Ltf$-functions yield
$$E(C(f_{1},\ldots,f_{n}))^{1/2} \leq \liminf_{k\to \infty} E(C(f_{1,k},\ldots,f_{n,k}))^{1/2} \leq \liminf_{k\to \infty} \sum_{i = 1}^n E(f_{i,k})^{1/2}.$$
According to Proposition~\ref{proposition:elementary cut-off properties},  we have $E(f_{i,k}) \leq  E(f_i)$. This finishes the proof of Claim~1.\qedc

{\em Claim 2:} If (iii) holds for all continuous Markovian forms on $\Ltf$ whose domain equals $\Ltf$, then it holds for the restriction of $E$ to $\Ltf$. 

{\em Proof of Claim 2.} According to Lemma~\ref{lemma:restriction to Ltf}, the restriction of $E$ to $\Ltf$ is a closed form on $\Ltf$. Furthermore, this restriction is Markovian. Proposition~\ref{proposition:Markovian approximation} shows that it can be approximated by continuous Markovian forms on $\Ltf$ whose domain is $\Ltf$.  This proves Claim~2.  \qedc

{\em Claim 3:} Assertion (iii) holds for continuous Markovian forms on $\Ltf$ with domain $\Ltf$. 

{\em Proof of Claim 3.} Let $E$ be a continuous Markovian form on $\Ltf$ with $D(E) = \Ltf$. As seen in Remark~\ref{remark:cutoff remark}, the statements of Lemma~\ref{lemma:cutoff for functions with disjoint support} and Lemma~\ref{lemma:cutoff for functions beeing one on support} remain true for the form $E$. For two measurable sets $A,B \subseteq X$, they imply the following inequalities.

\begin{itemize}
 \item[(a)] If $A \cap B = \emptyset$, then $E(1_A,1_B)\leq 0$.
 \item[(b)] If $A \subseteq B$, then $E(1_A,1_B) \geq 0$.
\end{itemize}
With these at hand, we can prove (iii) for simple functions. Let
$$f = \sum_{i=1}^n \alpha_i 1_{A_i}$$
with $A_i \cap A_j = \emptyset$ if $i\neq j$ be given. An elementary computation shows 
$$E(f) = \sum_{i,j = 1}^n b_{ij} (\alpha_i - \alpha_j)^2 + \sum_{i = 1}^n c_{i} \alpha_i ^2,$$
where 
$$b_{ij} = -E(1_{A_i},1_{A_j}) \text{ and } c_{i} = \sum_{j = 1}^n E(1_{A_i},1_{A_j}) = E(1_{A_i},1_{\cup_j A_j}).$$
The inequalities (a) and (b) yield $b_{ij} \geq 0$ and $c_i \geq 0$. For a normal contraction $C:\IR^n \to \IR$ and $f = C(f_1,\ldots,f_n)$ with simple functions $f_1,\ldots,f_n,$ the positivity of $b_{ij}$ and $c_i$ implies
$$E(f)^{1/2} \leq \sum_{i = 1}^n E(f_i)^{1/2}. $$
Therefore, (iii) holds for simple functions. Since simple functions are dense in $\Ltf$ and $E$ is continuous, this finishes the proof of Claim~3. \qedc

Taking into account all the claims finishes the proof of the implication (i) $\Rightarrow$ (iii). 
\end{proof}

\begin{remark}
\begin{itemize}
  \item The previous theorem is well known for Dirichlet forms, see \cite[Theorem~4.12 and Corollary~4.13 of Chaper~1]{MR}. That it remains valid for extended Dirichlet spaces is not explicitly stated in the literature. We could only find the equivalence (i) $\Leftrightarrow$ (ii), see \cite[Corollary~1.5.1]{FOT} and \cite[Theorem~1.1.5]{CF}. This is somewhat surprising since (iii) for extended Dirichlet forms can easily be deduced from the statement for Dirichlet forms and the definition of the extended Dirichlet space. In contrast, the known methods that show the equivalence of (iii) and (iv) for Dirichlet forms can not be easily adapted to extended Dirichlet spaces. Our approach to proving it by extending a normal contraction on a subset of $\IR^n$ to the whole space seems to be new.

  \item For resistance forms, Proposition~\ref{proposition:elementary cut-off properties} is known under the additional assumption that the underlying space equipped with the resistance metric is separable, see \cite[Proposition~3.15]{Kig3}. 
  
  \item The general idea behind the proof of the theorem is well known. In the end, everything boils down to the inequalities (a) and (b) in the proof of Claim~3 and to an approximation procedure. That the inequalities (a) and (b) play an important role for proving contraction properties was already observed in the seminal paper \cite{BD}.  The innovation here is to approximate in $\Ltf$. 
  
  \item The reason why we change the space from $L^0(m)$ to $\Ltf$ for  the approximation is quite simple. In many cases, there are no nontrivial continuous quadratic forms on $L^0(m)$. If $q$ is were  continuous quadratic form on $L^0(m)$, then for each $v \in L^0(m)$ the linear functional $L^0(m) \to \IR$, $u \mapsto q(v,u)$ is continuous. As discussed above, when $m$ is the Lebesgue measure, the space $L^0(m)$ does not admit any continuous linear functionals. 
\end{itemize}

\end{remark}

\section{Structure properties of energy forms} \label{section:structure properties}

In this  section  we study the structure of the domain and of the kernel of energy forms. We prove that domains of energy forms are lattices and that essentially bounded functions in the domain form an algebra. We investigate  continuity properties of energy forms in the spirit of Section~\ref{sect:closed forms}. Surprisingly, it turns out that the continuity of the embedding of the form domain into $L^0(m)$ is already guaranteed by the triviality of the form kernel. This observation  implies that the domain of an energy form with trivial kernel equipped with the induced inner product is a Hilbert space. Therefore, understanding the kernel of an energy form is quite important for understanding its properties. We give an explicit description of the kernel when it is nontrivial. At the end of this section, we discuss the support of an energy form.

We start with the order and the algebraic structure of the domain of energy forms. In the following theorem we use the convention $ 0 \cdot \infty = 0$.

\begin{theorem}[Algebraic and order structure] \label{theorem:algebraic and order properties} 
 Let $E$ be an energy form. For all functions $f,g \in L^0(m)$, the inequalities 
  $$E(f \wedge g)^{1/2} \leq E(f)^{1/2} + E(g)^{1/2}, \quad E(f \vee g)^{1/2} \leq E(f)^{1/2} + E(g)^{1/2} $$
 and 
  $$E(fg)^{1/2} \leq \|f\|_\infty E(g)^{1/2} + \|g\|_\infty E(f)^{1/2} $$
 hold. In particular, $D(E)$ is a lattice and $D(E) \cap L^\infty(m)$ is an algebra. 
\end{theorem}
\begin{proof}
 Let $f,g \in D(E)$ be given. We use the identity $f \wedge g = \frac{f + g - |f-g|}{2}$ and the contraction properties of $E$ to obtain
 $$E(f \wedge g)^{1/2} \leq \frac{1}{2} \left(E(f)^{1/2} +E(g)^{1/2} +E(|f - g|)^{1/2}\right) \leq E(f)^{1/2} +E(g)^{1/2}.$$
 The statement for the maximum $f\vee g$ can be inferred similarly. 
 
 When $f,g \in D(E) \cap L^\infty(m)$, the estimate on $E(fg)^{1/2}$ follows from Theorem~\ref{theorem:cutoff properties energy forms} since we have
 $$|f(x)g(x)| \leq \|f\|_\infty |g(x)| + \|g\|_\infty |f(x)|$$ and $$|f(x)g(x) - f(y)g(y)|  \leq \|f\|_\infty |g(x) - g(y)| +  \|g\|_\infty |f(x) - f(y)|.$$
 It remains to treat the case when $\|f\|_\infty = \infty$ and $E(g) = 0$. For  $n \in  \IN$ we define $f_n := (f \wedge n) \vee (-n)$.  The lower semicontinuity of $E$ and what we have already shown yield
 $$E(fg)^{1/2} \leq \liminf_{n\to \infty} E(f_n g)^{1/2} \leq \liminf_{n\to \infty} \left(n E(g)^{1/2} + \|g\|_\infty E(f_n)^{1/2}\right) \leq \|g\|_\infty E(f)^{1/2}. $$
 The 'in particular' part follows immediately.
\end{proof}

\begin{remark}
The previous theorem is well known for Dirichlet forms and  extended Dirichlet forms, see \cite[Theorem~1.4.2 and Corollary~1.5.1]{FOT}, and  its proof is standard. For resistance forms the fact that bounded functions in the domain form an algebra seems to be known only in the case when the underlying space is separable with respect to the resistance metric, see \cite{Kig3}.
\end{remark}

The following proposition provides two important dense subspaces of the form domain. 

\begin{proposition}\label{proposition:approximation by bounded functions}
 Let $E$ be an energy form. For each $f \in D(E)$, the following holds true in the form topology of $E$.
 \begin{itemize}
  \item[(a)]  $(f \wedge n) \vee (-n) \to f, \text{ as } n \to \infty.$
  \item[(b)] $(f - \alpha)_+ - (f + \alpha)_- = f - (f\wedge \alpha)\vee (-\alpha) \to f,  \text{ as } \alpha \to 0+.$
 \end{itemize}
In particular, the subspaces $D(E) \cap L^\infty(m)$ and 
$$\{f \in D(E) \mid \text{ there exists } \psi \in D(E)\text{ with } 1_{\{|f| > 0\}} \leq \psi \}$$
are dense in $D(E)$ with respect to the form topology. 
\begin{proof}
The mappings $\IR \to \IR$,  $x \mapsto (x \wedge n) \vee (-n)$ and $x \mapsto (x - \alpha)_+ - (x + \alpha)_-$ are normal contractions.   From Theorem~\ref{theorem:cutoff properties energy forms} we infer
$$\limsup_{n \to \infty} E((f \wedge n) \vee (-n)) \leq E(f) \text{ and } \limsup_{\alpha \to 0+}E((f - \alpha)_+ - (f + \alpha)_-)  \leq E(f).$$
   Since both (a) and (b) hold in the topology of convergence in measure, we can apply Lemma~\ref{lemma:characterization convergence in form topology lsc forms} to obtain the statement. The denseness of $D(E)\cap L^\infty(m)$ is an immediate consequence of (a). For the second density statement we note that $|f| \in D(E)$ by the contraction properties and
   $$\alpha^{-1} |f| \geq 1_{\{|(f - \alpha)_+ - (f + \alpha)| > 0\}}. $$
   This finishes the proof. 
\end{proof}
 \end{proposition}
 \begin{remark}
 For Dirichlet forms these approximations are well known, see \cite[Theorem~1.4.2]{FOT}. Our proof is a bit more direct since it makes use of the lower semicontinuity of $E$ and Lemma~\ref{lemma:characterization convergence in form topology lsc forms}.
 \end{remark}

 Further properties of an energy form depend on its kernel. It will turn out that we can completely describe it with the help of the following concepts.
 
 \begin{definition}[Recurrence and transience] \label{definition:recurrence}
  An energy form $E$ is called {\em transient} if $\ker E = \{0\}$. It is {\em recurrent} if $1 \in \ker E$.
 \end{definition}

 \begin{remark}
 This definition is motivated by the classical characterization of recurrent and transient Dirichlet forms. It is well known that a Dirichlet form $\E$ is recurrent if and only if $1 \in \ker \Ee$, see \cite[Theorem~1.6.3]{FOT}, and that $\E$ is transient if and only if $\ker \Ee = \{0\}$, see \cite[Theorem~1.6.2]{FOT}. In other words, recurrence/ transience of a Dirichlet form in the classical sense coincides with recurrence/transience of the extended Dirichlet form in the above sense. 
 \end{remark}

 \begin{theorem} \label{theorem:continuous embedding of energy forms}
 Let $E$ be an energy form. The following assertions are equivalent. 
 \begin{itemize}
 \item[(i)] $E$ is transient.
  \item[(ii)] The embedding  
  $$(D(E),\|\cdot\|_E) \to (L^0(m),\tau(m)),\quad f \mapsto f$$
is continuous.
\item[(iii)] $(D(E),E)$ is a Hilbert space. 
 \end{itemize}
\end{theorem}
\begin{proof}
 (i) $\Rightarrow$ (ii): Let $(f_i)$ be a net in $D(E)$ with $\|f_i\|_E \to 0$.  We use Proposition~\ref{prop:weak compactness in Ltf} to obtain a subnet $(f_{i_j})$ such that $(|f_{i_j}|\wedge 1)$ converges to some $\psi \in\Ltf$ with respect to the weak topology $\sigma(\Ltf,\Ltf')$. The characterization of $\Ltf'$, Proposition~\ref{proposition:characterization duals of Lpf}, shows that for each $U \in \Bf$  we have
 $$ \lim_{j} \int_U |f_{i_j}|\wedge 1\, {\rm d}m = \int_U \psi\, {\rm d}m.$$
 Therefore, it suffices to show $\psi = 0$. Since the restriction of $E$ to $\Ltf$ is closed,  Theorem~\ref{theorem:characterization closedness lc} yields that it is lower semicontinuous with respect to $\sigma(\Ltf,\Ltf')$. This observation and the contraction properties of $E$ imply
 $$E(\psi) \leq \liminf_j E(|f_{i_j}|\wedge 1) \leq \liminf_j E(f_{i_j}) = 0.$$
 By the assumption $\ker E = \{0\}$, we obtain $\psi = 0$.  
 
 (ii) $\Rightarrow$ (iii): Assertion (ii) implies that the form topology is induced by the form norm $\|\cdot\|_E$. In particular, it is metrizable and we can apply Theorem~\ref{theorem:characterization closedness hilbert space} from which the claim follows.
 
 (iii) $\Rightarrow$ (i): This is clear as in Hilbert spaces inner products are nondegenerate. 
 \end{proof}

\begin{remark} \label{remark:continuous embedding of energy forms}
\begin{itemize}
 \item The previous theorem is quite surprising. It is not clear why the vanishing of the kernel of a quadratic form should imply the continuity of the embedding of the form domain into the underlying space. Here, the 'secret' reasons are the weak compactness of order bounded sets in $\Ltf$ and the contraction properties of the given energy form. Whenever $m$ is finite, the compactness statement reduces to weak compactness of bounded sets in $L^2(m)$.  In this case,  we believe that the given proof is the shortest possible for the theorem.
 
 \item The equivalence of assertions (i) and (iii) is known for extended Dirichlet spaces when $m$ is $\sigma$-finite, see \cite[Theorem~2.1.9]{CF}. However, the proof given there is very different from ours and somewhat less transparent. There, the main technical ingredient, which requires a rather long series of claims, is the following. For a Dirichlet form $\E$, the condition $\ker \Ee = \{0\}$ implies the existence of a strictly positive function $g \in L^1(m)$ such that 
 $$\int_{X} |f|g \, {\rm d} m \leq \Ee(f)^{1/2} \text{ for all } f\in D(\Ee).$$
  This inequality is somewhat stronger than (ii) and it does not directly follow from our arguments.
\end{itemize}
\end{remark}
When the kernel of an energy form is nontrivial we need the following concept of invariant sets to describe it. 
\begin{definition}[Invariant sets and irreducibility] \label{definition:invariant sets}
Let $E$ be an energy form. A measurable set $A \subseteq X$ is called $E$-{\em invariant} if for every $f \in D(E)$ the equality
$$E(f) = E(1_A f) + E(1_{X\setminus A} f)$$
holds. The collection of all $E$-invariant sets is denoted by $\mathcal{I}(E)$. The energy form $E$ is called {\em irreducible} or {\em ergodic} if every  $A \in \mathcal{I}(E)$ satisfies $m(A)  = 0$ or $m(X\setminus A) =  0$. 
\end{definition}

\begin{remark}
\begin{itemize}
 \item If the underlying space $X$ has at least two measurable sets which do not coincide $m$-a.e., then  irreducible energy forms  are nontrivial.
 \item The notion of irreducibility is also inspired by the corresponding one for Dirichlet forms, cf. \cite[Section~1.6]{FOT}. 
\end{itemize}
\end{remark}

The following lemma provides a useful criterion for checking whether or not a given set is invariant.

\begin{lemma}\label{lemma:characterization invariant sets}
Let $E$ be an energy form and let $A \subseteq X$ be measurable. The following assertions are equivalent.
\begin{itemize}
 \item[(i)] $A$ is $E$-invariant.
 \item[(ii)] For all $f \in D(E)$, the inequality
 $$E(1_A f) \leq E(f)$$
 holds.
\end{itemize}
\end{lemma}
\begin{proof}
 It suffices to prove the implication (ii) $\Rightarrow$ (i). To this end, let a measurable $A\subseteq X$ be given such that for all $f \in D(E)$ the inequality  
 $$E(1_A f) \leq E(f) $$
 holds. This inequality implies that for $f \in D(E)$ we have $1_{A} f  \in D(E)$ and, since $D(E)$ is a vector space, $1_{X\setminus A} f = f - 1_A f\in D(E)$. For   $\varepsilon >0$, we obtain
 $$E(1_{A} f) = E(( 1_{A} f \pm \varepsilon 1_{X \setminus A} f)1_{A})  \leq  E(1_A f) \pm 2 \varepsilon E(1_A f,  1_{X\setminus A} f) + \varepsilon^2 E( 1_{X\setminus A} f) .$$
Letting $\varepsilon \to 0+$ shows $E(1_{A} f, 1_{X\setminus A} f) = 0$. Therefore, $A$ is $E$-invariant.  
\end{proof}

\begin{lemma}
 Let $E$ be an energy form. The collection of $E$-invariant sets $\mathcal{I}(E)$ is a $\sigma$-algebra.
\end{lemma}
\begin{proof}
 Obviously,  $X \in \mathcal{I}(E)$ and $A \in \mathcal{I}(E)$ implies $X \setminus A \in  \mathcal{I}(E)$. Next, we prove that $\mathcal{I}(E)$ is stable under finite intersections. For $A,B \in \mathcal{I}(E)$ and $f \in D(E)$, we obtain
 $$E(1_{A \cap B}f)  = E(1_A (1_Bf)) \leq E(1_Bf) \leq E(f). $$
 Hence, Lemma~\ref{lemma:characterization invariant sets} yields $A \cap B \in \mathcal{I}(E)$. 
 
 What we have shown also implies that $\mathcal{I}(E)$ is stable under finite unions. It remains to prove its stability under countable unions. To this end, let $A_1, A_2 ,\ldots \in \mathcal{I}(E)$ be given. For each $f \in D(E)$, the equation
 $$1_{\bigcup_{i=1}^\infty A_i} f = \lim_{n \to \infty} 1_{\bigcup_{i=1}^n A_i} f  $$
 holds in the $L^0(m)$-topology. The lower semicontinuity of $E$ implies
 $$E\left(1_{\bigcup_{i=1}^\infty A_i} f \right) \leq \liminf_{n \to \infty} E\left( 1_{\bigcup_{i=1}^n A_i} f\right) \leq E(f). $$
 Another application of Lemma~\ref{lemma:characterization invariant sets} shows $\bigcup_{i=1}^\infty A_i \in \mathcal{I}(E)$. This finishes the proof. 
\end{proof}

\begin{remark}
 The lower semicontinuity of $E$ even implies a stronger statement about $\mathcal{I}(E)$. It shows that after identifying sets that differ on a set of measure $0$ the collection $\mathcal{I}(E)/\hspace*{-1mm}\sim$ is a Dedekind complete subalgebra of the measure algebra of $m$. 
\end{remark}

The following theorem is the main result about the kernel of an energy form when it is nontrivial. 

\begin{theorem} \label{thm:kernel of an energy form}
 Let $E$ be an energy form. Then $\ker E \subseteq L^0(X,\mathcal{I}(E),m)$. In particular, if $E$ is irreducible, then  $\ker E \subseteq \IR 1$. 
\end{theorem}
\begin{proof}
 Let $h \in \ker E \setminus\{0\}$ be given. By the contraction properties of $E$ we can assume $h \geq 0$. For $\alpha \geq 0$, we set $A_\alpha := \{ h > \alpha \}$. It suffices to show that $A_\alpha$ is invariant. We have  $ (n (h - h\wedge \alpha) )\wedge 1 \overset{m}{\to} 1_{A_\alpha}$,  as $n \to \infty$. From the lower semicontinuity of $E$ and its contraction properties, we infer
 $$E(1_{A_\alpha}) \leq \liminf_{n\to \infty} E((n (h - h\wedge \alpha) )\wedge 1) =  0.$$
 This observation and Theorem~\ref{theorem:algebraic and order properties} yield  
 $$E(1_{A_\alpha} f) \leq E(f) \text{ for all }f \in  D(E).$$
 Hence, Lemma~\ref{lemma:characterization invariant sets} shows the invariance of $A_\alpha$. For the 'in particular' statement note that the triviality of invariant sets implies $L^0(X,\mathcal{I}(E),m) = \IR  1$.
\end{proof}

\begin{remark}
\begin{itemize}
 \item In the proof of the previous theorem we have used that  the irreducibility of $E$ implies that $L^0(X,\mathcal{I} (E),m)$ is isomorphic to $\IR $. We warn the reader that this is not a topological but a purely algebraic statement. Indeed, if $m(X) = \infty$ and any $A \in  \mathcal{I} (E)$ satisfies $m(X \setminus A) = 0$ or $m(A) = 0$, then $ L^0(X,\mathcal{I} (E),m)$ carries the trivial topology.

 \item The statement for irreducible Dirichlet forms and their extended spaces is certainly well known, see e.g. \cite[Theorem~2.1.11]{CF}. For forms which are not irreducible, we could not find a reference.
 \end{itemize}
\end{remark}

We now prove several corollaries to the previous theorem. We first show that irreducible energy forms satisfy the main assumption of Theorem~\ref{theorem:characterization closedness hilbert space}, i.e., their form topology is metrizable. After that we discuss the kernel of  recurrent energy forms and  a dichotomy of recurrence and transience for irreducible forms.
 
 \begin{corollary}\label{corollary:irreducible forms are metrizable}
  Let $E$ be an irreducible energy form and assume $m(X) > 0$. Let $(f_i)$ be a net in $D(E)$. The following assertions are equivalent.
  \begin{itemize}
   \item[(i)] $f_i \to f$ with respect to the form topology $\tau(m)_E$.
   \item[(ii)] $\| f_i - f\|_E \to 0$ and there exists some $U \in \Bf$ with $m(U)>0$ such that 
   $$\int _U |f_i-f| \wedge 1\, {\rm d}m \to 0.$$
  \end{itemize}
In particular, the form topology $\tau(m)_E$ is metrizable. 
 \end{corollary}
 \begin{proof}
 (i) $\Rightarrow$ (ii): This follows from the definition of the form topology and the fact that the localizability of $m$ and $m(X) > 0$ imply the existence of a set $U \in \Bf$ with $m(U)>0$.
 
(ii) $\Rightarrow$ (i): For $U \in \Bf$ with $m(U) >0$, we consider the quadratic form 
$$E_U:L^0(m) \to [0,\infty],\, f \mapsto E_U(f):= E(f) + \int_{U} f^2 {\rm d}m.$$
As a sum of two energy forms, it is an energy form itself.  Furthermore, the irreducibility of $E$ and Theorem~\ref{thm:kernel of an energy form} imply $\ker E_U = \{0\}$. 

Let $(f_i)$ be a net in $D(E)$ with $\|f_i\|_E \to 0$ and suppose
$$\int _U |f_i| \wedge 1\, {\rm d}m \to 0$$
for some $U \in \Bf$ with $m(U) > 0$. By the contraction properties of $E$, we  also have $\||f_i|\wedge 1\|_E \to 0$. Therefore, we obtain $E_U( |f_i|\wedge 1 ) \to 0$. Now, Theorem~\ref{theorem:continuous embedding of energy forms} applied to $E_U$ yields $|f_i|\wedge 1 \to 0$  in $L^0(m)$ and finishes the proof. 
 \end{proof}
\begin{remark}
   For proving the metrizability of the form topology $\tau(m)_E$ the condition that $E$ is irreducible can be weakened. In fact, it would be sufficient to find countably many invariant sets $A_1,A_2,\ldots$ with $X = \cup_n A_n$ such that for each $n$ the form $f \mapsto E(1_{A_n}f)$ is irreducible. We refrain from giving details. 
\end{remark}

 In view of this corollary, Theorem~\ref{theorem:characterization closedness hilbert space} and Theorem~\ref{theorem:continuous embedding of energy forms} and by studying the examples of Section~\ref{section:the definition and main examples}, one is tempted to make the following conjecture.

 \begin{conjecture} \label{conjecture:hilbert space}
  If $E$ is an irreducible energy form, then  $(D(E)/\ker E,E)$ is a Hilbert space.
 \end{conjecture}

  Unfortunately, we were not able to prove it nor to find a counterexample. To the best of our knowledge, its validity is even unknown for irreducible extended Dirichlet forms. A partial result is contained in \cite{Osh}. There, it is shown that for an irreducible Harris recurrent Dirichlet form $\E$ there exists a nonnegative function $g \in L^1(m)$ such that 
  $$\int_X |f - L(f)|g\D m \leq \Ee(f) \text{ for all } f \in D(\Ee)^{1/2},$$
  where $L(f) = \|g\|_1^{-1} \int_X f g \D m.$ From this Poincar\'{e} type inequality, the conjecture can  be easily deduced with the help of Theorem~\ref{theorem:characterization closedness hilbert space}.

The next corollary shows that the kernel of a recurrent energy form can be computed explicitly.

\begin{corollary}\label{corollary:kernel recurrent forms}
An energy form $E$ is recurrent if and only if $\ker E =  L^0(X,\mathcal{I} (E),m)$.
\end{corollary}
\begin{proof}
Constant functions are always measurable. Therefore, we obtain $1 \in \ker E$ whenever $\ker E =  L^0(X,\mathcal{I} (E),m)$. It remains to prove the opposite implication. By the previous theorem, it suffices to show that recurrence yields $ L^0(X,\mathcal{I} (E),m) \subseteq \ker E$.
 
   To this end, assume that $E$ is recurrent and let $A \in \mathcal{I}(E)$ be given. By the invariance of $A$ we obtain
 $$E(1_A) = E(1_A 1) \leq E(1) = 0.$$
 Hence, every simple function in $L^0(X,\mathcal{I} (E),m)$ belongs to $\ker E$. The lower semicontinuity of $E$ implies that $\ker E$ is a closed subspace of $L^0(m)$. Since simple functions are dense in $L^0(X,\mathcal{I} (E),m)$ with respect to convergence in measure, we obtain the statement. 
\end{proof}

The last corollary to  Theorem~\ref{thm:kernel of an energy form} which we present is the following.

\begin{corollary}[Dichotomy of recurrence and transience] \label{corolllary:dichotomy recurrence and transience}
 An irreducible energy form is recurrent if and only if it is not transient. 
\end{corollary}
\begin{proof}
 Let $E$ be an irreducible energy form. If $E$ is recurrent, it cannot be transient by definition. Assume $E$ is not transient, i.e., $\ker E \neq \{0\}$. In this case, Theorem~\ref{thm:kernel of an energy form} shows $\ker E = \IR 1$ and so $E$ is recurrent.  This finishes the proof. 
\end{proof}

\begin{remark}
 The dichotomy of recurrence and transience is well-known for Dirichlet forms. Surprisingly, we were not able to find a version of Corollary~\ref{corollary:kernel recurrent forms} for extended Dirichlet forms in the literature.
\end{remark}

With a bit more effort, we can strengthen the observation of the previous corollary and obtain a decomposition of $E$ into its recurrent and transient part. For a measurable subset $A \subseteq X$, we let $\iota_A:L^0(A,\cB \cap A,m|_{\cB \cap A}) \to L^0(X,\cB,m)$ be the canonical extension operator that extends a function by $0$ outside of $A$. Here, $\cB \cap A$ denotes the $\sigma$-algebra of measurable subsets of $A$ and $m|_{\cB \cap A}$ is the restriction of $m$ to this $\sigma$-algebra.  

\begin{theorem}[Decomposition into recurrent and transient part]\label{theorem:decomposition into recurrent and transient part}
Let $E$ be an energy form on $L^0(X,\cB,m)$. There exists a set $S \in \mathcal{I}(E)$ such that 
$$E_R:L^0(S,\cB \cap S,m|_{\cB \cap S}) \to [0,\infty], \quad f \mapsto E(\iota_S f)$$
is a recurrent energy form and 
$$E_T:L^0(X\setminus S,\cB \cap (X\setminus S),m|_{\cB \cap (X\setminus S)}) \to [0,\infty], \quad f \mapsto E(\iota_{X\setminus S} f)  $$
is a transient energy form.   In particular, for each $f \in L^0(m)$ we have
$$E(f) = E_R(f|_S) + E_T(f|_{X \setminus S}). $$
\end{theorem}
\begin{proof}
 We let  $\mathcal{C}:= \{C \in \cB \mid 1_C \in \ker E\}$. Recall that we order sets by inclusion up to $m$ measure zero. Hence, the net $(1_C)_{C\in \mathcal{C}}$ is monotone increasing and bounded by the constant function $1$. According to Proposition~\ref{proposition:convergence of monotone nets}, it converges in $L^0(m)$ to the function 
 $$f:= \sup_{C\in \mathcal{C}} 1_C.$$
 Clearly, $f$ only takes the values $0$ or $1$ and we let $S:= \{f = 1\}$ such that $f = 1_S$. The lower semicontinuity of $E$ implies
 $$E(1_S) \leq \liminf_{C\in \mathcal{C}} E(1_C) = 0.$$
Since $1_S \in \ker E$, Theorem~\ref{theorem:algebraic and order properties} combined with Lemma~\ref{lemma:characterization invariant sets} yield that $S$ is invariant. 

We define $E_R$ and $E_T$ as in the statement of the theorem. The invariance of $S$ implies that $E_R$ and $E_T$ are energy forms. From the equality
 $$E_R(1) = E(1_S) = 0$$
 we infer that the form $E_R$ is recurrent. It remains to show that $E_T$ is transient. To this end, we let $h \in \ker E_T$ be given. Without loss of generality we assume $h \geq 0$. For $\alpha > 0$, we set $A_\alpha := \{h > \alpha\} \subseteq X\setminus S$. With the same argument as in the proof of Theorem~\ref{thm:kernel of an energy form}, we obtain $E(1_{A_\alpha}) = E_T(1_{A_\alpha})  = 0$, which implies $ A_\alpha \in \mathcal{C}.$ Since $S$ is an essential supremum of $\mathcal{C}$ and $A_\alpha \subseteq X \setminus S$, we must have $m(A_\alpha) = 0$. Therefore, we obtain $h=0$. This finishes the proof. 
\end{proof}

\begin{remark}
  The previous theorem shows that we can develop the theory for recurrent and transient forms separately without loosing information. This is of course well-known for Dirichlet forms, see \cite[Lemma~1.6.4]{FOT}. Here, the novelty lies in the fact that we derived this decomposition by form methods only.
\end{remark}

In Dirichlet form theory it is common to assume that the domain  of the form is dense in the underlying $L^2$-space.  So-far, a lack of a similar condition did not cause any problems but it will in the next section. This is why we introduce the following concept.

\begin{definition}[Support of an energy form]
 Let $E$ be an energy form and let
 $$\mathcal{S} := \{ \{|f| \geq \alpha\} \mid f \in D(E), \alpha > 0\}.$$
 An essential supremum of $\mathcal{S}$ is called the {\em support of $E$} and is denoted by ${\rm supp}\,E$. The form $E$ is said to have full support if $\sE= X$. 
\end{definition}

\begin{remark}
\begin{itemize}
 \item The existence of $\sE$ is guaranteed by the localizability of $m$.
 \item The support of an energy form is only defined up to a set  of measure zero. When treating it as a concrete set we always need to choose a representative. In this sense, the condition $\sE = X$ is to be understood as $m(X \setminus \sE) = 0$. As already mentioned in Subsection~\ref{section:lebesgue spaces}, this causes some subtleties that we do not address explicitly. 
\end{itemize}
\end{remark}

The following proposition shows that each energy form can be viewed as an energy form of full support on a smaller underlying space.  As above, we tacitly extend functions by zero outside of their domains.  
 
\begin{proposition} \label{proposition:support of an energy form}
 Let $E$ be an energy form on $L^0(X,\cB,m)$. Its support is $E$-invariant and the quadratic form 
 $$\ow{E}:L^0({\rm supp} E, \cB \cap {\rm supp} E, m|_{\cB \cap {\rm supp} E}) \to [0,\infty],\quad f \mapsto  E(\iota_{{\rm supp} E} f)$$
 is an energy form of full support. 
\end{proposition}
\begin{proof}
 The $E$-invariance of ${\rm supp} E$ follows from the fact that for each $f \in D(E)$ the equality $f = f 1_{{\rm supp} E}$ holds. The rest is clear.
\end{proof}

\section{Superharmonic and excessive functions} \label{section:superharmonic functions}

This section is devoted to superharmonic and excessive functions. We study their properties and provide characterizations in terms of contraction properties for the form that are not induced by normal contractions.  These additional contractions properties are a main technical ingredient in later sections.

\begin{definition}[Superharmonic and excessive functions] \label{definition:superharmonic function}
 Let $E$ be an energy form. A  function $h \in D(E)$ is called {\em $E$-superharmonic} if for each nonnegative  $\psi \in D(E)$ the inequality
 $$E(h,\psi) \geq 0$$
 holds. The collection of all superharmonic functions is denoted by $S(E)$. A function is called {\em $E$-excessive} if it belongs to the $\tau(m)$-closure of $S(E)$ in $L^0(m)$. 
\end{definition}

\begin{remark}
When $E$ is an extended Dirichlet form we will see in the discussion after Theorem~\ref{theorem:characterization of excessive functions} that the above notion of excessive function and the one for (extended) Dirichlet forms coincide. 
\end{remark}

We first show that superharmonic functions can be characterized as minimizers of the energy. As a preparation, we need the following lemma. It provides further contraction properties, which are not induced by normal contractions. 

\begin{lemma} \label{lemma:superharmonic functions as cutoff}
 Let $E$ be an energy form and let $h$ be $E$-excessive. For each $f \in D(E)$, the following inequalities hold.
 \begin{itemize}
  \item[(a)] $E(f \wedge h) \leq E(f).$
  \item[(b)] $E((f - h)_+) \leq E(f).$
 \end{itemize}
 If, additionally, $h$ is $E$-superharmonic, then
  \begin{itemize}
 \item[(c)] $E(f \wedge h) \leq E(f \wedge h,f)$.
  \end{itemize}
\end{lemma}
\begin{proof} 
We start with statement (c). The identity $f \wedge h = 1/2 (f + h - |f - h|)$, the contraction properties of $E$ and the superharmonicity of  $h$ imply
\begin{align*}
E(f \wedge h,f)-  E(f \wedge h)  &= \frac{1}{4} \left( E(f) - E(h) - E(|f-h|) + 2 E(h,|f-h|)\right)\\
&\geq  \frac{1}{4} \left( E(f) - E(h) - E(f-h) + 2 E(h,|f-h|)\right)\\
&= \frac{1}{2} E(h, f - h + |f-h|)\\
&= E(h,(f-h)_+) \geq 0.
\end{align*}

By the lower semicontinuity of $E$, it suffices to show the statements (a) and (b) in the case when $h$ is superharmonic. Applying (c), we obtain 
$$E(f \wedge h) \leq E(f \wedge h,f) \leq E(f \wedge h)^{1/2}E(f)^{1/2},$$
 which proves (a). For inequality (b), we use $(f -h)_+ = f - f \wedge h$ and (c) to conclude
$$E((f-h)_+,f) - E((f-h)_+) = E(f - f \wedge h, f \wedge h ) = E(f,f\wedge h) - E(f \wedge h) \geq 0.$$
This computation yields 
$$E((f-h)_+) \leq E((f-h)_+,f) \leq E((f-h)_+)^{1/2}E(f)^{1/2}$$
and finishes the proof.
\end{proof}

\begin{remark}
\begin{itemize}
 \item Let $\E$ be a Dirichlet form with associated semigroup $(T_t)$. \cite[Proposition~1.2]{MR} implies that for any $\alpha > 0$,  an $\E_\alpha$-superharmonic function  $h$ is $\alpha$-excessive in the classical sense, i.e., it satisfies $e^{-t \alpha} T_t h \leq h$ for all $t>0$. In this situation, assertion (a) of the previous lemma is contained in \cite[Theorem~2.6]{MOR}. For extended Dirichlet forms and excessive functions in the classical sense  (which are not to be confused with $\alpha$-excessive functions and which will be shown to be $\Ee$-excessive in our sense) inequality (a) of the previous lemma is given by \cite[Proposition~2.10]{Kaj}. There, the proof uses methods from \cite{Shi}.  In contrast to the known proofs, we emphasize that ours is short and (almost) purely algebraic.  
 
 \item It will turn out that this rather innocent looking lemma is one of the main lemmas of this thesis. 
\end{itemize}
\end{remark}

The following characterization of superharmonic functions is (almost) a direct consequence of the previous lemma.
 
\begin{theorem}[Characterization of superharmonic functions] \label{theorem:characterization of superharmonic functions}
 Let $E$ be an energy form and let $h \in D(E)$. The following assertions are equivalent.
 \begin{itemize}
  \item[(i)] $h$ is superharmonic.
  \item[(ii)] There exists $g \in L^0(m)$ with $g \leq h$ such that 
  $$E(h) = \inf\{ E(f) \mid f \in L^0(m) \text{ with }f \geq g\}.$$
  \item[(iii)] $E(h) = \inf \{E(f) \mid f \in L^0(m) \text{ with }f \geq h\}.$
  \item[(iv)] For all $f \in L^0(m)$, the inequality $E(f \wedge h) \leq E(f)$ holds.
 \end{itemize}
\end{theorem}
\begin{proof}
 (i) $\Rightarrow$ (iv): This follows from Lemma~\ref{lemma:superharmonic functions as cutoff}.
 
 (iv) $\Rightarrow$ (iii):  Let $(f_n)$ be a sequence in $D(E)$ with $f_n \geq h$ and
 $$E(f_n) \to \inf\{E(f) \mid f \in L^0(m) \text{ with } f \geq h \}, \text{ as } n \to \infty.$$
 By (iv) we can assume $f_n \leq h$ and obtain $f_n = h$ for all $n$. This shows (iii). 
 
 (iii) $\Rightarrow$ (ii): This is obvious.
 
 (ii) $\Rightarrow$ (i): Let $\varepsilon > 0$ and let $\psi\in D(E)$ nonnegative. By $g \leq h$ we have $h + \varepsilon \psi \geq g$. Hence, (ii) implies
 $$E(h) \leq E(h + \varepsilon \psi) =  E(h) + 2\varepsilon E(h,\psi) + \varepsilon^2 E(\psi). $$
 Since $\varepsilon > 0$ was arbitrary, this computation shows $E(h,\psi) \geq 0$, i.e., $h$ is superharmonic.
\end{proof}

\begin{remark}
 \begin{itemize}
  \item The fact that certain minimizers of the energy are superharmonic is well known in potential theory. The idea of the proof of (ii) $\Rightarrow$ (i) is standard, see e.g. the proof of \cite[Lemma~2.1.1]{FOT}.

  \item For excessive functions, a somewhat weaker characterization in terms of contractions can be obtained under the additional assumption that the form has full support, see Theorem~\ref{theorem:characterization of excessive functions} below.
 \end{itemize}
\end{remark}

\begin{corollary}
 Let $E$ be an energy form and let $h$ be excessive. Then $h \in D(E)$ if and only if $h$ is superharmonic. 
\end{corollary}
\begin{proof}
 Let $h\in D(E)$ be excessive. By Lemma~\ref{lemma:superharmonic functions as cutoff} it satisfies $E(f \wedge h) \leq E(f)$ for each $f \in D(E)$. Therefore, the previous theorem yields the statement. 
\end{proof}

\begin{corollary}\label{corollary:concatenation superharmonic functions}
 Let $E$ be an energy form. Let $h  \in D(E)$ be superharmonic and let $h_1,h_2 \in L^0(m)$ be excessive. Then $h \wedge h_1$ is superharmonic and $h_1 \wedge h_2$ is excessive. Furthermore, for each constant $C \geq 0$, the function $h \wedge C$ is superharmonic and $h_1 \wedge C$ is excessive. 
\end{corollary}
\begin{proof}
 Let $h$ superharmonic and $h_1,h_2$ excessive.  Lemma~\ref{lemma:superharmonic functions as cutoff} implies $E(h \wedge h_1) \leq E(h)$, i.e., $h \wedge h_1 \in D(E)$. Another application of Lemma~\ref{lemma:superharmonic functions as cutoff} yields
 $$E(f \wedge h \wedge h_1) \leq E(f \wedge h) \leq E(f)  %
 \text{ for all }f \in D(E).$$
 With this at hand the superharmonicity of $h \wedge h_1$ follows from Theorem~\ref{theorem:characterization of superharmonic functions}. That $h_1\wedge h_2$ is excessive is an immediate consequence and the 'furthermore' statement can be inferred along the same lines.
\end{proof}

After having characterized superharmonic functions as minimizers of the energy we can now prove their existence. Indeed, we show the existence of a whole family of superharmonic functions.

\begin{proposition}\label{proposition:existence superharmonic functions}
 Let $E$ be an energy form. For every $f \in D(E)$ and all $\alpha> 0$, there exists a superharmonic function $h_{f,\,\alpha}$ such that
 $$ 1_{\{|f| \geq \alpha\}} \leq h_{f,\,\alpha} \leq  1 \text{ and } E(h_{f,\,\alpha}) = \inf \{ E(g) \mid g \in L^0(m) \text{ with } g \geq 1_{\{|f| \geq \alpha\}}\}. $$
 Moreover, these functions can be chosen such that $h_{g,\, \beta} \geq h_{f,\, \alpha}$ whenever $|g| \geq |f|$ and $\beta \leq \alpha$. In particular, if $E$ is nontrivial, then there exists a nontrivial superharmonic function. 
\end{proposition}
\begin{proof} 
By Theorem~\ref{theorem:decomposition into recurrent and transient part}, we can assume that $E$ is either recurrent or transient. If $E$ is recurrent, then there is nothing to show since in this case $1 \in \ker E$ is superharmonic and satisfies all the required properties. 

 Now, assume $\ker E  = \{0\}$. For $f \in D(E)$ and $\alpha> 0$, we consider the convex set 
 $$ K_{f,\alpha}:= \{g \in D(E) \mid g \geq 1_{\{|f| \geq \alpha\}} \}.$$
 By the contraction properties of $E$, we have $\alpha^{-1}|f| \in K_{f,\alpha}$ and by Theorem~\ref{theorem:continuous embedding of energy forms}, $K_{f,\alpha}$ is closed in the Hilbert space $(D(E),E)$. Approximation theory in Hilbert spaces yields the existence of a projection $P:D(E)\to K_{f,\alpha}$ such that $P(h)$ is the unique element in $K_{f,\alpha}$ that satisfies
 $$E(h - P(h)) = \inf_{g \in K_{f,\alpha}} E(h - g).$$
 According to Theorem~\ref{theorem:characterization of superharmonic functions}, the function $h_{f,\, \alpha} := P(0) \in K_{f,\alpha}$ is superharmonic. Its is boundedness by $1$ follows from the contraction properties of $E$. 

 For the 'moreover' statement, we let $g \in D(E)$ with $|g| \geq |f|$ and $\beta \leq \alpha$ be given.  We have $h_{g,\, \beta} \geq 1_{\{|f| \geq \alpha\}}$, which implies $h_{g,\, \beta} \wedge h_{f,\, \alpha} \in K_{f, \alpha}$.  Since $h_{g,\beta}$ is superharmonic, Theorem~\ref{theorem:characterization of superharmonic functions} yields $E(h_{g,\, \beta} \wedge h_{f,\, \alpha})  \leq E(h_{f,\, \alpha})$. As $h_{f,\,\alpha}$ is the unique minimizer of $E$ on $K_{f,\alpha}$, we obtain $h_{g,\, \beta} \wedge h_{f,\, \alpha}  = h_{f,\, \alpha}$, i.e.,   $h_{g,\, \beta} \geq h_{f,\, \alpha}.$  This finishes the proof.
\end{proof}

The previous proposition set the stage for characterizing excessive functions in terms of contractions properties.

\begin{theorem}[Characterization excessive functions]\label{theorem:characterization of excessive functions}
 Let $E$ be an energy form of full support. For $h \in L^0(m)$ the following assertions are equivalent. 
 \begin{itemize}
 \item[(i)] $h$ is $E$-excessive. 
  \item[(ii)] For all $f \in D(E)$ the inequality $E(f \wedge h) \leq E(f)$ holds.
 \end{itemize}
\end{theorem}
\begin{proof}
 (i) $\Rightarrow$ (ii): This follows from Lemma~\ref{lemma:superharmonic functions as cutoff}.
 
 (ii) $\Rightarrow$ (i): For $f \in D(E)$ and $\alpha > 0$, we choose a superharmonic function $h_{f,\, \alpha}$ as in Proposition~\ref{proposition:existence superharmonic functions}. Combining assertion (ii) with Theorem~\ref{theorem:characterization of superharmonic functions} shows that for each $n$ the function 
 $$h_{(n,f,\,\alpha)} := h \wedge (n h_{f,\, \alpha})\in D(E)$$
 is superharmonic. We introduce a preorder on the set of triplets
 $$I:= \{ (n,f,\alpha) \mid n \in \IN, \alpha >0, f\in D(E)\}$$
 by letting $(n,f,\alpha) \prec (n',f',\alpha')$ if  and only if $n \leq n', \alpha \geq \alpha'$ and $|f| \leq |f'|$. Clearly, $(I,\prec)$ is upwards directed. By Proposition~\ref{proposition:existence superharmonic functions} the net $(h_i)_{i \in I}$ is monotone increasing. Since it is also bounded by $h$, Proposition~\ref{proposition:convergence of monotone nets} shows that it converges in $L^0(m)$ to its supremum. The assumption that $E$ has full support and the properties of the $h_{f,\,\alpha}$ yield
 $$\sup\{h_{f,\, \alpha} \mid  f \in D(E),\alpha >0\} = 1.$$
 As a consequence, we obtain 
 $$h = \sup_{i \in I} h_i  = \lim_{i \in I} h_i.$$
 This finishes the proof. 
 \end{proof}

 \begin{remark}
 \begin{itemize}
  \item The theorem is not true for general energy forms. Consider the set $X = \{a,b\}$ equipped with the counting measure $\mu$  and the energy form
  $$E:L^0(\mu) \to [0,\infty], \, f \mapsto E(f):= \begin{cases}
            f(a)^2 & \text{ if } f(b) = 0\\
            \infty & \text{ else}
           \end{cases}.
$$
A function $h$ is $E$-superharmonic if and only if $h(a) \geq 0$ and $h(b) = 0$. Therefore, the constant function $1$  satisfies $E(f \wedge 1) \leq E(f)$ but cannot be approximated by superharmonic functions. 

Proposition~\ref{proposition:support of an energy form} shows that we can always consider the energy form on its support as an energy form of full support. In this sense, the condition on the support is not  very restrictive. In the example, we have ${\rm supp}\,E = \{a\}$. Therefore, $1$ is $E$-excessive when $E$ is considered on the space $L^0(\{a\},\mu)$.
 
 \item  Classically, for a Dirichlet form $\E$ with associated semigroup $(T_t)$, an $h \in L^0(m)$ is called excessive if $T_t|h| < \infty$ $m$-a.e. and $T_t h \leq h$ for all $t>0$. It is shown in \cite[Proposition~2.10]{Kaj} that this is equivalent to $\Ee (f \wedge h) \leq \Ee(f)$ for all $f \in D(\Ee)$. The Dirichlet forms in \cite{Kaj} are densely defined on $L^2(m)$ and therefore have full support. Hence, our theorem and the previous discussion show  that  $h$ is $\Ee$-excessive in our sense if and only if it is  excessive in the classical sense. In the light of the this remark, it seems to be a new observation that excessive functions are precisely the limits of  superharmonic functions in the topology of local convergence in measure.  
 \end{itemize}
 \end{remark}

After having proven that there exist superharmonic (and excessive) functions we show that their sign and their support are not arbitrary. 

\begin{proposition}\label{lemma:superharmonic functions are positive}
Let $E$ be an energy form and let $h$ be $E$-excessive. Then 
$$E(h_-) = 0.$$
In particular, if $E$ is transient, then every $E$-excessive function is nonnegative. 
\end{proposition}
\begin{proof}
Let $h$ be $E$-superharmonic. According to Lemma~\ref{lemma:cutoff for functions with disjoint support}, the inequality $E(h_-,h_+) \leq 0$ holds. We obtain
$$ 0 \leq E(h,h_-) = E(h_+,h_-)  -  E(h_-) \leq 0, $$
which implies $E(h_-) = 0$. 

For an $E$-excessive $h$  we consider a net of $E$-superharmonic functions $(h_i)$ with $h_i \overset{m}{\to } h$.  The lower semicontinuity of $E$ yields
$$E(h_-) \leq \liminf_i E( (h_i)_-) = 0.$$
This finishes the proof. 
\end{proof}

The previous lemma already suggests that more can be said about excessive functions when the form is recurrent or transient. Before discussing this relation, we need the following lemma.

\begin{lemma}\label{lemma:support of superharmonic function is invariant}
 Let $E$ be an energy form and let $h \geq 0$ be $E$-excessive. The set $\{h>0\}$ is $E$-invariant.
\end{lemma}
\begin{proof}
 For $f \in D(E)$ and $n\geq0$ we set 
 $$f_n:= (f \wedge (n h)) \vee (-nh).$$
 Since $nh$ is excessive, Lemma~\ref{lemma:superharmonic functions as cutoff} implies
 $$ E(f_n) \leq E(f).$$
 From the convergence $f_n \overset{m}{\to} f 1_{\{h>0\}}$, as $n \to \infty$, and the lower semicontinuity of $E$ we infer
 $$E(f 1_{\{h>0\}}) \leq \liminf_{n \to \infty} E(f_n) \leq E(f).$$
 According to Lemma~\ref{lemma:characterization invariant sets}, this shows the invariance of $\{h >0\}$.
\end{proof}

\begin{proposition} \label{proposition: superharmonic functions are strictly positive}
 Let $E$ be an irreducible transient energy form. Then every nonzero $E$-excessive function is strictly positive. If, additionally, $E$ is nontrivial, then there exists a strictly positive $E$-superharmonic function.
\end{proposition}
\begin{proof}
Let $h\neq 0$ be $E$-excessive. By Proposition~\ref{lemma:superharmonic functions are positive} we have $h\geq 0$. Furthermore,  Lemma~\ref{lemma:support of superharmonic function is invariant} shows that $\{h>0\}$ is invariant. Since $E$ is irreducible and $h$ is nontrivial, this implies $h > 0$ $m$-a.e., what we claimed.

For the 'in particular' part we  use Proposition~\ref{proposition:existence superharmonic functions} to obtain a superharmonic function $0\neq h \in D(E)$. It is strictly positive by the above. 
\end{proof}

One application of the previous observation is that nontrivial irreducible energy forms always have full support. 

\begin{corollary}\label{corollary: irreducible forms have full support}
 Let $E$ be a nontrivial irreducible energy form. Then $E$ has full support. 
\end{corollary}
\begin{proof}
 According to Corollary~\ref{corolllary:dichotomy recurrence and transience}, the form $E$ is either recurrent or transient. In the recurrent case, $1 \in \ker E \subseteq D(E)$ implies $\sE = X$. In the transient case, the previous proposition guarantees the existence of a strictly positive function in $h \in D(E)$. But then 
 $$1_X = \sup_{\alpha> 0} 1_{\{h \geq \alpha\}}, $$
 and we obtain $\sE = X$. This finishes the proof.
\end{proof}

We finish this section by discussing excessive functions for recurrent forms. 
  
\begin{proposition} \label{proposition:superharmonic functions are constant for recurrent forms}
  Let $E$ be a recurrent energy form. Every $E$-excessive function belongs to $\ker E$.  If, additionally, $E$ is  irreducible, then every $E$-excessive function is constant.
\end{proposition}
\begin{proof}
 Let $h$ be $E$-excessive. By the recurrence of $E$, we have $1 \in \ker E$. Using the lower semicontinuity of $E$ and Lemma~\ref{lemma:superharmonic functions as cutoff} we obtain
 $$E(h) \leq \liminf_{n \to \infty} E(h \wedge (n 1)) \leq \liminf_{n \to \infty} E(n 1) = 0,$$
 which shows $h \in \ker E$.  It follows from Theorem~\ref{thm:kernel of an energy form} that $h$ is constant when $E$ is irreducible. This finishes the proof.
\end{proof}

If the form has full support and the domain of the energy form contains a nonconstant function, a converse to the previous proposition holds true.

\begin{theorem} \label{theorem:recurrence in terms of constant excessive functions}
 Let $E$ be an energy form of full support. The following assertions are equivalent.
 \begin{itemize}
  \item[(i)] $E$ is recurrent.
  \item[(ii)] Every $E$-excessive function belongs to $\ker E$. 
  \end{itemize}
   If, additionally, $E$ is irreducible and $D(E)$ contains a nonconstant function, then both assertions are equivalent to the following.
  \begin{itemize}
  \item[(iii)] Every $E$-excessive function is constant. 
 \end{itemize}
\end{theorem}
\begin{proof}
 The implications (i) $\Rightarrow$ (ii) and (i) $\Rightarrow$ (iii) (under the additional assumption) follow from Proposition~\ref{proposition:superharmonic functions are constant for recurrent forms}.
 
 (ii) $\Rightarrow$ (i): Since $E$ has full support, Theorem~\ref{theorem:characterization of excessive functions} shows that the constant function $1$ is  $E$-excessive. By (ii) it belongs to $\ker E$, i.e., $E$ is recurrent.

 Assume that $E$ is irreducible and $D(E)$ contains a nonconstant function.

 (iii) $\Rightarrow$ (i): Let $g \in D(E)$ be not constant.  After possibly rescaling $g$ and using the contraction properties of $E$, we can assume that $0 \leq g \leq 1$, $\|g\|_\infty = 1$ and $\|1-g\|_\infty = 1$. 
 
 Assume  $E$ were transient. With the same arguments as in the proof of Proposition~\ref{proposition:existence superharmonic functions}, we obtain unique $E$-superharmonic functions $h_1$  and $h_2$ that are bounded by $1$ and satisfy
 $$E(h_1) = \inf \{E(f) \mid f \in L^0(m) \text{ with } f \geq g\}$$
 and
 $$E(h_2) = \inf \{E(f) \mid f \in L^0(m) \text{ with }  f \geq 1-g\}.$$
 According to assertion (iii), $h_1$ and $h_2$ must be constant. Since $\|g\|_\infty = \|1-g\|_\infty = 1$, we conclude $h_1 = h_2  = 1$.  For $0 < \alpha < 1$, the properties of $h_1$ and $h_2$ imply 
 $$E(1) = E(h_1)  \leq E(g + \alpha (1-g)) = \alpha^2 E(1) + (1-\alpha)^2 E(g) + 2\alpha (1-\alpha) E(1,g)$$
 and 
 $$E(1) = E(h_2)  \leq E(1-g + \alpha g) = E(1) + (1-\alpha)^2 E(g) - 2 (1-\alpha) E(1,g).$$
 Rearranging these inequalities and dividing by $1-\alpha > 0$ amounts to
 $$(1+\alpha) E(1) \leq (1-\alpha) E(g) + 2\alpha E(1,g)$$
 and 
 $$2 E(1,g) \leq (1-\alpha)E(g).$$
 Letting $\alpha \to  1-$ yields $E(1) \leq E(1,g) \leq 0$, a contradiction to the transience of $E$. 
\end{proof}

\begin{remark}
\begin{itemize}
 \item We have already noted that excessive functions in our sense and in the sense of Dirichlet form theory coincide. Therefore, the previous theorem is a generalization of the main result of \cite{Kaj} to energy forms. Again, we would like to point out the algebraic character of our proof. 
 \item Without the assumption on the existence of a nonconstant function in the domain of the energy form the previous theorem does not hold. For example, consider $X = \{a\}$ equipped with the counting measure $\mu$. The energy form $$E:L^0(\mu) \to [0,\infty],\, f\mapsto E(f) = f(a)^2$$ is transient and only possesses constant functions in its domain. 
\end{itemize}
\end{remark}

\section{Capacities}
 
In this section  we introduce the notion of capacity for energy forms.  We prove that capacity theory  is trivial for recurrent forms and provide some properties of capacities in the transient case. Since no new observations or proofs appear here, we only prove the theorems that we use later and refer the reader to  \cite{FOT} and \cite{MR} for more detailed accounts.  

\begin{definition}[Capacity]
Let $E$ be an energy form.  A nonnegative function  $h \in L^0(m)$ is called {\em admissible} if 
$$E(f \wedge h) \leq E(f)\text{ for all }f \in D(E).$$
Let $h\in L^0(m)$ be admissible and let $U\subseteq X$ be measurable. The {\em capacity of $U$ with respect to $E$ and $h$} is defined by 
$$\ceh(U) := \inf\{E(f) \mid f \in L^0(m) \text{ with } f 1_U \geq h1_U \}.$$
 If, additionally, $X$ is a topological space and $m$ is a Borel measure on $X$, for any $A \subseteq X$ we let the {\em topological capacity of $A$ with respect to $E$ and $h$} be given by
$$\ceh^{\rm top}(A) = \inf \{ \ceh(O)  \mid O \text{ open and } A \subseteq O\}.$$
\end{definition}

\begin{remark}
\begin{itemize}
 \item The contraction properties of energy forms imply that constant functions are admissible and Lemma~\ref{lemma:superharmonic functions as cutoff} shows that excessive functions and superharmonic functions are admissible as well. Indeed,  Theorem~\ref{theorem:characterization of excessive functions} yields that all admissible functions are excessive when the form has full support. 
 
 \item The advantage of using superharmonic functions instead of arbitrary admissible ones is that superharmonic functions always belong to the domain of the energy form and so the whole space has finite capacity. 
 
 \item We warn the reader not to confuse  $\ceh$ with $\ceh^{\rm top}$, which is studied in \cite{FOT} or \cite{MR}. In particular, $\ceh$ is only defined on the underlying $\sigma$-algebra and has the feature that $\ceh(U) = \ceh(V)$ whenever $m(U\setminus V \cup V \setminus U) = 0$.
 
 \item Below, we use $\ceh$ to measure the smallness of certain families of sets rather than the smallness of individual sets.  If $X$ is a topological space and $m$ is a Borel measure, then $\ceh^{\rm top}(A)$ can be seen as a measure for the smallness of the family of all open sets that contain $A$.  
\end{itemize}
\end{remark}

Depending on whether or not $E$ is transient, capacities are useful or not. This is discussed next.  

\begin{theorem}\label{theorem:capacities for recurrent forms}
 Let $E$ be an energy form. The following assertions are equivalent.
 \begin{itemize}
  \item[(i)] $E$ is recurrent.
  \item[(ii)] For every admissible $h$ and every measurable $U \subseteq X$ the equation $\ceh(U) = 0$ holds. 
  \item[(iii)] For every admissible $h$ the equation $\ceh(X) = 0$ holds. 
  \item[(iv)] The equation ${\rm cap}_{E,1} (X) = 0$ holds.
 \end{itemize}
\end{theorem}
\begin{proof}
 (i) $\Rightarrow$ (ii): Let $E$ be recurrent and let $h$ be admissible. We can argue exactly as in the proof of Proposition~\ref{proposition:superharmonic functions are constant for recurrent forms} to show that $h$ belongs to $\ker E$. Therefore, we obtain $\ceh (U) = 0$ for each measurable $U \subseteq X$.  
 
 (ii) $\Rightarrow$ (iii) $\Rightarrow$ (iv): This is trivial.
 
 (iv) $\Rightarrow$ (i): Let $f_n \geq 1$ be a sequence with $E(f_n) \to 0$.  By the contraction properties of $E$ we obtain
 $$E(1) =  E(f_n \wedge 1) \leq E(f_n).$$
 This shows $1 \in \ker E$ and finishes the proof. 
\end{proof}

 \begin{remark}
  \begin{itemize}
   \item An alternative formulation for the previous theorem is that capacity theory of an energy form is utterly useless if and only if the form is recurrent.
   \item If $X$ is a topological space and $m$ is a Borel measure on $X$, the previous theorem remains true with $\ceh$ being replaced by $\ceh^{\rm top}$. We refrain from giving details. 
  \end{itemize}

 \end{remark}

 The following theorem shows that situation is better for transient forms.

\begin{theorem} \label{theorem:superharmonic functions and capacities}
 Let $E$ be a transient energy form and let $h$ be admissible. For each measurable $U \subseteq X$ there exists $h_U \in L^0(m)$ with $h1_U \leq h_U \leq h $ and
 $$\ceh(U) = E(h_U).$$
 If, additionally, $\ceh(U) < \infty$, then $h_U$ is $E$-superharmonic and uniquely determined. 
\end{theorem}
\begin{proof}
 If $\ceh(U) = \infty$, then $h$ does not belong to $D(E)$. In this case, we can set $h_U = h$ and the claim follows. 
 
 Now, assume  $\ceh(U) < \infty$. We argue as in the proof of Proposition~\ref{proposition:existence superharmonic functions} to show the existence of $h_U$. The set
 $$K := \{f \in D(E) \mid  f1_U \geq h1_U \}$$
 is  convex  and nonempty. According to Theorem~\ref{theorem:continuous embedding of energy forms}, the space $(D(E),E)$ is a Hilbert space and $K$ is a closed subset.  The projection theorem in Hilbert spaces implies the existence of a unique $h_U \in D(E)$ that satisfies 
 $$E(h_U) = \inf\{E(f - 0) \mid f \in K\} = \ceh(U). $$
 Since $h$ is admissible, we obtain $E(h_U \wedge h) \leq E(h_U)$. The minimality of $E(h_U)$ and the uniqueness of $h_U$ show $h_U \leq h$. The superharmonicity of $h_U$ follows from Theorem~\ref{theorem:characterization of superharmonic functions}.  This finishes the proof.
\end{proof}

\begin{remark}
 The idea of the proof of the previous theorem  and its statement are well known in potential theory, cf. the proof of \cite[Lemma~2.1.1]{FOT}. We have already used it in the proof of Proposition~\ref{proposition:existence superharmonic functions}.
\end{remark}

 For later purposes, we note the following alternative formula for the capacity with respect to a superharmonic function. It seems to be new.

\begin{lemma}\label{lemma:alternative characterization capacity}
 Let $E$ be an energy form and let $h$ be admissible.  Then $\ceh(X) < \infty$ if and only if $h$ is $E$-superharmonic. In this case, for any measurable $U \subseteq X$ the equation 
 $$\ceh (U) = \inf \{E(h-g) \mid g \in D(E) \text{ with } g 1_U = 0\}$$
 holds.
\end{lemma}
\begin{proof}
 By definition superharmonic functions belong to the form domain. Therefore, if $h$ is $E$-superharmonic, then $\ceh(X) < \infty$.
 
 If $\ceh(X) < \infty$, there exists an $f \in D(E)$ with $f \geq h$. Since $h$ is admissible, we infer
 $$E(h) = E(f \wedge h) \leq E(f) < \infty,$$
 i.e., $h \in D(E)$.  It follows from Theorem~\ref{theorem:characterization of superharmonic functions} that admissible functions which belong to the form domain are superharmonic.
 
 For the alternative expression of the capacity we first consider a function $g \in D(E)$ with $g 1_U = 0$. It satisfies $(h - g)1_U \geq h 1_U$ and we obtain
 $$\ceh (U) \leq \inf \{E(h-g) \mid g \in D(E) \text{ with } g 1_U = 0\}.$$
 For the reverse inequality, we let $\varepsilon > 0$ and choose $f \in D(E)$ with $f1_U  \geq h1_U$ and $$E(f) \leq \ceh(U)  + \varepsilon.$$  Since $(h - f\wedge h) 1_U = 0 $ and $E(f \wedge h) \leq E(f)$, we obtain
 $$\ceh (U) + \varepsilon  \geq E(f\wedge h) = E(h - (h- f\wedge h)) \geq  \inf \{E(h-g) \mid g \in D(E) \text{ with } g 1_U = 0\}.$$
 This finishes the proof. 
\end{proof}

%
 
\section{Nests and local spaces}\label{section:nests and local spaces}

In this section  we introduce the space of functions that locally belong to the form domain of an energy form. Classically, thinking of Sobolev-type spaces or local $L^p$-spaces, we would set the local form domain to be the space of all functions that agree with functions in the form domain on relatively compact open sets. Since in general we have no topology on the underlying space at our disposal, we first introduce alternative classes of sets on which we compare arbitrary functions with functions in the form domain, so called nests.  We study properties of these nests and discuss the local form domain at the end of this section. 	

\begin{definition}[Nest]
 Let $E$ be an energy form. A collection of measurable sets $\cN$ is called {\em $E$-nest} if 
 $$D(E)_\cN := \{f \in D(E) \mid \text{ there exists } N \in  \cN \text{ s.t. } f 1_{X\setminus N} = 0\}$$
 is a dense subspace of $D(E)$ with respect to the form topology. 
\end{definition}

\begin{remark} \label{remark:nest}
 \begin{itemize}
 \item There are two assumptions in the previous definition. One is the denseness of $D(E)_\cN$ in $D(E)$. It reflects the idea that a nest should  exhaust the underlying space. The other requirement is that $D(E)_\cN$ is a vector space. This corresponds to the fact that sets in $\cN$ can be compared. For example, it is satisfied if for all $N_1,N_2 \in \cN$, there exists a set $N_3 \in \cN$ with $m(N_1 \cup N_2 \setminus N_3) = 0.$ In other words, it is satisfied when $\cN$ is upwards directed with respect to the usual preorder on measurable sets that we introduced in Subsection~\ref{section:lebesgue spaces}. 
 
   \item In the classical topological situation one considers the sets in a nest to be closed (or quasi-closed) and the nest itself to be countable and increasing, cf. \cite{FOT,MR}. For our purposes, we do not need any of these properties. A version of our notion of nest was introduced in \cite{AH}. 
   
  \item The standard example of a nest when $E$ is a regular energy form is 
  $$\mathcal{K} := \{G \subseteq X \mid G \text{ open and relatively compact}\}.$$
  In this situation, we have 
  $$C_c(X) \cap D(E) \subseteq D(E)_{\mathcal{K}}$$
  which implies the denseness of $D(E)_{\mathcal{K}}$. Its vector space properties follow from the previous discussion and the fact that the union of two open, relatively compact sets is again open and relatively compact. 
 \end{itemize}
\end{remark}

The following proposition shows that for a nest $\cN$ the algebraic properties of $D(E)$ pass to $D(E)_\cN$.

\begin{proposition}\label{proposition:properties of den}
 Let $E$ be an energy form and let $\cN$ be an $E$-Nest.
 \begin{itemize}
  \item[(a)] For every normal contraction $C:\IR^n \to \IR$ and $f_1,\ldots, f_n \in D(E)_\cN$, we have
  $$C(f_1,\ldots, f_n) \in D(E)_\cN.$$
  \item[(b)] $D(E)_\cN$ is a lattice and $D(E)_\cN \cap L^\infty(m)$ is an algebra. 
 \end{itemize}

\end{proposition}
\begin{proof}
 (a): Let $f_1,\ldots,f_n \in D(E)_\cN$ and let $C:\IR^n \to \IR$ be a normal contraction. The statement follows from Theorem~\ref{theorem:cutoff properties energy forms}, the inequality 
 $$|C(f_1,\ldots, f_n)|  \leq \sum_{i = 1}^n |f_i|$$
and the fact that $D(E)_\cN$ is a vector space.
 
 (b): This follows from the corresponding results for $D(E)$, the inequalities
 $$|f\wedge g| \leq |f| + |g| \text{ and } |f\vee g| \leq |f| + |g|,$$
 and the fact that $D(E)_\cN$ is a vector space. 
\end{proof}

Not all nests are equally well-suited for studying energy forms. The following class is of  importance for almost all our considerations below that involve nests.

\begin{definition}[Special nest] \label{definition:special nest}
Let $E$ be an energy form. An $E$-nest $\cN$ is called {\em special} if for each $N \in \cN$ there exists a function $g_N \in D(E)_\cN$ such that $g_N \geq 1_N$.
\end{definition}

\begin{remark}
 \begin{itemize}
  \item Special nests comprise two contrary features. A special nest is large enough such that it 'exhausts' the underlying space while its individual elements are 'small' in the sense that there exists some function in the form domain that equals one on them. Indeed, the sets are even so small that the function which equals one on them can be taken from the smaller space $D(E)_\cN$.
  
  \item We use the term 'special' for the described property of nests because it has a similar meaning as  the term 'special' in the concept of special standard core of a regular Dirichlet form, cf. \cite[Section~1.1]{FOT}. 
 \end{itemize}

\end{remark}

The following lemma shows that each energy form has a special nest.

\begin{lemma}\label{lemma:special nest}
Every energy form possesses a special nest. 
\end{lemma}
\begin{proof}
 Let $E$ be an energy form. We show that the system 
 $$\cN := \left\{\{|f| \geq \alpha\}\, \middle |\, f \in D(E) \text{ and } \alpha>0\right\}$$
 is an $E$-nest with the desired properties. 
 
We first prove the denseness of $D(E)_\cN $ in $D(E)$ with respect to the form topology. For $f\in D(E)$ and $\alpha > 0$ we set 
 $$f_\alpha := f - (f\wedge \alpha) \vee (-\alpha).$$
These $f_\alpha$ satisfy
 $$\{f_\alpha \neq 0\} = \{|f| \geq \alpha\} \in \cN,$$
 which shows $f_\alpha  \in D(E)_\cN$. Furthermore, Proposition~\ref{proposition:approximation by bounded functions} implies $f_\alpha \to f$ in the form topology, as $\alpha\to 0+$. 
 
 Using that $D(E)$ is a lattice, we obtain 
 $$\{|f| \geq \alpha\} \cup \{|g| \geq \beta\} \subseteq \{|f|\vee |g| \geq  \alpha \vee \beta\} \in \cN.$$
 This shows that $D(E)_\cN$ is a vector space.
 
 For proving that $\cN$ is special, we note that for each $f\in D(E)$ and $\alpha > 0$ we have $|f|_{\alpha/2} \in D(E)_\cN$ and 
 $$2 \alpha^{-1} |f|_{\alpha/2} \geq 1_{\{|f|\geq \alpha\}}.$$
 This finishes the proof.
\end{proof}

For regular energy forms, the collection of relatively compact open sets is the most prominent nest. The next lemma shows that it is special.

\begin{lemma}\label{lemma:special nest regular form}
 Let $E$ be a regular energy form.   The collection
$$\mathcal{K} := \{G \subseteq X \mid G \text{ open and relatively compact}\}$$ 
 is a special $E$-nest. 
\end{lemma}
\begin{proof}
 In Remark~\ref{remark:nest} we have already seen that $\mathcal{K}$ is a nest. It remains to show that $\mathcal{K}$ is special.  Let $G\subseteq X$ open and relatively compact. Locally compact separable metric spaces are completely regular. Hence, there exists $\varphi \in C_c(X)$ with $\varphi(x) \geq 2$ for each $x \in G$. Since $C_c(X) \cap D(E)$ is uniformly dense in $C_c(X)$, there exists a function $\ow{\varphi} \in C_c(X) \cap D(E)$ with $\|\varphi - \ow{\varphi}\|_\infty \leq 1/2$. It satisfies $\ow{\varphi} \in D(E)_\mathcal{K}$ and $\ow{\varphi}(x) \geq 1$ for each $x \in G$. 
\end{proof}

\begin{remark}
 In view of the previous lemma, it is sometimes helpful to think of special nests as generalizations of the collection of relatively compact open sets. 
\end{remark}

Special nests have another feature that allows them to serve as index sets for nets.  Recall the preorder that we introduced on measurable sets in Section~\ref{section:lebesgue spaces}. Namely, we say that two measurable sets $N,N' \subseteq X$ satisfy $N\prec N'$ if and only if $m(N \setminus N') = 0$. The following lemma shows that special nests are upwards directed with respect to this preorder.

\begin{lemma} \label{lemma:special nests are directed}
 Let $E$ be an energy form and let $\cN$ be a special $E$-nest. Then $(\cN,\prec)$ is upwards directed.
\end{lemma}
\begin{proof}
 For $i = 1,2$, let $N_i \in \cN$ and  $g_{N_i} \in D(E)_\cN$ with $g_{N_i} \geq 1_{N_i}$ be given. Since $D(E)_\cN$ is a vector space, we have $g_{N_1} + g_{N_2} \in D(E)_\cN$. Therefore, there exists a set $N_3 \in \cN$ such that $(g_{N_1} + g_{N_2})1_{X \setminus N_3} = 0$. The inequality $g_{N_1} + g_{N_2} \geq 1_{N_1 \cup N_2}$  implies $m((N_1 \cup  N_2) \setminus N_3) = 0$ and finishes the proof.
\end{proof}

\begin{lemma}[Refinement of nests]\label{lemma:refinement of nests}
Let $E$ be an energy form and let $\cN_1$ and $\cN_2$ be two $E$-nests. Their {\em refinement}
$$\cN_1 \wedge \cN_2 : = \{N_1 \cap N_2 \mid  N_1 \in  \cN_1, N_2\in \cN_2\}$$
is an $E$-nest. Moreover, if $\cN_1$ and $\cN_2$ are special, then $\cN_1 \wedge \cN_2$ is special.
\end{lemma}
\begin{proof} It follows from the definition of $\cN_1 \wedge \cN_2$ that the identity
$$D(E)_{\cN_1 \wedge\,  \cN_2} = D(E)_{\cN_1} \cap D(E)_{\cN_2}$$
holds. Therefore,  $D(E)_{\cN_1 \wedge\,  \cN_2}$ is a vector space and it remains to prove the denseness of $D(E)_{\cN_1 \wedge\, \cN_2}$ in $D(E)$ with respect to the form topology. Let $\ow{E}$ be the closure of the restriction of $E$ to $D(E)_{\cN_1 \wedge\, \cN_2}$. Since $E$ is an extension of $\ow{E}$, it suffices to show the inclusion $D(E) \subseteq D(\ow{E})$. 

Let $f\in D(E)$. There exists a net $(f_i)$ in $D(E)_{\cN_1}$ and and a net $(g_j)$ in  $D(E)_{\cN_2}$ that converge to $f$ in the form topology. We set  
$$h_{i,j} := (f_i \wedge |g_j|) \vee (-|g_j|) \in  D(E)_{\cN_1 \wedge\, \cN_2}.$$
From the lower semicontinuity of $\ow{E}$ we infer
$$\ow{E}(f) \leq \liminf_i \liminf_j \ow{E}( h_{i,j}) = \liminf_i \liminf_j E( h_{i,j}).  $$
Theorem~\ref{theorem:algebraic and order properties} yields the estimate
$$E(h_{i,j})  ^{1/2} \leq E(f_i)^{1/2} + 2 E(|g_j|)^{1/2} \leq E(f_i)^{1/2} + 2 E(g_j)^{1/2}. $$
Altogether, we obtain $\ow{E}(f) < \infty$, i.e., $f \in D(\ow{E})$.  This proves the denseness of $D(E)_{\cN_1 \wedge\, \cN_2}$ in $D(E)$.

Now, assume that $\cN_1$ and $\cN_2$ are special. For $N_i \in \cN_i$, $i=1,2$, we choose $\varphi_i \in D(E)_{\cN_i}$ with $\varphi_i \geq 1_{N_i}$. By the contraction properties of $E$ we can assume that these functions are bounded. Since $D(E)\cap L^\infty(m)$ is an algebra, we have $\varphi_1 \varphi_2 \in D(E)$ and $\varphi_1 \varphi_2 \geq 1_{N_1 \cap N_2}$. Moreover, if $\varphi_i$ vanishes outside of $M_i \in \cN_i$, then $\varphi_1 \varphi_2$ vanishes on $X \setminus (M_1 \cap M_2)$. This finishes the proof.
\end{proof}


For transient forms which possess a strictly positive superharmonic function nests can be characterized in terms of capacity. This is the content of the following proposition. 

\begin{proposition}\label{proposition:capacity characterisation of nests}
 Let $E$ be a transient energy form and let  $h\in D(E)$ be a strictly positive $E$-superharmonic function.  The following assertions are equivalent.
 \begin{itemize}
  \item[(i)] $\cN$ is an $E$-nest. 
  \item[(ii)] $D(E)_\cN$ is a vector space and $\inf_{N\in \cN} \ceh (X\setminus N)= 0.$
 \end{itemize}
\end{proposition}
\begin{proof}  We start with the following observation.  

{\em Claim:} The function $h$ belongs to the form topology closure of $D(E)_\cN$ if and only if $\inf_{N\in \cN} \ceh (X\setminus N)= 0.$ 

{\em Proof of the claim.} 
Lemma~\ref{lemma:alternative characterization capacity} shows that for each $N \in \cN$ we have
$$\ceh (X \setminus N) = \inf \{E(h - g) \mid g \in D(E)  \text{ with } g1_{X \setminus N} = 0 \}.$$
This proves the claim. \qedc

(i) $\Rightarrow$ (ii): This is a direct consequence of the previous claim.  

(ii) $\Rightarrow$ (i): We let $\ow{E}$ be the closure of the restriction of $E$ to $D(E)_\cN$. Since $E$ is an extension of $\ow{E}$, it suffices to show the inclusion $D(E) \subseteq D(\ow{E})$. 

According to the claim, assertion (ii) implies $h \in D(\ow{E})$, i.e., there exists a net $(h_i)$ in $D(E)_\cN$ that converges to $h$ in the form topology. Let $ f \in D(E)$ nonnegative. Since $h_i \in D(E)_\cN$, we have $f \wedge h_i \in D(E)_\cN$ and $E(f \wedge h_i) = \ow{E}(f \wedge h_i)$. The lower semicontinuity of $\ow{E}$ and Theorem~\ref{theorem:algebraic and order properties} imply
$$\ow{E}(f \wedge h)^{1/2}  \leq \liminf_i \ow{E}(f \wedge h_i)^{1/2} = \liminf_i E(f \wedge h_i)^{1/2} \leq E(f)^{1/2}+ E(h)^{1/2} < \infty.  $$
 The same argumentation yields that for $n \in \IN$ we have $f \wedge (nh) \in D(\ow{E})$ and, therefore,  $\ow{E} (f \wedge (nh)) = E (f \wedge (nh))$. From the lower semicontinuity of $\ow{E}$ and  Lemma~\ref{lemma:superharmonic functions as cutoff} we conclude
$$\ow{E}(f) \leq \liminf_{n\to \infty} \ow{E} (f \wedge (nh)) = \liminf_{n\to \infty} E (f \wedge (nh)) \leq E(f)< \infty.$$
Since any function in $D(E)$ can be split into positive and negative part, this finishes the proof.
 \end{proof}

\begin{remark}
\begin{itemize}
 \item In Proposition~\ref{proposition: superharmonic functions are strictly positive} we have seen that every nontrivial irreducible transient energy form possesses a strictly positive superharmonic function. For transient energy forms which do not satisfy this assumption, the proposition remains true when (ii) is replaced by
 \begin{itemize}
  \item[(ii)'] $D(E)_\cN$ is a vector space and for all $E$-superharmonic functions $h$ we have $\inf_{N\in \cN} \ceh (X\setminus N)= 0$.  
 \end{itemize}

 \item   It is well known for Dirichlet forms that nests (in the classical sense) can be characterized in terms of capacity, see \cite[Theorem~III.2.11]{MR}. 
\end{itemize}
\end{remark}

\begin{definition}[Local space]
 Let $E$ be an energy form. For an $E$-nest $\cN$ the {\em local space with respect to $\cN$} is
 \begin{align*} 
 D(E)_{{\rm loc },\, \cN} := \{ f \in L^0(m) \mid  \text{ for all } N \in \cN 
 \text{ there exists }  f_N \in D(E) \text{ with } f 1_N = f_N 1_N\}. 
 \end{align*}
 The {\em local space of $E$} is given by
 \begin{align*} 
  D(E)_{\rm loc} := \bigcup_{\cN \text{ is }E\text{-nest}} D(E)_{{\rm loc },\, \cN}.
 \end{align*}
\end{definition}

\begin{remark}
 When $\E$ is a regular Dirichlet form the local Dirichlet space is usually introduced in a similar spirit as local $L^p$-spaces or local Sobolev spaces. Namely, one considers
 $$\{f \in L^0(m) \mid \text{ for all }G \in \mathcal{K} \text{ there exists }  f_G \in D(\E)\text{ with }f 1_G = f_G1_G\}.$$
 In our terminology, this corresponds to the local space with respect to the nest of all open relatively compact sets $\mathcal{K}$. Unfortunately, for many purposes this space is too small. We shall encounter this problem   in Chapter~\ref{chapter:silverstein extensions} and in Chapter~\ref{chapter:weak solutions}. 
 
\end{remark}

From the definition of $D(E)_{\rm loc}$ it is not clear that it allows algebraic manipulations. However, with the help of Lemma~\ref{lemma:refinement of nests} we obtain the following structural results. 

\begin{proposition}
 Let $E$ be an energy form. The following assertions hold.
 \begin{itemize}
  \item[(a)] $D(E)_{\rm loc}$ is a vector space.
  \item[(b)] For $f_1,\ldots,f_n \in D(E)_{\rm loc}$ and every normal contraction $C:\IR^n \to \IR$ we have $$C(f_1,\ldots,f_n) \in D(E)_{\rm loc}.$$
  \item[(c)] $D(E)_{\rm loc}$ is a lattice and $D(E)_{\rm loc} \cap L^\infty(m)$ is an algebra.  
 \end{itemize}
\end{proposition}
\begin{proof}
 It follows from Lemma~\ref{lemma:refinement of nests} that $D(E)_{\rm loc}$ is a vector space. Similarly, assertions (b) and (c) follow from Lemma~\ref{lemma:refinement of nests} combined with Theorem~\ref{theorem:cutoff properties energy forms} or Theorem~\ref{theorem:algebraic and order properties}, respectively.
\end{proof}
 
 Different nests of an energy form $E$ can lead to the same local space, i.e., it may happen that $D(E)_{{\rm loc},\, \cN} = D(E)_{{\rm loc},\, \cN'}$ for different $E$-nests $\cN$ and $\cN'$. This phenomenon is important for later considerations and so we note the following. 
 
\begin{definition}[Equivalent nests] \label{definition:equivalent nests}
 Let $E$ be an energy form.  Two $E$-nests $\cN$ and $\cN'$ are called {\em equivalent} if the following is satisfied: 
  \begin{itemize}
  \item For each $N \in \cN$ there exists $N' \in \cN'$ with $N \subseteq N'$.
  \item For each $N' \in \cN'$ there exists $N  \in \cN$ with $  N' \subseteq N$.
 \end{itemize}
\end{definition}

\begin{lemma}\label{lemma:local domain equivalent nests}
 Let $E$ be an energy form and let $\cN,\cN'$ be equivalent $E$-nests. Then $D(E)_\cN = D(E)_{\cN'}$ and $D(E)_{{\rm loc},\, \cN} = D(E)_{{\rm loc},\, \cN'}$.
\end{lemma}
\begin{proof}
 This is immediate from the definitions.
\end{proof}

\chapter{Silverstein extensions of energy forms} \label{chapter:silverstein extensions}

 Silverstein extensions are an important tool in the theory of Dirichlet forms. They were introduced by Silverstein in \cite[Terminology~20.2]{Sil} and then used in \cite{Sil2} to study the boundary behavior of symmetric Markov processes. Roughly speaking, each Dirichlet form corresponds to a symmetric Markov process on the underlying space. A Silverstein extension of such a form provides a symmetric Markov process with the same local behavior,  but with different global behavior as the original process. One example is Brownian motion on a bounded domain. It can either be absorbed (killed) once it hits the boundary of the domain or it can be reflected back into the domain upon hitting the boundary. Reflected Brownian motion comes from a Silverstein extension of the Dirichlet form of absorbed Brownian motion. The book \cite{CF} is devoted to studying the stochastic interpretation of Silverstein extensions (and time changes) in the Dirichlet form case and we refer the reader to it for details. 
 
 In this chapter we extend the concept of Silverstein extensions to energy forms. This is possible since Silverstein extensions are defined by a simple algebraic property of the form domain; an extension of an energy form is a Silverstein extension if the domain of the smaller form is an ideal in the domain of the larger form. 
 In the first section we prove that this ideal property can either be formulated in terms of the order structure or in terms of the multiplicative structure of the domains and that it can be characterized by the fact that the domain of the larger form lies in the local domain of the smaller form, see Theorem~\ref{theorem:Ideal properties of energy forms}. This observation is then illustrated with several examples. More specifically, we discuss the precise relation of energy forms and Dirichlet forms, see Proposition~\ref{proposition:energy forms vs. dirichlet forms}, characterize Silverstein extensions of Dirichlet forms, see Theorem~\ref{theorem:silverstein extensions of dirichlet forms}, and show that Markovian extensions and Silverstein extensions agree for quadratic forms of hypoelliptic operators, see Theorem~\ref{theorem:markov v.s. silverstein extensions} and Corollary~\ref{corollary:hypoelliptic operators}.
 
 The second section of this chapter is devoted to uniqueness of Silverstein extensions. We prove that recurrent energy forms are Silverstein unique, see Theorem~\ref{theorem:recurrent forms are Silverstein unique},  and we provide an abstract version of the meta theorem that Silverstein uniqueness can be characterized by the vanishing of the capacity of 'the boundary', see Theorem~\ref{theorem:capacity characterization of uniqueness of Silverstein extensions}. This result allows us to obtain a unified treatment and to generalize most of the Silverstein uniqueness results for Dirichlet forms that are formulated in terms of capacity, see Theorem~\ref{theorem:silverstein uniqueness boundary inside space} and Theorem~\ref{theorem:silverstein uniqueness boundary outside space}. 
 
 In the last section of this chapter, we construct the maximal Silverstein extension of a given energy form. For (quasi-)regular Dirichlet forms this is usually done with the help of representation theory viz the Beurling-Deny formula of the form or the Fukushima decomposition of the associated process. These rather complicated procedures stand in sharp contrast with the definition of Silverstein extensions in terms of simple ideal properties of the form domains. We provide a new construction of the maximal Silverstein extension that only relies on these algebraic properties.  Along the way, we split a given energy form into its main part, see Definition~\ref{definition:main part}, and its killing part, see Definition~\ref{definition:killing part}, and introduce the reflected energy form, see Definition~\ref{definition:reflected form}. For Dirichlet forms this provides a new construction of the so-called reflected Dirichlet space. We give criteria when the reflected energy form is the maximal Silverstein extension, see Theorem~\ref{theorem:maximality er}, but also show that not every energy form, indeed not even every regular Dirichlet form, has a maximal Silverstein extension, see Example~\ref{example:counterexample maximal Silverstein extension energy form}, Example~\ref{example:counterexample maximal Silverstein extension regular dirichlet form} and Proposition~\ref{proposition:counterexample maximality er}. These insights reveal some mistakes in the literature, where it is claimed that all regular Dirichlet forms possess a maximal Silverstein extension.  We conclude this chapter with the structural insight that recurrence and uniqueness of Silverstein extensions are almost the same concept, see Theorem~\ref{theorem:recurrence and uniqueness}. This is the reason why both can be characterized in terms of capacity.

\section{Ideals of energy forms} \label{section:ideals}

In this section we introduce Silverstein extensions of energy forms. We prove that they can be characterized in terms of ideal properties of the form domains and in terms of local spaces. The notion of Silverstein extensions provides a language that allows us to precisely formulate the relation of Dirichlet forms and energy forms.  We finish this section with various examples and applications.

\begin{definition}[Algebraic ideal and order ideal]
Let $A,B \subseteq L^0(m)$. The set $A$ is called {\em order ideal in $B$} if $f \in A,$ $g \in B$  and $|g| \leq |f|$ implies $g \in A$. It is an {\em algebraic ideal in $B$} if for each $f \in A$ and $g \in B$ we have $fg \in A$. 
\end{definition}

The following theorem shows that both notions of ideal coincide for form domains of extensions of energy forms.

\begin{theorem}[Ideal properties of the form domain] \label{theorem:Ideal properties of energy forms}
 Let $E$ and $\ow{E}$ be energy forms such that $\ow{E}$ is an extension of $E$. The following assertions are equivalent. 
 \begin{itemize}
  \item[(i)] $D(E)$ is an order ideal in $D(\ow{E})$.
  \item[(ii)] $D(E)\cap L^\infty(m)$ is an algebraic ideal in $D(\ow{E})\cap L^\infty(m)$. 
  \item[(iii)] $D(\ow{E}) \subseteq \DEl$.
  \item[(iv)] There exists a special $E$-nest $\cN$ such that $D(\ow{E}) \cap L^\infty(m) \subseteq \DEln$.
  \item[(v)] For all special $E$-nest $\cN$ we have $D(\ow{E}) \cap L^\infty(m) \subseteq \DEln$.
 \end{itemize}
\end{theorem}
\begin{proof}
 (i) $\Rightarrow$ (ii): Let $f \in D(E)\cap L^\infty(m)$ and $g \in D(\ow{E})\cap L^\infty(m)$. Their product satisfies
 $$|fg| \leq \|g\|_\infty |f|.$$
 Since $D(E)$ is an order ideal in $D(\ow{E})$, this implies $fg \in D(E)$.   
 
 (ii) $\Rightarrow$ (i): Let $f \in D(E)$ and $g \in D(\ow{E})$ with $|g| \leq |f|$. We have to prove $g \in D(E)$. By the contraction properties of $E$ and by Proposition~\ref{proposition:approximation by bounded functions} we can assume that $f$ and $g$ are nonnegative and bounded. The inequality $0 \leq g \leq f$ implies the convergence
 $$   \frac{f}{f + \varepsilon}g \overset{m}{\longrightarrow} g, \text{ as } \varepsilon \to 0+. $$
  From the lower semicontinuity of $E$ we infer 
 $$E(g) \leq \liminf_{\varepsilon \to 0+} E(f(f+\varepsilon)^{-1}g).$$
 Therefore, it suffices to show that $E(f(f+\varepsilon)^{-1}g)$ is finite and uniformly bounded in $\varepsilon$.  In the interior of $S:=\{(w,z) \in \IR^2 \mid 0\leq w \leq z\}$  the partial derivatives of the function 
 $$C_\varepsilon:S \to \IR, \, (w,z) \mapsto \frac{w}{w+\varepsilon} z$$
 are bounded by $1$. This observation, the continuity of $C_\varepsilon$ and the inequality $0 \leq g \leq f$ yield 
 $$|C_\varepsilon(f(x),g(x)) - C_\varepsilon(f(y),g(y))| \leq |f(x) - f(y)| + |g(x) - g(y)|$$
 and
 $$|C_\varepsilon(f(x),g(x))| \leq |f(x)| + |g(x)|.$$
The contraction properties of $\ow{E}$ imply
 $$\ow{E} (C_\varepsilon(f,g))^{1/2} \leq \ow{E}(f)^{1/2} + \ow{E}(g)^{1/2}.$$
 A similar argument applied to the function $D_\varepsilon:[0,\infty) \to \IR,\, x  \mapsto x/(x+\varepsilon)$ shows
 $$f(f+\varepsilon)^{-1} = D_\varepsilon(f)  \in D(E)\cap L^\infty(m).$$ 
 According to (ii), the algebra $D(E) \cap L^\infty(m)$ is an algebraic ideal in $D(\ow{E}) \cap L^\infty(m)$ and we obtain $C_\varepsilon(f,g) =f(f+\varepsilon)^{-1}g  \in D(E) \cap L^\infty(m)$. Since $\ow{E}$ is an extension of $E$, this and the previous considerations  imply
 $$E(C_\varepsilon(f,g))^{1/2} = \ow{E}(C_\varepsilon(f,g))^{1/2} \leq \ow{E}(f)^{1/2} + \ow{E}(g)^{1/2}.$$
 Therefore,  $E(C_\varepsilon(f,g)))$ is finite and bounded independently of $\varepsilon$ and implication (ii) $\Rightarrow$ (i) is proven.
 
 (ii) $\Rightarrow$ (v):   Let $f \in D(\ow{E}) \cap L^\infty(m)$  and let $\mathcal{N}$ be a special $E$-nest. For $N \in \cN$ we choose $g_N \in D(E)_\cN$ with $1_N \leq g_N \leq 1$.  Since $D(E) \cap L^\infty(m)$ is an algebraic ideal in $D(\ow{E}) \cap L^\infty(m)$, we we have $f g_N \in D(E)$. Furthermore, the choice of $g_N$ yields $f g_N  = f1_N $ and we obtain $f \in D(E)_{{\rm loc},\, \cN}$.  
 
 (v) $\Rightarrow$ (iv): This follows from the existence of special $E$-nests, Lemma~\ref{lemma:special nest}. 
 
 (iv) $\Rightarrow$ (iii): Let $\cN$ be a special $E$-nest such that $D(\ow{E}) \cap L^\infty(m) \subseteq \DEln$ and let $f \in D(\ow{E})$. Assertion (iv) and the contraction properties of $\ow{E}$ imply 
 $$(f \wedge n) \vee (-n) \in \DEln \subseteq \DEl \text{ for all }n \in \IN.$$
 As a consequence, we obtain that for each  
 $$M \in \cN_{|f|} :=\{N\cap \{|f| \leq n \} \mid N \in \mathcal{N}, n \in \IN\} = \cN \wedge \{\{|f| \leq n\} \mid n \in \IN\}$$
 there exists an $f_{M} \in D(E)$ with $f_{M} = f$ on $M$. Therefore, it suffices to show that $ \cN_{|f|}$ is an $E$-nest. According to Lemma~\ref{lemma:refinement of nests}, it remains to check that $\ow{\cN}:= \{\{|f| \leq n\} \mid n \in \IN\}$ is an $E$-nest.
 
We let $E'$ be the closure of the restriction of $E$ to $D(E)_{\ow{\cN}}$ and show $D(E) \subseteq D(E')$. According to Theorem~\ref{theorem:closure is markovian}, the form $E'$ is an energy form. Since bounded functions are dense in the domain of an energy form, it suffices to prove  $D(E) \cap L^\infty(m) \subseteq D(E') \cap L^\infty(m).$ To this end, for $\varphi \in D(E) \cap L^\infty(m)$ and $n \in \IN$ we set
 $$\varphi_n := \varphi - \varphi \cdot \frac{|f| \wedge n}{n} = \begin{cases} 0 & \text{ if } |f| > n\\ 
  \varphi\left(1-\frac{|f|}{n}\right) & \text{ else }\end{cases}.$$
 Since $D(E)\cap L^\infty(m)$ is an algebraic ideal in $D(\ow{E})\cap L^\infty(m)$, we have $\varphi_n \in D(E)$ and by definition $\varphi_n$ vanishes on $\{|f| > n\}$. This shows $\varphi_n \in D(E)_{\ow{\cN}}$. Furthermore,  $\varphi_n \overset{m}{\to}\varphi$, as $n \to \infty$. We use the lower semicontinuity of $E'$ and that $E,E'$ and $\ow{E}$ coincide on $D(E)_{\ow{\cN}}$ to conclude
 $$E'(\varphi) \leq \liminf_{n \to \infty} E'(\varphi_n) =  \liminf_{n \to \infty} E(\varphi_n) = \liminf_{n \to \infty} \ow{E}(\varphi_n). $$
 Theorem~\ref{theorem:algebraic and order properties} applied to the form $\ow{E}$ yields
 \begin{align*}
 \ow{E}(\varphi_n)^{1/2} &\leq \ow{E}(\varphi)^{1/2} + \| n^{-1}(|f|\wedge n)\|_\infty  \ow{E}(\varphi)^{1/2}  +  \|\varphi\|_\infty \ow{E}( n^{-1}(|f|\wedge n))^{1/2}  \\ &\leq 2\ow{E}(\varphi)^{1/2} + \|\varphi\|_\infty \ow{E}(f)^{1/2}.
 \end{align*}
 These calculations show  $\varphi \in D(E')$ and finish the proof of implication  (iv) $\Rightarrow$ (iii).  
 
 (iii) $\Rightarrow$ (ii): Let $g \in D(E) \cap L^\infty(m)$ and $f \in D(\ow{E})\cap L^\infty(m)$. We have to show $fg \in D(E)$. By (iii) there exists an $E$-nest $\cN$ such that $f \in \DEln$.  We choose a net $(g_i)$ in  $D(E)_\mathcal{N}$ that converges to $g$ with respect to the form topology. According to Lemma~\ref{lemma:characterization convergence in form topology lsc forms} and the contraction properties of $E$, we can assume $\|g_i\|_\infty \leq \|g\|_\infty$. It follows from the definitions and Theorem~\ref{theorem:algebraic and order properties} that $D(E)_\cN \cap L^\infty(m)$ is an algebraic ideal in $\DEln \cap L^\infty(m)$. Therefore, we have $fg_i \in D(E)$ and obtain
 $$E(fg) \leq \liminf_{i} E(g_if) = \liminf_{i} \ow{E}(g_if).$$
Another application of Theorem~\ref{theorem:algebraic and order properties} yields
 $$ \ow{E}(g_i f)^{1/2} \leq  \|g_i\|_\infty \ow{E}(f)^{1/2} + \|f\|_\infty \ow{E}(g_i)^{1/2} \leq  \|g\|_\infty \ow{E}(f)^{1/2} + \|f\|_\infty E(g_i)^{1/2}. $$
 Combining these computations shows $E(fg) < \infty$ and  finishes the proof.
\end{proof}

\begin{definition}[Silverstein extension]
Let $E$ be an energy form. An extension $\ow{E}$ of $E$ that satisfies one of the equivalent conditions of the previous theorem is called {\em Silverstein extension of $E$}. The set of all Silverstein extensions of $E$ is denoted by ${\rm Sil}(E)$. The form $E$ is called {\em Silverstein unique} if ${\rm Sil}(E) = \{E\}$.
\end{definition}

\begin{remark} \label{remark:discussion silverstein}
\begin{itemize} 
 \item For Dirichlet forms I learned  of the equivalence of (i) and (ii) in a discussion with Hendrik Vogt and Peter Stollmann in December~2012 and the proof of the implication (ii) $\Rightarrow$ (i) as it is presented here is due to Hendrik Vogt. They were familiar with \cite{Ouh}, which shows that (i) for the domains of  Dirichlet forms is equivalent to domination of the associated semigroups. I knew that (ii)  is also related with domination of the associated semigroups from  the paper \cite{HKLMS}, which I wrote with co-authors at the time.   Later, I learned that the equivalence of (ii) and the domination of the associated semigroups was already known and stated in \cite[Theorem~21.2]{Sil}   under various technical assumptions on the underlying space. There, the proof is rather complicated and uses probabilistic methods.  Independently, the equivalence of (i) and (ii) was recently proven in \cite[Proposition~2.7]{Rob} by different methods. For a precise formulation of the just mentioned results for Dirichlet forms, see Theorem~\ref{theorem:silverstein extensions of dirichlet forms} below.
 
 \item The implication (ii) $\Rightarrow$ (iii) was observed for quasi-regular Dirichlet forms in \cite[Remark~6.6.2]{CF} with a slightly different notion of local space. In particular, the idea that for an unbounded $f \in D(\ow{E})$ one should consider $\{\{|f| \leq n\}\mid n \in \IN\}$  and then take a refinement with a given special nest can, in some disguise,   already be found there. Here, the technical novelty lies in the fact that we had to show that $\{\{f \leq n\}\mid n \in \IN\}$ is an $E$-nest while this is automatically satisfied in \cite{CF}. That the reverse implication (iii) $\Rightarrow$ (ii) also holds seems to be new. We shall see at the end of this section that it is particularly useful for constructing examples.
\end{itemize}
\end{remark}

In the proof of the previous theorem we did not only show the inclusion $D(\ow{E}) \subseteq \DEl$, but explicitly constructed the corresponding nests. We fix this construction for later purposes. Whenever $\cN$ is an $E$-nest and $f \in L^0(m)$, we let
$$\cN_f := \cN \wedge \{\{f \leq n\} \mid  n\in \IN\} = \{N \cap \{f\leq n\} \mid N\in \cN \text{ and } n \in \IN\}.$$

\begin{lemma} \label{lemma:silverstein extension local space}
 Let $E$ be an energy form and let $\ow{E}$ be a Silverstein extension of $E$. Let $\cN$ be a special $E$-nest. For all $f \in D(\ow{E})$ the collection $\cN_{|f|}$ is a special $E$-nest and
 $$f \in D(E)_{{\rm loc},\,\cN_{|f|}}.$$
\end{lemma}
\begin{proof}
 As seen in the proof of Theorem~\ref{theorem:Ideal properties of energy forms} (iv) $\Rightarrow$ (iii),  $\cN_{|f|}$ is a nest and $f \in D(E)_{{\rm loc},\,\cN_{|f|}}$. It remains to prove that $\cN_{|f|}$ is special.  For $N \in \cN$ we choose some $g_N \in D(E)_\cN$ with $1_N \leq g_N \leq 1$.  For $n \in \IN$, the function
 $$f_{n,N}:= 2 \left(g_N - \frac{|f|}{2n}\right)_+$$
 belongs to $D(\ow{E})$, satisfies $f_{n,N} \geq 1_{N \cap \{|f| \leq n\}}$ and $f_{n,N} = 0$ on $\{|f|>2n\}$. Since $D(E)$ is an order ideal in $D(\ow{E})$ and  $|f_{n,N}| \leq g_N$, we obtain $f_{n,N} \in D(E)_\cN$. This finishes the proof.
\end{proof}

In many applications energy forms are defined as closures of some energy functional. Sometimes one has a pair of such functionals for which the ideal properties are only known on the smaller domains before taking the closure.  The following proposition shows that this is irrelevant.

\begin{proposition}\label{proposition:silverstein extension of closure}
 Let $E$ and $\ow{E}$ be closable Markovian forms such that $\ow{E}$ is an extension of $E$.  If $D(E) \cap L^\infty(m)$ is an algebraic ideal in $D(\ow{E})\cap L^\infty(m)$, then the closure of $\ow{E}$ is a Silverstein extension of the closure of $E$.
 \end{proposition}
\begin{proof}
 We denote the closure of $E$ by $E_1$ and the closure of $\ow{E}$ by $E_2$ and prove that the algebra $D(E_1)  \cap L^\infty(m)$ is an ideal in the algebra $D(E_2)\cap L^\infty(m)$. To this end, we let $f \in D(E_1)  \cap L^\infty(m)$ and $g \in D(E_2)\cap L^\infty(m)$ and choose nets $(f_i)$ in $D(E)$ and $(g_j)$ in $D(\ow{E})$ that converge towards $f$ and $g$, respectively, in the corresponding form topology. By the contraction properties of $E_1$ and $E_2$ we can assume that the $f_i$ are uniformly bounded by $\|f\|_\infty$ and that the $g_j$ are uniformly bounded by $\|g\|_\infty$. Using the lower semicontinuity of $E_1$ and that $D(E)  \cap L^\infty(m)$ is an algebraic ideal in $D(\ow{E})\cap L^\infty(m)$, we obtain
$$E_1(fg) \leq \liminf_i \liminf_j E_1(f_i g_j) = \liminf_i \liminf_j E(f_i g_j) = \liminf_i \liminf_j \ow{E}(f_i g_j).$$
Theorem~\ref{theorem:algebraic and order properties} applied to $E_2$ yields
$$\ow{E}(f_i g_j)^{1/2} \leq \|f_i\|_\infty \ow{E}(g_j)^{1/2} + \|g_j\|_\infty \ow{E}(f_i)^{1/2}\leq \|f\|_\infty \ow{E}(g_j)^{1/2} + \|g\|_\infty \ow{E}(f_i)^{1/2}. $$
Combining these computation shows $fg \in D(E_1)$ and proves the claim. 
 \end{proof}

We have seen that extended Dirichlet forms provide an example for energy forms. The previous proposition and the notion of Silverstein extensions allows us to discuss their relation more precisely.

\begin{proposition}[Energy forms v.s. Dirichlet forms] \label{proposition:energy forms vs. dirichlet forms}
 Let $E$ be an energy form and let $\E := E|_{L^2(m)}$ its restriction to $L^2(m)$. Then $\E$ is a Dirichlet form and $E$ is a Silverstein extension of $\Ee$. Moreover, if $m$ is finite, then $\Ee  = E$. 
\end{proposition}
\begin{proof}
 The form $\E$ is Markovian and as a restriction of a closed form to the complete space $L^2(m)$, which continuously embeds into $L^0(m)$, it is closed, see Lemma~\ref{lemma:form restriction}. Furthermore, $D(\E) \cap L^\infty(m) = D(E) \cap L^2(m) \cap L^\infty(m)$ is an algebraic ideal in $D(E)\cap L^\infty(m)$. Since $\Ee$ is the closure of $\E$ in $L^0(m)$, Proposition~\ref{proposition:silverstein extension of closure} shows that $E$ is a Silverstein extension of $\Ee$. If $m$ is a finite measure, we have 
 $$D(E) \cap L^\infty(m) \subseteq D(E) \cap L^2(m) \subseteq D(\Ee).$$
 With this at hand, the density of bounded functions in the form domain $D(E)$ and the fact that $E$ is an extension of $\Ee$ yield the claim.
\end{proof}

\begin{remark}
  The previous proposition shows that in the case of finite measure there is a one-to-one correspondence between extended Dirichlet forms and energy forms. When the measure is infinite this correspondence breaks down. For example, consider a resistance form $\E$ as in Subsection~\ref{subsection:resistance forms} and assume that all balls with respect to the resistance metric are uncountable. This is a typical situation for forms associated with fractals, concrete examples of this kind can be found in \cite{Kig1}. All functions in $D(\E|_{L^2(\mu)})$ have to have countable support (recall that $\mu$ is the counting measure) and, due to (RF4), functions in $D(\E)$ are continuous with respect to the resistance metric. Since the balls are uncountable, the only functions in $D(\E)$ with countable support need to vanish and we obtain $D(\E_{L^2(\mu)}) = \{0\}.$ Similar considerations apply to all measures $m$ for which the spaces $L^0(m)$ and $L^0(\mu)$ agree.
\end{remark}

Since the motivation for studying Silverstein extensions originally comes from  Dirichlet form theory, we mention the following relation, which was already discussed in Remark~\ref{remark:discussion silverstein}. Recall that for a Dirichlet form $\E$, we let $\E_1 = \E + \|\cdot\|_2^2$, which is an energy form.  

\begin{theorem}\label{theorem:silverstein extensions of dirichlet forms}
 Let $\E,\ow{\E}$ be Dirichlet forms such that $\ow{\E}$ is an extension of $\E$. Let $(G_\alpha),(\ow{G}_\alpha)$ be the associated resolvents and $(T_t),(\ow{T}_t)$ the associated semigroups, respectively. The following assertions are equivalent.
 \begin{itemize}
  \item[(i)] $D(\E) \cap L^\infty(m)$ is an algebraic ideal in $D(\ow{\E}) \cap L^\infty(m)$.
  \item[(ii)] $D(\E)$ is an order ideal in $D(\ow{\E})$.
  \item[(iii)] $\ow{\Ee}$ is a Silverstein extension of $\Ee$.
  \item[(iv)] $\ow{\E}_1$ is a Silverstein extension of $\E_1$.
  \item[(v)] $(\ow{G}_\alpha)$ dominates $(G_\alpha)$, i.e., for all $f \in L^2(m)$ and all $\alpha >0$, we have $|G_{\alpha} f| \leq \ow{G}_{\alpha}|f|$.
  \item[(vi)] $(\ow{T}_t)$ dominates $(T_t)$, i.e., for all $f \in L^2(m)$ and all $t >0$, we have  $|T_t f| \leq \ow{T}_t |f|$.
 \end{itemize}
\end{theorem}
\begin{proof}
The equivalence of (i), (ii) and (iv) follows from the fact that $\E_1$ and $\ow{\E}_1$ are energy forms whose domains satisfy $D(\E_1) = D(\E)$ and $D(\ow{\E}_1) = D(\ow{\E})$, and Theorem~\ref{theorem:Ideal properties of energy forms}. The implication (i) $\Rightarrow$ (iii) is a direct consequence of Proposition~\ref{proposition:silverstein extension of closure}  and  implication (iii) $\Rightarrow$ (ii) follows from the identities
$$D(\E) = D(\Ee) \cap L^2(m) \text{ and } D(\ow{\E}) = D(\ow{\Ee}) \cap L^2(m).$$
The equivalence of (ii), (iv) and (v) was observed in \cite{Ouh}. This finishes the proof.
\end{proof}

\begin{remark}
 \begin{itemize}
  \item We had to formulate (i) and (ii) in the previous proposition in the way they are since Dirichlet forms are not energy forms (in general they are not closed with respect to convergence in measure). In the sequel we say that an extension $\ow{\E}$ of a Dirichlet form $\E$ is a {\em Silverstein extension of $\E$} if  one of the above equivalent conditions is satisfied.  
  \item For a further discussion of the previous theorem, we refer back to Remark~\ref{remark:discussion silverstein}.
 \end{itemize}

\end{remark}

We finish this  section by providing various  examples for Silverstein extensions.

\begin{example}
 Let $E$ be an energy form and let $A \subseteq X$ measurable. We set
 $$E^A:L^0(m) \to [0,\infty],\quad  f \mapsto \begin{cases}
                                E(f) &\text{ if } f1_{X \setminus A} = 0\\
                                \infty &\text{ else}
                               \end{cases}.
$$
 $E^A$ is an energy form, $D(E^A)$ is a order ideal in $D(E)$ and $E^A$ is a restriction of $E$. Hence, $E$ is a Silverstein extension of $E^A$. 
\end{example}

\begin{example}
 Let $X$ be a locally compact separable metric space and let $m$ be a Radon measure of full support. For an energy form $E$ on $L^0(m)$, we let $E^0$ be the closure of the restriction of $E$ to $C_c(X) \cap D(E)$.
 
 \begin{proposition} \label{proposition:Silverstin example topological situation}
  If $C(X)\cap D(E)$ is dense in $D(E)$ with respect to the form topology, then $E$ is a Silverstein extension of $E^0$.
 \end{proposition}
 \begin{proof}
 We can apply Proposition~\ref{proposition:silverstein extension of closure} to the restriction of $E$ to $D(E)\cap C(X)$  and to the restriction of $E^0$ to $C_c(X) \cap D(E)$. 
 \end{proof}
\end{example}
\begin{remark}
 The previous two examples combined with Theorem~\ref{theorem:Ideal properties of energy forms} provide a good idea of how to think about Silverstein extensions. If $E$ is an energy form and $\ow{E}$ is a Silverstein extension, then functions in the domain of $E$ and in the domain of $\ow{E}$ locally look the same, they only differ on some 'boundary points'. In the first example these boundary points are given by $X \setminus A$, while in the second example the boundary points are lying at 'infinity' outside the space $X$. Indeed, with the help of Gelfand theory (applied to the uniform closure of the algebra $D(\ow{E}) \cap L^\infty(m)$) one can replace the underlying space $X$ by a compact space and then show that $D(E)$ corresponds to continuous functions that vanish on a closed subset of this space. For regular Dirichlet forms a variant of this idea has been discussed in \cite[Section~A.4]{FOT}. We refrain from giving details. 
\end{remark}

\begin{example}
 There are also extensions of energy forms which are not Silverstein extensions. Typically, this occurs when $E$ is a restriction of $\ow{E}$ with some kind of periodic boundary conditions. The simplest example is the following. We let $X = \{a,b\}$ and let $\mu$ be the counting measure on all subsets of $X$.  We define
 $$E:L^0(\mu) \to [0,\infty],\quad f \mapsto E(f) := \begin{cases}
            f(a)^2 + f(b)^2 &\text{ if } f(a) = f(b)\\
            \infty &\text{ else} 
           \end{cases}
$$
and 
$$\ow{E}:L^0(\mu) \to [0,\infty],\quad f \mapsto \ow{E}(f) := f(a)^2 + f(b)^2.$$
In this case, $\ow{E}$ is an extension of $E$ but $D(E)$ is no order ideal in $D(\ow{E})$.
\end{example}

The previous example demonstrates that not every extension of an energy form is necessarily a Silverstein extension. Nevertheless, in some situations the class of Silverstein extension is rather rich. The next two examples show that this is the case for regular forms when the underlying space has no smoothness at all or when the forms under considerations derive from hypoelliptic operators.

\begin{example}[Lack of smoothness in the underlying space] \label{example:lack of smoothness}
Let $X$ be a countable set equipped with the discrete topology and let $m$ be Radon measure of full support. In this case, $L^0(m) = C(X)$, $C_c(X)$ is the space of finitely supported functions and $\tau(m)$ is the topology of pointwise convergence. For an energy form $E$ on $L^0(m)$, this implies $D(E)\cap C(X) = D(E)$.  Therefore, Proposition~\ref{proposition:Silverstin example topological situation} shows that $E$ is a Silverstein extension of $E^0$, the closure of the restriction of $E$ to $C_c(X)$. In particular, any extension of a regular energy form on $L^0(m)$ is a Silverstein extension. 
\end{example}

\begin{example}[Dirichlet forms associated with hypoelliptic operators] \label{example:hypoelliptic}
 Let $S$ be a densely defined symmetric operator on $L^2(m)$ with domain $D(S)$ such that the quadratic form 
 $$\E_{S}:L^2(m) \to [0,\infty], \, f \mapsto \begin{cases}
                                               \as{Sf,f} &\text{ if } f\in D(S)\\
                                               \infty &\text{ else}
                                              \end{cases}
  $$
 is Markovian. By the Friedrichs extension theorem the form $\E_S$ is closable and by \cite[Theorem~3.1.1]{FOT} its closure $\bar{\E}_S$ is a Dirichlet form. The associated self-adjoint operator is an extension of $S$. 
 
 There is a vast amount of literature on determining all self-adjoint extensions of $S$ whose associated quadratic form is a Dirichlet form, see e.g. the comprehensive treatise \cite{Ebe} and references therein.  We denote the collection of all such extensions by ${\rm Ext}(S)_{\rm M}$. It is of particular interest to find criteria which guarantee that this set consists of only one element, i.e., which guarantee that $S$ is {\em Markov unique}. In general, there is no reason why one should hope to find a good description of ${\rm Ext}(S)_{\rm M}$ (and its triviality), which goes beyond the abstract results available in the theory of self-adjoint extensions. In contrast, the situation is much better for those self-adjoint extensions of $S$ whose associated form is a Silverstein extension of $\bar{\E}_S$. We denote their collection by ${\rm Ext}(S)_{\rm Sil}$. In this case, there are many criteria which guarantee that ${\rm Ext}(S)_{\rm Sil}$ consists of only one element, see e.g. Section~\ref{section:uniqueness of silverstein extensions}. Even if there is more than one such extension, one can give an explicit description of the maximal element of ${\rm Ext}(S)_{\rm Sil}$ (ordered in the sense of the associated quadratic forms) purely in terms of $\bar{\E}_S$, see Section~\ref{subsection:maximal Silverstein extension}.  In this light, one should read the following theorem. Recall that in our notation $\bar{\E}_{S,1} = \bar{\E}_{S} + \|\cdot\|_2^2$.   Furthermore, we denote the $L^2$-adjoint of $S$ by $S^*$.
 
 \begin{theorem} \label{theorem:markov v.s. silverstein extensions}
  Let $S$ be a densely defined symmetric operator on $L^2(m)$ such that $\E_S$ is Markovian. If 
  $$\ker (S^* + 1) \subseteq D(\bar{\E}_{S,1})_{\rm loc},$$
  then 
  $${\rm Ext}(S)_{\rm Sil} = {\rm Ext}(S)_{\rm M}.$$
 \end{theorem}
 \begin{proof}
  Let $\ow{S}$ be a self-adjoint extension of $S$ such that the associated quadratic form $\ow{\E}$ is a Dirichlet form. In this case, the domain of $\bar{\E}_{S}$ is a closed subspace in the Hilbert space $(D(\ow{\E}),\ow{\E}_1)$. Therefore, we can orthogonally decompose $D(\ow{\E})$ with respect to $\ow{\E}_1$ into
  $$D(\ow{\E}) = D(\bar{\E}_S) \oplus \{f \in D(\ow{E}) \mid \ow{\E}_1(f,\psi) = 0 \text{ for all } \psi \in D(\bar{\E}_S)\}.$$
  Since $\ow{S}$ is a self-adjoint extension of $S$ and $\ow{\E}$ is its associated form, this implies
  $$D(\ow{\E}) = D(\bar{\E}_S) \oplus D(\ow{\E})\cap \ker (S^* + 1).$$
  In particular, $D(\ow{\E}_1) = D(\ow{\E}) \subseteq D(\bar{\E}_{S,1})_{\rm loc}$. Theorem~\ref{theorem:silverstein extensions of dirichlet forms} shows that $\ow{\E}$ is a Silverstein extension of $\bar{\E}_S$. 
 \end{proof}
The condition on the kernel of the adjoint of $S$ can be seen as a weak regularity assumption on the operator $S$. A  typical class of operators where it is satisfied are hypoelliptic operators. More precisely, for a smooth Riemannian manifold $(M,g)$ we let $\mathcal{D}(M)$ denote the space of distributions on  $C_c^\infty(M)$. A continuous linear operator $P:C_c^\infty(M) \to C_c^\infty(M)$, where $C_c^\infty(M)$ is equipped with the usual locally convex topology, is called {\em hypoelliptic}, if for  all $u \in \mathcal{D}(M)$ the inclusion $P' u \in C^\infty(M)$ implies $u \in C^\infty(M)$. Here, $P'$ is the dual operator $P' : \mathcal{D}(M) \to \mathcal{D}(M)$ of $P$. For hypoelliptic operators the following variant of the previous theorem holds.
\begin{corollary}\label{corollary:hypoelliptic operators}
 Let $P:C_c^\infty(M) \to C_c^\infty(M)$ be a continuous linear operator that is symmetric on $L^2({\rm vol}_g)$ and whose associated form $\E_P$ is Markovian. If $P+1$ is hypoelliptic, then  
 $${\rm Ext}(P)_{\rm Sil} = {\rm Ext}(P)_{\rm M}.$$
\end{corollary}
\begin{proof}
 By definition the form $\bar{\E}_{P,1}$ is a regular energy form.  In particular, the collection of all  relatively compact open sets is an $\bar{\E}_{P,1}$-nest, see Lemma~\ref{lemma:special nest regular form}. According to the hypoellipticity of $P+1$, we have $\ker (P^* + 1) \subseteq \ker (P' + 1) \subseteq C^\infty(M)$. Standard results on smooth partitions of manifolds imply that for each  relatively compact open $G \subseteq M$ and all $f \in C^\infty(M)$ there exists $f_G \in C_c^\infty(M)$ with $f = f_G$ on $G$. This shows   $\ker (P^* + 1) \subseteq D(\bar{\E}_{P,1})_{{\rm loc}}$ and finishes the proof.
\end{proof}
\end{example}

\begin{remark}
\begin{itemize}
  \item There is a vast amount of results on when a (pseudo-)differential operator on a manifold is hypoelliptic. We refer the reader to \cite{Hoer}, which among many other subjects treats   hypoellipticity.
  
  \item  The equality ${\rm Ext}(P)_{\rm Sil} = {\rm Ext}(P)_{\rm M}$ for certain hypoelliptic differential operators $P$ on open subsets of $\IR^n$ was first observed in \cite{Tak} and then extended to general uniformly elliptic operators with smooth coefficients in \cite{FOT}. There, it is also proven for the considered class of differential operators that the self-adjoint operator of a Silverstein extension of $\bar{\E}_P$ is automatically an extension of $P$. For general operators $S$ this latter statement is not true. However, in Section~\ref{subsection:maximal Silverstein extension} we construct the maximal Silverstein extension of $\bar{\E}_S$ and prove (under mild assumptions on $S$) that its associated operator is an extension of $S$. The novelty of our result lies in proving the previous corollary through Theorem~\ref{theorem:markov v.s. silverstein extensions}. The considerations in \cite{Tak,FOT} made explicit use of the concrete forms of the given operators.
  
  \item Hypoellipticity is indeed a too strong assumption  for the previous corollary to hold true. We used it to have a rather large class of operators at hand for which we can apply Theorem~\ref{theorem:markov v.s. silverstein extensions} without having to state too many technical details. For elliptic operators $A$ with rather rough coefficients, the equality ${\rm Ext}(A)_{\rm Sil} = {\rm Ext}(A)_{\rm M}$  is part of \cite[Theorem~1.1]{RS}. There, it is shown that the domain of $A^*$ is contained in some local Sobolev space and so our result could be applied. We refrain from giving details.
  
  \item In \cite{Ebe} a somewhat weaker statement is proven for so-called diffusion operators on finite dimensional spaces. If $L$ is such a diffusion operator, it is shown that ${\rm Ext}(L)_{\rm Sil}$ consists of one element if and only if ${\rm Ext}(L)_{\rm M}$ consists of one element.  The main ingredient for the proof given there is the construction of the maximal element of ${\rm Ext}(L)_{\rm M}$. Contrary to our narrative, this is possible since  diffusion operators have a lot of intrinsic structure. 
 
 \item Another situation where Theorem~\ref{theorem:markov v.s. silverstein extensions} can be applied are graph Laplacians of locally finite graphs. For a detailed discussion and related results, see \cite{HKLW}. Note that the associated forms also fall in the context of Example~\ref{example:lack of smoothness}.
\end{itemize}
\end{remark}

\section{Uniqueness of Silverstein extensions} \label{section:uniqueness of silverstein extensions}

As already mentioned in the discussion of Example~\ref{example:hypoelliptic}, it is an important problem to provide criteria that guarantee Silverstein uniqueness. For certain Sobolev spaces Silverstein uniqueness appears in the form of asking whether $W^{1,2} = W^{1,2}_0$ and for symmetric Markov process it corresponds to the question whether the local behavior of a process determines its global behavior uniquely, see e.g. \cite{CF}. In this section, we give two abstract criteria for this property. First, we prove that recurrence implies Silverstein uniqueness and then we characterize Silverstein uniqueness for transient energy forms in terms of capacities of nests. At the end of this section, we discuss how the abstract criterion can be employed to recover and to generalize known uniqueness results for Silverstein extensions of Dirichlet forms which deal with capacities of boundaries. 

\begin{theorem}\label{theorem:recurrent forms are Silverstein unique}
 Let $E$ be an energy form. If $1 \in D(E)$, then  ${\rm Sil}(E) = \{E\}$. In particular, all recurrent energy forms are Silverstein unique.
\end{theorem}
\begin{proof}
 Assume that $\ow{E}$ is a  Silverstein extension of $E$. The algebra $D(E) \cap L^\infty(m)$ is an algebraic ideal in $D(\ow{E}) \cap L^\infty(m)$.  Hence, for all $f \in D(\ow{E}) \cap L^\infty(m)$ we  have $f = 1f \in D(E)$.  Since bounded functions are dense in $D(E)$ with respect to the form topology, we infer $E = \ow{E}$. 
\end{proof}

As an immediate corollary of this theorem and Theorem~\ref{theorem:silverstein extensions of dirichlet forms} we obtain the following. 

\begin{corollary}
 Let $\E$ be a Dirichlet form. If $1\in D(\Ee)$, then $\E$ does not have any nontrivial Silverstein extensions. 
\end{corollary}

\begin{remark}
\begin{itemize}
 \item It is well known that recurrence implies Silverstein uniqueness of Dirichlet forms, see e.g. \cite[Theorem~6.2]{Kuw}. We believe that a proof for this fact becomes more transparent when given on the level of extended Dirichlet spaces. 
 \item It was our narrative in the last section that Silverstein extensions of a given energy form are a rather rich class in the collection of its Markovian extensions. The only non-Silverstein extensions that we encountered so-far were Markovian extensions of forms with periodic boundary conditions. The previous theorem opens the way to other non-examples. When one considers Dirichlet forms on infinite dimensional state spaces, it is typical that the measure is finite and that the constant function $1$ is contained in the form domain. Therefore, they are always Silverstein unique but might have other Markovian extensions, see \cite{Ebe} for concrete examples.
\end{itemize}
\end{remark}

We shall see below (see Theorem~\ref{theorem:recurrence and uniqueness})  that uniqueness of Silverstein extensions and recurrence are more intimately linked than the previous theorem suggests. This is also why both can be characterized in terms of capacity. 

\begin{theorem}\label{theorem:capacity characterization of uniqueness of Silverstein extensions} 
 Let $E$ be an energy form and let $\ow{E}$ be a transient Silverstein extension of $E$. Let $h$ be a strictly positive $\ow{E}$-superharmonic function. The following assertions are equivalent.
 \begin{itemize}
  \item[(i)] There exists a special $E$-nest $\mathcal{N}$ such that 
  $$ \inf_{N \in \mathcal{N}} {\rm cap}_{\ow{E},h} (X \setminus N) = 0.$$

   \item[(ii)] For all $E$-nests $\mathcal{N}$ we have
  $$\inf_{N \in \mathcal{N}} {\rm cap}_{\ow{E},h} (X \setminus N) = 0.$$
  \item[(iii)] $E = \ow{E}$.
 \end{itemize}
\end{theorem}
\begin{proof}
 (i) $\Rightarrow$ (iii): Let $\cN$ be a special $E$-nest. According to  Proposition~\ref{proposition:capacity characterisation of nests}, assertion (i) implies that $\cN$ is also an $\ow{E}$-nest. In particular, the space $D(\ow{E})_\cN \cap L^\infty(m)$ is dense in $D(\ow{E})$. Since $E$ is an extension of $\ow{E}$, it suffices to prove the following.
 
 {\em Claim:} The equality $D(\ow{E})_\cN \cap L^\infty(m) = D(E)_\cN \cap L^\infty(m)$ holds. 
 
 {\em Proof of the claim.} For $f \in D(\ow{E})_\cN \cap L^\infty(m)$ there exists an $N \in \cN$ with $f 1_{X \setminus N} = 0$. By our assumptions on the nest $\cN$, we can choose a $g_N \in D(E)$ with $1_N \leq g_N \leq 1 $. Since $D(E) \cap L^\infty(m)$ is an algebraic ideal in $D(\ow{E}) \cap L^\infty(m)$, we obtain $f = f g_N \in D(E)$. This shows the claim. \qedc

 (iii) $\Rightarrow$ (ii): This is a consequence of Proposition~\ref{proposition:capacity characterisation of nests}. 
 
 (ii) $\Rightarrow$ (i): This follows from the existence of a special nest, Lemma~\ref{lemma:special nest}.
\end{proof}
\begin{remark}
\begin{itemize}
 \item It is well known in several situations that uniqueness of Silverstein extensions can be characterized with the help of capacities, see also the discussion below. Our approach should be considered as one that provides a structural insight. The equality $\ow{E} = E$ is equivalent to the fact that each $E$-nest is an $\ow{E}$-nest, which can then be characterized in terms of capacity.
 
 \item Choosing a superharmonic function in the previous theorem guarantees that the whole space has finite capacity. This is quite essential for the proof and also for the theorem to hold true. We will discuss examples for its failure in the case of infinite capacity below.    
 
 \item If $\E$ is a Dirichlet form and $\ow{\E}$ is a Silverstein extension of $\E$, then $\E_1$ is a transient energy form and $\ow{\E}_1$ is a transient Silverstein extension of $\E_1$. Therefore, the previous theorem applied to $\E_1$ and $\ow{\E}_1$ is perfectly well suited for deciding whether or not $\E$ coincides with $\ow{\E}$. 
 
 \item The theorem is not true without the transience assumption on the extension. Indeed, if the extension is recurrent, then assertion (i) and (ii) are always satisfied but (iii) need not hold.  For example, consider $X = \{a,b\}$ equipped with the counting measure $\mu$ and the energy forms
$$E:L^0(\mu) \to [0,\infty], \, f \mapsto E(f):= \begin{cases} f(a)^2 & \text{ if } f(b) = 0 \\ \infty & \text{ else }\end{cases}$$
and 
$$\ow{E}:L^0(\mu) \to [0,\infty], \, f  \mapsto \ow{E}(f):= (f(a) - f(b))^2.$$
In this case, $\ow{E}$ is a Silverstein extension of $E$ that satisfies assertion (ii) but differs from $E$.
\end{itemize}
\end{remark} 

The main applications of the previous theorem are the following two criteria for uniqueness of forms  when the underlying space is a topological space.
\begin{example}[Uniqueness of Silverstein extensions with boundary inside the space] 
Let $E$ be a regular energy from and let $\Gamma \subseteq X$ closed. In this case, the space $C_c(X \setminus \Gamma)$ can be identified with the subspace of all functions in $C_c(X)$ that vanish on an open neighborhood of $\Gamma$. We denote the closure of the restriction of $E$ to $C_c(X \setminus \Gamma) \cap D(E)$ by $E_\Gamma$. The definition of the domain of $E_\Gamma$ is a precise version of $\{f \in D(E) \mid f|_\Gamma = 0 \}$ when $m(\Gamma) = 0$. It could also be characterized by functions in the domain of $E$ that vanish quasi-everywhere on $\Gamma$. Since we did not introduce quasi-notions in this text, we refrain from giving details. The following theorem characterizes when $E$ and $E_\Gamma$ coincide.

\begin{theorem} \label{theorem:silverstein uniqueness boundary inside space}
 Let $E$ be a transient regular energy form and let $\Gamma \subseteq X$ closed. Let $h$ be a strictly positive $E$-superharmonic function. The following assertions are equivalent.
 \begin{itemize}
  \item[(i)] $E = E_\Gamma$.
  \item[(ii)] ${\rm cap}_{E,h}^{\rm top}(\Gamma) = 0$.
 \end{itemize}
\end{theorem}
\begin{proof}
 It follows from Proposition~\ref{proposition:silverstein extension of closure} that $E$ is a Silverstein extension of $E_\Gamma$.  Consider the collection
 $$\mathcal{K}_\Gamma:= \{K \subseteq X \setminus \Gamma \mid K\text{ compact}\}.$$
 {\em Claim:} $\mathcal{K}_\Gamma$ is a special $E_\Gamma$-nest. 
 
 {\em Proof of the claim.} It follows immediately from the definitions that $\mathcal{K}_\Gamma$ is a nest, cf. Remark~\ref{remark:nest}. For proving that $\mathcal{K}_\Gamma$  is special, we use almost the same arguments as in the proof of Lemma~\ref{lemma:special nest regular form}. Let $K \in \mathcal{K}_\Gamma$. Since $X$ is a locally compact separable metric space, there exists a function $\varphi \in C_c(X)$ such that $\varphi \geq 3$ on $K$ and ${\rm supp}\, \varphi \subseteq X \setminus \Gamma$. The regularity of $E$ yields the existence of a function $\ow{\varphi} \in C_c(X) \cap D(E)$ with $\|\varphi - \ow{\varphi}\|_\infty < 1$. We set $\psi:= \ow{\varphi} - (\ow{\varphi} \wedge 1)\vee (-1)$. The contraction properties of $E$ imply $\psi \in D(E)\cap C_c(X)$, and the definition of $\psi$  and the properties of $\ow{\varphi}$ yield ${\rm supp}\, \psi \subseteq {\rm supp}\, \varphi \subseteq X \setminus \Gamma$. Therefore, $\psi \in D(E_\Gamma)_{\mathcal{K}_\Gamma}$ and for $x \in K$ it satisfies
 $$\psi(x) =  \ow{\varphi}(x) - (\ow{\varphi}(x) \wedge 1)\vee (-1) \geq \ow{\varphi}(x) - 1 \geq \varphi(x) - 2 \geq 1.$$
 This finishes the proof of the claim. \qedc
 
 The claim and Theorem~\ref{theorem:capacity characterization of uniqueness of Silverstein extensions} show that $E = E_\Gamma$ if and only if
 $$\inf_{K \in  \mathcal{K}_\Gamma} \ceh(X \setminus K)= 0.$$
 The definition of the topological capacity reads
 $${\rm cap}_{E,h}^{\rm top}(\Gamma) = \inf \{\ceh(U) \mid  U\text{ open},\, \Gamma \subseteq U \} = \inf_{C \in  \mathcal{C}_\Gamma} \ceh(X \setminus C),$$
 where $\mathcal{C}_\Gamma := \{ C \subseteq X\setminus \Gamma \mid C \text{ closed}\}.$ Hence, we need to prove 
 $$I_{\mathcal{C}_\Gamma}:=\inf_{C \in \mathcal{C}_\Gamma} \ceh(X \setminus C) = 0$$ 
 if and only if 
 $$I_{\mathcal{K}_\Gamma}:= \inf_{K \in  \mathcal{K}_\Gamma} \ceh(X \setminus K)= 0.$$
 Obviously, the inequality $I_{\mathcal{C}_\Gamma} \leq I_{\mathcal{K}_\Gamma}$ holds and, therefore, $I_{\mathcal{K}_\Gamma} = 0$ implies $I_{\mathcal{C}_\Gamma} = 0$.  
 
 For the other implication we note that Lemma~\ref{lemma:alternative characterization capacity} yields
 $$I_{\mathcal{C}_\Gamma} = \inf_{C \in \mathcal{C}_\Gamma} \inf\{E(h - g) \mid g \in D(E) \text{ and } g 1_{X\setminus C} = 0\} $$
 and
 $$I_{\mathcal{K}_\Gamma} = \inf_{K \in \mathcal{K}_\Gamma} \inf\{E(h - g) \mid g \in D(E) \text{ and } g 1_{X\setminus K} = 0\}.$$
 In particular, $ 0  = I_{\mathcal{C}_\Gamma}$ if and only if $h$ belongs to the $E$-closure of $D(E)_{\mathcal{C}_\Gamma}$ and $0 = I_{\mathcal{K}_\Gamma}$ if and only if $h$ belongs to the $E$-closure of $D(E)_{\mathcal{K}_\Gamma}$. Assume $0  = I_{\mathcal{C}_\Gamma}$ and let $h_n \in D(E)_{\mathcal{C}_\Gamma}$ such that $E(h-h_n) \to 0$, as $n\to \infty$. Since $h$ is nonnegative, we may assume that the $h_n$ are also nonnegative. The form $E$ is regular. Therefore, there exist nonnegative $\varphi_n \in C_c(X)$ with $E(h_n - \varphi_n) < 1/n.$ We set $f_n := h_n \wedge \varphi_n$. Since $h_n$ vanishes outside some set in $\mathcal{C}_\Gamma$ and since $\varphi_n$ vanishes outside some compact subset of $X$, the function $f_n$ vanishes outside some compact subset of $X \setminus \Gamma$. In other words, $f_n \in D(E)_{\mathcal{K}_\Gamma}$. The transience of $E$ implies $h_n \overset{m}{\to} h$ and $\varphi_n \overset{m}{\to} h$. Furthermore, we have 
 $$\limsup_{n \to \infty} E(f_n)^{1/2} \leq \limsup_{n \to \infty}( E(\varphi_n)^{1/2}  + E(h_n)^{1/2}) \leq 2 E(h)^{1/2}.$$
 According to Lemma~\ref{lemma:existence of a weakly convergent subnet}, the previous observations show that the sequence $(f_n)$ converges $E$-weakly to $h$. Since the $E$-weak closure of $D(E)_{\mathcal{K}_\Gamma}$  and  its $E$-closure coincide, we obtain that $h$ belongs to the $E$-closure of $D(E)_{\mathcal{K}_\Gamma}$. As seen above, this yields $I_{\mathcal{K}_\Gamma} = 0$ and finishes the proof. 
\end{proof}

\begin{remark}
   The theorem can be seen as a direct consequence of Theorem~\ref{theorem:capacity characterization of uniqueness of Silverstein extensions} and a continuity statement for the topological capacity. It is a well known meta theorem in Dirichlet form theory and a classical result for Sobolev spaces, which is due to Maz'ya, see  \cite[Theorem~9.2.2]{Maz}. For Dirichlet forms it is usually stated with $h = 1$. If $\E$ is a Dirichlet form, the constant function $1$ is $\E_1$-superharmonic if and only if the underlying measure is finite and $1 \in D(\E)$. This is typically the case when $\E$ comes from an elliptic operator with Neumann-type boundary conditions on a bounded domain. Such forms can be considered to be regular on the closure of the domain and so our theorem can be applied.  However, \cite[Proposition~4.1]{RS}  also treats unbounded domains, a case  where $h = 1$ is not superharmonic but only excessive. It can be shown that our theorem remains true when $h$ is excessive as-well. This requires some additional arguments that lead away from the theory we develop here, and so we do not spell out the details.
\end{remark}

\end{example}
\begin{example}[Uniqueness of Silverstein extensions with boundary at infinity]
Let $E$ be a regular energy form on $L^0(m)$. For any other energy form $\ow{E}$ on $L^0(m)$ and any admissible function $h$ of $\ow{E}$ we define the topological capacity of the boundary of $X$ by
$${\rm cap}^{\rm top}_{\ow{E},h} (\partial X) := \inf \{{\rm cap}_{\ow{E},h} (X \setminus K) \mid K \subseteq X \text{ compact}\}. $$
\begin{theorem}\label{theorem:silverstein uniqueness boundary outside space}
 Let $E$ be a regular energy form and let $\ow{E}$ be a transient Silverstein extension of $E$. Let $h$ be a strictly positive $\ow{E}$-superharmonic function. The following assertions are equivalent.
 \begin{itemize}
  \item[(i)] $E = \ow{E}$. 
  
   \item[(ii)] ${\rm cap}^{\rm top}_{\ow{E},h} (\partial X) = 0$.
 \end{itemize}
\end{theorem}
\begin{proof}
 By Theorem~\ref{theorem:capacity characterization of uniqueness of Silverstein extensions} it suffices to show that the set $\cN := \{K\subseteq X \mid K \text{ compact}\}$  is a special $E$-nest. This follows from Lemma~\ref{lemma:special nest regular form}.
\end{proof}
 \begin{remark}\begin{itemize}
                \item It is quite remarkable that one can characterize $E = \ow{E}$ by the vanishing of the capacity of 'the boundary' even though no boundary is around. The definition of the capacity of 'the boundary' is motivated by the following observation.

 If $\hat{X}$ is a compactification of $X$  and $U$ is an open neighborhood of the topological boundary $\hat{X} \setminus X$ in $\hat{X}$, then $\hat{X} \setminus U = X\setminus U$ is a compact subset of $X$. This shows that the capacity of the boundary equals
  $${\rm cap}^{\rm top}_{\ow{E},h} (\partial X) = \inf \{{\rm cap}_{\ow{E},h} (X \cap U) \mid U \text{ is an open neighborhood of } \hat{X}\setminus X\}.$$
  When $\ow{E}$ is considered as an energy form on the space $L^0(\hat{X},\hat{m}$), for some suitable extension $\hat{m}$ of $m$ that satisfies $\hat{m}(\hat{X} \setminus X)  = 0$,  then the definition of ${\rm cap}^{\rm top}_{\ow{E},h} (\partial X)$ is the 'correct' one. 
 
  \item Related results have been proven in \cite{GM} and \cite{HKMW} for capacities of boundaries that arise from certain metric completions of the underlying space $X$.   In contrast to Theorem~\ref{theorem:silverstein uniqueness boundary inside space}, it is quite essential here to assume that 
  $${\rm cap}^{\rm top}_{\ow{E},h} (\partial X) < \infty.$$
  Otherwise, a trivial counterexample to the theorem is provided by any regular energy form whose boundary has infinite capacity.  For manifolds, see \cite[Proposition~4.6]{GM} and for graphs, see \cite[Example~5.2]{HKMW}, there are nontrivial examples for pairs of energy forms $E$ and $\ow{E}$ with ${\rm cap}^{\rm top}_{\ow{E},1}(\partial X)  = \infty$ that satisfy $E = \ow{E}$. 
  \end{itemize}
 \end{remark}
\end{example}

\section{On maximal Silverstein extensions} \label{subsection:maximal Silverstein extension}


This section is devoted to the construction of  maximal Silverstein extensions with respect to the natural order on quadratic forms (cf. Subsection~\ref{section:basics on quadratic forms}). More precisely, we decompose a given energy form into its main part and a monotone killing part and then show that the main part provides the maximal Silverstein extension when the killing vanishes. Unfortunately, if a killing is present, then a maximal Silverstein extension need not exist. At the beginning of this section we provide  examples for this phenomenon, which  show that some of the theorems  on the existence of maximal Silverstein extensions that can be found in the literature are not correct.  This is followed by a subsection where we construct the reflected energy form, whose maximality properties are then investigated in the last subsection.

The following is probably the easiest example of an energy form that does not admit a maximal Silverstein extension.

\begin{example} \label{example:counterexample maximal Silverstein extension energy form}
 Let $X = \{a,b\}$ and let $\mu$ be the counting measure on all subsets of $X$. The energy form 
 $$E:L^0(\mu) \to [0,\infty], \quad f \mapsto E(f) := \begin{cases}
            f(a)^2 & \text{if } f(b) = 0\\
            \infty &\text{else}
           \end{cases}
$$
 does not possess a maximal Silverstein extension.
\end{example}
\begin{proof}
 Assume that $\ow{E}$ is the maximal Silverstein extension of $E$. The energy form
 $$E_1:L^0(\mu) \to [0,\infty],\quad f \mapsto E_1(f) := (f(a) - f(b))^2$$
 is a Silverstein extension of $E$. Hence, the maximality of $\ow{E}$ implies $L^0(\mu) = D(E_1) \subseteq D(\ow{E})$ and  $\ow{E}(1) \leq E_1(1) = 0$. In particular, $\ow{E}$ is recurrent. The energy form 
  $$E_2:L^0(\mu) \to [0,\infty],\quad f \mapsto E_2(f) := f(a)^2$$
  is also a Silverstein extension of $E$. The maximality of $\ow{E}$, its recurrence and that it is an extension of $E$ yield
  $$1 = E(1_{\{a\}}) = \ow{E}(1_{\{a\}}) = \ow{E}(1_{\{a\}} - 1) \leq E_2(1_{\{a\}} - 1) = 0,$$
  a contradiction. 
\end{proof}

The previous example can be interpreted as a Dirichlet form on the discrete set $\{a,b\}$ but it is not regular. In the literature, the existence of maximal Silverstein extensions is claimed for all regular Dirichlet forms. This is why we also include the next example.

\begin{example} \label{example:counterexample maximal Silverstein extension regular dirichlet form}
We let $\lambda$ be the Lebesgue measure on all Borel subsets of the open interval $(-1,1)$  and let  
$$W^1((-1,1)) = \{f \in L^2(\lambda) \mid f' \in L^2(\lambda)\}$$
the usual Sobolev space of first order. It is folklore that $W^1((-1,1))$ equipped with the norm
$$\|\cdot\|_{W^1}:W^1((-1,1)) \to [0,\infty),\, f \mapsto  \|f\|_{W^1} := \sqrt{\|f\|_2^2 + \|f'\|_2^2} $$
is a Hilbert space, which continuously embeds into $(C([-1,1]),\|\cdot\|_\infty)$. In particular,  any function in $W^1((-1,1))$ can be uniquely extended to the boundary points $-1$ and $1$. We let 
$$W_0^1((-1,1)) := \{ f\in W^1((-1,1)) \mid f(-1) = f(1) = 0\} $$
and note that $W_0^1((-1,1))$ coincides with the closure of $C_c((-1,1)) \cap W^1((-1,1))$ in the space $(W^1((-1,1)),\|\cdot\|_{W^1})$.

\begin{proposition}
The Dirichlet form 
$$\E:L^2(\lambda) \to [0,\infty],\quad f \mapsto \E(f) = \begin{cases} 
                                                        \int_{-1}^1|f'|^2 \D \lambda + f(0)^2 & \text{ if } f \in W_0^1((-1,1))\\
                                                        \infty &\text{ else}
                                                       \end{cases}
$$
is regular and does not possess a maximal Silverstein extension. 
\end{proposition}
\begin{proof}
 The regularity of $\E$ follows from the properties of the Sobolev space $W_0^1((-1,1))$. Assume that $\ow{\E}$ is a Dirichlet form that is the maximal Silverstein extension of $\E$. The Dirichlet forms 
  $$\E_1:L^2(\lambda) \to [0,\infty],\quad f \mapsto \E_1(f) = \begin{cases} 
                                                        \int_{-1}^1|f'|^2 \D \lambda + (f(0)-f(1))^2 & \text{ if } f \in W^1((-1,1))\\
                                                        \infty &\text{ else}
                                                       \end{cases}
$$
 and
 $$\E_2:L^2(\lambda) \to [0,\infty],\quad f \mapsto \E_2(f) = \begin{cases} 
                                                        \int_{-1}^1|f'|^2 \D \lambda + f(0)^2 & \text{ if } f \in W^1((-1,1))\\
                                                        \infty &\text{ else}
                                                       \end{cases}
$$
are  Silverstein extensions of $\E$, as obviously $D(\E)$ is an order ideal in $D(\E_1)$ and $D(\E_2)$. The maximality of $\ow{\E}$ implies $W^1((-1,1)) = D(\E_1) \subseteq (\ow{\E})$ and 
$$\ow{\E}(1) \leq \E_1(1) = 0.$$
We choose a function $f \in W_0^1((-1,1))$ with $f(0) = 1$. The equality $\ow{\E}(1) = 0$ and the maximality of $\ow{E}$ yield 
\begin{align*}
   \int_{-1}^1|f'|^2 \D \lambda + 1 = \E(f) = \ow{\E}(f) = \ow{\E}(f - 1) \leq \E_2(f-1) = \int_{-1}^1|f'|^2 \D \lambda,
\end{align*}
a contradiction. This shows that $\E$ does not possess a maximal Silverstein extension.  
\end{proof}
\end{example}

The last example contradicts \cite[Theorem~5.1]{Kuw} and \cite[Theorem~6.6.9]{CF}, which claim that every (quasi-)regular Dirichlet form has a maximal Silverstein extension. The main reason for this failure is the following: When extending an energy form, a killing term such as  $f(a)^2$ can be generated either by the killing term $f(a)^2$ itself or by the jump-type expression $(f(a) - f(b))^2$. In certain situations, this implies that  a maximal Silverstein extension can not have a killing, see also Proposition~\ref{proposition:counterexample maximality er}.  In the proof of \cite[Theorem~6.6.9]{CF}, the authors note that a killing  can be generated by a jump-type form,  but then use a wrong inequality at the very last line of the proof. This is why the theorem is not correct. We shall see below that the killing part is indeed the only obstruction to the existence of the maximal Silverstein extension. Moreover, if the killing vanishes, then the maximal Silverstein extension of a Dirichlet form is given by its so-called reflected Dirichlet form  as claimed in \cite{Kuw,CF}.

For regular Dirichlet forms the reflected Dirichlet form  was first introduced in \cite{Sil1,Sil} via two different approaches and later some gaps in these construction were closed in \cite{Che}.  For quasi-regular Dirichlet forms the reflected Dirichlet form was defined in \cite{Kuw}. These known approaches towards constructing the reflected Dirichlet forms have in common that they either use the Beurling-Deny formula or the theory of Hunt processes. Both of these tools require the underlying space to be sufficiently regular and are technically rather involved. This stands in contrast to the fact that Silverstein extensions are characterized by a simple algebraic property.  We give a new construction of reflected forms that only uses the algebraic properties of energy forms and therefore works for all energy forms. It is rather explicit and provides a splitting of a given energy form into its main part and its killing part. Since we use the insights of this section for the theory of weak solutions in Chapter~\ref{chapter:weak solutions}, we provide more details than strictly necessary for the construction of the maximal Silverstein extension alone.

\subsection{Construction of the reflected energy form} \label{subsection:construction of er}

In this subsection we construct the reflected energy form of a given energy form. We start with an example that outlines the strategy of the construction and provides interpretations for the involved formulas.

\begin{example}[Jump-type form on a topological space] \label{example:maximal Silverstein extension}
  We let $X$ be a locally compact separable metric space and let $m$ be a Radon measure of full support. We assume that $J$ is a symmetric Radon measure on $X\times X$ that satisfies the following conditions.
 \begin{itemize}
  \item[(J1)] $m(A) = 0$ implies $J(A \times X) = J(X \times A) = 0$.
  \item[(J2)] For each compact $K \subseteq X$ there exists a relatively compact neighborhood $G$ of $K$ with $J(K \times (X \setminus G)) < \infty$.
 \end{itemize}
 Furthermore, we let $0 \leq V \in L^1_{\rm loc}(m)$ and define
 $$E_{J,V}:L^0(m) \to [0,\infty], \quad f \mapsto \int_{X \times X} (f(x) - f(y))^2 \D J(x,y) + \int_X f(x)^2V(x)\D m (x).$$
 As seen in Section~\ref{section:the definition and main examples}, the assumption~(J1) yields that $E_{J,V}$ is an energy form. The assumption that $J$ is a Radon measure combined with (J2) and the integrability condition on $V$ imply $C_c(X) \subseteq D(E_{J,V})$. 
%
%
 We let $E_{J,V}^0$ be the closure of the restriction of $E_{J,V}$ to $C_c(X)$. It is a regular energy form. 
 \begin{proposition}
  $E_{J,V}$ is a Silverstein extension of $E^0_{J,V}$.
 \end{proposition}
\begin{proof}
 According to Proposition~\ref{proposition:silverstein extension of closure}, it suffices to show that all $f \in D(E_{J,V})\cap L^\infty(m)$ with compact support belong to $D(E_{J,V}^0)$. Let such an $f$ be given. Since $m$ is a Radon measure, there exists a sequence $\varphi_n \in C_c(X)$ and a compact set $K$ such that ${\rm supp}\,f,{\rm supp}\, \varphi_n \subseteq K$, $\varphi_n  {\to}f$ $m$-a.e. and $\|\varphi_n\|_\infty \leq \|f\|_\infty$. Let $G$ be a relatively compact open neighborhood of $K$ with $J(K \times (X \setminus G)) < \infty$. We set $\psi_n := f - \varphi_n$ and obtain
 $$E_{J,V}(\psi_n) = \int_{G \times G} (\psi_n(x) - \psi_n(y) )^2 \D J(x,y) + 2 \int_{K \times X \setminus G} \psi_n(x)^2 \D J(x,y) + \int_{K} \psi_n(x)^2 V \D m(x).$$
The $\psi_n$ are uniformly bounded by $2\|f\|_\infty$ and the assumptions on $J$ and $V$ imply that the above integrations are over sets of finite measure. Therefore, Lebesgue's theorem yields $ E_{J,V}(\psi_n) \to 0$, as $n\to \infty$. This finishes the proof. 
\end{proof}
\begin{remark}
 The condition that $J$ is a Radon measure on $X \times X$ excludes measures that are singular near the diagonal $D:= \{(x,y) \mid x,y \in X\}$. In particular, the case when $E$ is the extended Dirichlet form of a symmetric $\alpha$-stable process on $\IR^n$ is not included in the setting presented here (but it is included in the setting of Section~\ref{section:the definition and main examples}). Note however, that we restricted the attention to Radon measures only to keep the technical details of this example simple. Similar results can be proven for measures $J$ that are only Radon on the smaller space $X \times X \setminus D$.
\end{remark}

  It will follow from the theory we develop below that if the potential $V$ vanishes, then $E_{J,V}$ is the maximal Silverstein extension of $E_{J,V}^0$. Since we intend to prove similar results for arbitrary energy forms, this observation raises two challenges: We need to formulate the condition $V = 0$ purely in terms of the form $E_{J,V}^0$ and construct $E_{J,V}$ from the knowledge of $E_{J,V}^0$ on its possibly smaller domain. This can be achieved as follows.  
 
 For $K \subseteq X$ compact we choose a function $\varphi_K \in C_c(X)$ with $1_K \leq \varphi_K \leq 1$. Since $E_{J,V}$ is a Silverstein extension of $E_{J,V}^0$, for $f \in L^\infty(m) \cap D(E_{J,V})$ we have $\varphi_K f , \varphi_K f^2  \in D(E_{J,V}^0)$. A simple computation then shows
 \begin{gather}
  E^0_{J,V}(\varphi_K f) -  E^0_{J,V}(\varphi_K f^2,\varphi_K) = \int_{X \times X} \varphi_K(x) \varphi_K(y) (f(x) - f(y))^2 \,{\rm d}J(x,y). \tag{$\spadesuit$} \label{equation:idea silverstein1}
 \end{gather}
 Letting $K\nearrow X $ (which is possible due to separability) we can reconstruct the $J$-part of $E_{J,V}$ from the knowledge of $E^0_{J,V}$. For $f \in D(E_{J,V}) \cap L^\infty(m)$ Equation~\eqref{equation:idea silverstein1} implies 
 \begin{gather}
  \int_X f^2\varphi_K^2 \, V\D m = E^0_{J,V}(f\varphi_K) - \lim_{K\nearrow X} \left(E^0_{J,V}(\varphi_K^ 2 f) -  E^0_{J,V}(\varphi_K (f\varphi_K)^2,\varphi_K) \right). \tag{$\clubsuit$} \label{equation:idea silverstein2}
 \end{gather}
 Therefore, we can also recover the $V$-part of $E_{J,V}$ from the knowledge of $E^0_{J,V}$. It will turn out that Equation~\eqref{equation:idea silverstein1} and a variant of Equation~\eqref{equation:idea silverstein2} can be employed in the construction of the maximal Silverstein extension.
\end{example}

For the rest of this section, $E$ is a fixed energy form on $L^0(m)$. We now turn to constructing its maximal Silverstein extension by using the strategy that the example suggests, i.e., by exploiting  equations \eqref{equation:idea silverstein1} and \eqref{equation:idea silverstein2}. They involve taking the product of functions in the form domain. Since this operation is only well-behaved for bounded functions, we first consider forms on $L^\infty(m)$ and then extend them to $L^0(m)$. We use a hat to indicate preliminary versions of forms, which are to be extended later.   

For $\varphi \in D(E)$ with $0 \leq \varphi \leq 1$, we define the functional $\Eph :L^\infty(m) \to \IR$ by
$$\Eph (f) := \begin{cases}
               E(\varphi f) - E(\varphi f^2,\varphi) &\text{if } \varphi f , \varphi f^2   \in D(E)\\
               \infty &\text{else}
              \end{cases}
.$$
Equation~\eqref{equation:idea silverstein1} suggests that $\Eph$ is a Markovian quadratic form, which is monotone in the parameter $\varphi$ and dominated by $E$ on $D(E) \cap L^\infty(m)$. For proving these properties  we need the following technical lemma.
\begin{lemma} \label{lemma:maximal silverstein extension technical lemma}
 Let $\varphi \in D(E)$ with $0 \leq \varphi \leq 1$ and let $f \in L^\infty(m)$. 
 \begin{itemize}
  \item[(a)] Let $C:\IR \to \IR$ be $L$-Lipschitz with $C(0) = 0$ and let $M:= \sup \{|C(x)| \mid |x|\leq \|f\|_\infty\}$. Then
  $$E( \varphi\, C(f))^{1/2} \leq L E(\varphi f)^{1/2} + (M + L\|f\|_\infty) E(\varphi)^{1/2}.$$

  In particular, $\varphi f  \in D(E)$ implies $\varphi\, C(f)  \in D(E)$. 
  \item[(b)] If $\psi \in D(E)$ with $0 \leq \psi \leq \varphi$, then $\varphi f  \in D(E)$ implies $\psi f  \in D(E)$. 
 \end{itemize}
\end{lemma}
\begin{proof}
 (a): Let $A :=\{(x,y) \in \IR^2 \mid |x| \leq |y| \cdot \|f\|_\infty\}$ and consider the function 
$$\ow{C}: A \to \IR, \quad (x,y) \mapsto \ow{C}(x,y) := \begin{cases}C\left({x}/{y}\right)y &\text{if } y \neq 0\\0 & \text{if }  y = 0 \end{cases}.$$
We show that $\ow{C}$ is Lipschitz with appropriate constants. The statement  then follows from Theorem~\ref{theorem:cutoff properties energy forms} and the identity $\varphi\, C(f) = \ow{C}(\varphi f, \varphi)$. For $(x_1,y_1),(x_2,y_2) \in A$ with $y_1,y_2 \neq 0$, we have
\begin{align*}
 |\ow{C}(x_1,y_1) - \ow{C}(x_2,y_2)| &\leq |y_1| |C(x_1/y_1) - C(x_2/y_2)| + |C(x_2/y_2)| |y_1 - y_2|.
\end{align*}
Since $|x_i|\leq |y_i| \|f\|_\infty$, $i=1,2$, we obtain $|C(x_2/y_2)| \leq M$ and 
$$|C(x_1/y_1) - C(x_2/y_2)| \leq L \left|\frac{x_1y_2 - x_2 y_1}{y_1y_2}\right| \leq L \left|\frac{x_1  - x_2}{y_1}\right| + L \|f\|_\infty \left|\frac{y_1 - y_2}{y_1}\right|. $$
Altogether, these considerations amount to 
$$|\ow{C}(x_1,y_1) - \ow{C}(x_2,y_2)| \leq L|x_1 - x_2| + (M + L \|f\|_\infty) |y_1 - y_2|.$$
By the continuity of $\ow{C}$ on $A$, this inequality extends to the case when $y_2 = x_2 = 0$, in which it reads  
$$|\ow{C}(x_1,y_1)| \leq L|x_1| + (M + L \|f\|_\infty) |y_1|.$$
From Theorem~\ref{theorem:cutoff properties energy forms} we infer
$$E(\varphi\, C(f))^{1/2} = E(\ow{C}(\varphi f,\varphi))^{1/2} \leq L E(\varphi f)^{1/2} + (M + L\|f\|_\infty) E(\varphi)^{1/2}. $$
This finishes the proof of assertion (a).

(b):  We let $B:=\{(x,y,z) \in  \IR^3  \mid z \geq 0, |x| \leq z \cdot \|f\|_\infty \text{ and } |y| \leq z\}$. For $\varepsilon> 0$, we consider the function
$$C_\varepsilon:B \to \IR,\quad (x,y,z) \mapsto C_\varepsilon(x,y,z):= xy/(z + \varepsilon).$$
From the inequality $0 \leq \psi \leq \varphi$ we obtain
$$C_\varepsilon(\varphi f, \psi ,\varphi) =  \psi  \frac{\varphi}{\varphi + \varepsilon} f \overset{m}{\longrightarrow}  \psi f,  \text{ as } \varepsilon \to 0+. $$
The lower semicontinuity of $E$ implies
$$E(\psi f) \leq \liminf_{\varepsilon \to 0+} E(C_\varepsilon(\varphi f ,\psi,\varphi)). $$
Thus, it suffices to prove that the right-hand side of the above inequality is finite. 

The partial derivatives of $C_\varepsilon$ satisfy $|\partial_1 C_\varepsilon| \leq 1$, $|\partial_2 C_\varepsilon| \leq \|f\|_\infty$ and $|\partial_3 C_\varepsilon| \leq \|f\|_\infty$ in the interior of $B$. This yields
$$|C_\varepsilon(x_1,y_1,z_1) - C_\varepsilon(x_2,y_2,z_2)| \leq |x_1 - x_2| + \|f\|_\infty |y_1-y_2| + \|f\|_\infty |z_1 - z_2|,$$
 for $(x_i,y_i,z_i), i = 1,2,$ in the interior of $B$. Since  $C_\varepsilon$ is continuous and any point in $B$ can be approximated by interior points, we can argue similarly as in the proof of assertion (a) to obtain 
$$E(C_\varepsilon(\varphi f, \psi,\varphi))^{1/2} \leq E(\varphi f)^{1/2} + \|f\|_\infty E(\psi)^{1/2} + \|f\|_\infty E(\varphi)^{1/2} < \infty. $$
This finishes the proof of assertion (b).
\end{proof}

 In view of Equation~\eqref{equation:idea silverstein1}, the following properties of $\Eph$ are not so surprising. 
\begin{lemma}\label{lemma:properties of eph} Let $\varphi,\psi \in D(E)$ with $0\leq \varphi,\psi \leq 1$. 
 \begin{itemize}
 \item[(a)] $\Eph$ is a nonnegative quadratic form on $L^\infty(m)$. Its domain satisfies
 $$D(\Eph) = \{f \in L^\infty(m) \mid  \varphi f \in D(E)\}.$$
 \item[(b)] For every normal contraction $C:\IR \to \IR$ and every $f \in L^\infty(m)$ the inequality 
 $$\Eph(C\circ f) \leq \Eph(f)$$
  holds. 
  \item[(c)] If $\psi \leq \varphi$, then for all $f \in L^\infty(m)$ we have $\widehat{E}_{\psi}(f) \leq \Eph (f)$. 
  \item[(d)] For all $f \in D(E) \cap L^\infty(m)$ the inequality $\Eph (f) \leq E(f)$ holds.
 \end{itemize}
\begin{proof}
(a): Let $f \in L^\infty(m)$. Lemma~\ref{lemma:maximal silverstein extension technical lemma}~(a) applied to the function
$$C:\IR \to \IR,\, x \mapsto x^2 \wedge \|f\|_\infty^2$$
shows that $\varphi  f\in D(E)$ implies $\varphi f^2  \in D(E)$. Therefore, $\Eph$ is a quadratic form with domain $ D(\Eph) = \{f \in L^\infty(m) \mid \varphi f \in D(E)\}$. The nonnegativity of $\Eph$ can be obtained from assertion (c) by letting $\psi = 0$. 

For proving assertions (b), (c) and (d),  we argue similarly as in the proof of Theorem~\ref{theorem:cutoff properties energy forms} and reduce everything to the case when $E$ is a continuous Markovian form on $\Ltf$. Here, the reduction is a bit more complicated since the definition of $\Eph$ involves taking off-diagonal values of $E$, which are only well-defined on $D(E)$. This is why we first prove the following claim.

{\em Claim:} It suffices to show (b), (c) and (d) under the additional assumptions that the restriction of $E$ to $\Ltf$ is continuous and satisfies $D(E) \cap \Ltf = \Ltf$.  

{\em Proof of the claim.} The restriction of $E$ to $\Ltf$ is closed and Markovian. By Proposition~\ref{proposition:Markovian approximation} there exists a net of continuous Markovian forms $(E_i)$  on $\Ltf$  with $D(E_i) = \Ltf$ such that
$$E(f) = \lim_i E_i (f)$$
holds for all $f \in \Ltf$. In particular, for $g_1,g_2 \in D(E) \cap \Ltf$ this implies
$$E(g_1,g_2) = \lim_i E_i(g_1,g_2).$$
Since $f \in D(\Eph)$ yields  $\varphi f , \varphi f^2  \in D(E) \cap \Ltf$,  we obtain 
$$\Eph(f) = \lim_i \widehat{E}_{i,\varphi}(f) \text{ for all }f \in D(\Eph).$$
We assume that (b), (c) and (d) hold true for each of the $(E_i)$ and show that they also hold true for $E$. 

(b): Let $f \in D(\Eph)$ and let $C:\IR \to \IR$ be a normal contraction. From Lemma~\ref{lemma:maximal silverstein extension technical lemma} we infer $ \varphi\, C(f) \in D(E)$, which implies $C(f) \in D(\Eph)$ by (a). We obtain 
$$\Eph(C\circ f) = \lim_i \widehat{E}_{i,\varphi}(C\circ f ) \leq \lim_i \widehat{E}_{i,\varphi}(f) = \Eph(f), $$
where we used the validity of (b) for $E_i$. 

(c): We let $f \in D(\Eph).$ In this case, Lemma~\ref{lemma:maximal silverstein extension technical lemma} yields $ f\psi \in D(E)$ and  (a) implies $f \in D(\widehat{E}_\psi)$. We obtain
$$\widehat{E}_\psi(f) = \lim_i \widehat{E}_{i,\psi}(f) \leq \lim_i \widehat{E}_{i,\varphi}(f) = \Eph(f), $$
where we used the validity of (c) for $E_i$.

(d): Since $D(E)\cap L^\infty(m)$ is an algebra, the inclusion $D(E)\cap L^\infty(m) \subseteq D(\Eph)$ holds. Thus, for $f \in D(E) \cap L^\infty(m)$ we obtain 
$$\Eph(f) = \lim_i \widehat{E}_{i,\varphi}(f) \leq \lim_i E_i(f) = E(f),$$
where  we used the validity of (d) for $E_i$. This finishes the proof of the claim. \qedc
%

 According to the previous claim, we can assume that the restriction of $E$ to $\Ltf$ is continuous and satisfies $D(E) \cap \Ltf = \Ltf$. Therefore, it suffices to prove the assertions for simple functions.  Let $f$ be of the form
 $$f = \sum_{i=1}^n \alpha_i 1_{A_i}$$
 with pairwise disjoint measurable sets $A_i$. We  obtain
 $$\Eph(f) = \sum_{i,j = 1}^n b^\varphi_{ij} (\alpha_i - \alpha_j)^2 + \sum_{i = 1}^n c^\varphi_{i} \alpha_i^2$$
 with 
 $$b^\varphi_{ij} = -\Eph (1_{A_i},1_{A_j}) =   - E(\varphi 1_{A_i} , \varphi 1_{A_j}) \text{ and } c^\varphi_i = \Eph(1_{A_i}, 1_{\cup_j A_j}) = E(\varphi 1_{A_i},\varphi (1_{\cup_j A_j} - 1)), $$
  and 
 $$E(f) = \sum_{i,j = 1}^n b_{ij} (\alpha_i - \alpha_j)^2 + \sum_{i = 1}^n c_{i} \alpha_i^2$$
 with 
 $$b_{ij} = -E(1_{A_i},1_{A_j})  \text{ and } c_i = E(1_{A_i}, 1_{\cup_j A_j}).$$
 By Remark~\ref{remark:cutoff remark} we can apply Lemma~\ref{lemma:cutoff for functions with disjoint support} to continuous Markovian forms on $\Ltf$. It implies that whenever $0\leq \psi \leq \varphi$ the inequalities
 $$0 \leq b_{ij}^\psi  \leq b_{ij}^\varphi \leq b_{ij} \text{ and } 0 \leq c_i^\psi \leq c_i^\varphi$$
 hold. They show (c). Furthermore, (b)  is a consequence of the nonnegativity of $b_{ij}^\varphi$ and $c_j^\varphi$, cf. proof of Theorem~\ref{theorem:cutoff properties energy forms}. It remains to prove the inequality $c_i^\varphi  \leq c_i$ to obtain (d). We compute
 \begin{align*}
  c_i - c_i^\varphi &= E(1_{A_i},1_{\cup_j A_j}) +  E(\varphi 1_{A_i}, \varphi (1-1_{\cup_j A_j})) \\
  &= E((1-\varphi)1_{A_i},1_{\cup_j A_j}) +  E(\varphi 1_{A_i}, 1_{\cup_j A_j} + \varphi(1-1_{\cup_j A_j})). 
 \end{align*}
 Lemma~\ref{lemma:cutoff for functions beeing one on support} implies that right-hand side of the above equation is positive. This finishes the proof. 
\end{proof}
\end{lemma}
\begin{remark}
 \begin{itemize}
  \item The proof of the previous lemma uses the same strategy as the proof of Theorem~\ref{theorem:cutoff properties energy forms}. The main difficulty is that an approximation of $\Eph$ by continuous forms is only valid on its domain. 
  \item The definition of $\Eph$ is similar to the form characterization of energy measures of Dirichlet forms. When $\E$ is a regular Dirichlet form, the energy measure $\Gamma(f)$  of a function  $f \in D(\E) \cap L^\infty(m)$ is uniquely determined by the equation
  $$\int_X \varphi \D \Gamma(f) = 2 \E(f\varphi,f) - \E(f^2,\varphi), \text{ for all } \varphi \in D(\E) \cap C_c(X).$$
  That the latter formula can be applied for constructing the maximal Silverstein extension was observed in \cite{Tak} for elliptic operators and in  \cite{GM} for manifolds.  For general local forms similar considerations can be found in \cite{Rob}. The locality of the given form implies that the energy measure $\Gamma(f)$ and $2 \E(f\varphi,f) - \E(f^2,\varphi)$ can be extended to the local form domain. The advantage of using $\Eph$ rather than the form characterizations of energy measures is that it can be  directly applied to functions outside the form domain.  
 \end{itemize}
\end{remark}

  Equation~\eqref{equation:idea silverstein1} provides an  explicit formula for $\Eph$ when $E = E^0_{J,V}$. In this case, $\Eph$ is a  jump-type form and therefore it is closable, see the discussion in Section~\ref{section:the definition and main examples}.  The following lemma shows that the closability of $\Eph$  is true for general energy forms.
  
  \begin{lemma}\label{lemma:closability of ep} 
    Let $\varphi \in D(E)$ with $0\leq \varphi \leq 1$. The form $\Eph$ is lower semicontinuous with respect to $\tau(m)$, i.e., for any net $(f_i)$ in $L^\infty(m)$ and $f \in L^\infty(m)$ the convergence $f_i \overset{m}{\to} f$ implies
 $$\Eph(f) \leq \liminf_i \Eph(f_i).$$
 In particular, $\Eph$ is closable on $L^0(m)$. Its closure $\Ep$ is a recurrent energy form whose domain satisfies 
 $$D(\Ep) \cap L^\infty(m) = D(\Eph).$$
  \end{lemma}
  \begin{proof}
   Let $(f_i)$ be a net in $L^\infty(m)$ and  let $f \in L^\infty(m)$ with $f_i \overset{m}{\to} f$. We show
 $$\Eph(f)\leq \liminf_i \Eph(f_i). $$
 Without loss of generality, we assume $f_i \in D(\Eph)$ for all $i$ and
 $$\lim_i \Eph(f_i) = \liminf_i \Eph(f_i) < \infty. $$
This allows us to also assume that $(\Eph(f_i))$ is bounded. Furthermore, by Lemma~\ref{lemma:properties of eph}~(b) it suffices to consider the case when $\|f_i\|_\infty \leq \|f\|_\infty$, as otherwise we could replace $f_i$ by $(f_i \wedge \|f\|_\infty)\vee(-\|f\|_\infty).$ 
 
 Lemma~\ref{lemma:maximal silverstein extension technical lemma}~(a) applied to the function $C:\IR \to \IR, \, C(x) = x^2 \wedge \|f\|_\infty ^2$ yields
 $$E(\varphi f_i^2)^{1/2} = E(\varphi\, C(f_i))^{1/2} \leq 2 \|f\|_\infty E(\varphi f_i )^{1/2} + 3 \|f\|_\infty ^2 E(\varphi)^{1/2}.$$
 From this inequality, we infer
 \begin{align*} 
  \Eph(f_i) &= E(\varphi f_i) - E(\varphi f_i^2,\varphi) \\
  &\geq E(\varphi f_i) - E(\varphi f_i^2)^{1/2}E(\varphi)^{1/2}\\
  &\geq E(\varphi f_i)^{1/2} \left(E(\varphi f_i)^{1/2}  - 2\|f\|_\infty E(\varphi)^{1/2} \right) - 3\|f\|^2_\infty E(\varphi).
 \end{align*}
Therefore, the boundedness of $(\Eph(f_i))$ yields the boundedness of $(E(\varphi f_i))$ and $(E(\varphi f_i^2))$. Together with the lower semicontinuity of $E$ this implies $\varphi f, \varphi f^2 \in D(E)$, i.e., $f \in D(\Eph)$. From the boundedness of $(E(\varphi f_i^2))$ and the convergence $\varphi f_i^2 \overset{m}{\to}\varphi f^2$, we obtain the $E$-weak convergence $\varphi f_i^2   \to \varphi f^2$, see Lemma~\ref{lemma:existence of a weakly convergent subnet}. Altogether, these observations amount to 
$$\Eph(f) = E(\varphi f) - E(\varphi f^2,\varphi) \leq \liminf_i E(\varphi f_i) - \lim_i E(\varphi f_i^2,\varphi) = \liminf_i \Eph(f_i).$$
This shows lower semicontinuity.

For the 'in particular' part, we note that the lower semicontinuity of $\Eph$ and Proposition~\ref{prop:characterization closability}  yield the closability of $\Eph$. Since $\Eph$ is Markovian, its closure $\Ep$ is an energy form, see Proposition~\ref{theorem:closure is markovian}. It is recurrent, as $\Ep(1) = \Eph(1) = E(\varphi) - E(\varphi) = 0$.  The statement on the domain follows directly from the lower semicontinuity of $\Eph$.
  \end{proof}

\begin{definition}
 Let $\varphi \in D(E)$ with $0\leq \varphi \leq 1$. The closure of $\Eph$ is denoted by $\Ep$.
\end{definition}

\begin{remark}
The previous lemma contains a stronger statement than the closability of $\Eph$ on $L^0(m)$, cf. the assumption of Proposition~\ref{prop:characterization closability}. It states that $\Eph$ is a lower semicontinuous form on the (incomplete)  topological vector space $(L^\infty(m),\tau(m))$. 
\end{remark}

The next lemma shows that the properties of $\Eph$ pass to its closure $\Ep$.

\begin{lemma}\label{lemma:properties of ep} Let $\varphi, \psi \in D(E)$ with $0\leq \varphi,\psi \leq 1$. 
\begin{itemize}
 \item[(a)] If $\psi \leq \varphi$, then $E_{\psi}(f) \leq \Ep  (f)$ for all $f \in L^0(m)$. 
  \item[(b)] For all $f \in D(E) $ the inequality $\Ep (f) \leq E(f)$ holds.
  \item[(c)] For $f,g \in L^0(m)$ the identity $\varphi f  =  \varphi g$ implies $\Ep(f) = \Ep(g)$.
\end{itemize}

\end{lemma}
\begin{proof} Assertions (a) and (b) are immediate consequences of the denseness of $D(\Eph)$ in $D(\Ep)$ and Lemma~\ref{lemma:properties of eph}. 

(c): Without loss of generality we can assume $f \in D(\Ep)$. For $n \in \IN$, Lemma~\ref{lemma:closability of ep} implies $f^{(n)}:= (f \wedge n) \vee (-n) \in D(\Eph)$. In particular, we have  $ \varphi f^{(n)} \in D(E)$. The  assumption $\varphi g =  \varphi f$ yields $\varphi g^{(n)} =  \varphi  f^{(n)}  \in D(E)$, and so  $g^{(n)} \in D(\Eph) = D(\Ep) \cap L^\infty(m)$ by Lemma~\ref{lemma:properties of eph}. Since $\Eph$ only depends on the values of the function on the support of $\varphi$,  we conclude
$$\Ep(g) = \lim_{n \to \infty}\Ep(g^{(n)}) = \lim_{n \to \infty} \Eph(g^{(n)}) = \lim_{n \to \infty}\Eph(f^{(n)}) = \lim_{n \to \infty}\Ep(f^{(n)}) = \Ep(f),$$
where we used Proposition~\ref{proposition:approximation by bounded functions} for the first and the last equality.
\end{proof}
With these preparations we introduce the part of $E$ that corresponds to the $J$-integral in Example~\ref{example:maximal Silverstein extension}. 

\begin{definition}[Main part] \label{definition:main part}
 The functional $\Egm: L^0(m) \to [0,\infty]$ defined by 
$$\Egm(f):= \sup \left\{\Ep(f) \mid \varphi \in D(E) \text{ with } 0\leq \varphi \leq 1\right\}$$
is called the {\em main part of $E$}.
\end{definition}

\begin{lemma}
 $\Eg$ is a recurrent energy form. 
\end{lemma}
\begin{proof}
 It follows from its definition and Lemma~\ref{lemma:inequality characterization of quadratic forms} that $\Eg$ is a quadratic form. Its Markov property  and the recurrence can be inferred from Lemma~\ref{lemma:closability of ep}. The closedness of $\Eg$ is a consequence of the closedness of the $\Ep$ and Lemma~\ref{lemma:monotone nets}.
\end{proof}
Equation~\eqref{equation:idea silverstein1} shows that in general $\Eg$ is not an extension of the given form $E$. In order to obtain an extension, we need to compensate the error by using Equation~\eqref{equation:idea silverstein2}. We define  
$$\Ek: D(E) \to [0,\infty),\quad f \mapsto \Ek(f) := E(f) - \Eg(f).$$
It is a perturbation of $\Eg$ which is rather well behaved; namely, it is monotone with respect to the absolute value of a given function. This is discussed in the next lemma. 
\begin{lemma} \label{lemma:properties of ek}
Let $f,g \in D(E)$.
\begin{itemize}
 \item[(a)] $\Ek$ is a nonnegative quadratic form on $D(E)$. 
 \item[(b)] $|f| \leq |g|$ implies $\Ek(f) \leq \Ek(g)$.
 \item[(c)] $fg\geq  0$ implies $\Ek(f,g) \geq 0.$
 \item[(d)] If $h,l \in D(E)$ satisfy $fg = hl$, then $\Ek(f,g) = \Ek(h,l)$.
\end{itemize}
\end{lemma}
\begin{proof}
 (a): This follows from Lemma~\ref{lemma:properties of ep}~(c) and the fact that $E$ and $\Eg$ are quadratic forms. 
 
(c): The inequality  $fg \geq 0$ implies $|f+g|\geq |f-g|$. Hence, assertion (b) yields
$$\Ek(f,g) = \frac{1}{4} \left(\Ek(f+g) - \Ek(f-g)\right) \geq 0.$$

 (b):  $\Ek$ is continuous with respect to $E$. Since bounded functions are dense in $D(E)$ with respect to the form topology, see  Proposition~\ref{proposition:approximation by bounded functions}, we can assume $f,g \in D(E) \cap L^\infty(m)$.
 
 {\em Claim:} The inequality $\Ek(f) \leq \Ek(g)$ holds under the additional assumption that there exists a function $\psi\in D(E)$ with $1_{\{|g| > 0\}} \leq \psi \leq 1$. 
 
 {\em Proof of the claim.}   Let $\varepsilon > 0$. By the definition of $\Eg$ and the monotonicity of $\Ep$ in the parameter $\varphi$, see Lemma~\ref{lemma:properties of ep}~(a), there exists $\varphi \in D(E) \cap L^\infty(m)$ with $ \psi \leq \varphi \leq 1$ and 
 $$|\Eg(g) - \Ep(g)| < \varepsilon \text{ and } |\Eg(f) - \Ep(f)| < \varepsilon .$$
 From the inequalities $ 1 \geq \varphi \geq \psi \geq 1_{\{|g| > 0\}}$ and $|f| \leq |g|$   we infer $f \varphi = f$ and $g \varphi = g$. This yields
 $$\Ek(g) - \Ek(f) \geq E(g) - E(f) - E_\varphi (g) + E_\varphi (f) - 2\varepsilon = E(g^2 - f^2,\varphi) - 2\varepsilon.$$
 Since $\varphi = 1$ on the set $\{g^2 - f^2 > 0\}$, Lemma~\ref{lemma:cutoff for functions beeing one on support} implies $E(g^2 - f^2,\varphi) \geq 0$ and the claim is proven. \qedc
 
 For $\alpha > 0$ we let
 $$ C_\alpha:\IR \to \IR,\quad x \mapsto C_\alpha(x) := (x-\alpha)_+ - (x+\alpha)_-.$$
 Proposition~\ref{proposition:approximation by bounded functions} shows that for any $f \in D(E) \cap L^\infty(m)$ we have $E(C_\alpha \circ f -  f) \to 0$, as $\alpha \to 0+$. Since $\Ek$ is continuous with respect to $E$, the same holds true for convergence with respect to $\Ek$. The inequality $|f| \leq |g|$ implies $|C_\alpha \circ f| \leq |C_\alpha \circ g|$ and the function 
 $$\psi_\alpha:= (\alpha^{-1} |g|) \wedge 1 \in D(E)$$ 
 satisfies $\psi_\alpha \geq 1_{ \{|C_\alpha g| > 0\}}$. Therefore, we can apply the previous claim to the functions $C_\alpha \circ f$ and $C_\alpha \circ g$ to obtain
 $$\Ek(f) = \lim_{\alpha \to 0+} \Ek(C_\alpha \circ f) \leq \lim_{\alpha \to 0+} \Ek(C_\alpha \circ g) = \Ek(g).$$

 (d): As in the proof of (b), we can assume that all the functions are bounded and that there exists a function $\psi \in D(E)$ that is equal to one on their supports. For $\varphi \in D(E)$ with $\psi\leq \varphi \leq 1$, we obtain 
 \begin{align*}
  E(f,g) - \Ep(f,g) &= E(f,g) - E(\varphi f,\varphi g) + E(\varphi f g,\varphi)\\
  &= E(\varphi f g,\varphi)\\
  &= E(\varphi h l,\varphi)\\
  &= E(h,l) - \Ep(h,l).
  \end{align*}
 With this identity at hand, the claim follows from the definition of $\Ek$. This finishes the proof. 
\end{proof}

The previous lemma shows that $\Ek$ behaves like a quadratic form that is induced by a measure. In general, such forms tend to be not closable.  In order to extend it, we use its monotonicity with respect to the absolute value of functions. 
\begin{definition}[Killing part]  \label{definition:killing part}
The functional $\Ekm:L^0(m)\to [0,\infty]$ defined by  
$$\Ekm(f) :=  \begin{cases}
               \sup \left\{\Ek(g)\, \middle | \, g \in D(E) \text{ with } |g| \leq |f|\right\} &\text {if } f \in D(\Egm)\\
               \infty& \text{else}
              \end{cases}$$
  is called the {\em killing part of E}. 
\end{definition}

The next lemma discusses properties of $\Ekm$.

\begin{lemma}\label{lemma:characterization of k}
\begin{itemize}
 \item[(a)] $\Ekm$ is a nonnegative quadratic form on $L^0(m)$, which extends $\Ek$.
 \item[(b)] For $f,g \in D(\Egm)$, the inequality $|f| \leq |g|$ implies $\Ekm(f) \leq \Ekm(g)$. 
 \item[(c)]  If $f,g \in D(\Ekm)$ satisfy $fg \geq 0$, then $\Ekm(f,g) \geq 0$. Furthermore, if $h,l \in D(\Ekm)$ satisfy $fg = hl$, then $\Ekm(f,g) = \Ekm(h,l)$. 
 \item[(d)]   Let $f \in D(\Egm)$ and, for $n\in \IN$,  set $f^{(n)} := (f\wedge n) \vee (-n)$. Then 
$$\Ekm(f) = \sup \left\{\Ek\left(\varphi f^{(n)}\right)\, \middle | \,   \varphi \in D(E) \text{ with } 0\leq \varphi \leq 1,\, n \in \IN \right\}.$$

 \item[(e)] Let $(f_i)$ be a net in $L^0(m)$ and let $f \in  L^0(m)$ such that $f_i \overset{m}{\to} f$. If 
 $$\liminf_i \Egm(f_i) < \infty,$$
 then
 $$\Ekm(f) \leq \liminf_i \Ekm(f_i).$$
\end{itemize}

\end{lemma}
\begin{proof}
(d): Let $\varphi \in D(E)$ with  $0\leq \varphi \leq 1$.  Lemma~\ref{lemma:closability of ep} shows that $f \in D(\Egm)$ implies $f^{(n)} \in D(\Eph)$ and, therefore, $\varphi f^{(n)} \in D(E) \cap L^\infty(m)$.    Consequently, 
 $$k(f) := \sup \left\{\Ek\left(\varphi f^{(n)}\right)\, \middle | \,   \varphi \in D(E) \text{ with } 0\leq \varphi \leq 1,\, n \in \IN \right\}  $$
 is well defined. The inequality $|\varphi f^{(n)}| \leq |f|$ yields $k(f) \leq \Ekm(f)$. Hence, it suffices to show $k(f) \geq \Ekm(f)$. 
 
 To this end, we choose a sequence $(f_l)$ in $D(E) \cap L^\infty(m)$ that satisfies $|f_l| \leq |f|$ and $\Ek(f_l) \nearrow \Ekm(f)$, as $l \to \infty$. By Lemma~\ref{lemma:properties of ek}~(b) we can assume $f,f_l \geq 0$. For $\varepsilon > 0$, consider the functions $C_1,C_2: [0,\infty) \to \IR$ defined by
 $$C_1(x) = \frac{x}{x+\varepsilon} \text{ and } C_2(x) = \frac{x^2}{x+\varepsilon}.$$
 They satisfy $C_1(0) = C_2(0) = 0$, $|C_1'(x)| \leq 1/\varepsilon$ and $|C'_2(x)| \leq 1 $. Therefore, the contraction properties of $E$ imply 
 $$\varphi_{\varepsilon,l} :=  \frac{f_l}{f_l + \varepsilon} = C_1 \circ f_l \in D(E)$$
 and 
 $$E(\varphi_{\varepsilon,l} f_l ) = E(C_2 \circ f_l) \leq E(f_l).$$
 According to Lemma~\ref{lemma:characterization convergence in form topology lsc forms}, this last inequality and the convergence $\varphi_{\varepsilon,l} f_l  \overset{m}{\to} f_l$, as $\varepsilon \to 0+$, yield $E(f_l - \varphi_{\varepsilon,l} f_l ) \to 0$, as $\varepsilon \to 0+$.  Since the form $\Ek$ is continuous with respect to   $E$, we obtain
 $$\Ek(f_l) = \lim_{\varepsilon \to 0+} \Ek(\varphi_{\varepsilon,l} f_l)  \leq  \limsup_{\varepsilon \to 0+}\Ek(\varphi_{\varepsilon,l} (f\wedge \|f_l\|_\infty) ) \leq k(f).$$
 Note that we used Lemma~\ref{lemma:properties of ek}~(b) for the first inequality. 
 
 (a): The monotonicity of $\Ek$, see Lemma~\ref{lemma:properties of ek}~(b), and the definition of $\Ekm$  imply that $\Ekm$ is an extension of $\Ek$.  

We use Lemma~\ref{lemma:inequality characterization of quadratic forms} to prove that $\Ekm$ is a quadratic form. For $\lambda \in \IR$ and   $f \in D(\Ekm)$  the inequality
$$\Ekm(\lambda f) \leq |\lambda|^2\Ekm(f) $$
follows from the definition of $\Ekm$ and the fact that $\Ek$ is a quadratic form. Next, for $f,g \in D(\Ekm)$ we prove the inequality
$$2\Ekm(f) + \Ekm(g) \leq \Ekm(f+g) + \Ekm(f-g),$$
from which the claim follows. To this end, let $f,g \in D(\Ekm)\subseteq D(\Egm) $ and let
$$(\varphi,n) \in I := \{(\psi,m) \mid \psi \in D(E) \text{ with } 0\leq \psi \leq 1, m\in \IN\}.$$
Since $\Ek$ is a quadratic form on $D(E)$ and $\varphi f^{(n)},\varphi g^{(n)} \in D(E)$, we obtain 
$$\Ek(\varphi (f^{(n)}+g^{(n)})) + \Ek(\varphi(f^{(n)}-g^{(n)})) = 2 \Ek(\varphi f^{(n)}) + 2 \Ek(\varphi g^{(n)}),$$
and  the  monotonicity of the expressions in $\varphi$ and $n$ yields 
$$
 \sup_{(\varphi,n) \in I} \left[ 2 \Ek(\varphi f^{(n)}) + 2 \Ek(\varphi g^{(n)}) \right] 
 = 2\sup_{(\varphi,n) \in I}  \Ek(\varphi f^{(n)}) +  2 \sup_{(\varphi,n) \in I}  \Ek(\varphi g^{(n)}).
$$
%
%
Combining these identities and using (d) we infer 
\begin{align*}
 2 \Ekm(f) + 2 \Ekm(g) &= 2\sup_{(\varphi,n) \in I}  \Ek(\varphi f^{(n)}) +  2 \sup_{(\varphi,n) \in I}  \Ek(\varphi g^{(n)})\\
 &= \sup_{(\varphi,n) \in I} \left[ 2 \Ek(\varphi f^{(n)}) + 2 \Ek(\varphi g^{(n)}) \right] \\
 &= \sup_{(\varphi,n) \in I}  \left[\Ek(\varphi (f^{(n)}+g^{(n)})) + \Ek(\varphi(f^{(n)}-g^{(n)}))\right]\\
 &\leq \Ekm(f+g) + \Ekm(f-g).
\end{align*}
For the last inequality we used  $|\varphi(f^{(n)}+g^{(n)})| \leq |f+ g|$ and $|\varphi(f^{(n)}-g^{(n)})| \leq |f-g|$, and the definition of $\Ekm$. This  finishes the proof of (a).

 (e): Let $\varphi \in D(E)$ with $0\leq \varphi \leq 1$ and let $n\in \IN$. We first prove $\varphi f^{(n)} \in D(E) = D(\Ek)$ and
 $$\Ek(\varphi f^{(n)} ) \leq \liminf_i \Ek(\varphi f_i^{(n)}).$$
 The closedness of $\Egm$  yields $f\in D(\Egm)$, which implies $\varphi f^{(n)} \in D(E).$ Furthermore, from Theorem~\ref{theorem:algebraic and order properties} we obtain
 $$\Eg(\varphi f_i^{(n)})^{1/2} \leq n \Eg(\varphi)^{1/2} + \Eg(f_i)^{1/2}.$$
 Hence, by the assumption $\liminf_i \Eg(f_i) < \infty$, we can assume that the net $(\Egm(\varphi f_i^{(n)}))_i$ is bounded. The definition of $\Ek$ and Lemma~\ref{lemma:properties of ek}~(b) yield
 $$E(\varphi f_i^{(n)}) = \Eg(\varphi f_i^{(n)}) + \Ek(\varphi f_i^{(n)}) \leq \Eg(\varphi f_i^{(n)}) + n^2 \Ek(\varphi).$$
 Therefore, the net $(E(\varphi f_i^{(n)}))_i$ can be assumed to be bounded as well. These observations combined with $\varphi f_i^{(n)} \overset{m}{\to} \varphi f^{(n)}$ let us infer  $\varphi f_i^{(n)} \to \varphi f^{(n)}$ $E$-weakly and $\Eg$-weakly from Lemma~\ref{lemma:existence of a weakly convergent subnet}. As a consequence, we have the  $\Ek$-weak convergence  $\varphi f_i^{(n)} \to \varphi f^{(n)}$. The inequality 
 $$\Ek(\varphi f^{(n)} ) \leq \liminf_i \Ek(\varphi f_i^{(n)})$$
 follows from the lower semicontinuity of quadratic forms with respect to weak convergence, see Lemma~\ref{lemma:form weak convergence implies lower limit}. Here, we used (a), i.e., that $\Ekm$ is a quadratic form. With this at hand, the statement is a consequence of (d) and the inequality $\Ek (\varphi f_i^{(n)}) \leq \Ekm (f_i)$.

(b): This follows from definition of $\Ekm$ and Lemma~\ref{lemma:properties of ek}~(b). 

(c): The first statement can be deduced from (b) as in the proof of Lemma~\ref{lemma:properties of ek}~(c). The 'furthermore' statement follows from (d) and Lemma~\ref{lemma:properties of ek}~(d).
\end{proof}

 \begin{remark}
 \begin{itemize}
 \item The assertions of the lemma are valid for functions from different domains. In general, the domain of $\Ekm$ is strictly smaller than the domain of $\Eg$. This is why assertion~(c) is stated the way it is. Furthermore, assertion (b) can fail for functions that do not belong to the domain of $\Eg$.  
 
 \item Assertion~(d) states that $\Ekm$ is a closed quadratic form on the topological vector space $(D(\Eg),\tau(m)_{\Eg})$.
 
 \item  In concrete applications the main part of an energy form is usually coming from the geometry of the underlying space while the killing part is some perturbation by a positive measure.

  In Example~\ref{example:maximal Silverstein extension}, the main part and the killing part of the form $E_{J,V}^0$ can be computed explicitly with the help of Equation~\eqref{equation:idea silverstein1} and  Equation~\eqref{equation:idea silverstein2}. With them it is more or less straightforward that
  $$(E_{J,V}^0)^{(M)} = E_{J,0} \text{ and } (E_{J,V}^0)^{(k)} = E_{0,V}|_{D(E_{J,0})},$$
  which is consistent with our narrative.
  
  \item One may wonder why we restricted the domain of $\Ekm$ to $D(\Egm)$. Its formal definition via the supremum could easily be extended to $L^0(m)$. However, in this case $\Ekm$ can fail to be a quadratic form as the following example shows. 
  
  Let $X = (-1,1)$ and let $\lambda$ the Lebesgue measure. As before, we denote the usual first oder Sobolev space by  $W^1((-1,1)) = \{f \in L^2(\lambda) \mid f' \in L^2(\lambda)\}$. Consider the energy form    

  $$E:L^0(\lambda) \to [0,\infty],\quad f \mapsto E(f) := \begin{cases}
             \int_{-1}^1 |f'(x)|^2\D \lambda + f(0)^2 &\text{if }f \in W^1((-1,1))\\
             \infty &\text{else }
            \end{cases}.
$$
Since $1 \in D(E)$, it is easy to check that its preliminary killing part $\Ek$ is given by the (not closable) form
$$\Ek:W^1((-1,1))\to [0,\infty), \, f \mapsto \Ek(f) = f(0)^2.$$
For the moment, we let
$$k:L^0(\lambda) \to [0,\infty],\quad f \mapsto k(f) :=  \sup \left\{\Ek(g)\, \middle | \, g \in D(E) \text{ with } |g| \leq |f|\right\}.$$
Consider the indicator function $1_{(0,1)}$. Since functions in $W^1((-1,1))$ are continuous, any  $g \in W^1((-1,1))$ with $|g| \leq 1_{(0,1)}$ satisfies $g(0) = 0$. Therefore, $k(1_{(0,1)}) = 0$. The same argumentation also yields $k(1_{(-1,0)}) = 0$. Since $|1_{(0,1)} \pm 1_{(-1,0)}| = 1$ in $L^0(\lambda)$, and since $k$ only depends on the absolute value of the function, this implies $k (1_{(0,1)} \pm 1_{(-1,0)}) = k(1) = 1$. Altogether, we obtain
$$k(1_{(0,1)}  +  1_{(-1,0)}) + k(1_{(0,1)} - 1_{(-1,0)}) = 2 > 0 = 2 k(1_{(0,1)}) +  2 k(1_{(-1,0)}),$$
 which shows that $k$ is not a quadratic form. 
 \end{itemize}
 \end{remark}
 With the help of all the preparations made in this subsection we can now introduce the reflected energy form and show that it is a Silverstein extension of the given energy form.
 
\begin{definition}[Reflected energy form]\label{definition:reflected form}
 The form $\Er:= \Eg + \Ekm$ is called the {\em reflected energy form of E}.   
\end{definition}

 \begin{theorem} \label{theorem:properties of er}
  Let $E$ be an energy form. Its reflected energy form $\Er$ is an energy form and a Silverstein extension of $E$.  
 \end{theorem}
\begin{proof}
 The closedness of $\Er$ follows from the closedness of $\Eg$ and Lemma~\ref{lemma:characterization of k}. That $\Er$ is an extension of $E$ is an immediate consequence of the definition of $\Ek$ and the fact that $\Ekm$ extends $\Ek$. For  $\varphi  \in D(E)$ with $0 \leq \varphi \leq 1$, Lemma~\ref{lemma:closability of ep} implies
 $$D(\Er) \cap L^\infty(m) \subseteq D(\Ep) \cap L^\infty(m) = D(\Eph) =  \{f \in L^\infty(m) \mid \varphi f \in D(E) \}.$$
 This shows that $D(E)\cap L^\infty(m)$ is an algebraic ideal in $D(\Er) \cap L^\infty(m)$, i.e., $\Er$ is a Silverstein extension of $E$.  
\end{proof}
\begin{remark}
 In Example~\ref{example:maximal Silverstein extension} the reflected energy form of $E_{J,V}^0$ is given by $E_{J,V}$.  Moreover, when $\E$ is a (quasi-)regular Dirichlet form   similar computations as in Example~\ref{example:maximal Silverstein extension} show that $(\Ee^{{\, \rm ref}}, D(\Ee^{{\, \rm ref}}))$ coincides with the so-called reflected Dirichlet space and that $(\Ee^{{\, \rm ref}}, D(\Ee^{{\, \rm ref}}) \cap L^2(m))$ coincides with the so-called active reflected Dirichlet space as defined in \cite{Kuw,CF}; we refrain from giving details. This justifies the name reflected energy form. More examples of reflected spaces are discussed at the end of the next subsection. 
\end{remark}
In view of the previous remark we also make the following definition for Dirichlet forms.
\begin{definition}[Reflected Dirichlet form] \label{definition:reflected dirichlet form}
 Let $\E$ be a Dirichlet form on $L^2(m)$. The {\em main part of $\E$} is the restriction of $\Ee^{(M)}$ to $L^2(m)$ and is denoted by $\E^{(M)}$, and the {\em killing part of $\E$} is the restriction of $\Ee^{(k)}$ to $L^2(m)$ and is denoted by $\E^{(k)}.$ The {\em reflected Dirichlet form of $\E$} is defined by $\E^{{\rm ref}} :=\E^{(M)} + \E^{(k)}$.
\end{definition}

We finish this subsection with an  alternative formula for $\Ekm$ and $\Er$, which holds for certain functions in $D(E)$. For this purpose, we let 
$$I_E:= \{\varphi \in D(E) \mid 0 \leq \varphi  \leq 1\}$$
and order this set with the natural order relation on functions. The monotonicity of $\Ep$ in $\varphi$ yields
$$\Eg(f) = \lim_{\varphi \in I_E} \Ep(f) \text{ for  all } f \in D(\Eg).$$
Furthermore, for $f \in L^0(m)$ and $n \in \IN$, we set $f^{(n)}:= (f \wedge n) \vee (-n)$. 
\begin{lemma} \label{lemma:alternative formula for er}
 Let $\cN$ be a special $E$-nest. For $f \in D(\Er)$ and  $\psi \in D(E)_\cN$, we have
 $$\Ekm(f,\psi) = \lim_{n \to \infty} \lim_{\varphi \in I_E} \lim_{m \to \infty}  E( f^{(n)}\psi^{(m)}, \varphi)$$
 and 
 $$\Er(f,\psi) =  \lim_{n \to \infty} \lim_{\varphi \in I_E} E(\varphi f^{(n)},\psi). $$
\end{lemma}
\begin{proof}
We first prove the statement about $\Ekm$. Since $\cN$ is special and $\psi \in D(E)_\cN$, there exist $\varphi_0,\varphi_1 \in D(E)_\cN$ with $1_{\{|\psi| > 0\}} \leq \varphi_0 \leq 1$ and $1_{\{\varphi_0 > 0\}} \leq \varphi_1 \leq 1$. We use that $\Ekm$ is continuous with respect to $\Er$ and Lemma~\ref{lemma:characterization of k} to obtain 
$$\Ekm(f,\psi) = \lim_{n \to \infty} \Ekm(f^{(n)},\psi) =\lim_{n \to \infty} \Ekm(\varphi_0 f^{(n)},\psi).$$
The form $\Er$ is a Silverstein extension of $E$, hence $\varphi_0 f^{(n)} \in D(E)$. From the definition of $\Ekm$ on $D(E)$ we infer
\begin{align*}
 \Ekm(\varphi_0 f^{(n)},\psi) &= E(\varphi_0 f^{(n)}, \psi) - \Eg(\varphi_0 f^{(n)}, \psi) \\
 &= E(\varphi_0 f^{(n)}, \psi) - \lim_{\varphi \in I_E} \Ep (\varphi_0 f^{(n)}, \psi)\\
 &= E(\varphi_0 f^{(n)}, \psi) - \lim_{\varphi \in I_E} \lim_{m \to \infty} \left[  E(\varphi \varphi_0  f^{(n)}, \varphi \psi^{(m)}) - E(\varphi \varphi_0 f^{(n)} \psi^{(m)}, \varphi) \right].  
\end{align*}
 The choice of $\varphi_0$ and $\varphi_1$ yields that for $\varphi \in I_E$ with $\varphi \geq \varphi_1$ we have  $\varphi \varphi_0  f^{(n)} = \varphi_0  f^{(n)}$, $\varphi \psi^{(m)} = \psi^{(m)}$ and    $\varphi \varphi_0 f^{(n)} \psi^{(m)} = f^{(n)} \psi^{(m)}.$ Therefore, the  the previous equation simplifies to
\begin{align*}
  \Ekm(\varphi_0 f^{(n)},\psi) = \lim_{\varphi \in I_E} \lim_{m \to \infty} E(f^{(n)} \psi^{(m)}, \varphi).
\end{align*}
Combining these computations shows 
$$\Ekm(f,\psi) = \lim_{n \to \infty} \lim_{\varphi \in I_E} \lim_{m \to \infty}  E( f^{(n)}\psi^{(m)}, \varphi).$$
Since  $\varphi \psi^{(m)} = \psi^{(m)}$ and $\varphi f^{(n)}\psi^{(m)} = f^{(n)}\psi^{(m)}$ whenever $\varphi_0  \leq \varphi \leq 1$, we obtain 
\begin{align*}
 \Eg(f,\psi) &= \lim_{n\to \infty} \lim_{\varphi \in I_E} \lim_{m\to \infty}\Ep(f^{(n)},\psi^{(m)})\\
 &= \lim_{n\to \infty} \lim_{\varphi \in I_E} \lim_{m\to \infty} \left[ E(\varphi f^{(n)}, \varphi \psi^{(m)}) - E(\varphi f^{(n)} \psi^{(m)},\varphi)\right]\\
 &=  \lim_{n\to \infty} \lim_{\varphi \in I_E} \lim_{m\to \infty} \left[ E(\varphi f^{(n)}, \psi^{(m)}) - E( f^{(n)} \psi^{(m)},\varphi)\right]\\
 &= \lim_{n\to \infty} \lim_{\varphi \in I_E} E(\varphi f^{(n)}, \psi) - \Ekm(f,\psi).
\end{align*}
This finishes the proof. 
\end{proof}

\begin{remark}\label{remark:alternative formula for er}
 For the previous lemma to hold true, the condition that $\psi\in D(E)_\cN$ for some special nest $\cN$ can be weakened. Indeed, in the proof we only used that there exist $\varphi_0,\varphi_1 \in D(E)$ with $1_{\{|\psi|>0\}} \leq \varphi_0 \leq 1$ and $1_{\{\varphi_0 > 0\}} \leq \varphi_1 \leq 1$.
\end{remark}

 \subsection{The maximality of the reflected energy form and applications}

 This subsection is devoted to maximality properties of the reflected energy form. More precisely, we prove that the reflected energy form is the maximal Silverstein extension whenever the killing part vanishes and we provide an abstract version of the counterexamples at the beginning of this section.   Moreover, we show that recurrence and uniqueness of Silverstein extensions are much more intimately related than it could be expected. We finish this subsection by explicitly computing the reflected energy form for several examples.
 
 The following lemma is the main ingredient for proving maximality statements on $\Er$.  
 
 \begin{lemma} \label{lemma:maximality of eg}
  Let $E$ be an energy form and let $\ow{E}$ be a Silverstein extension of $E$. Its domain satisfies $D(\ow{E}) \subseteq D(\Eg)$ and for all $f \in D(\ow{E})$ the inequality
  $$\ow{E}^{(M)}(f) \geq \Eg(f)$$
  holds. Moreover, for $g \in D(E)$ we have
  $$ \ow{E}^{(k)} (g) \leq \Ekm(g).$$
 \end{lemma}
 \begin{proof}
 For proving the first inequality we can assume $f \in D(\ow{E}) \cap L^\infty(m)$. Let $\varphi \in D(E)$ with $0 \leq \varphi \leq 1$. Since $\ow{E}$ is a Silverstein extension, we have $\varphi f \in D(E)$ and $\varphi f^2 \in D(E)$. This observation and the definition of $\ow{E}^{(M)}$ yield
 $$\ow{E}(f) \geq \ow{E}^{(M)}(f) \geq \ow{E}_\varphi(f) = \ow{E}(\varphi f) - \ow{E}(\varphi f^2,\varphi) = E(\varphi f) - E(\varphi f^2, \varphi) =  \Ep(f).$$
 Taking the supremum over all such $\varphi$ implies $D(\ow{E}) \subseteq D(\Eg)$ and $\ow{E}^{(M)} (f) \geq \Eg(f)$. For $g \in D(E)$,  the inequality on the main parts yields
 $$\Ekm(g) = E(g) - \Eg(g) \geq \ow{E}(g) - \ow{E}^{(M)}(g) = \ow{E}^{(k)}(g).$$
 This finishes the proof. 
 \end{proof}

 The following theorem is the main theorem concerning the maximality of the reflected energy form. 
 
 \begin{theorem}[Maximality of $\Er$] \label{theorem:maximality er}
  Let $E$ be an energy form. Let $\ow{E}$ be a Silverstein extension of $E$ such that for all $f \in D(E)$ the inequality $\Ekm(f) \leq \ow{E}^{(k)}(f)$ holds. Then  $\ow{E} \leq \Er$. In particular, if $\Ekm = 0$, then $\Er = \Eg$ is the maximal Silverstein extension of $E$. 
 \end{theorem}
\begin{proof}
  Let $\ow{E}$ be a Silverstein extension of $E$ and let $f \in D(\ow{E})$.  Lemma~\ref{lemma:maximality of eg} yields $\ow{E}^{(M)}(f) \geq \Eg(f)$. Furthermore,  the monotonicity of $\ow{E}^{(k)}$ and the definition of $\Ekm$ imply
  \begin{align*}
   \ow{E}^{(k)}(f)&\geq \sup\{\ow{E}^{(k)}(g) \mid g \in D(E) \text{ with } |g| \leq |f|\}\\
   &\geq  \sup\{E^{(k)}(g) \mid g \in D(E) \text{ with } |g| \leq |f|\}\\
   &= \Ekm(f).
  \end{align*}
  In terms of the order on quadratic forms, which we introduced in Subsection~\ref{section:basics on quadratic forms}, this means $\ow{E} \leq \Er$. The statement on the maximality of $\Er$ when $\Ekm = 0$ follows from what we have shown so far and the fact that $\Er$ is a Silverstein extension of $E$, see Theorem~\ref{theorem:properties of er}. This finishes the proof.
\end{proof}

 For Dirichlet forms the contents of the previous theorem read as follows.

\begin{corollary}\label{corollary:maximality er dirichlet form}
 Let $\E$ be a Dirichlet form on $L^2(m)$. Let $\ow{\E}$ be a Silverstein extension of $\E$ such that for all $f \in D(\E)$ the inequality $\E^{(k)}(f) \leq \ow{\E}^{(k)}(f)$ holds. Then  $\ow{\E} \leq \E^{{\rm ref}}$. In particular, if $\E^{(k)} = 0$, then $\E^{{\rm ref}}= \E^{(M)}$ is the maximal Silverstein extension of $\E$. 
\end{corollary}

\begin{proof}
 This follows from the previous theorem and Theorem~\ref{theorem:silverstein extensions of dirichlet forms}.
\end{proof}

\begin{remark} \label{remark:maximality er}
\begin{itemize}
 \item As discussed at the beginning of this section, the previous theorem and its corollary are extensions and a correction of the main result of \cite{Kuw} and of \cite[Theorem~6.6.9]{CF}. Due to some additional assumptions, the maximality statements on the reflected Dirichlet form in \cite{Sil} and in \cite{Tak} are correct. In \cite{Sil} the maximality of the reflected form is only claimed among all Silverstein extensions whose associated operators are contained in the 'local generator' of the given form. The notion of 'local generator' does not make sense for arbitrary Dirichlet forms but we can prove a similar statement for forms whose associated operator is a self-adjoint extension of a given densely defined symmetric operator, see Theorem~\ref{theorem:maximality er operator}. The work \cite{Tak} treats Dirichlet forms of certain differential operators  whose killing part vanishes.
 
 \item  The form $\Eg$ satisfies a more general maximality statement. If $\ow{E}$ is an energy form such that $D(E) \cap L^\infty(m)$ is an algebraic ideal in $D(\ow{E})\cap L^\infty(m)$ and for all $f,g \in D(E)$ the inequality $fg \geq 0$ implies
   $$E(f,g) \geq \ow{E}(f,g),$$
    then $\ow{E} \leq \Eg$ in the sense of quadratic forms. Note that in this case, $\ow{E}$ need not be an extension of $E$. Here, we only prove a special case, see Theorem~\ref{theorem:maximality eg}; the details for general forms will be given elsewhere.  When considering restrictions to $L^2(m)$, the condition $E(f,g) \geq \ow{E}(f,g)$ whenever $fg \geq 0$ is related to the domination of the associated semigroups, see \cite{MVV}. In this sense, the restriction of $\Eg$ to $L^2(m)$ is the maximal form whose associated semigroup dominates the semigroup of the restriction of $E$ to $L^2(m)$.
\end{itemize}
\end{remark}

 As already discussed, maximal Silverstein extensions do not always exists. The following proposition shows that our examples at the beginning of this section are a prototype for this phenomenon.
 
 \begin{proposition} \label{proposition:counterexample maximality er}
  Let $E$ be an energy form.  If $E$ has a recurrent Silverstein extension and $\Ekm \neq 0$, then $\Er$ is not the maximal Silverstein extension of $E$. If, additionally, $1 \in D(\Ekm)$, then $E$ does not possess a maximal Silverstein extension. 
 \end{proposition}
 \begin{proof}
  Let $E_1$ be a recurrent Silverstein extension of $E$. Assume that $\Er$ is the maximal Silverstein extension of $E$. We obtain 
  $$\Ekm(1) \leq \Er(1) \leq E_1(1) = 0.$$
  Due to the monotonicity of $\Ekm$, this observation implies $\Ekm = 0$, a contradiction to the assumption that $\Ekm \neq 0$.
  
  Now, assume $1 \in D(\Ekm)$ and that $\ow{E}$ is the maximal Silverstein extension of $E$. Due to its maximality, we have $\ow{E}(1) \leq E_1(1) = 0$, i.e., $\ow{E}$ is recurrent. For $f \in D(E)$ and $\alpha \in \IR$, we obtain 
  $$\Eg(f) + \Ekm(f) = E(f) = \ow{E}(f) = \ow{E}(f-\alpha) \leq \Er (f-\alpha) = \Eg(f) + \Ekm(f-\alpha).$$
  This implies 
  $$2 \alpha \Ekm(f,1) \leq \alpha^2 \Ekm(1).$$
  Dividing by $\alpha$ and letting $\alpha \to 0\pm$ yields $\Ekm(f,1) = 0$ for all $f \in D(E)$. According to Lemma~\ref{lemma:properties of ek}, for each  $f \in D(E)\cap L^\infty(m)$ we have 
  $$\Ekm(f) = \Ekm(f^2,1) = 0.$$
  Since bounded functions are $E$-dense in $D(E)$ and since $\Ekm$ is continuous with respect to $E$, these considerations show $\Ekm = 0$, a contradiction to the assumption $\Ekm \neq 0$. This finishes the proof.   
 \end{proof}

 \begin{remark}
 The observation that lies at the heart of the previous proposition is that maximal Silverstein extensions of forms that possess a recurrent Silverstein extension can not have a killing part. In particular, $\Er$ can not always be the maximal Silverstein extension. 
 \end{remark}

  The following theorem is a weaker version of the  general maximality statement for the form $\Eg$ that was mentioned in Remark~\ref{remark:maximality er}.  One of its corollaries is quite important for the theory in Chapter~\ref{chapter:weak solutions}.

 \begin{theorem}\label{theorem:maximality eg}
  Let $E$ and $\ow{E}$ be energy forms with the following properties:
  \begin{itemize}
   \item $D(E) \cap L^\infty(m)$ is an algebraic ideal in $D(\ow{E}) \cap L^\infty(m)$.
   \item For all $f\in D(E) \cap L^\infty(m)$ the identity $E(f) = \ow{E}(f) + \Ekm(f)$ holds.  
  \end{itemize}
 Then $D(\ow{E}) \subseteq D(\Eg)$ and for all  $f \in D(\ow{E})$ we have $\Eg(f) \leq \ow{E}(f)$. 
 \end{theorem}
 \begin{proof}
  Let $f \in D(\ow{E})\cap L^\infty(m)$ and let $\varphi \in D(E)$ with $0\leq \varphi \leq 1$.  Since $D(E) \cap L^\infty(m)$ is an algebraic ideal in $D(\ow{E}) \cap L^\infty(m)$, we have $\varphi f \in D(E)$. Furthermore, the second property of $\ow{E}$ yields $\varphi \in D(\ow{E})$. With the help of Lemma~\ref{lemma:properties of ep} applied to $\ow{E}$, we deduce
  \begin{align*}
   \ow{E}(f) &\geq \ow{E}_\varphi(f)\\
   &= \ow{E} (\varphi f) - \ow{E}(\varphi f^2, \varphi)\\
   &= \ow{E} (\varphi f) + \Ekm(\varphi f) -  \Ekm(\varphi f^2 ,\varphi) - \ow{E}(\varphi f^2, \varphi)\\
   &= E(\varphi f) - E(\varphi f^2,\varphi)\\
   &= \Ep(f).
  \end{align*}
  Here, we used the identity $\Ekm(\varphi f) =  \Ekm(f^2\varphi,\varphi)$  of Lemma~\ref{lemma:characterization of k}~(c). Taking the supremum over all such $\varphi$ finishes the proof.
 \end{proof}
 There is an immediate corollary to  the previous theorem that provides an alternative characterization of $\Eg$ in terms of special nests (cf. Section~\ref{section:nests and local spaces}).  
 \begin{corollary}\label{corollary:alternative formula for eg} 
  Let $E$ be an energy form. Let $\cN$ be a special $E$-nest and let 
  $$I_\cN := \{\varphi \in D(E)_\cN \mid 0 \leq \varphi \leq 1\}.$$
  For all $f \in L^0(m)$ we have
  $$\Eg(f) = \sup_{\varphi \in I_\cN} \Ep(f).$$
 \end{corollary}
 \begin{proof}
  We define the form $\ow{E}:L^0(m) \to [0,\infty]$ by
  $$\ow{E}(f):= \sup_{\varphi \in I_\cN} \Ep(f).$$
  The same arguments that we used to show that $\Eg$ is an energy form apply to $\ow{E}$. Therefore, $\ow{E}$ is an energy form. For proving the statement, it suffices to confirm the two assumptions of Theorem~\ref{theorem:maximality eg}. 
  
   We start by showing the identity $E(f) = \ow{E}(f) + \Ekm(f)$ for $f \in D(E) \cap L^\infty(m)$. To this end, we first let $f \in D(E)_\cN \cap L^\infty(m)$. Since $\cN$ is special, there exists a function $\psi \in I_\cN$ such that $\psi \geq 1_{\{|f| > 0\}}$. For any $\varphi \in D(E)$ with $\psi \leq \varphi \leq 1$ we have $\varphi f = f$ and, therefore,
  $$E(f) - \Ep(f) = E(f) - E(\varphi f) + E(\varphi f^2, \varphi) = E(f^2,\varphi).$$
  The set  $I_\cN$ is dense in  $\{\varphi \in D(E) \mid 0\leq \varphi \leq 1\}$ with respect to the form topology and, for $\varphi \geq \psi$,  the expression $E(f^2,\varphi)$ is nonnegative and monotone decreasing in $\varphi$, see Lemma~\ref{lemma:cutoff for functions with disjoint support} and Lemma~\ref{lemma:cutoff for functions beeing one on support}.  These considerations imply
  \begin{align*}
   \Ekm(f) &= E(f) - \Eg(f)\\
   &= \inf \{E(f) - \Ep(f) \mid \varphi \in D(E) \text{ with } \psi \leq \varphi \leq 1\}\\
   &= \inf \{E(f^2,\varphi) \mid \varphi \in D(E) \text{ with } \psi \leq \varphi \leq 1\}\\
   &= \inf \{E(f^2,\varphi) \mid \varphi \in I_\cN \text{ with } \psi \leq \varphi \}\\
   &= E(f) - \ow{E}(f).
  \end{align*}
  Since $D(E)_\cN \cap L^\infty(m)$ is dense in $D(E) \cap L^\infty(m)$ and $\ow{E},\Ekm$ are continuous with respect to $E$, the above equation extends to all $f \in D(E) \cap L^\infty(m)$. 
  
  It remains to prove that $D(E)\cap L^\infty(m)$ is an algebraic ideal in $D(\ow{E})\cap L^\infty(m)$. By the definition of $\ow{E}$, the algebra $D(E)_\cN \cap L^\infty(m)$ is an algebraic ideal in $D(\ow{E})\cap L^\infty(m)$; we have to show that this is stable under taking limits. Let $f \in D(E) \cap L^\infty(m)$ and let $g \in D(\ow{E})\cap L^\infty(m)$. We choose a net $(f_i) \in D(E)_\cN \cap L^\infty(m)$ that satisfies $\|f_i\|_\infty \leq \|f\|_\infty$ and $f_i \to f$ in the form topology of $E$. What we have already shown yields $\ow{E}(f_i) \leq E(f_i)$, $f_i g \in D(E)\cap L^\infty(m)$ and $E(f_i g) = \ow{E}(f_ig) + \Ekm(f_i g)$. From the lower semicontinuity of $E$  we obtain 
  $$E(fg) \leq \liminf_i E(f_i g) = \liminf_i\left[ \ow{E}(f_ig)  +   \Ekm(f_i g)\right].$$
  Since $\ow{E}$ is an energy form, Theorem~\ref{theorem:algebraic and order properties} yields
  $$\ow{E}(f_i g)^{1/2} \leq \|f_i\|_\infty\ow{E}(g)^{1/2} + \|g\|_\infty \ow{E}(f_i)^{1/2} \leq  \|f\|_\infty\ow{E}(g)^{1/2} + \|g\|_\infty E(f_i)^{1/2}. $$
  Furthermore, the properties of $\Ekm$ imply
  $$\Ekm(f_i g) \leq  \|g\|_\infty^2\Ekm(f_i) \leq \|g\|_\infty^2E(f_i). $$
  Combining these observations, we infer $E(fg) < \infty$. This finishes the proof. 
 \end{proof}

  We have not only obtained the maximal Silverstein extension but also constructed a decomposition of the given energy form into a main part and a killing part. With this at hand, we can give a further characterization of recurrence. It shows that from a conceptual viewpoint uniqueness of Silverstein extensions and recurrence are almost the same.

\begin{theorem} \label{theorem:recurrence and uniqueness}
The following assertions are equivalent.
\begin{itemize}
 \item[(i)] $E$ is recurrent.
 \item[(ii)] $E = \Egm$.
 \item[(iii)] ${\rm Sil}(E) = \{E\}$ and $\Ekm = 0$. 
\end{itemize}
\end{theorem}
\begin{proof}
 (i) $\Rightarrow$ (ii): Theorem~\ref{theorem:recurrent forms are Silverstein unique} shows that recurrent energy forms are Silverstein unique.  According to Theorem~\ref{theorem:maximality er}, it suffices to prove $\Ekm = 0$. For an $f \in D(E) \cap L^\infty(m)$ the recurrence of $E$ and the monotonicity of $\Ekm$ yield
 $$\Ekm(f) \leq \|f\|_\infty^2 \Ekm(1) \leq \|f\|_\infty^2 E(1) = 0. $$
 Since bounded functions are $E$-dense in $D(E)$, this finishes the proof of the claim.
 
 (ii) $\Rightarrow$ (iii): Since $E= \Eg$, the form $E$ is recurrent. Theorem~\ref{theorem:recurrent forms are Silverstein unique} implies its Silverstein uniqueness and the vanishing of $\Ekm$ follows from the identity $E = \Egm$.
 
 (iii) $\Rightarrow$ (i): The form $\Er$ is a Silverstein extension of $E$. Therefore, the conditions ${\rm Sil}(E) = \{E\}$ and $\Ekm = 0$ yield $E = \Er  =\Eg$. This proves the recurrence of $E$.
\end{proof}

\begin{corollary}
 Let $m$ be finite and let $\E$ be a Dirichlet form on $L^2(m)$ with $\E^{(k)} = 0$. The following assertions are equivalent.
 \begin{itemize}
  \item[(i)]  $\E$ is Silverstein unique.
  \item[(ii)] $\Ee$ is recurrent.
 \end{itemize}
\end{corollary}
\begin{proof}
 This follows from the previous theorem, Theorem~\ref{theorem:silverstein extensions of dirichlet forms} and the fact that the finiteness of $m$ yields $1 \in D(\Ee^{(M)}) \cap L^2(m) =  D(\E^{(M)})$.
\end{proof}
\begin{remark}
  A slightly weaker form of the previous theorem for quasi-regular Dirichlet forms is contained in \cite[Corollary~6.1]{Kuw} and the corollary can be found in \cite{HKLMS} for regular Dirichlet forms.
\end{remark}

We finish  this subsection with some examples of the reflected energy forms.

\begin{example}[Jump-type form on a topological space]
 In Example~\ref{example:maximal Silverstein extension}, the reflected energy form of $E_{J,V}^0$ is given by $E_{J,V}$.  Moreover, if $V = 0$, then $E_{J,0}$ is the maximal Silverstein extension of $E_{J,0}^0$. 
\end{example}

  \begin{example}[Energy forms on Riemannian manifolds]
   Let $(M,g)$ be a Riemannian manifold.  Recall the energy forms $E_{(M,g)}^0$ and $E_{(M,g)}$, which were introduced in Subsection~\ref{subsection:manifolds}, and the notation of Appendix~\ref{appendix:manifolds}. Using the developed theory, we shall prove that $E_{(M,g)}$ is the maximal Silverstein extension of  $E_{(M,g)}^0$. Unfortunately, for an arbitrary $\varphi \in D(E_{(M,g)}^0)$ with $0 \leq \varphi \leq 1$ it is somewhat tedious and technical to determine the precise domain of $E_{(M,g),\, \varphi}^{0}$. We limit ourselves to the following special case from which the maximality will follow. For the rest of this example, we let $Q:= E_{(M,g)}^0$ to simplify notation, i.e., to avoid too many subscripts and superscripts.
   
   \begin{lemma}\label{lemma:ep for manifolds}
   Let $f \in W^1_{\rm loc}(M)$ and let $\varphi \in D(Q)$ with $0 \leq \varphi \leq 1$ have compact support. Then $f \in D(Q_\varphi)$ and 
   $$Q_\varphi(f) = \int_M \varphi ^2 |\nabla f|^2 \D {\rm vol}_g.$$
   \end{lemma}
   \begin{proof}
   We choose a sequence of bounded smooth normal contractions $C_n:\IR \to \IR$ that satisfy $C_n(x) = x$ for all $x \in [-n,n]$ and set $f_n := C_n \circ f$. According to Lemma~\ref{lemma:contraction manifold},  we have  $f_n \in L^\infty({\rm vol}_g)$ and $\nabla f_n = C'_n(f) \nabla f \in \Ltlm$. Since $\varphi$ is bounded and has compact support, $\varphi f_n$ belongs to  $L^2({\rm vol}_g)$ and has compact support as well. Furthermore, an application of Lemma~\ref{lemma:product rule} and the properties of $\varphi$ and $f_n$  show
   $$\nabla(\varphi f_n) = \varphi \nabla f_n + f_n \nabla \varphi \in L^2({\rm vol}_g).$$
   With this at hand, Lemma~\ref{lemma:product in w0} yields $\varphi f_n \in D(Q)$, i.e., $f_n \in D(Q_\varphi)$. The  product rule, Lemma~\ref{lemma:product rule}, and the identity $\nabla f_n = C_n'(f) \nabla f$ imply
   \begin{align*}
   Q_\varphi(f_n) &= Q(\varphi f_n) - Q(\varphi f_n^2,\varphi)\\
    &= \int_M \varphi^2 |\nabla f_n|^2 \D {\rm vol}_g\\
    &= \int_M \varphi^2 C_n'(f)^2 |\nabla f|^2 \D {\rm vol}_g.
   \end{align*}
   This computation shows that $Q_\varphi(f_n)$ is bounded. Since $f_n \overset{m}{\to} f$, the lower semicontinuity of $Q_\varphi$ yields $f \in D(Q_\varphi)$. The $f_n$ are normal contractions of $f$ and so we have
   $$Q_\varphi(f_n) \leq Q_\varphi(f).$$
   According to Lemma~\ref{lemma:characterization convergence in form topology lsc forms}, we obtain
   $$Q_\varphi(f) = \lim_{n\to \infty} Q_\varphi(f_n) = \lim_{n\to \infty} \int_M \varphi^2 C_n'(f)^2 |\nabla f|^2 \D {\rm vol}_g = \int_M \varphi^2 |\nabla f|^2 \D {\rm vol}_g.$$
   For the last  equality, we used that the derivatives $C'_n$ are uniformly bounded by one and pointwise converge to one. This finishes the proof.
   \end{proof}

   \begin{theorem} 
    The  form $E_{(M,g)}$ is the maximal Silverstein extension of $Q$ ($=E_{(M,g)}^0)$. 
   \end{theorem}
   \begin{proof}
    We start by proving that $E_{(M,g)}$ is a Silverstein extension of $Q$. To this end, let $f \in D(E_{(M,g)}) \cap L^\infty(m)$ and let  $\psi \in C_c^\infty(M)$. It follows from  Lemma~\ref{lemma:product in w0} that $f \psi \in D(Q)$. Since $C_c^\infty(M)$ is dense in $D(Q)$ with respect to the form topology, Proposition~\ref{proposition:silverstein extension of closure} shows that $E_{(M,g)}$ is a Silverstein extension of $Q$. 
    
    Next, we prove that the killing of $Q$ vanishes. We let $(\varphi_n)$ be a sequence in $C_c^\infty(M)$ with $0 \leq \varphi_n \leq 1$ and $\varphi_n \nearrow 1$ pointwise, as $n \to \infty$. For $f \in D(Q)$, the previous lemma and the properties of $Q^{(M)}$ show
    $$Q(f) \geq Q^{(M)}(f) \geq \lim_{n \to \infty}Q_{\varphi_n} (f) = \lim_{n \to \infty} \int_M \varphi_n^2 |\nabla f| \D {\rm vol}_g = Q(f). $$ 
    Therefore, the killing part of $Q$ vanishes. With this at hand, Theorem~\ref{theorem:maximality er} implies that the form  $Q^{(M)}$ is the maximal Silverstein extension of $Q$ and it remains to prove $Q^{(M)} \leq E_{(M,g)}$. Let $f \in D(Q^{(M)}) \cap L^\infty({\rm vol}_g)$ be given. Since $Q^{(M)}$ is a Silverstein extension of $Q$, we have $\varphi f \in D(Q)$ for all $\varphi \in C_c^\infty(M)$. Due to the locality of the operator $\nabla$, this implies $|\nabla f| \in \Ltlm$. From the previous lemma and the monotone convergence theorem, we infer
    $$Q^{(M)}(f) \geq \lim_{n \to \infty} Q_{\varphi_n} (f) =  \lim_{n \to \infty} \int_M \varphi_n ^2 |\nabla f|^2 \D {\rm vol}_g = \int_M  |\nabla f|^2 \D {\rm vol}_g.$$
    Hence, $f \in D(E_{(M,g)})$ and $Q^{(M)}(f)  \geq E_{(M,g)}(f)$. Since bounded functions are dense in the form domains, this finishes the proof.
   \end{proof} 
   
   \begin{remark}
    The maximality result for the $L^2$-restrictions of $E_{(M,g)}$ and $E_{(M,g)}^0$ is explicitly proven in \cite{GM}, but also follows from the (corrected) considerations in \cite{Kuw,CF}. The maximality statement without the restriction to $L^2({\rm vol}_g)$ seems to be new.
   \end{remark}

  \end{example}
  \begin{example}[Dirichlet forms of hypoelliptic operators]
  Here, we use the notation and the assumptions of Example~\ref{example:hypoelliptic}, i.e., we assume that $S$ is a densely defined symmetric operator on $L^2(m)$ such that the form
   $$\E_S:L^2(m) \to [0,\infty], f \mapsto \E_S(f) = \begin{cases} 
                                                      \as{Sf,f} &\text{if } f \in D(S)\\
                                                      \infty &\text{else}
                                                     \end{cases}
$$
   is Markovian. In this subsection, we studied (maximal) Silverstein extensions of quadratic forms, while  Example~\ref{example:hypoelliptic} deals with self-adjoint extensions of the  operator $S$. In general, closed extensions of the form $\E_S$ need not provide self-adjoint extensions of $S$. However,  we show, under some mild assumptions on the domain of the  operator $S$, that the operator of the reflected Dirichlet form is an extension of $S$ and the maximal element in ${\rm Ext}(S)_{\rm Sil}$.  In particular, this result can be applied to certain hypoelliptic operators on Riemannian manifolds. 
   
    To shorten notation, we let $\E := \bar{\E}_S$, the closure of $\E_S$, and denote  the self-adjoint operator of $\E^{\, \rm ref}$ by $S^{\rm ref}$. Recall that we order the sets ${\rm Ext}(S)_{\rm Sil}$ and  ${\rm Ext}(S)_{\rm M}$ with  the order relation of the associated quadratic forms. 
    
   \begin{theorem}\label{theorem:maximality er operator} If there exists a special $\Ee$-nest $\cN$ such that $D(S) \subseteq D(\Ee)_\cN$, then  $S^{\,\rm ref}$ is an extension of $S$. If, additionally, $D(S)\cap L^\infty(m)$ is an algebra that is $\E$-dense in $D(\E)$ and satisfies $S (D(S)\cap L^\infty(m))  \subseteq L^1(m)$, then $S^{\,\rm ref}$ is the maximal element of  ${\rm Ext}(S)_{\rm Sil}$.
   \end{theorem}
   \begin{proof}
     We first prove that $S^{\,\rm ref}$  is a restriction of $S^*$, which implies that it is a self-adjoint extension of $S$. Let $\cN$ be a special $\Ee$-nest such that $D(S) \subseteq D(\Ee)_\cN$.   For $f \in D(S^{\,\rm ref})$ and  $\psi \in D(S)$, Lemma~\ref{lemma:alternative formula for er} shows
    $$\as{S^{\,\rm ref} f,\psi} = \E^{\, \rm ref}(f,\psi) = \lim_{n \to \infty} \lim_{\varphi \in I_{\Ee}}  \Ee (\varphi f^{(n)},\psi), $$
    where $f^{(n)} = (f \wedge n) \vee (-n)$ and  $I_{\Ee} = \{\varphi \in D(\Ee) \mid 0 \leq \varphi \leq 1\}$ is ordered by the natural order relation on functions. Since for $\varphi \in I_{\Ee}$ we have  $\varphi f^{(n)} \in D(\Ee) \cap L^2(m) = D(\E)$, this simplifies to 
    $$\lim_{n \to \infty}\lim_{\varphi \in I_{\Ee}} \Ee (\varphi f^{(n)},\psi) = \lim_{n \to \infty}\lim_{\varphi \in I_{\Ee}}   \E (\varphi f^{(n)},\psi) = \lim_{n \to \infty}\lim_{\varphi \in I_{\Ee}}  \as{\varphi f^{(n)},S \psi}. $$
    The density of $D(S)$ in $L^2(m)$ implies $\lim_{\varphi \in I_{\Ee}} \varphi = 1$ in $L^0(m)$.  Therefore,  Lebesgue's dominated convergence theorem, Lemma~\ref{lemma:Lebesgue's theorem}, yields
    $$ \lim_{n \to \infty}\lim_{\varphi \in I_{\Ee}}  \as{\varphi f^{(n)},S \psi} = \as{f,S \psi}.$$

    These considerations show that for all $\psi \in D(S)$ the equality $\as{S^{\,\rm ref} f,\psi} =  \as{f,S \psi}$ holds.  Consequently, $S^{\, {\rm ref}}$ is a restriction of $S^*$.

   Now, assume that $D(S)\cap L^\infty(m)$ is an algebra that is $\E$-dense in $D(\E)$ and satisfies $S (D(S)\cap L^\infty(m))  \subseteq L^1(m)$. Let $\ow{S}$ be a self-adjoint extension of $S$ such that the associated form $\ow{\E}$ is a Silverstein extension of $\E$. According to Corollary~\ref{corollary:maximality er dirichlet form}, it suffices to prove that  $\ow{\E}^{(k)}$and $\E^{(k)}$ coincide on $D(\E)$. This is discussed next.
   
   For $f \in D(S) \cap L^\infty(m) \subseteq D(\Ee)_\cN$,   Lemma~\ref{lemma:alternative formula for er} and Remark~\ref{remark:alternative formula for er} yield
   $$\ow{\E}^{(k)}(f) = \lim_{\varphi \in I_{\ow{\E}_{{\rm e}}}} \ow{\Ee}(f^2,\varphi).$$
   Let $\varphi \in I_{\ow{\E}_{{\rm e}}}$. We choose a net $(\varphi_i)$ in $D(\ow{\E})$ that converges to $\varphi$ in the form topology of $\ow{\Ee}$. Without loss of generality  we can assume $\varphi_i \geq 0$. The net $\ow{\varphi}_i := \varphi_i \wedge \varphi$ is $\ow{\E}_{{\rm e}}$-bounded and converges to $\varphi$ in $L^0(m)$. It belongs to $D(\ow{\Ee}) \cap L^2(m) = D(\ow{\E})$, and by Lemma~\ref{lemma:existence of a weakly convergent subnet} it $\Ee$-weakly converges to $\varphi$. According to our assumptions, $\ow{S}$ is an extension of $S$ and $D(S) \cap L^\infty(m)$ is an algebra.  Combining these observations, we obtain  
   $$\ow{\Ee}(f^2,\varphi) = \lim_i \ow{\Ee}(f^2,\ow{\varphi}_i) = \lim_i \ow{\E}(f^2,\ow{\varphi}_i)  = \lim_i \as{S(f^2),\ow{\varphi}_i} =   \int_X S(f^2) \varphi \D m.$$
   For the last equality, we used Lebesgue's theorem, Lemma~\ref{lemma:Lebesgue's theorem}. Altogether, this amounts to
   $$\ow{\E}^{(k)}(f) =  \lim_{\varphi \in I_{\ow{\E}_{{\rm e}}}} \ow{\Ee}(f^2,\varphi) = \lim_{\varphi \in I_{\ow{\E}_{{\rm e}}}} \int_X S(f^2) \varphi \D m = \int_X S(f^2) \D m, $$
   where we used $\lim_{\varphi \in I_{\ow{\E}_{{\rm e}}}} \varphi = 1$, $S(f^2) \in L^1(m)$ and Lebesgue's theorem. Since these considerations also apply to  $\E$, for  $f \in D(S) \cap L^\infty(m)$ we conclude
   $$\ow{\E}^{(k)}(f) =  \int_X S(f^2)\D m =   \E^{(k)}(f).$$
   The $\E$-density of $D(S) \cap L^\infty(m)$  in $D(\E)$ yields that this equality extends to $D(\E)$ and we obtain the desired statement. 
   \end{proof}
 
  As an immediate corollary, we obtain the following result for hypoelliptic operators on Riemannian manifolds. 
 
   \begin{corollary}\label{corollary:maximal markovian extension}
    Let $(M,g)$ be a Riemannian manifold and let $P:C_c^\infty(M) \to C_c^\infty(M)$ be a continuous linear operator that is symmetric on $L^2({\rm vol}_g)$ and whose associated form $\E_P$ is Markovian.  Then $P^{\, {\rm ref}}$ is the maximal element in ${\rm Ext}(S)_{\rm Sil}.$ If, additionally, $P + 1$ is hypoelliptic, then  $P^{\, {\rm ref}}$ is the maximal element in ${\rm Ext}(S)_{\rm M}.$ 
   \end{corollary}
   \begin{proof}
    The form $\bar{\E}_{P,{\rm e}}$ is regular (it is the $L^0$-closure of a regular Dirichlet form) and so the collection of all relatively compact open subsets of $M$ is a special $\bar{\E}_{P,{\rm e}}$-nest  with the required properties. Therefore, the previous theorem shows the maximality of $P^{\,\rm ref}$ in  ${\rm Ext}(S)_{\rm Sil}.$ The statement on the maximality of $P^{\, {\rm ref}}$  in ${\rm Ext}(S)_{\rm M}$ under the condition that $P+1$ is hypoelliptic follows from Corollary~\ref{corollary:hypoelliptic operators}.
   \end{proof}
   
   \begin{remark}
   \begin{itemize}
    \item The condition that $D(S)  \subseteq D(\Ee)_\cN$ for some special $\Ee$-nest $\cN$ means that functions in $D(S)$ vanish at 'infinity' and the assumption that $D(S) \cap L^\infty(m)$ is an algebra is more or less natural when dealing with Silverstein extensions.  In contrast, the assumption $S (D(S)\cap L^\infty(m))  \subseteq L^1(m)$ is somewhat technical. 
    
     \item If $X$ is a locally compact separable metric space and $m$ is a Radon measure on $X$, the previous theorem   applies to operators that map a dense subalgebra of $C_c(X)$ to $C_c(X)$. One class of examples for this situation is given in the previous corollary, another class of examples are graph Laplacians on locally finite graphs, for which the previous corollary is more or less contained in \cite[Theorem~2.10]{HKLW}.
    
    \item  As already mentioned in Remark~\ref{remark:maximality er}, the  theorem and its corollary can be seen as a generalization of \cite[Theorem~15.2]{Sil}. 
    
    \item The main insight for the proof of the theorem is that the killing part of a Dirichlet form that is associated with an extension of $S$ is determined by $S$.

   \end{itemize}
   \end{remark}
  \end{example}

  \chapter{Weak solutions to the Laplace equation} \label{chapter:weak solutions}

For some  Dirichlet forms it is known that global properties, such as Silverstein uniqueness, recurrence and stochastic completeness, can be characterized in terms of weak (super-) solutions to the Laplace equation. In these cases, the notion of weak solution is defined using particular features of the considered model. For Dirichlet forms on manifolds one uses the  theory of distributions, see e.g. \cite{Gri}, for Dirichlet forms on graphs one uses discrete Laplacians, see e.g. \cite{Schmi}, and for regular Dirichlet forms one uses a notion of weak solutions in terms of the Beurling-Deny formula, see e.g. \cite{Stu1,FLW}. These  definitions of weak Laplacians or weak Laplace equations have in common that they are not intrinsic; they either use extrinsic information on the model or representation theory.  This makes them somewhat unsatisfactory. The first two approaches are certainly powerful but limited to concrete examples, while the third one is technically involved and  requires the form to be quasi-regular. 

In this chapter we propose an intrinsic approach towards weak solutions of the Laplace equation that only uses algebraic properties of the given form. This allows us to consider weak Laplace equations for {\bf all} energy forms.  It is our goal is to convince the reader that our definition is the 'correct' one. To this end, we demonstrate that it is flexible enough to prove  known characterizations of global properties in the greatest possible generality.  Another application that we have in mind are Dirichlet forms and energy forms on noncommutative spaces. We believe that the concepts and theorems of this chapter can be extended to the noncommutative world. Unfortunately, this is beyond the scope of this thesis; it will be discussed elsewhere.

In the previous chapter we constructed the reflected energy form. For this purpose we extended the given form on its diagonal.  In the first section of this chapter, we recycle some of the auxiliary forms that were introduced in the previous chapter to define the weak form extension, an off-diagonal extension of the given form, see Definition~\ref{definition:weak form}. This weak form is then employed to introduce weak solutions to the Laplace equation, see Definition~\ref{definition:weak solution}. Furthermore, we prove that in concrete examples our notion of the weak form extension coincides with the 'classical' one. The second section is devoted to characterizing excessive functions as weakly superharmonic functions, see Theorem~\ref{theorem:characterization excessive functions as superharmonic functions}. It is followed by a section where the existence of minimal solutions to the weak Laplace equation is investigated and characterized in terms of the extended potential operator, see Theorem~\ref{theorem:maximal principle potential operator}. We conclude this chapter with several applications. In Subsection~\ref{subsection:uniqueness of bounded weak solutions} we study the existence and uniqueness of bounded weak solutions to the Laplace equation. More specifically, uniqueness of such solutions is characterized by a conservation property for the potential operator, see Theorem~\ref{theorem:uniqueness of bounded weak solutions}. In Subsection~\ref{section:lp resolvents} we apply the developed theory to $L^p$-resolvents of Dirichlet forms. We show that they provide weak solutions to an eigenvalue equation, see Theorem~\ref{theorem:lp resolvents}, and characterize a generalization of stochastic completeness in terms of uniqueness of solutions to this eigenvalue equation for bounded functions, see Theorem~\ref{theorem:characterization stochastic completeness}. Subsection~\ref{subsection:recurrence} concludes this chapter with a characterization of recurrence in terms of weakly superharmonic functions.

\section{Off-diagonal extension and the weak form domain} \label{section:off-diagonal extension and the weak form domain}

In the previous chapter we  constructed the maximal Silverstein extension of a given energy form by extending $E$ on its diagonal. In this section  we we do off-diagonal extension. The goal is to define $E(f,\psi)$ for pairs $(f,\psi)$, where $f$ should come from a large class of functions and $\psi$ plays the role of a test function.

Before discussing the general case, we come back to Example~\ref{example:maximal Silverstein extension} of Section~\ref{subsection:maximal Silverstein extension}. When $E = E^0_{J,V}$ (or $E= E_{J,V})$ it is intrinsically clear how to define $E(f,\psi)$. One would just set
\begin{align} E(f,\psi) := \int_{X \times X} (f(x) - f(y))(\psi(x) - \psi(y)) {\rm d} J(x,y) + \int_X f(x)\psi(x) V(x){\rm d}m(x)\label{equation:weak form}\tag{$\blacklozenge$}\end{align}
whenever the integrals make sense, i.e., when $(f \otimes 1 - 1 \otimes f)(\psi \otimes 1 - 1 \otimes \psi) \in L^1(J)$ and $f\psi \in L^1(V\cdot m)$.  Similar as in Example~\ref{example:maximal Silverstein extension}, this raises the question of how to compute \eqref{equation:weak form} from the knowledge of $E_{J,V}^0$ on its (rather small) domain.  To answer it, we employ the auxiliary  forms $\Ep$ and $\Ekm$, which we  introduced in Subsection~\ref{subsection:construction of er}. The value of
$$\Ep(f,\psi) = \int_{X \times X} \varphi(x)\varphi(y) (f(x) - f(y))(\psi(x) - \psi(y)) {\rm d} J(x,y) $$
only depends on the restriction of $f$ to the set  $\{|\varphi|>0\}$. Therefore, $\Ep(f,\psi)$ is well-defined if there exists a function $f_\varphi \in D(E^0_{J,V})$ that equals $f$ on $\{|\varphi|>0\}$. Likewise, if $f_\psi \in D(E)$ is a function that agrees with $f$ on $\{|\psi|>0\}$, the properties of $\Ekm$ yield
$$\Ekm(f_\psi,\psi) = \int_X f(x)\psi(x) V(x){\rm d}m(x).$$
 As a consequence, we obtain 
$$E(f,\psi) = \lim_{\varphi \nearrow 1} \Ep(f,\psi) + \Ekm(f_\psi,\psi),$$
whenever $f$ locally belongs to the form domain and the integrals in Formula~\eqref{equation:weak form} exist. Unfortunately, the latter condition cannot be phrased in terms of the form $E$. As a substitute we will require the existence of the limit $\lim_{\varphi \nearrow 1} \Ep(f,\psi)$. In the example, this corresponds to some form of improper integration.

  The following two lemmas show that a construction as suggested by the previous discussion can be performed for arbitrary energy forms. Namely, they show that both $\Ep$ and $\Ekm$ can be extended to the local form domain of $E$. 

\begin{lemma} \label{lemma:local space is contained in dep}
 Let $E$ be an energy form. Let $\cN$ be an $E$-nest and let $\varphi \in D(E)_{\cN}$ with $0 \leq \varphi \leq 1$. Then $\DEln \subseteq D(\Ep)$.
 
\end{lemma}
\begin{proof}
 Let $f \in \DEln$ and let $N \in \cN$ such that $\varphi 1_{X\setminus N} = 0$. By definition of the local space, there exists a function $f_N \in D(E)$ such that $f = f_N$ on $N$. In particular, $\varphi f_N = \varphi f$.  Lemma~\ref{lemma:properties of ep} implies
 $$\Ep(f) = \Ep(f_N) \leq E(f_N) < \infty.$$
 This finishes the proof. 
\end{proof}

\begin{lemma}\label{lemma:extension of k}
 Let $E$ be an energy form and let $\cN$ be an $E$-nest. The bilinear form 
 $$\Ekm:D(E) \times D(E)_\cN \to \IR,\, (f,\psi) \mapsto \Ekm(f,\psi)$$
 can be uniquely extended to a bilinear form 
 $$\Ekm:D(E)_{{\rm loc},\,\cN} \times D(E)_\cN \to \IR,$$
 such that for $f \in D(E)_{{\rm loc},\,\cN}$ and $\psi \in  D(E)_\cN$ the inequality $f\psi \geq 0$ implies 
  $$\Ekm(f,\psi) \geq 0.$$
 In particular, if $f \in D(E)_{{\rm loc},\,\cN}$, $\psi \in D(E)_\cN$ and $f_\psi \in D(E)$ with $f_\psi \psi = f \psi$, then
  $$\Ekm(f,\psi) = \Ekm(f_\psi,\psi).$$
\end{lemma}
\begin{proof}
 {\em Existence:} For $f \in D(E)_{{\rm loc},\,\cN}$ and $\psi \in D(E)_\cN$, we choose a function $f_\psi \in D(E)$ with $f_\psi \psi = f \psi$ and define
 $$\Ekm(f,\psi) := \Ekm(f_\psi ,\psi).$$
 Lemma~\ref{lemma:characterization of k}~(b) shows that this definition is independent of the particular choice of $f_\psi$ and that it has the desired properties. The 'in particular' statement is an immediate consequence of the definition. 
 
 {\em Uniqueness:} Let $k:D(E)_{{\rm loc},\,\cN} \times D(E)_\cN \to \IR$ be another bilinear form with the stated properties. Furthermore, let $f \in D(E)_{{\rm loc},\, \cN}$, $\psi \in D(E)_\cN$ and choose a function $f_\psi \in D(E)$ with $f\psi = f_\psi \psi$. The identity $(f-f_\psi) \psi = 0$ and the properties of $k$ imply $k(f-f_\psi,\psi) = 0$. Since $k$ is bilinear and coincides with $\Ekm$ on $D(E) \times D(E)_\cN$, we obtain
 $$k(f,\psi) = \Ekm(f_\psi,\psi).$$
 This shows uniqueness. 
\end{proof}
With these preparations we introduce the weak form domain.  For a special $E$-nest $\cN$ we let  
$$I_\cN := \{\varphi \in D(E)_\cN \mid 0 \leq \varphi \leq 1\}$$
and order it by the natural order relations on functions. Since $D(E)_\cN$ is a lattice, see Proposition~\ref{proposition:properties of den}, the set $I_\cN$ is upwards directed with respect to this order. 
\begin{definition}[Weak form domain and weak form extension] \label{definition:weak form}
 Let $\cN$ be a special $E$-nest. The {\em weak form domain of $E$ with respect to $\cN$} is given by
$$D(E)_{w,\, \cN} := \{f \in \DEln \mid \lim_{\varphi \in I_\cN} \Ep(f,\psi) \text{ exists for all } \psi \in D(E)_\cN\}.$$
For $f \in D(E)_{w,\, \cN}$ and $\psi \in D(E)_\cN$ we set
$$\Ew(f,\psi)  := \Ew_{\cN}(f,\psi) :=  \lim_{\varphi \in I_\cN} E_\varphi(f,\psi) +  \Ekm(f, \psi). $$
%
\end{definition}

\begin{remark}
\begin{itemize}
\item The definition of $\Ew$ is along the lines of the discussion  at the beginning of this section, where it corresponded to some form of improper integration. Since improper integration is rather sensitive to the way limits are taken, the value of $\Ew(f,\psi)$ can depend on the chosen nest. More precisely, this means that for $f \in D(E)_{w,\, \cN_1} \cap D(E)_{w,\, \cN_2} $ and $\psi \in D(E)_{\cN_1} \cap D(E)_{\cN_2}$ it can happen that $\Ew_{\cN_1}(f,\psi) \neq \Ew_{\cN_2}(f,\psi)$. Nevertheless, we  omit the subscripts  in the notation when it is clear from the context with respect to which nest the limit is taken. Below, we shall see that for several classes of functions the definition of $\Ew$ is indeed independent of the underlying nest. 

 \item In principle we could have introduced the weak form domain with respect to any nest  without assuming it to be special. In this case, some of the theorems that we aim at are not correct, while for others the proofs become substantially more difficult. This is why we restrict the focus to the weak form domain with respect to special nests. 
 
 \item Let $\E$ be a regular Dirichlet form. In this case $\mathcal{K}$, the collection of all relatively compact open subsets of the underlying space, is a special $\Ee$-nest. Classically, $\E$ (or $\Ee$) is extended to $D(\E)^*_{{\rm loc},\, {\mathcal{K}}} \times D(\E)_\mathcal{K}$ with the help of the Beurling-Deny-formula, see e.g. \cite{Stu1} for local Dirichlet forms and \cite{FLW} for non-local Dirichlet forms. Here, $D(\E)^*_{{\rm loc},\, {\mathcal{K}}} = D(\E) _{{\rm loc},\, {\mathcal{K}}}$ in the local situation, while $D(\E)^*_{{\rm loc},\, {\mathcal{K}}}$ is a certain subspace of $D(\E)_{{\rm loc},\, {\mathcal{K}}}$ in the non-local situation. It can be shown that our weak form is an extension of the classical extension. We do not give all the details but discuss the situation for jump-type forms, see Example~\ref{example:jump-type form weak extension}, and manifolds, see Example~\ref{example:manifolds}. 
\end{itemize}
\end{remark}

The following theorem is the main observation of this section. It shows that $\Ew$ does indeed provide an off-diagonal extension of the form $E$.
\begin{theorem}\label{theorem:ew extends e} Let $E$ be an energy form and let $\cN$ be a special $E$-nest. The functional 
 %
 $$\Ew_\cN:D(E)_{w,\, \cN} \times D(E)_\cN \to \IR, \, (f,\psi) \mapsto \Ew_\cN(f,\psi)  $$
 is a bilinear form that extends the bilinear form
 $$E:D(E) \times D(E)_\cN \to \IR, \, (f,g) \mapsto E(f,g).$$
 %
\end{theorem}
The theorem is an immediate consequence of the following lemma and the fact that $\Er$ extends $E$. 
\begin{lemma} \label{lemma: weak form as extension}
Let $E$ be an energy form and let $\cN$ be a special $E$-nest. For $f,g \in D(\Egm)$ the identity 
$$\Egm(f,g) = \lim_{\varphi \in I_\cN} E_\varphi(f,g)$$
holds. In particular,  $f \in D(E)_{{\rm loc},\, \cN} \cap D(\Egm)$ implies $f \in D(E)_{w,\, \cN}$ and
$$\Ew_\cN(f,\psi) = \Egm(f,\psi) + \Ekm(f,\psi) \text{ for all }\psi \in D(E)_{\cN}.$$

\end{lemma}

\begin{proof}
The statement on $\Eg$ follows from Corollary~\ref{corollary:alternative formula for eg}. The 'in particular'-part is an immediate consequence of this observation and the definition of $\Ew_\cN$. 
\end{proof}

 \begin{remark} \label{remark:weak extension of er} The previous lemma is somewhat technical in its statement and falls just a bit short of saying that $\Ew_\cN$ extends $\Er$. This is due to the fact that in general there is no special $E$-nest $\cN$ such that $D(\Er) \subseteq \DEln$. However, since $\Er$ is a Silverstein extension of $E$, Lemma~\ref{lemma:silverstein extension local space} yields that for each individual $f \in D(\Er)$ there exists a special $E$-nest $\cN$ such that $f \in D(E)_{{\rm loc},\, \cN}$. In this sense, the weak form extends $\Er$. 
 
 If one only considers bounded functions one avoids these difficulties. According to Theorem~\ref{theorem:Ideal properties of energy forms}, for any special $E$-nest $\cN$ the inclusion $D(\Er)\cap L^\infty(m) \subseteq \DEln$ holds. Therefore, the previous lemma implies that for any  special $E$-nest $\cN$ the bilinear form 
  $$\Ew_\cN:D(E)_{w,\, \cN} \times D(E)_\cN  \to \IR, \, (f,\psi) \mapsto \Ew_cN(f,\psi)  $$
  is an extension of the bilinear form 
  $$
  \Er :D(\Er) \cap L^\infty(m) \times D(E)_\cN \to \IR,\, (f,g) \mapsto \Er(f,g).$$
 \end{remark}

As mentioned above, the weak form domain and the weak form can depend on the chosen nest. The following lemma shows that this is not the case for equivalent nests, cf. Definition~\ref{definition:equivalent nests}.

\begin{lemma}\label{lemma:independence of nests}
 Let $E$ be an energy form and let $\cN,\cN'$ be two equivalent special $E$-nests. Then $D(E)_\cN = D(E)_{\cN'}$ and  $D(E)_{w,\, \cN} = D(E)_{w,\, \cN'}$. Moreover, for all $f \in  D(E)_{w,\, \cN}$ and all $\psi \in D(E)_\cN$ we have
 $$\Ew_\cN (f,\psi) = \Ew_{\cN'}(f,\psi).$$
\end{lemma}
\begin{proof}
 This follows from the definition of $\Ew$ and Lemma~\ref{lemma:local domain equivalent nests}.
\end{proof}

The concept of the weak form extension allows us to introduce weak solutions to the Laplace equation with respect to a given energy form. The main aim of this chapter is to prove that existence and uniqueness of certain solutions to this equation determines, and is determined, by global properties of the given form. We first introduce weak solutions to the Laplace equation and then  discuss where this equation appears in concrete examples. Existence and uniqueness of solutions is then studied in later sections.

\begin{definition}[Weak solutions to the Laplace equation]\label{definition:weak solution}
 Let $E$ be an energy form. Let $\ell:D(\ell) \to \IR$ be a linear functional on $L^0(m)$. We say that a function $f \in L^0(m)$ is a {\em weak supersolution} to the {\em  Laplace equation}
 $$ 
 \Ew(g,\cdot) = \ell 
$$
 if there exists a special $E$-nest $\cN$ such that $D(E)_\cN \subseteq D(\ell)$, $f \in D(E)_{w,\, \cN}$ and 
 $$\Ew_\cN(f,\psi) \geq \ell(\psi) \text{ for all nonnegative }  \psi \in D(E)_\cN.$$
 If, furthermore,
 $$\Ew_\cN(f,\psi) = \ell(\psi) \text{ for all } \psi \in D(E)_\cN,$$
 then $f$ is a {\em weak solution} to the Laplace equation. A function $f \in L^0(m)$ is a  {\em weak subsolution} to the equation $\Ew(g,\cdot) = \ell$ if $-f$ is a weak supersolution. A weak solution to the equation  $\Ew(g,\cdot) = 0$ is called {\em weakly $E$-harmonic}. Likewise,  a weak supersolution to the equation $\Ew(g,\cdot) = 0$ is called {\em weakly $E$-superharmonic} and a weak  subsolution to this equation is called {\em weakly $E$-subharmonic}.
\end{definition}

\begin{remark}
 We warn the reader that a function $f$ that is a weak supersolution and a weak subsolution to the equation $\Ew(g,\cdot) = \ell$ need not be a weak solution. This is due to the fact that the value of $\Ew$ can depend on the choose nest. It will follow from the discussion below that this pathological behavior does not appear for regular functionals.
\end{remark}

We finish this section by providing several examples for the weak form domain and the weak form.

 \begin{example}[Jump-type form on a topological space] \label{example:jump-type form weak extension} We start with the weak form domain of the jump form $E_{J,V}^0$ of Example~\ref{example:maximal Silverstein extension}, which we also discussed at the beginning of this section. The form $E_{J,V}^0$ is regular and, therefore, the collection $\mathcal{K}$, of all open relatively compact subsets of $X$, is a special $E_{J,V}^0$-nest. Recall that in this case 
$$E^0_{J,V,\varphi}(f) = \int_{X \times X}\varphi (x) \varphi(y) (f(x) - f(y))^2 \D J(x,y)$$
whenever $f \in D(E^0_{J,V,\varphi})$. Hence, for $f \in D(E)_{{\rm loc},\, \mathcal{K}}$,  the integrability condition 
$$\int_{X \times X} |f(x) - f(y)||\psi(x) - \psi(y)| {\rm d} J(x,y) < \infty \text{ for all }\psi \in D(E)_{\mathcal{K}}$$
and Lebesgue's theorem imply $f \in D(E_{J,V}^0)_{w,\, \cN}$ and 
$$ E^{0,(w)}_{J,V} (f, \psi) = \int_{X \times X} (f(x) - f(y))(\psi(x) - \psi(y)) \D J(x,y) + \int_X f(x)\psi(x) V(x) \D m(x).$$
For an $f \in D(E)_{{\rm loc},\, \mathcal{K}}$, one way to guarantee the integrability condition is to require 
\begin{align*}
\int_{G\times X} (f(x) - f(y))^2 \D J(x,y) < \infty \text{ for all } G \in \mathcal{K}. 
\end{align*}
 Weak extensions of regular jump-type Dirichlet forms to functions satisfying this condition paired with functions in $D(E)_\mathcal{K}$ (i.e. functions in $D(E)$ with compact support) have been studied in  \cite{FLW}. In this sense, our notion of weak extension is a generalization of the one given there. 
\end{example}

  \begin{example}[Jump-type forms on discrete spaces] \label{example:graphs} A more concrete class of examples of jump-type forms, which have been studied quite extensively, are forms induced by graphs, see Subsection~\ref{subsection:jump forms}. We assume the setting of  Example~\ref{example:maximal Silverstein extension} and additionally require that $X$ is countable and equipped with the discrete topology. As discussed in Subsection~\ref{subsection:jump forms}, the measure $J$ is induced by a symmetric function $b: X \times X \to [0,\infty]$ via the identity
 $$J(A\times B) = \sum_{(x,y) \in A\times B} b(x,y)\text{ for all } A,B \subseteq X.$$
 It can be interpreted as the edge weight of a graph over $X$. The condition that $J$ is a Radon measure yields $b(x,y) < \infty$ for all $x,y \in X$. Combined with the assumption (J2) we obtain 
 $$\sum_{y \in X} b(x,y) < \infty \text{ for all } x\in X.$$
 Similarly, the Radon measure of full support $m$ is induced by a function $m:X \to (0,\infty)$ via the identity
 $$m(A) = \sum_{x \in A} m(x) \text{ for all } A \subseteq X.$$
 As already seen in Subsection~\ref{subsection:jump forms}, for $f \in D(E_{J,V})$ we have 
 $$E_{J,V}(f) = \sum_{x,y \in X}b(x,y) (f(x) -f(y))^2 + \sum_{x \in X} f(x)^2 V(x) m(x).$$
 The space $X$ carries the discrete topology and $m$ is equivalent to the counting measure. Therefore, $L^0(m) = C(X)$ and the topology $\tau(m)$ coincides with the topology  of pointwise convergence. The collection of all open relatively compact subsets of $X$ is given by 
 $$\mathcal{K} = \{F \subseteq X \mid F \text{ finite}\},$$
 and $C_c(X)$ is the space of finitely supported functions.  In order to simplify notation, we let $q := E^0_{J,V}$, the closure of the restriction of $E_{J,V}$ to $C_c(X)$. The discreteness of the space and the definition of $q$ yield $D(q)_{\mathcal{K}} = C_c(X)$. On the domain 
 $$D(\ow{L}):= \left\{f \in C(X)\, \middle|\, \sum_{y \in X} b(x,y)|f(y)| < \infty \text{ for all } x\in X \right\}$$
 we define the linear operator $\ow{L}:D(\ow{L}) \to C(X)$ via
 $$\ow{L}f(x) :=  2 \sum_{y\in X} b(x,y) (f(x) - f(y)) + V(x) f(x).$$
 We can explicitly compute the weak form $q^{(w)}_\mathcal{K}$ in terms of the operator $\ow{L}$.
 
 \begin{proposition}\label{proposition:example graphs}
 The weak form domain of $q$ with respect to the special nest $\mathcal{K}$ is given by $D(q)_{w,\, \mathcal{K}} = D(\ow{L})$. Moreover, for all $f \in D(\ow{L})$ and $\psi \in C_c(X)$ we have
 $$q_\mathcal{K}^{(w)}(f,\psi) = \sum_{x \in X} \ow{L}f(x) \psi(x).$$
 \end{proposition}
 \begin{proof}
  It is shown in \cite{HK} that   $f \in D(\ow{L})$   and $\psi \in C_c(X)$ implies
 $$\sum_{x,y \in X}b(x,y) |f(x)-f(y)||\psi(x) - \psi(y)| < \infty $$
 and 
 $$ \sum_{x \in X} \ow{L}f(x) \psi(x) =  \sum_{x,y \in X} b(x,y) (f(x)-f(y))(\psi(x) - \psi(y)) + \sum_{x \in X} f(x)\psi(x).$$
 Therefore, the discussion in Example~\ref{example:jump-type form weak extension} yields $D(\ow{L}) \subseteq D(q)_{w,\, \mathcal{K}}$ and the identity 
 $$q_\mathcal{K}^{(w)}(f,\psi) = \sum_{x \in X} \ow{L}f(x) \psi(x),$$
 whenever  $f \in D(\ow{L})$ and $\psi \in C_c(X)$. 
 
 It remains to prove the inclusion  $D(q)_{w,\, \mathcal{K}} \subseteq D(\ow{L})$.  We order the collection $\mathcal{K}$ by set inclusion and the collection $J:=\{1_K \mid  K \in \mathcal{K}\}$ by the usual order relation on functions.  Clearly, $J$ is cofinal in $I_\mathcal{K} = \{\varphi \in D(q)_\mathcal{K}\mid  0 \leq \varphi \leq 1\} = \{\varphi \in C_c(X) \mid 0 \leq \varphi \leq 1\}.$ Therefore, $f \in D(q)_{w,\, \mathcal{K}}$ implies that for all $\psi \in C_c(X)$ we have
 $$\lim_{\varphi \in I_\mathcal{K}} q_\varphi(f,\psi)  = \lim_{\varphi \in J} q_\varphi(f,\psi) = \lim_{K\in \mathcal{K}} q_{1_K}(f,\psi).$$
 For $\psi = 1_{\{z\}}$  with $z \in K$ we obtain
 \begin{align*}
  q_{1_K}(f,1_{\{z\}}) &= \sum_{x,y \in K} b(x,y) (f(x)-f(y)) (1_{\{z\}}(x) - 1_{\{z\}}(y))\\
  &= 2 \sum_{y \in K} b(z,x)(f(z) - f(x))\\
  &= 2 f(z) \sum_{y \in K} b(z,y)  -  \sum_{y \in K} b(z,y)f(y).
 \end{align*}
 As discussed above, the function $b$ satisfies $\sum_{y \in X}b(z,y) < \infty$ for all $z \in X$ and, therefore,
 $$\lim_{K \in \mathcal{K}}  \sum_{y \in K} b(z,y) = \sum_{y \in X}b(z,y).$$
 Together with the previous discussion, this implies that for all $z \in X$ the limit
 $$\lim_{K \in \mathcal{K}} \sum_{y \in K} b(z,y)f(y) $$
 exists, i.e., the function $X \to \IR, y \mapsto b(z,y)f(y)$ is unconditionally summable. The classical Riemann series theorem  yields  
 $$\sum_{y \in X} b(z,y)|f(y)|< \infty.$$
 We obtain $f \in D(\ow{L})$ and the proposition is proven.
 \end{proof} 
 
 \begin{remark}
  As remarked already in Subsection~\ref{subsection:jump forms}, jump type forms on discrete spaces have been studied quite intensively and it is known that the operator $\ow{L}$  is  important in this context. Starting with \cite{KL}, in recent years  a particular focus lay on graphs that are possibly locally infinite. In this case, $D(\ow{L})$ is a proper subset of $C(X)$ and it was somewhat unclear how $\ow{L}$ and its domain  $D(\ow{L})$ are intrinsically related with the form $q$. The previous proposition answers this question. 
 \end{remark}

 \end{example}
 
 \begin{example}[Riemannian manifolds]\label{example:manifolds}  Let $(M,g)$ be a Riemannian manifold. Recall the regular energy form $E_{(M,g)}^0$, which was introduced in Subsection~\ref{subsection:manifolds}, and the notation of Appendix~\ref{appendix:manifolds}. The regularity of $E_{(M,g)}^0$  implies that $\mathcal{K}$, the collection of all open relatively compact subsets of $M$, is a special  $E_{(M,g)}^0$-nest. As before, we set $Q:= E_{(M,g)}^0$ to shorten notation.
 
 \begin{proposition} \label{proposition:weak form manifold}
  The weak form domain of $Q$ with respect to $\mathcal{K}$ satisfies
  $$D(Q)_{w,\, \mathcal{K}} = D(Q)_{{\rm loc},\, \mathcal{K}} = W^1_{\rm loc}(M),$$
  and for $f \in W^1_{\rm loc}(M)$ and $\psi \in D(Q)_\mathcal{K}$ we have 
  $$Q^{(w)}_\mathcal{K}(f,\psi) = \int_M g(\nabla f, \nabla \psi) \D {\rm vol}_g.$$
  Moreover, if $f \in \Ltlm$ with $\Delta f \in \Ltlm$, then $f \in D(Q)_{w,\, \mathcal{K}}$ and 
  $$Q^{(w)}_\mathcal{K}(f,\psi) = -\int_M \Delta f \psi \D {\rm vol}_g.$$
 \end{proposition}
 \begin{proof}
  We start by proving $D(Q)_{{\rm loc},\, \mathcal{K}} = W^1_{\rm loc}(M)$. It follows from the definition of $Q$ that $f \in D(Q)_{{\rm loc},\, \mathcal{K}}$ implies $f \in \Ltlm$ and $|\nabla f| \in \Ltlm$, i.e., $f \in W^1_{\rm loc}(M)$. For the converse inclusion, let $f \in \Ltlm$ with  $|\nabla f| \in \Ltlm$ be given. For $G \in \mathcal{K}$, we choose a function $\varphi \in C_c^\infty(M)$ that satisfies $1_G \leq \varphi \leq 1$. From Lemma~\ref{lemma:product in w0}, we infer $\varphi f \in D(Q)$. Since $\varphi f = f$ on $G$,  we obtain $f \in D(Q)_{{\rm loc},\, \mathcal{K}}$. 
   
  Next, we show the identity for $Q^{(w)}$. Let $f \in W^1_{\rm loc}(M)$ and let $\varphi \in I_\mathcal{K}$. According to Lemma~\ref{lemma:ep for manifolds}, we have 
  $$Q_\varphi(f) = \int_M \varphi^2 |\nabla f|^2\D {\rm vol}_g.$$
  Since $g(\nabla f, \nabla \psi)$ vanishes outside a compact neighborhood of the support of $\psi$, this implies
  $$\lim_{\varphi \in I_\mathcal{K}} Q_\varphi (f,\psi) = \lim_{\varphi \in I_\mathcal{K}} \int_M \varphi^2 g(\nabla f,\nabla\psi) \D {\rm vol}_g = \int_M g(\nabla f,\nabla \psi) \D {\rm vol_g}. $$
  The 'moreover' statement follows form Green's formula for compactly supported $\psi$ and elliptic regularity theory, which shows that $f \in \Ltlm$ and $\Delta f \in \Ltlm$ implies $f \in W^1_{\rm loc}(M)$. This finishes the proof. 
 \end{proof}
 \begin{remark} \label{remark:comparing notions}
 \begin{itemize} 
 \item Similar computations as in the above proof yield that our notion of the weak form extension extends the classical one for strongly local regular Dirichlet forms. We refrain from giving details.
 
  \item  The examples on graphs and manifolds show that weak solutions to the Laplace equation as introduced in Definition~\ref{definition:weak solution}  can be interpreted as weak solutions in the classical sense, i.e., in the sense of distributions on manifolds and in the sense of pointwise solutions on graphs. 
 \end{itemize}

  \end{remark}
 \end{example}

\section{Weakly superharmonic and excessive functions} \label{section:Weakly superharmonic and excessive functions}

In this section  we provide first evidence for the utility of our concept of the weak form extension. Namely, we prove that the nonnegative excessive functions are exactly the nonnegative weakly superharmonic functions. We recall the following notation of Section~\ref{section:ideals}. For a nest $\cN$ and a function $f \in L^0(m)$ we let 
$$\cN_f:= \cN \wedge \{\{f \leq n\} \mid n \in \IN\} = \{N \cap \{f\leq n\} \mid N \in \cN \text{ and } n\in \IN\}. $$
The following theorem is the main observation in this section.
\begin{theorem} \label{theorem:characterization excessive functions as superharmonic functions}
Let $E$ be an energy form of full support and let $h \in L^0(m)$ nonnegative. The following assertions are equivalent.
\begin{itemize}
 \item[(i)] $h$ is $E$-excessive.
 \item[(ii)] $h$ is weakly $E$-superharmonic.
  \item[(iii)]  For all special $E$-nests $\cN$  the collection $\cN_h$ is a special $E$-nest with $h \in D(E)_{w,\, \cN_h}$, and for all nonnegative  $\psi \in D(E)_{\cN_h}$ it satisfies
 $$E^{(w)}_{\cN_h}(h,\psi) \geq 0.$$

\end{itemize}
\end{theorem}
\begin{remark}
 \begin{itemize}
  \item In view of our discussion in Remark~\ref{remark:comparing notions},   this theorem can be seen as an extension of \cite[Lemma~3]{Stu1}, which provides a characterization of nonnegative bounded excessive functions. The reason why \cite{Stu1} only treats bounded excessive functions is that it deals with regular Dirichlet forms and one fixed special nest, namely the collection of all relatively compact open subsets of the underlying space; however, not all excessive functions need to belong to the local space with respect to this nest. 
  
  \item The theorem fails when $E$ does not have full support. In this case, the constant function $1$ is weakly superharmonic but not excessive. Nevertheless, the implication (i) $\Rightarrow$ (iii) is true without the assumption that $E$ has full support. 
 \end{itemize}
\end{remark}

Before proving the theorem we discuss a couple of lemmas, which are also important for later purposes. The first two show that every excessive function belongs to the local form domain.

\begin{lemma}\label{lemma:excessive functions provide nests}
 Let $E$ be an energy form and let $h$ be $E$-excessive. Then 
 $$\{\{h \leq n\} \mid n \in \IN\}$$
 is an $E$-nest. 
\end{lemma}
\begin{proof}
  Let $ f \in D(E)$ be nonnegative and let $n \in \IN$. We set
 $$f_n := (f \wedge n - (h/n) \wedge n)_+. $$
 These functions satisfy $f_n = 0$ on $\{h > n^2\}$. It remains to prove $f_n \in D(E)$ and $f_n \to f$, as $n \to \infty$,  in the form topology.  The convergence $f_n \overset{m}{\to} f$ is obvious. By Corollary~\ref{corollary:concatenation superharmonic functions} the function $(h/n) \wedge n$ is excessive. Hence, Lemma~\ref{lemma:superharmonic functions as cutoff} and the contraction properties of $E$ yield
 $$  E(f_n) \leq E(f\wedge n) \leq E(f).$$
 This shows $f_n \in D(E)$ for all $n \in \IN$. Furthermore,  the convergence $f_n \to f$ in the form topology of $E$ follows from this inequality and Lemma~\ref{lemma:characterization convergence in form topology lsc forms}. This finishes the proof. 
\end{proof}
\begin{lemma}\label{lemma:excessive functions belong to the local form domain}
 Let $E$ be an energy form and let $\cN$ be a special $E$-nest. If $h$ is a nonnegative $E$-excessive function, then $\cN_h$ is a special $E$-nest and $h\in D(E)_{{\rm loc},\, \cN_h}$. In particular, if $h$ is bounded, then $h \in D(E)_{{\rm loc},\, \cN}$.
\end{lemma}
\begin{proof}
  Lemma~\ref{lemma:excessive functions provide nests} shows that $\{\{h \leq n\} \mid n\geq 0\}$ is a nest. According to the lemma on the refinement of nests, Lemma~\ref{lemma:refinement of nests},  $\cN_h$ is a nest as well. To prove the other statements, for each $N \in \cN$ we choose a function $g_N \in D(E)_\cN$ that satisfies $1_N \leq g_N \leq 1$. 
  
  {\em Claim~1:} $\cN_h$ is special.
  
  {\em Proof of Claim~1.}  For $n \in \IN$ and $N \in \cN$ consider the function
 $$g_{n,N}:=  2\left(g_N- \frac{h}{2n}\right)_+.$$
 The inequality $g_{n,N} \geq 1_{N \cap \{h \leq n\}}$ holds and by Lemma~\ref{lemma:superharmonic functions as cutoff} it belongs to $D(E)$. Furthermore, it satisfies  $g_{n,N} = 0$ on $\{h > 2n\}$ and $|g_{n,N}| \leq g_N$. This shows that if $g_N$ vanishes on the complement of $N' \in \cN$, then $g_{n,N}$ vanishes on the complement of $N'\cap \{h\leq 2n\}$. We obtain $g_{n,N} \in D(E)_{\cN_h}$ and the claim is proven. \qedc
  
 {\em Claim~2:} $h \in D(E)_{{\rm loc},\, \cN_h}$.
 
 {\em Proof of Claim~2.} Lemma~\ref{lemma:superharmonic functions as cutoff} and the contraction properties of $E$ imply
 $$E((n g_N) \wedge (h \wedge n) ) \leq E(n g_N) < \infty.$$
 Since $ (n g_N) \wedge (h \wedge n)  = h $ on $N \cap \{h \leq n\}$ it follows that $h \in D(E)_{{\rm loc},\, \cN_h}.$ \qedc
 
 For the 'in particular' statement we note that the boundedness of $h$ implies $\cN \subseteq \cN_h$. This finishes he proof.
\end{proof}

\begin{remark}
 \begin{itemize}
  \item The previous two lemmas are the main reason why it is essential to allow all (special) nests in the definition of the local space and the weak form extension.
  \item If $h$ is excessive and $\cN$ is a special nest, then $\cN_h$ is always a nest. However, if $h$ is not nonnegative, then $\cN_h$ need not be special.
 \end{itemize}

\end{remark}

\begin{lemma}\label{lemma:excessive function ideal}
  Let $E$ be an energy form and let $\cN$ be a special $E$-nest. If $h$ is a nonnegative $E$-excessive function, then for any $\eta \in D(E)_{\cN_h} \cap L^\infty(m)$  we have  $h\eta \in D(E)_{\cN_h}\cap L^\infty(m)$.
\end{lemma}
\begin{proof}
 According to the definition of $\cN_h$, for each $\eta \in D(E)_{\cN_h} \cap L^\infty(m)$ there exists $n \in \IN$ such that $  h \eta =  (h \wedge n) \eta$.  Lemma~\ref{lemma:excessive functions belong to the local form domain} shows that $h \wedge n \in D(E)_{{\rm loc}, \, \cN_h}$. Since $\eta$ is bounded we obtain  $h\eta =  (h \wedge n) \eta \in D(E)\cap L^\infty(m)$ from Theorem~\ref{theorem:algebraic and order properties}. 
\end{proof}

The following three lemmas provide a variant of Lemma~\ref{lemma:alternative formula for er} for excessive functions. In their proofs we employ basically the same computations as in the proof of Lemma~\ref{lemma:alternative formula for er}, with the additional difficulty that we have to guarantee that certain limits exist.

\begin{lemma} \label{lemma: Alternative formula for ek}
Let $E$ be an energy form, let $\cN$ be a special $E$-nest and let $h$ be a nonnegative $E$-excessive function. Let $\psi \in D(E)_{\cN_h}$ and, for $n \in \IN$, set $\psi^{(n)} := (\psi \wedge n) \vee (-n)$. Then  the equality
 $$\Ekm(h,\psi) = \lim_{\varphi \in I_{\cN}} \lim_{n\to \infty} E( h\psi^{(n)},\varphi)$$
 holds.
\end{lemma}
\begin{proof} Lemma~\ref{lemma:excessive functions belong to the local form domain} and Lemma~\ref{lemma:excessive function ideal} show that the involved quantities are well-defined.  We choose a function $\ow{\varphi} \in I_{\cN_h}$ that satisfies $\ow{\varphi} \geq 1_{\{|\psi|>0\}}$ and set $h_\psi := h \ow{\varphi}$. According to Lemma~\ref{lemma:excessive function ideal}, this function belongs to $D(E)_{\cN_h}\cap L^\infty(m)$. Furthermore, by the choice of $\ow{\varphi}$ we have $h\psi = h_\psi \psi$.  From the properties of the weak extension of $\Ekm$, see Lemma~\ref{lemma:extension of k}, and the definition of $\Ekm$ on $D(E)$ we obtain 
 $$\Ekm(h,\psi) = \Ekm(h_\psi,\psi) = E(h_\psi,\psi) - \Eg(h_\psi,\psi).$$
Corollary~\ref{corollary:alternative formula for eg} implies
 $$\Eg(h_\psi,\psi) = \lim_{\varphi \in I_{\cN}}\Ep(h_\psi,\psi).$$
 Since the $\Ep$ are energy forms, Proposition~\ref{proposition:approximation by bounded functions} yields
 $$\Eg(h_\psi,\psi) = \lim_{\varphi \in I_{\cN}} \lim_{n\to \infty} \Ep(h_\psi,\psi^{(n)}).$$

 We have  $h_\psi,\psi \in D(E)_{\cN_h} \subseteq D(E)_{\cN}$ and the nest $\cN$ is special. Hence, there exists $\varphi_0 \in I_{\cN}$ with $h_\psi \varphi_0 = h_\psi$ and $\psi \varphi_0 = \psi$. With this observation the limit simplifies to 
 \begin{align*}
  \lim_{\varphi \in I_{\cN}}  \lim_{n\to \infty} \Ep(h_\psi,\psi^{(n)}) &= \lim_{\varphi \in I_{\cN}} \lim_{n\to \infty}\left[E(\varphi h_\psi, \varphi \psi^{(n)}) - E(\varphi h_\psi \psi^{(n)}, \varphi) \right]\\
  &= E(h_\psi, \psi) -  \lim_{\varphi \in I_{\cN}}\lim_{n\to \infty}E(\varphi h  \psi^{(n)}, \varphi)\\
  &= E(h_\psi, \psi) -  \lim_{\varphi \in I_{\cN}}\lim_{n\to \infty}E(h  \psi^{(n)}, \varphi).
 \end{align*}
 Combining all the computations concludes the proof of the lemma.
\end{proof}

\begin{remark}
 In the previous lemma we assumed $\psi \in D(E)_{\cN_h}$, but the limit is taken with respect to $\cN$. In view of later applications, it would be interesting to know whether 
 $$\Ekm(h,\psi) = \lim_{m \to \infty} \lim_{\varphi \in I_{\cN}} \lim_{n\to \infty} E( h^{(m)}\psi^{(n)},\varphi)$$
 holds for all $\psi \in D(E)_\cN$, which is the precise analogue of Lemma~\ref{lemma:alternative formula for er}. Our proof does not seem to work in this case.
\end{remark}

\begin{lemma}\label{lemma: main lemma for excessive functions in weak domain}
 Let $E$ be an energy form, let $\cN$ be a special $E$-nest and let $h$ be a nonnegative $E$-excessive function. For $n \in \IN$ we set $h^{(n)} := h \wedge n$. Let $\psi \in D(E)_{\cN_h}$ nonnegative and let $m_\psi \in \IN$ such that $ \{\psi > 0\} \subseteq \{h \leq m_\psi\}$. Then the map
 $$\{\varphi \in I_{\cN} \mid \varphi \geq 1_{\{\psi>0\}}\} \times \{n \in \IN \mid n \geq m_\psi\} \to \IR,\, (\varphi,n) \mapsto E(\varphi h^{(n)},\psi)$$
 is well-defined, takes only nonnegative values and is monotone decreasing. 
\end{lemma}
\begin{proof}
  According to Corollary~\ref{corollary:concatenation superharmonic functions}, the function $h^{(n)}$ is $E$-excessive. Since it is also bounded, we have $\cN \subseteq \cN_{h^{(n)}}$ and so Lemma~\ref{lemma:excessive function ideal} shows that the map under consideration is well-defined. Let $\varphi,\varphi' \in I_{\cN}$ with $\varphi \geq \varphi' \geq 1_{\{\psi>0\}}$. We apply Lemma~\ref{lemma:cutoff for functions with disjoint support} to the nonnegative functions $(\varphi - \varphi')h^{(n)} $ and $\psi$ which satisfy $(\varphi - \varphi')h^{(n)} \wedge \psi = 0$ and obtain
$$E(\varphi h^{(n)} ,\psi) - E(\varphi' h^{(n)} , \psi) = E((\varphi - \varphi')h^{(n)} ,\psi) \leq 0.$$
This shows the monotonicity in $\varphi$. Using that $n \geq m_\psi$, the monotonicity in $n$ can be inferred similarly. 

Let $\varepsilon > 0$. The inequality $1_{\{\psi > 0\}} \leq \varphi \leq 1$  and the nonnegativity of $h^{(n)} $ imply
$$h^{(n)} \varphi =  (h^{(n)} \varphi + \varepsilon \psi) \wedge h^{(n)} .$$
Since $h^{(n)}$ is excessive, Lemma~\ref{lemma:superharmonic functions as cutoff} yields
$$E(h^{(n)} \varphi) = E( (h^{(n)} \varphi + \varepsilon \psi) \wedge h^{(n)} ) \leq E(h^{(n)} \varphi + \varepsilon \psi) = E(h^{(n)} \varphi) + 2\varepsilon E(h^{(n)} \varphi,\psi) + \varepsilon^2 E(\psi). $$
Letting $\varepsilon\to 0+$,   we obtain $E(\varphi h^{(n)} , \psi) \geq 0$. This finishes the proof. 
\end{proof}

All the preparations  made so-far accumulate in the following lemma. It can be seen as the main step in the proof Theorem~\ref{theorem:characterization excessive functions as superharmonic functions}.

\begin{lemma} \label{lemma:excessive functions are in the weak form domain}
 Let $E$ be an energy form, let $\cN$ be a special $E$-nest and let $h$ be a nonnegative $E$-excessive function. If $h \in \DEln$, then for each $\psi \in D(E)_{\cN_h}$  the limit $\lim_{\varphi \in I_{\cN}} \Ep(h,\psi)$ exists and satisfies 
 $$\lim_{\varphi \in I_{\cN}} \Ep(h,\psi) = \lim_{\varphi \in I_{\cN}} \lim_{n\to \infty} E(\varphi h^{(n)},\psi)  - \Ekm(h,\psi).$$
 In particular, $h \in D(E)_{w,\, \cN_h}$ and for all $\psi \in D(E)_{\cN_h}$ we have
 $$E^{(w)}_{\cN_h}(h,\psi) =  \lim_{\varphi \in I_{\cN_h}} \lim_{n\to \infty} E(\varphi h^{(n)},\psi) = \lim_{\varphi \in I_{\cN_h}} E(\varphi h,\psi).$$
\end{lemma}
\begin{proof}
 It suffices to consider the case when $\psi \in D(E)_{\cN_h}$ is nonnegative. We choose $m_\psi \in \IN$ such that $\{\psi > 0 \} \subseteq \{h \leq m_\psi\}$ and set
 $$J_{\cN} := \{\varphi \in I_\cN \mid \varphi \geq 1_{\{\psi>0\}}\}.$$
 Since $\cN$ is special, the directed set $J_\cN$ is cofinal in $I_{\cN}$. This observation combined with the monotonicity statement of Lemma~\ref{lemma: main lemma for excessive functions in weak domain} yield that $\lim_{\varphi \in I_{\cN}} \lim_{n\to \infty} E(\varphi h^{(n)},\psi)$ exists. For $m > 0$, we let $\psi^{(m)} := \psi \wedge m$. Proposition~\ref{proposition:approximation by bounded functions} and the definition of $\Ep$ imply
 \begin{align*}
 \lim_{\varphi \in I_{\cN}} \lim_{n\to \infty} E(\varphi h^{(n)},\psi) &=  \lim_{\varphi \in J_{\cN}} \lim_{n\to \infty} \lim_{m \to \infty} E(\varphi h^{(n)}, \varphi \psi^{(m)})\\
  &= \lim_{\varphi \in J_{\cN}} \lim_{n \to \infty} \lim_{m \to \infty} \left[\Ep(h^{(n)}, \psi^{(m)}) + E(\varphi h^{(n)} \psi^{(m)},\varphi) \right].\\
 \end{align*}
 For $\varphi \in J_{\cN}$ and $n \geq m_\psi$  we have 
 $$E(\varphi h^{(n)} \psi^{(m)},\varphi) = E(h \psi^{(m)},\varphi).$$
 According to Lemma~\ref{lemma: Alternative formula for ek}, this implies
 $$ \lim_{\varphi \in J_{\cN}} \lim_{n \to \infty} \lim_{m \to \infty}E(\varphi h^{(n)} \psi^{(m)},\varphi) = \lim_{\varphi \in J_{\cN}}   \lim_{m \to \infty} E(h \psi^{(m)},\varphi) = \Ekm(h,\psi).$$
 Combining these considerations, we arrive at
 $$\lim_{\varphi \in J_{\cN}} \lim_{n \to \infty} \lim_{m \to \infty} \Ep(h^{(n)}, \psi^{(m)}) =  \lim_{\varphi \in I_{\cN}} \lim_{n\to \infty} E(\varphi h^{(n)},\psi) - \Ekm(h,\psi).$$
 Since $ h \in \DEln$, Lemma~\ref{lemma:local space is contained in dep} yields $h \in D(\Ep)$ for all $\varphi \in I_\cN$. Consequently, by  Proposition~\ref{proposition:approximation by bounded functions} the previous equation simplifies to
 $$\lim_{\varphi \in J_{\cN}}  \Ep(h, \psi) =  \lim_{\varphi \in I_{\cN}} \lim_{n\to \infty} E(\varphi h^{(n)},\psi) - \Ekm(h,\psi).$$
 Since $J_\cN$ is cofinal in $I_\cN$ this proves the first statement.
 
  For the 'in particular' part we apply what we have already proven to the nest $\cN_h$.  According to Lemma~\ref{lemma:excessive functions belong to the local form domain}, $\cN_h$ is special and we have $h \in D(E)_{{\rm loc},\, \cN_h}$. Furthermore, for $\varphi \in I_{\cN_h}$ the equality $\varphi h^{(n)} = \varphi h$ holds whenever $n$ is large enough. For $\psi \in D(E)_{\cN_h}$, this observation and our previous considerations show  
 $$\lim_{\varphi \in I_{\cN_h}} \Ep(h,\psi)  = \lim_{\varphi \in I_{\cN_h}}  \lim_{n\to \infty} E(\varphi h^{(n)},\psi) - \Ekm(h,\psi) = \lim_{\varphi \in I_{\cN_h}}   E(\varphi h,\psi) - \Ekm(h,\psi).$$
 This implies $h \in D(E)_{w,\, \cN_h}$ and finishes the proof. 
\end{proof}

\begin{proof}[Proof of Theorem~\ref{theorem:characterization excessive functions as superharmonic functions}] 
 (i) $\Rightarrow$ (iii):  Let $\cN$ be a special $E$-nest. According to Lemma~\ref{lemma:excessive functions belong to the local form domain}, $\cN_h$ is a special $E$-nest. Lemma~\ref{lemma:excessive functions are in the weak form domain} shows $h \in D(E)_{w,\, \cN_h}$ and that for $\psi \in D(E)_{\cN_h}$ we have
 $$\Ew_{\cN_h}(h,\psi) =  \lim_{\varphi \in I_{\cN_h}} \lim_{n\to \infty} E(\varphi h^{(n)},\psi).$$
 If, additionally, $\psi$ is nonnegative, then Lemma~\ref{lemma: main lemma for excessive functions in weak domain} yields that this limit is nonnegative and implication (i) $\Rightarrow$ (iii) is proven.

(iii) $\Rightarrow$ (ii): This follows from the existence of a special $E$-nest. 

(ii) $\Rightarrow$ (i): Let $\cN$ be a special $E$-nest such that $h \in D(E)_{w,\, \cN}$ and $\Ew_\cN(h,\psi) \geq 0$ for all nonnegative $\psi \in D(E)_{\cN}$.  According to Theorem~\ref{theorem:characterization of excessive functions}, we need to show
$$E(f\wedge h) \leq E(f) \text{ for all } f \in D(E).$$
Since $D(E)_\cN $ is dense in $D(E)$, it suffices to prove this inequality for $f \in D(E)_\cN$. To this end, we basically follow the proof of Lemma~\ref{lemma:superharmonic functions as cutoff}.

 Let $\varphi \in I_{\cN}$ and $f \in D(E)_\cN$. Since $h$ is nonnegative and belongs to $\DEln$, we have $f\wedge h = f_+ \wedge h - f_- \in D(E)_\cN$ and $(f-h)_+ \in D(E)_\cN$.  Furthermore,  Lemma~\ref{lemma:local space is contained in dep} shows $h \in D(\Ep)$.  The contraction properties of $\Ep$ yield
\begin{align*}
E_\varphi (f \wedge h,f)-  E_\varphi(f \wedge h)  &= \frac{1}{4} \left( E_\varphi(f) - E_\varphi(h) - E_\varphi(|f-h|) + 2 E_\varphi(h,|f-h|)\right)\\
&\geq  \frac{1}{4} \left( E_\varphi(f) - E_\varphi(h) - E_\varphi(f-h) + 2 E_\varphi(h,|f-h|)\right)\\
&= \frac{1}{2} E_\varphi(h, f - h + |f-h|)\\
&= E_\varphi(h,(f-h)_+).
\end{align*}
The identity $h (f-h)_+ = f\wedge h (f-h)_+$ and Lemma~\ref{lemma:extension of k} imply
$$\Ekm(f\wedge h,f) - \Ekm(f\wedge h) = \Ekm(f \wedge h, (f - h)_+) = \Ekm(h, (f- h)_+). $$
From Theorem~\ref{theorem:ew extends e} and the assumption on $h$ we infer 
\begin{align*}
 E (f \wedge h,f)-  E(f \wedge h)  &=\lim_{\varphi \in I_\cN} \left[\Ep (f \wedge h,f) -  \Ep(f \wedge h)\right] + \Ekm(f\wedge h,f) - \Ekm(f\wedge h)\\
 &\geq \lim_{\varphi \in I_\cN}E_\varphi(h,(f-h)_+) + \Ekm(h,(f-h)_+)\\
 &= \Ew_\cN (h,(f-h)_+) \geq 0.
\end{align*}
That the inequality $E (f \wedge h,f)-  E(f \wedge h) \geq 0$ yields $E(f\wedge h) \leq E(f)$ can be proven as in the proof of Lemma~\ref{lemma:superharmonic functions as cutoff}. This finishes the proof of the theorem. 
\end{proof}

For an excessive function $h$ and a special nest $\cN$, Theorem~\ref{theorem:characterization excessive functions as superharmonic functions} and the lemmas preceding its proof make  claims for the nest $\cN_h$. In some applications however, one is given other excessive functions $h'$ and one needs to compare the weak form coming from $\cN_h$ and $\cN_{h'}$. The following proposition shows that this is possible and demonstrates that bounded excessive functions are particularly well adapted to the theory. It should be read in view of Lemma~\ref{lemma:independence of nests}.

\begin{proposition} \label{proposition:nh vs nh'}
 Let $E$ be an energy form and let $\cN$ be a special $E$-nest.  For two nonnegative $E$-excessive functions $h,h'$, the following holds true.
 \begin{itemize}
 \item[(a)] If $h \in L^\infty(m)$, then  the nests $\cN_{h}$ and $\cN$ are equivalent.  In particular, $h \in D(E)_{w,\, \cN}$ and for all $\psi \in D(E)_\cN$ we have
 $$\Ew_{\cN} (h,\psi) = \Ew_{\cN_{h}}(h,\psi).$$
  \item[(b)] If $h \leq h'$, then  the nests $\cN_{h'}$ and $(\cN_{h'})_{h}$ are equivalent. In particular, $h \in D(E)_{w,\, \cN_{h'}}$ and for all $\psi \in D(E)_{\cN_{h'}}$ we have
  $$\Ew_{\cN_{h'}} (h,\psi) = \Ew_{(\cN_{h'})_h} (h,\psi). $$

 \end{itemize}
\end{proposition}
 \begin{proof}
 We only need to show the claimed equivalences of the nests, the 'in particular' parts follow from Lemma~\ref{lemma:independence of nests} and Lemma~\ref{lemma:excessive functions are in the weak form domain}.
 
 (a):  For $N\in \cN$ and $n \in \IN$, the inclusion $N \cap \{h \leq n\} \subseteq N$ holds. Furthermore, since $h$ is essentially bounded, we have $N = N \cap \{h\leq n\}$ whenever $n$ is large enough. This proves the equivalence of $\cN$ and $\cN_h$.

 (b): For $N \in \cN$ and $n,m \in \IN$,  the inequality $h\leq h'$ implies the inclusions
 $$N \cap \{h'\leq n\} \cap \{h \leq m\} \subseteq N \cap \{h'\leq n\}  \subseteq N \cap \{h'\leq n\} \cap \{h \leq n\}. $$
 This shows the equivalence of $\cN_{h'}$ and $(\cN_{h'})_{h}$.
 \end{proof}

  Excessive functions are originally defined to be limits of superharmonic functions and one may wonder whether or not the weak form respects this limit. Theorem~\ref{theorem:characterization excessive functions as superharmonic functions} and the further tools developed in this subsection allow us to give an affirmative answer under an additional condition.  

\begin{theorem} \label{theorem:computing the weak form for excessive functions}
  Let $E$ be a transient energy form of full support, let $\cN$ be a special $E$-nest and let $h$ be $E$-excessive. Assume that there exists a net of nonnegative functions $(h_i)$ in $D(E)$ with the following properties.
 \begin{itemize}
   \item[(a)] $h_i \overset{m}{\to} h$. 
   \item[(b)] For any $i$ the function $h-h_i$ is $E$-excessive.
   \item[(c)] For all $N \in \cN$ there exists an $i_N$ such that for all $i \succ i_N$ and all nonnegative $\psi \in D(E)_N$  the inequality $E(h_i,\psi) \geq 0$ holds. 
  \end{itemize}
  Then  $h \in D(E)_{w,\,\cN_h}$ and for all $\psi \in D(E)_{\cN_h}$ the identity
 $$\Ew_{\cN_h}(h,\psi) = \lim_i E(h_i,\psi)$$
 holds. 
\end{theorem}
\begin{remark}
\begin{itemize}
 \item Assumption (c) is of course satisfied if all the $h_i$ are superharmonic.  We included this slightly more general version since it will be useful later on. Here, the really important assumption is that $h-h_i$ is excessive. It yields some a-priori boundedness for the limit. 
 \item When we introduced the weak form, we remarked that its value may depend on the chosen nest. The theorem shows that this is not the case for excessive functions that can be approximated by superharmonic functions obeying (b). 
\end{itemize}
\end{remark}
Before proving the theorem, we need one further lemma. It weakens the assumptions of Lemma~\ref{lemma:superharmonic functions as cutoff}.
\begin{lemma}\label{lemma:cutoff}
 Let $E$ be an energy form and let $f,h \in D(E)$. If $E(h,(f-h)_+) \geq 0$, then 
 $$E(f \wedge h) \leq E(f)  \text{ and } E((f-h)_+) \leq E(f).$$
\end{lemma}
\begin{proof}
The inequality  $E(h,(f-h)_+) \geq 0$ is all that is required in the proof of Lemma~\ref{lemma:superharmonic functions as cutoff}. Therefore, the same computations that were used there show the claim.   
\end{proof}

\begin{proof}[Proof of Theorem~\ref{theorem:computing the weak form for excessive functions}] Without loss of generality we can assume that $\psi \in D(E)_{\cN_h}$ is nonnegative. The transience of $E$ yields that $h-h_i$ is nonnegative. According to Proposition~\ref{proposition:nh vs nh'},  $h-h_i \in D(E)_{w,\,\cN_h}$ and for all $\psi \in D(E)_{\cN_h}$ we have
$$\Ew_{\cN_h}(h-h_i,\psi) = \Ew_{(\cN_h)_{h-h_i}}(h-h_i,\psi).$$
With this observation at hand, Theorem~\ref{theorem:ew extends e} and Theorem~\ref{theorem:characterization excessive functions as superharmonic functions} show $h \in D(E)_{w,\, \cN_h}$  and 
$$\Ew_{\cN_h}(h,\psi) - E(h_i,\psi) = \Ew_{\cN_h}(h-h_i, \psi) \geq 0.$$
Hence, it suffices to prove $\limsup_i \Ew(h-h_i, \psi) \leq 0$. Lemma~\ref{lemma:excessive functions are in the weak form domain} yields
$$\Ew_{\cN_h}(h-h_i, \psi) = \lim_{\varphi \in I_{\cN_h}} E(\varphi(h-h_i),\psi),$$
and according to Lemma~\ref{lemma: main lemma for excessive functions in weak domain},  the map
$$\{\varphi \in I_{\cN_h} \mid \varphi \geq 1_{\{\psi>0\}}\} \to \IR,\, \varphi \mapsto E(\varphi(h-h_i),\psi)$$
is monotone deceasing and takes nonnegative values. Together with Lemma~\ref{lemma:monotone nets}, these observations imply
$$ \limsup_i \Ew_{\cN_h}(h-h_i, \psi) = \limsup_i \lim_{\varphi \in I_{\cN_h}} E(\varphi(h-h_i), \psi) \leq    \limsup_{\varphi \in I_{\cN_h}} \limsup_i E(\varphi(h-h_i), \psi).$$
Therefore, it remains to prove that for each $\varphi \in I_{\cN_h}$ we have $ \limsup_i E(\varphi(h-h_i), \psi) = 0$. 

 In order to establish this identity, we employ Lemma~\ref{lemma:existence of a weakly convergent subnet} to show the $E$-weak convergence $\varphi(h-h_i) \to 0$. The properties of the net $(h_i)$  imply $\varphi(h-h_i) \to 0$ in $L^0(m)$ and so it remains to check that $i \mapsto E(\varphi(h-h_i))$ is bounded, whenever $i$ is large enough. To this end, let $N \in \cN_h$ such that $\varphi 1_{X\setminus N} = 0$ and choose a nonnegative function $h_N \in D(E)_{\cN_h} \cap L^\infty(m)$ with $h_N = h$ on $N$. The existence of such an $h_N$ follows from Lemma~\ref{lemma:excessive function ideal}. The choice of $N$ and $h_N$ and the inequality $h_i \leq h$ imply
$$\varphi (h - h_i) = \varphi (h - h_i)_+ = \varphi (h_N - h_i)_+.$$
 Theorem~\ref{theorem:algebraic and order properties} yields
\begin{align*}
  E(\varphi (h_N - h_i)_+) ^{1/2} &\leq  E((h_N - h_i)_+) ^{1/2} + \|(h_N - h_i)_+\|_\infty E(\varphi)^{1/2}\\
  &\leq E((h_N-h_i)_+)^{1/2} + \|h_N\|_\infty E(\varphi)^{1/2}.
\end{align*}
 Since $(h_N - h_i)_+ \leq h_N$, there exists $\ow{N} \in \cN_h$ such that $(h_N - h_i)_+ \in D(E)_{\ow{N}}$ for all $i$. Hence, assumption $(c)$ implies $E(h_i, (h_N - h_i)_+) \geq 0$ whenever $i$ is large enough. According to Lemma~\ref{lemma:cutoff}, this implies $E((h_N - h_i)_+) \leq E(h_N)$ whenever $i$ is large enough.  We obtain the boundedness of $E(\varphi(h-h_i))$ for large $i$ and the theorem is proven.
\end{proof}

We finish this section with two examples.The first one shows that Theorem~\ref{theorem:characterization excessive functions as superharmonic functions} provides a particularly nice description of excessive functions for the jump-type forms on discrete spaces, which we studied in the last section.

\begin{example}[Excessive functions of discrete jump-type forms]
 We use the notation and assumptions of Example~\ref{example:graphs}. 
 
 \begin{proposition}
  Let $h \in C(X)$ be nonnegative. The following assertions are equivalent.
  \begin{itemize}
   \item[(i)] $h$ is $q$-excessive.
   \item[(ii)] $h \in D(\ow{L})$ and $\ow{L}h \geq 0$.
  \end{itemize}
 \end{proposition}
 \begin{proof}
 $q$ is a regular energy form. Therefore, it has full support and we can apply the theorems proven in this section. 
 
  (i) $\Rightarrow$ (ii): The discreteness of the space implies that $\mathcal{K}$ is the collection of finite subsets of $X$ and so we have $\mathcal{K} \subseteq  \mathcal{K}_h \subseteq \mathcal{K} \cup \{\emptyset\}.$ Therefore,  Theorem~\ref{theorem:characterization excessive functions as superharmonic functions} yields $h \in D(q)_{w,\, \mathcal{K}}$ and $q^{(w)}_\mathcal{K} (h,\psi) \geq 0$ whenever $\psi \in D(q)_\mathcal{K}$ is nonnegative.  According to Proposition~\ref{proposition:example graphs}, this is equivalent to $h \in D(\ow{L})$ and $\ow{L}h \geq 0$.
  
  (ii) $\Rightarrow$ (i): By Proposition~\ref{proposition:example graphs}, assertion (ii) implies that $h$ is weakly $q$-superharmonic and so it is excessive by Theorem~\ref{theorem:characterization excessive functions as superharmonic functions}.
 \end{proof}
\end{example}

\begin{example}[Excessive functions on manifolds]
  We use the notation and assumptions of Example~\ref{example:manifolds}. 
  \begin{proposition}
  Let $h \in \Ltlm$ with $\Delta h \in \Ltlm$ be nonnegative. The following assertions are equivalent. 
  \begin{itemize}
   \item[(i)] $h$ is  $Q$-excessive. 
   \item[(ii)] $-\Delta h \geq 0$.
  \end{itemize}
  \end{proposition}
  \begin{proof}
   The form $Q$ is regular and so it has full support. Therefore, we can apply all the theorems of this section. 
  
   (ii) $\Rightarrow$ (i): If $-\Delta h \geq 0$, then Proposition~\ref{proposition:weak form manifold} shows that $h$ is weakly $Q$-superharmonic. Hence, Theorem~\ref{theorem:characterization excessive functions as superharmonic functions} implies that $h$ is excessive. 
   
   (i) $\Rightarrow$ (ii): Let $\mathcal{K}$ be the special $Q$-nest of relatively compact open subsets of $M$. For $n \in  \IN$, the function $h\wedge n$ is bounded and $Q$-excessive, see Corollary~\ref{corollary:concatenation superharmonic functions}. Therefore, Theorem~\ref{theorem:characterization excessive functions as superharmonic functions} implies $h\wedge n \in D(Q)_{w,\, \mathcal{K}}$ and 
   $$Q^{(w)}_{\mathcal{K}}(h\wedge n,\psi) \geq 0 \text{ for all } \psi \in D(Q)_{\mathcal{K}}. $$
   The space $C_c^\infty(M)$ is contained in $D(Q)_{\mathcal{K}}$  and $\Delta$ is continuous with respect to convergence of distributions. For nonnegative $\psi \in C_c^\infty(M)$, this implies
   \begin{align*}
    -\Delta h (\psi) &= \lim_{n \to \infty} -\Delta (h\wedge n) (\psi)\\
    &= \lim_{n \to \infty} \nabla(h\wedge n) (\nabla \psi)\\ 
    &= \lim_{n \to \infty} \int_M g(\nabla (h\wedge n),\nabla \psi) \D {\rm vol}_g \\
    &= \lim_{n \to \infty} Q^{(w)}(h\wedge n,\psi) \geq 0.
   \end{align*}
   Here, we viewed $\nabla (h\wedge n)$ and $-\Delta h$ as distributions and used Proposition~\ref{proposition:weak form manifold} that shows $h\wedge n \in W^1_{\rm loc}(M)$ and provides the formula for $Q^{(w)}_\mathcal{K}$. This finishes the proof.
  \end{proof}
  \begin{remark}
  We characterized excessive functions that satisfy $-\Delta h \in \Ltlm$. It is possible to drop this assumption and and require (ii) to hold true in the sense of distributions. We leave the technical details to the reader.
  \end{remark}
\end{example}

\section{The potential operator}

In this section we study the existence of weak solutions to the Laplace equation with a particular focus on  minimal solutions when the functional on the right-hand side of the equation is positive. More precisely, we introduce the potential operator and show that it provides minimal solutions. We start our considerations with solutions to the (weak) Laplace equation that belong to the form domain. 

For a transient energy form $E$, we denote the space of continuous linear functionals on the Hilbert space $(D(E),E)$ by $L_c(D(E))$, cf. Section~\ref{section:lebesgue spaces}. Due to the Riesz representation theorem the space $L_c(D(E))$ is isomorphic to $D(E)$.  For $\ell \in L_c(D(E))$ this isomorphism provides a unique solution to the equation
$$\begin{cases}
   E(g,\cdot) = \ell \\
   g \in D(E)
  \end{cases}.
$$
Therefore, it deserves a special name. 

\begin{definition}[Potential operator]
 Let $E$ be a transient energy form. The linear operator
 $$G:L_c(D(E)) \to D(E)$$
 that satisfies
 $$ E(G\ell, \psi) = \ell(\psi)$$
 for all $\ell \in L_c(D(E))$ and all $\psi \in D(E)$ is called the {\em potential operator}. The function $G\ell$ is called the {\em potential of $\ell$}.
 \end{definition}

 \begin{remark}
  The potential operator has featured in the literature in various forms. Here, we give some examples. Let $\E$ be a Dirichlet form and recall that for  $\alpha > 0$, we let $\E_{\alpha} = \E + \alpha \|\cdot\|_2^2$. As discussed previously, $\E_\alpha$ is a transient energy form. A  function $f \in L^2(m)$ gives rise to the  linear functional 
  $$ D(\E) \to \IR,\, g \mapsto \int fg\,{\rm d} m,$$
  which is continuous with respect to $\E_\alpha$. The potential of this functional (with respect to $\E_\alpha$) is given by the usual $L^2$-resolvent $G_\alpha f$. If, additionally, $\E$ is  regular further examples of functionals in $L_c(D(\E_\alpha))$ are given by measures of finite energy integral. Recall that a Radon measure $\mu$ is of finite energy integral with respect to $\E$ if there exists a constant $C>0$ such that
  $$\int |\varphi| {\rm d} \mu \leq C \sqrt{\E_1(\varphi)}\text{ for all } \varphi \in D(\E) \cap C_c(X).$$
  In this case, the linear functional 
  $$D(\E) \cap C_c(X) \to \IR,\, \varphi \mapsto \int \varphi\, {\rm d} \mu$$
  can be uniquely extended to a continuous linear functional on $(D(\E_\alpha),\E_\alpha)$. The potential of this functional (with respect to $\E_\alpha$) is the $\alpha$-potential $U_\alpha \mu$ in the sense of \cite[Section~2.2]{FOT}. Another instance of the potential operator appears in the realm of extended Dirichlet spaces. For a transient Dirichlet form $\E$ (which means that $\Ee$ is transient in our terminology) with associated semigroup $(T_t)$,  the Green operator of $\E$ is defined by  
  $$\ow{G}:L^1(m) \to L^0(m),\, f \mapsto \ow{G}f := \int_{0}^\infty T_tf\, {\rm d} t. $$
  For this operator, \cite[Theorem~1.5.4.]{FOT} shows that if $|f| \ow{G}|f| \in L^1(m)$, then $\ow{G} f \in D(\Ee)$ and for all $g \in D(\Ee)$ we have $f g \in L^1(m)$ and  
  $$\Ee(\ow{G}f,g) = \int fg\, {\rm d}m.$$
  In other words,   $|f|\ow{G}|f|\in L^1(m)$ implies that  $\ow{G}f$ is the potential (with respect to $\Ee$) of the functional
  $$D(\Ee) \to \IR,\, g \mapsto \int fg\, {\rm d} m.$$
  The $L^2$-resolvent, the $\alpha$-potential of measures of finite energy integral and the Green operator play an important role in the theory of Dirichlet forms. The previous discussion implies that the potential operator provides a common framework for them. For some physical background why $G$ is called potential operator we refer the reader to the introduction of \cite{Schmu3} and references therein.
 \end{remark}
 
  A large part of the theory of energy forms derives from the compatibility of the energy form with the order structure of the underlying space. From this point it is not so surprising that the potential operator also respects some form of order. The following proposition shows that it maps positive functionals on the form domain to nonnegative functions. We recall that $S(E)$ denotes the collection of all $E$-superharmonic functions and refer to Section~\ref{section:lebesgue spaces} for the definition of positive and regular functionals. 
 \begin{proposition}\label{proposition:potential operator is positive}
  Let $E$ be a transient energy form. 
  \begin{itemize}
   \item[(a)] $L_r(D(E)) \subseteq L_c(D(E))$, i.e., any regular linear functional on $D(E)$ is continuous. 
   \item[(b)]$G L_+(D(E)) = S(E)$. In particular, $G$ is positive, i.e., it maps positive functionals to nonnegative functions. 
  \end{itemize}
 \end{proposition}
 \begin{proof} (a): It suffices to show the continuity of positive functionals.  Let $\ell:D(E) \to \IR$ linear and positive and assume that it is not continuous. Then there exists a sequence $(f_n)$ in $D(E)$ with 
 $$|\ell(f_n)| \geq 4^n \text{ and } E(f_n) \leq 1. $$
 From the positivity of $\ell$ and the contraction properties of $E$ we infer 
 $$\ell(|f_n|) \geq |\ell(f_n)| \geq 4^n \text{ and }E(|f_n|) \leq E(f_n) \leq 1.$$
 Since $(D(E),E)$ is a Hilbert space and the $E$-norm of  $(|f_n|)$ is uniformly bounded, the limit
 $$g = \sum_{n = 1}^\infty 2^{-n} |f_n|$$
 exists in $(D(E),E)$. For each $n$ it satisfies $g \geq 2^{-n} |f_n|$. Hence, the positivity of $\ell$ and the properties of the $(f_n)$ yield that for each $n \geq 1$ the inequality
 $$\ell(g) \geq 2^{-n} \ell(|f_n|) \geq 2^n $$
 holds. This is a contradiction. 
 
 (b): By (a) we have $L_+(D(E)) \subseteq L_c(D(E))$. With this observation the equation $$G L_+(D(E)) = S(E)$$ is immediate from the definition of superharmonicity and of the potential operator. Furthermore, by Proposition~\ref{lemma:superharmonic functions are positive} all superharmonic functions are nonnegative.
 \end{proof}

 The potential of a continuous functional  is the unique solution to the  Laplace equation that belongs to the form domain. We are now going to extend the potential operator to larger classes of functionals and then show that this extension provides weak solutions to the Laplace equation. In general, the solutions that come from the extended potential operator need not be unique, but we prove that they satisfy some minimality property.  
 
 We start by studying the potential operator of certain restrictions of the given energy form. Let $E$ be a transient energy form and let $N \subseteq X$ measurable. Recall the notation
 $$D(E)_N  := D(E)_{\{N\}} = \{f \in D(E)\mid f1_{X \setminus N} = 0\}.$$
The space $D(E)_N$ is closed in $D(E)$ with respect to the form topology. Therefore, the restriction $E^N := E|_{D(E)_N}$ is a transient energy form with domain $D(E^N) = D(E)_N$.  We denote the potential operator of $E^N$ by $G^N$. Furthermore, we let $P^N$ be the Hilbert space projection in $(D(E),E)$ onto the closed subspace $D(E)_N$. The restriction of an arbitrary functional on $D(E)$  is given by
 $$R^N: L(D(E)) \to L(D(E)_N),\,  F \mapsto F|_{D(E)_N}.$$
 The following lemma shows that the potential operator, the projection and the restriction are compatible.
 \begin{lemma}\label{lemma:compatibility potential operator}
  Let $E$ be a transient energy form and let $N\subseteq X$ measurable.  
  \begin{itemize}
   \item[(a)] For all  $\ell \in L_c(D(E))$ the identity $P^N G \ell = G^N R^N \ell$ holds.
   \item[(b)] For all $\ell \in L_c(D(E^N))$ the identity $G^N \ell = P^N G (\ell\circ P^N)$ holds. 
  \end{itemize}
 \end{lemma}
 \begin{proof} (a): For $\psi \in D(E)_N$ and $\ell \in L_c(D(E))$ we have
 $$E^N(P^N G\ell, \psi) = E(P^N G\ell, \psi) = E(G\ell, \psi) = \ell(\psi) = R^N \ell(\psi) = E^N(G^N R^N \ell,\psi). $$
 This proves (a).
 
 (b): For $\ell \in L_c(D(E^N))$ the identity $\ell = R^N(\ell \circ P^N)$ holds. Therefore, the statement follows from (a). This finishes the proof. 
 \end{proof}
There is also some monotonicity among the potential operators of the restrictions. This is discussed next. 
\begin{lemma}\label{lemma: monotonicity pn}
 Let $E$ be a transient energy form and let $N \subseteq X$ measurable. For each nonnegative $f \in D(E)$ the inequality $ P_N f \leq f$ holds. In particular, for all $\ell \in L_+(D(E))$ we have 
 $$G^N R^N \ell \leq G \ell.$$
 \end{lemma}
 \begin{proof}
  We compute
  $$E(f - f \wedge P_N f) = E((f-P_Nf)_+) \leq E(f - P_N f).$$
  Since $f \wedge P_N f \in D(E)_N$ and $P_Nf$ is the unique minimizer of the functional $g \mapsto E(f-g)$ on $D(E)_N$, this implies $P_Nf = f \wedge P_N f$. The 'in particular' part follows from Proposition~\ref{proposition:potential operator is positive} and Lemma~\ref{lemma:compatibility potential operator}. 
 \end{proof}

\begin{remark}
 \begin{itemize}
  \item The idea for the proof of the previous lemma is taken from \cite[Proposition~B1]{SV}.  
  
  \item For measurable $M,N \subseteq X$ with $m(N \setminus M) = 0$, the previous lemma implies that for each $\ell \in L_+(D(E^M))$ the inequality $G^N R^N \ell \leq G^M \ell$ holds. In this sense, the family of potential operators $(G^N)$ is monotone in the parameter $N$. 
 \end{itemize}

\end{remark}
  As mentioned previously, we aim at extending the operator $G$ to larger classes of functionals. The following approximation for $G$ lies at the heart of these considerations. Recall that we equip $E$-nests with the preorder of inclusion up to sets of measure zero and that special $E$-nests are upwards directed with respect to this relation, see Lemma~\ref{lemma:special nests are directed}.  
 \begin{lemma}\label{lemma:convergence of projections}
  Let $E$ be a transient energy form and let $\cN$ be a special $E$-nest. For each $f \in D(E)$ the identity
  $$f = \lim_{N \in \cN} P^N f $$ 
  holds in the form topology. In particular, for each $\ell \in L_c(D(E))$ we have
  $$G \ell = \lim_{N \in \cN} G^N R^N \ell$$
  in the form topology.
 \end{lemma}
 \begin{proof}
  We first prove $ P^N f \to f$ $E$-weakly.   For $N\in \cN$ and $\psi \in D(E)_N$ the identity   
  $$E(P^N f,\psi) = E(f, \psi)$$
  holds. Hence, for all $\psi \in D(E)_\cN$ we have 
  $$\lim_{N \in \cN} E(P^N f,\psi) = E(f, \psi). $$
  Since the net $(P^N f)$ is $E$-bounded and $D(E)_\cN$ is dense in $D(E)$, this the shows $E$-weak convergence of $(P^N f)$ to $f$. 
  With this at hand, the claim follows from the transience of $E$ and the identity
   $$E(f  - P^Nf) = E(f) - 2 E(f,  P^Nf) + E(P^N f) = E(f) -  E(f,  P^Nf).$$
  The 'in particular' part is a consequence of what we have already proven and the formula $P^N G\ell = G^N R^N \ell$, see Lemma~\ref{lemma:compatibility potential operator}.
 \end{proof}

 \begin{remark} \label{remark:directed suffices}
  In the proof of the  lemma we did not use that the nest is special, only that it is upwards directed with respect to the preorder of inclusion up to sets of measure zero.
 \end{remark}

 The previous lemma hints on how to extend $G$ to more general functionals. The expression $G^N R^N \ell$ is well-defined for such $\ell$ whose restriction to $D(E)_N$ is linear and continuous with respect to $E^N$. After taking a limit along all $N$ in a given nest (if it exists) we obtain an extension of $G$. More precisely, for a special $E$-nest $\cN$ we let
 $$L_c(D(E)_\cN)  := \{\ell \in L(D(E)_\cN)\mid  R^N \ell \in L_c(D(E^N)) \text{ for each } N \in \cN\}. $$
  According to Proposition~\ref{proposition:potential operator is positive}, positive functionals with domain $D(E)_{\cN}$ and regular functionals with domain $D(E)_{\cN}$ are contained in $ L_c(D(E)_\cN) $. 
 \begin{remark}
 We could have equipped  $D(E)_\cN$ with a vector space topology such that the notation $L_c(D(E)_\cN)$ is consistent with the one introduced in Section~\ref{section:lebesgue spaces}. We refrain from giving details. 
 \end{remark}
 We define the domain of the  potential operator $G^\cN:D(G^\cN) \to L^0(m)$ with respect to the special nest $\cN$ as  
 $$D(G^\cN) := \{\ell \in L_c(D(E)_\cN) \mid \lim_{N\in \cN} G^N R^N \ell \text{ exists in } L^0(m) \},$$
 on which it acts by
 $$G^\cN \ell := \lim_{N\in \cN} G^N R^N \ell.$$
  In general, it is hard to determine the domain of $G^\cN$ and, like the weak form extension, its value may depend on the chosen nest. For positive functionals (and hence for regular functionals) things simplify since then the limit in the definition of $G^\cN$ is monotone. We employ this monotonicity to characterize when positive functionals belong to the domain of the extended potential operator and then show that its value is independent of the considered nest.  
\begin{proposition}\label{proposition:characterization of domain for positive functionals}
 Let $E$ be transient and let $\cN$ be a special nest. For a functional $\ell \in L_+(D(E)_\cN)$ the following assertions are equivalent. 
 \begin{itemize}
  \item[(i)] $\ell \in D(G^\cN).$  
  \item[(ii)] There exists a function $f \in L^0(m)$ such that for all $N \in \cN$ the inequality $G^NR^N \ell \leq f$ holds.
 \end{itemize}
In particular, if one of the above is satisfied, then 
$$G^\cN \ell= \sup_{N \in \cN} G^N R^N \ell.$$
\end{proposition}
\begin{proof}
 Lemma~\ref{lemma: monotonicity pn} shows that the positivity of $\ell$ implies that the net $(G^N R^N \ell)$ is  monotone increasing. By Proposition~\ref{proposition:convergence of monotone nets} it is convergent if and only if it is bounded. In this case, the limit is given by the supremum. This finishes the proof.  
\end{proof}
It is a consequence of the previous proposition that the operator $G$ can be unambiguously extended for positive functionals. This is discussed in the next lemma.
\begin{lemma} \label{lemma:potential operator well defined}
 Let $E$ be a transient energy from and let $\cN_i$, $i=1,2$, be two special $E$-nests. Furthermore, let $\ell_i \in L_+(D(E)_{\cN_i}),$ $i=1,2$, that coincide on $D(E)_{\cN_1} \cap D(E)_{\cN_2}$. Then  $\ell_1 \in D(G^{\cN_1})$ if and only if $\ell_2 \in D(G^{\cN_2})$. In this case, we have
 $$G^{\cN_1}\ell_1 = G^{\cN_2} \ell_2.$$
\end{lemma}
\begin{proof} We start with the following claim. 

{\em Claim:} Let $N\subseteq X$ be a measurable set, let $\cN$ be a special $E$-nest and let $\ell \in L_+(D(E^N))$. Then  following identity holds:
$$G^N   \ell  = \sup_{M \in \cN} G^{N \cap M} R^{N \cap M} \ell.$$

{\em Proof of the claim.} The set $\{N \cap M \mid M \in \cN\}$ is an $E^N$-nest. It might not be special, but since $\cN$ is upwards directed, it is upwards directed. Furthermore,  Lemma~\ref{lemma:compatibility potential operator} shows the identity
$$P^{N\cap M} (G^NR^N \ell) = G^{N \cap M} R^{N \cap M} \ell.$$
Therefore, the claim follows from Lemma~\ref{lemma:convergence of projections} and the monotonicity in $M$ (here we use that Lemma~\ref{lemma:convergence of projections} holds true for nests that are upwards  directed even though we only formulated it for special nests, cf. Remark~\ref{remark:directed suffices}). \qedc 

This claim  and Proposition~\ref{proposition:characterization of domain for positive functionals} show that   $\ell_1\in D(G^{\cN_1})$ implies
\begin{align*}
G^{\cN_1}\ell_1 &= \sup_{N_1 \in \cN_1} G^{N_1} R^{N_1} \ell_1 \\
&=  \sup_{N_1 \in \cN_1}\sup_{N_2 \in \cN_2} G^{N_1 \cap N_2} R^{N_1 \cap N_2} \ell_1\\
&= \sup_{N_1 \in \cN_1}\sup_{N_2 \in \cN_2} G^{N_1 \cap N_2} R^{N_1 \cap N_2} \ell_2\\
&=  \sup_{N_2 \in \cN_2} \sup_{N_1 \in \cN_1} G^{N_1 \cap N_2} R^{N_1 \cap N_2} \ell_2\\
&=  \sup_{N_2 \in \cN_2} G^{N_2} R^{N_2} \ell_2\\
&= G^{\cN_2}\ell_2.
\end{align*}
Here, we used properties of suprema in the Dedekind complete space $L^0(m)$. This finishes the proof. 
\end{proof} 
With the help of the previous lemma, we now introduce the extended potential operator for regular functionals, which does not depend on the chosen nest.  As seen in Section~\ref{section:lebesgue spaces}, any regular functional $\ell :D(\ell) \to \IR$ on $L^0(m)$ can be decomposed into $\ell = \ell_+ - \ell_-$, where 
$$\ell_+ = \sup\{\ell,0\}  \text{ and } \ell_- = \sup\{ -\ell,0\}$$
are positive functionals and the suprema are taken in the space $L(D(\ell))$. Furthermore,  recall the notation $|\ell| = \ell_+ + \ell_-$. If $\cN$ is an $E$-nest with $D(E)_\cN \subseteq D(\ell)$, we denote the restriction of $\ell$ to $D(E)_\cN$ by $R^\cN \ell$. Proposition~\ref{proposition:characterization of domain for positive functionals} and Lemma~\ref{lemma:potential operator well defined} guarantee that the following is well-defined.
\begin{definition}[Extended potential operator]\label{definition:extended potential operator}
 Let $E$ be a transient energy form.  The domain of the  {\em extended potential operator} $\Ge:D(\Ge) \to L^0(m)$ is 
$$
  D(\Ge) := \left\{\ell:D(\ell) \to \IR \text{ linear} \begin{rcases} \text{ex. a special } E\text{-nest } \cN \text{ with } D(E)_\cN \subseteq D(\ell)\\R^\cN \ell \in L_r(D(E)_\cN) \text{ and  }  |R^\cN \ell| \in D(G^\cN)\end{rcases} \hspace{-.5cm}\right\},
$$
  on which it acts by
$$\Ge \ell:= G^\cN (R^\cN \ell)_+ - G^\cN (R^\cN\ell)_-. $$
\end{definition}

\begin{remark}\label{remark:extended potential operator}
\begin{itemize}
\item According to Proposition~\ref{proposition:characterization of domain for positive functionals} and the monotonicity of $G^\cN$, the condition $|R^\cN \ell| \in D(G^\cN)$ implies $(R^\cN\ell)_+, (R^\cN\ell)_-  \in D(G^\cN)$. This in turn yields $R^\cN \ell \in D(G^\cN)$ and the linearity of $G^\cN$ shows
$$G^\cN (R^\cN \ell)_+ - G^\cN (R^\cN\ell)_- = G^\cN \ell.$$
\item  For $\ell_i \in D(\Ge)$, $i=1,2$, with corresponding special nests $\cN_i$, $i=1,2$, one can  set 
               $$\ell_1 + \ell_2:D(E)_{\cN_1 \wedge\, \cN_2} \to \IR,\, f \mapsto \ell_1(f) + \ell_2(f).$$
               It is not hard to check that in this case $\ell_1 + \ell_2 \in D(\Ge)$ and $\Ge(\ell_1 + \ell_2) = \Ge \ell_1 + \Ge \ell_2.$ 

\item  All of the examples that we discussed at the beginning of this section and that we will discuss in the next section are applications of $G$ and $G^r$ to regular functionals. In this sense, only considering regular functionals is not very restrictive.     
\end{itemize}
\end{remark}
 We now come to the main result of this section. It shows that the potential operator provides minimal solutions to the weak Laplace equation.  
 \begin{theorem} [A minimum principle for the potential operator] \label{theorem:maximal principle potential operator}
  Let $E$ be a  transient energy form of full support. Let $\ell$ be a positive linear functional on $L^0(m)$. The following assertions are equivalent.
  \begin{itemize}
   \item[(i)] There exists a nonnegative weak supersolution to the equation 
   $\Ew(g,\cdot) = \ell.$
   \item[(ii)] $\ell \in D(\Ge)$.
  \end{itemize}
  If one of the above is satisfied, then the following holds true.
  \begin{itemize}
   \item[(a)]The function $G^r \ell$ is a weak solution to the equation  $\Ew(g,\cdot) = \ell$. More precisely, $\Ge\ell$ is $E$-excessive and for each special $E$-nest $\cN$ and all $\psi \in D(E)_{\cN_{\Ge\ell}}$ we have
   $$\Ew_{\cN_{\Ge\ell}}(\Ge\ell,\psi) = \ell (\psi).$$
   \item[(b)] If $h$ is a weak supersolution to the equation $\Ew(g,\cdot) = \ell$,  then $\Ge \ell \leq h$. 
  \end{itemize}
 \end{theorem}
\begin{proof}
 (i) $\Rightarrow$ (ii): Let $h$ be a weak supersolution to the equation $\Ew(g,\cdot) = \ell$. By definition, there exists a special $E$-nest $\cN$ such that $D(E)_{\cN} \subseteq D(\ell)$,  $h \in D(E)_{w,\, \cN}$ and 
 $$\Ew_\cN(h,\psi) \geq \ell(\psi), \text{ for all nonnegative }  \psi \in D(E)_{\cN}.$$
  In particular, Theorem~\ref{theorem:characterization excessive functions as superharmonic functions} shows that $h$ is $E$-excessive and Lemma~\ref{lemma:excessive functions provide nests} implies that $\cN_h$ is a special nest. We prove  that for each $N \in \cN_h$ the inequality 
 $$G^N R^N \ell \leq h$$
 holds, which implies $\ell \in D(\Ge)$ and $\Ge \ell \leq h$ by Proposition~\ref{proposition:characterization of domain for positive functionals}.
 
 To this end, we let  $N \in \cN_h$ and choose a function $\varphi \in D(E)_{\cN_h}$ that satisfies $1_N \leq \varphi \leq 1$.  Lemma~\ref{lemma: main lemma for excessive functions in weak domain} and Lemma~\ref{lemma:excessive functions are in the weak form domain}  yield that for each nonnegative $\psi \in D(E)_N$ there exists $m_\psi \in \IN$ such that $\varphi (h\wedge n)\in D(E)$ and $E(\varphi (h\wedge n),\psi) \geq \Ew_\cN(h,\psi)$ for each $n \geq m_\psi$. Furthermore, since $\varphi \in I_{\cN_h}$, we have $\varphi (h\wedge n) = \varphi h$ for large enough $n$. Together with the assumptions on $h$, these two observations show that for all nonnegative $\psi \in  D(E)_N$  the following holds:
 $$E(\varphi h,\psi)   \geq \Ew_\cN(h,\psi) \geq \ell(\psi) = E(G^NR^N \ell, \psi). $$
 Rearranging this inequality yields that for all nonnegative $\psi \in D(E)_N$ we have
 $$E(P^N(\varphi h - G^NR^N \ell),\psi) = E( \varphi h - G^NR^N \ell,\psi) \geq 0.$$
  By definition, this means that the function $P^N(\varphi h - G^NR^N \ell)$ is $E^N$-superharmonic. According to Proposition~\ref{lemma:superharmonic functions are positive}, the transience of $E^N$ implies that it is nonnegative. This observation and Lemma~\ref{lemma: monotonicity pn} yield
 $$\varphi  h \geq P^N( \varphi h) \geq P^N G^NR^N \ell = G^NR^N \ell. $$
 Since $\varphi$ satisfies $1_N \leq \varphi \leq 1$, we obtain
 $$G^NQ^N \ell \leq h.$$
 This concludes the proof of the implication (i)$\Rightarrow$(ii) and shows assertion~(b).
 
 (ii) $\Rightarrow$ (i): Let $\cN$ be a special $E$-nest such that $\ell = |\ell| \in D(G^{\cN})$. We intend to apply Theorem~\ref{theorem:computing the weak form for excessive functions}  to the functions $(G^NR^N \ell)$ and $\Ge \ell$ to show assertion~(a), which implies (i). The following claim yields that this is possible.
 
 {\em Claim:}  $(G^NR^N \ell)$ and $\Ge \ell$ satisfy the assumptions of  Theorem~\ref{theorem:computing the weak form for excessive functions}, i.e., the following assertions hold true.
 \begin{itemize}
  \item  $\Ge \ell$ is $E$-excessive.
  \item For each $N \in \cN$ the function $G^N R^N \ell$ is $E^N$-superharmonic.
  \item For each $N \in \cN$ the function $G\ell - G^NR^N \ell$ is $E$-excessive. 
 \end{itemize}

 {\em Proof of the claim.} Let $N\in \cN$. The $E^N$-superharmonicity of $G^NR^N \ell$ follows from its definition. For $M \in \cN$ with $m(N\setminus M) = 0$, we set $g_{N,M} := G^MR^M \ell - G^N R^N \ell$. For all nonnegative $\psi \in D(E)_M$, we obtain 
 $$E(g_{N,M},\psi) = \ell(\psi) - E(G^N R^N \ell,P^N\psi) = \ell(\psi - P_N\psi) \geq 0, $$
 where we used Lemma~\ref{lemma: monotonicity pn} for the last inequality. By definition, this means that $g_{N,M}$ is $E^M$-superharmonic. For $\psi \in D(E)_\cN$,
 this observation combined with Lemma~\ref{lemma:superharmonic functions as cutoff} yields
 $$E(\psi \wedge g_{N,M}) \leq E(\psi)$$
 whenever $M$ is large enough and we obtain 
 $$E(\psi \wedge (\Ge\ell - G^NQ^N \ell)) \leq \liminf_{M \in \cN}  E(\psi \wedge g_{N,M})  \leq E(\psi).$$
 Since $D(E)_{\cN}$ is dense in $D(E)$, the previous inequality extends to all $\psi \in D(E)$. From Theorem~\ref{theorem:characterization of excessive functions} we infer that $\Ge\ell - G^N R^N \ell$ is excessive. Letting $N = \emptyset$ in the above computations shows that $\Ge\ell$ is excessive as well. \qedc
 
 The previous claim shows that the assumptions of Theorem~\ref{theorem:computing the weak form for excessive functions} are satisfied. It yields that  for all $\psi \in D(E)_{\cN_{G\ell}}$ we have
 $$\Ew_{\cN_{\Ge \ell}}(\Ge \ell, \psi) = \lim_{N \in \cN} E(G^N R^N \ell, \psi) = \ell(\psi).$$
This proves implication (ii) $\Rightarrow$ (i) and assertion (a). 
\end{proof}

 An immediate consequence of this theorem is the following characterization of the domain of the extended potential operator.

\begin{corollary} \label{corollary:potential operator and weak solutions}
  Let $E$ be a transient energy form of full support and let $\ell:D(\ell) \to \IR$ be a regular functional on $L^0(m)$. The following assertions are equivalent.
  \begin{itemize}
   \item[(i)] $\ell \in D(\Ge)$.
   \item[(ii)] There exists a nonnegative weak solution to the equation $\Ew(g,\cdot) = |\ell|$.
   \item[(iii)] There exist nonnegative weak solutions to the equations $\Ew(g,\cdot) = \ell_+$ and\\ $\Ew(g,\cdot) = \ell_-$.
  \end{itemize}
  If one of the above is satisfied, then $\Ge \ell$ is a weak solution to the equation $\Ew(g,\cdot) = \ell$. More precisely, for any special $E$-nest $\cN$ we have $\Ge \ell \in \cN_{\Ge |\ell|}$ and   
  $$\Ew_{\cN_{\Ge |\ell|}}(\Ge \ell,\psi) = \ell(\psi) \text{  for all } \psi \in D(E)_{\cN_{\Ge |\ell|}}.$$
 \end{corollary}
 
 \begin{proof}
  The implications (i) $\Rightarrow$ (ii) and (ii) $\Rightarrow$ (iii) follow from Theorem~\ref{theorem:maximal principle potential operator}. 
  
  (iii) $\Rightarrow$ (i):  Theorem~\ref{theorem:maximal principle potential operator} shows $\ell_+,\ell_- \in D(\Ge)$ and that $\Ge \ell_+, \Ge \ell_-$ are $E$-excessive.  The function $h:= \Ge \ell_+ + \Ge \ell_-$ is $E$-excessive as-well. According to Proposition~\ref{proposition:nh vs nh'}, the special nests $\cN_h$, $(\cN_{h})_{\Ge \ell_+}$ and $(\cN_{h})_{\Ge \ell_-}$ are equivalent. Therefore, Lemma~\ref{lemma:independence of nests} and Theorem~\ref{theorem:maximal principle potential operator} yield that for all $\psi \in D(E)_{\cN_h}$ we have 
  $$\Ew_{\cN_h}(h,\psi) = |\ell|(\psi) \text{ and } \Ew_{\cN_{h}} (\Ge \ell_+ - \Ge \ell_-,\psi) = \ell(\psi).$$
  Hence, the function $\Ge \ell_+ - \Ge \ell_-$ is a weak solution to the equation  $\Ew(g,\cdot) = \ell$. Since $h$ is a weak solution to the equation $\Ew(g,\cdot) = |\ell|$,  Theorem~\ref{theorem:maximal principle potential operator} implies $\ell \in D(\Ge)$. This finishes the proof.
 \end{proof}

\begin{remark}
 When $\ell$ is a positive functional, the theorem shows that $\Ge \ell$ is the minimal weak solution to the Laplace equation $\Ew(g,\cdot) = \ell$. As discussed at the beginning of this section, the operator $G$ can be seen as a generalization of the resolvent of a Dirichlet form, and so $\Ge$ is an extension of the resolvent as well. For some classes of Dirichlet forms it is well known that extended resolvents provide minimal weak solutions to the Laplace equation (in the 'classical' sense). For example, \cite[Theorem~8.4]{Gri} treats the case of $L^2_{\rm loc}$-solutions on manifolds and  \cite[Theorem~11]{KL} discusses solutions on  graphs. In these examples one has a 'canonical' notion of weak solutions in terms of a distributional operator, cf. Example~\ref{example:graphs} and Example~\ref{example:manifolds}.  The novelty of our result lies in the fact that it holds true for all energy forms with our universal notion of weak solutions. The situation for resolvents of Dirichlet forms is discussed in more detail in Subsection~\ref{section:lp resolvents}.
\end{remark}

For theoretical purposes the definition of $\Ge$ is the 'correct' one. However, for explicit computations the following might be a bit more useful. 
 \begin{theorem} \label{theorem:computation of resolvent for positive functionals} 
  Let $E$ be a transient energy form of full support and let $(\ell_i)$ be a monotone increasing net in $L(D(E))_+$. We set  
  $$D(\ell):= \{\psi \in D(E) \mid \lim_i \ell_i (\psi) \text{ exists}\} \text{ and }\ell: D(\ell) \to \IR,\, \psi \mapsto \lim_i \ell_i (\psi). $$
  The following assertions are equivalent.
  \begin{itemize}
   \item[(i)] $\ell \in D(\Ge)$.
   \item[(ii)] There exists $f \in L^0(m)$ such that $G \ell_i \leq f$ for all $i$.
  \end{itemize}
 In particular, if one of the above is satisfied, then
  $$\Ge \ell = \lim_i G \ell_i = \sup_i G\ell_i,$$
  where the convergence holds in $L^0(m)$. 
 \end{theorem}
 \begin{proof}
 (i) $\Rightarrow$ (ii): Let $\cN$ be a special nest such that $D(E)_\cN \subseteq D(\ell)$ and $R^\cN \ell  \in D(G^\cN)$.  We first prove that for $N\in \cN$  the identities
  $$G^N R^N \ell = \lim_i G^N R^N \ell_i =\sup_i G^N R^N \ell_i $$
  hold in $L^0(m)$. Once we show the fist one, the second one is a consequence of the monotonicity of the $\ell_i$ and the positivity of $G^NR^N$, see Proposition~\ref{proposition:potential operator is positive}. Using the positivity of $\ell - \ell_i$ and that $G^N R^N \ell_i$ is nonnegative, we compute
  $$E(G^N R^N \ell - G^N R^N \ell_i) = (\ell - \ell_i)(G^N R^N \ell - G^N R^N \ell_i) \leq (\ell - \ell_i)(G^N R^N \ell) \to 0. $$
  Since $E$ is transient, this implies $G^N R^N \ell_i \to G^N R^N \ell$ in $L^0(m)$. 
  
  These considerations combined with Lemma~\ref{lemma:convergence of projections} and Proposition~\ref{proposition:characterization of domain for positive functionals} yield
  $$\Ge \ell = \sup_{N \in \cN} G^N R^N \ell = \sup_{N \in \cN} \sup_i G^N R^N \ell_i = \sup_i \sup_{N \in \cN} G^N R^N \ell_i = \sup_i G\ell_i. $$
  Here, we used properties of suprema in the Dedekind complete space $L^0(m)$. We obtain (i) $\Rightarrow$ (ii) and the the 'in particular' statement. 
  
  (ii) $\Rightarrow$ (i): By our assumption, $h:= \sup_i G\ell_i = \lim_i G\ell_i$ exists in $L^0(m)$. Since the $G\ell_i$ are superharmonic, see Proposition~\ref{proposition:potential operator is positive}, $h$ is excessive.  For each $i$ we have $h - G \ell_i = \lim_{j} G\ell_{j} - G \ell_i$. Moreover, the monotonicity of $(\ell_i)$ implies that for $j \succ i$ the function $G\ell_j - G \ell_i$ is superharmonic. Consequently, $h - G\ell_i$ is excessive. We choose a special $E$-nest $\cN$ and apply Theorem~\ref{theorem:computing the weak form for excessive functions} to obtain
  $$\Ew_{\cN_h}(h,\psi) = \lim_i E(G\ell_i,\psi) = \lim_i \ell_i(\psi) \text{ for each } \psi \in D(E)_{\cN_h}. $$
  In particular, $D(E)_{\cN_h} \subseteq D(\ell)$ and $h$ is a nonnegative (super-)solution to the equation $\Ew(g,\cdot) =  \ell$. Since the restriction of $\ell$ to $D(E)_{\cN_h}$ is positive, Theorem~\ref{theorem:maximal principle potential operator} yields $\ell \in D(\Ge)$. This finishes the proof. 
 \end{proof}
 \begin{remark}
 Let $\E$ be a Dirichlet form. At the beginning of this section we discussed that the $\E_\alpha$-potential of an $L^2$-function $f$, when interpreted as a functional on $D(\E_\alpha)$, is given by the $L^2$-resolvent $G_\alpha f$. In Dirichlet form theory resolvents are extended to larger classes of functions, for example the $L^p$-spaces. For a nonnegative measurable function $g$ this extension is usually introduced by letting 
 $$G_\alpha g = \sup \{G_\alpha f \mid f \in L^2(m) \text{ with } 0 \leq f \leq g\},$$
 if the supremum exists, see e.g. \cite[Section~6]{KL},  or one uses variants thereof with specific sequences that converge monotone towards $f$, see e.g. \cite[Chapter~1]{CF}.    The previous theorem shows that this way of extending the resolvent is compatible with the extension of the potential operator that was introduced in this section. We provide more details in Section~\ref{section:lp resolvents}.
\end{remark}

\section{Applications}

 \subsection{Uniqueness of bounded weak solutions to the Laplace equation} \label{subsection:uniqueness of bounded weak solutions}
 
  In this subsection we prove that uniqueness of bounded weak solutions to the Laplace equation is equivalent to a conservation property of the extended potential operator. Let $E$ be a transient energy form. The killing part of $E$ gives rise to the {\em killing functional} $\kappa$ with domain
  $$D(\kappa) := \{\psi \in D(E) \mid \text{there  exists } \varphi \in D(E) \text{ with } 1_{\{|\psi| > 0\}} \leq \varphi \leq 1\},  $$
  on which it acts by
  $$\kappa(\psi) := \Ekm(\varphi,\psi).$$
  Lemma~\ref{lemma:characterization of k} shows that $\kappa$ is well-defined and positive.  If $\cN$ is a special $E$-nest, then $D(E)_\cN \subseteq D(\kappa)$ and $1 \in D(E)_{w,\, \cN}$. For $\psi \in D(E)_\cN$ we obtain
  $$\kappa(\psi) = \Ekm(1,\psi) = \Ew_\cN(1,\psi),$$
 where  $\Ekm(1,\psi)$ is understood in the sense of Lemma~\ref{lemma:extension of k}. Hence, $1$ is a weak solution to the equation $\Ew(g,\cdot) = \kappa$ (with respect to any special nest). If, additionally, the form $E$ has full support, then Theorem~\ref{theorem:maximal principle potential operator} shows $\kappa \in D(\Ge)$ and $\Ge \kappa \leq 1$.  With this observation we can formulate the following characterization of uniqueness of bounded weak solutions to the Laplace equation.  
 
\begin{theorem}[Uniqueness of bounded weak solutions]\label{theorem:uniqueness of bounded weak solutions}
 Let $E$ be a transient energy form of full support. The following assertions are equivalent.
 \begin{itemize}
  \item[(i)] $1 = \Ge \kappa$.
  \item[(ii)] Any nonnegative essentially bounded weakly $E$-subharmonic function equals zero.
  \item[(iii)] Any essentially bounded weakly $E$-harmonic function equals zero. 
  \item[(iv)] For all linear functionals $\ell \in D(\Ge)$ with $\Ge |\ell| \in L^\infty(m)$, the equation $\Ew(g,\cdot) = \ell$ has a unique essentially bounded weak solution. 
 \end{itemize}
\end{theorem}

\begin{remark}
  When the potential operator is given by the resolvent of a Dirichlet form, the condition $\Ge \kappa = 1$ is a generalization of the concept of stochastic completeness. This is discussed in  detail in Subsection~\ref{section:lp resolvents}.
\end{remark}

The main ingredient for the proof of the previous theorem is the following lemma. 

\begin{lemma} \label{lemma:killing potential}
 Let $E$ be a transient energy form of full support. The function 
 $$w = 1 - \Ge \kappa$$
 is weakly $E$-harmonic and satisfies $0 \leq w \leq 1$. If $f$ is a weakly $E$-subharmonic function with $0 \leq f \leq 1$, then $f \leq w$. 
\end{lemma}
\begin{proof}
 The inequality $1 - \Ge \kappa \geq 0$ was discussed at the beginning of this subsection and $1 - \Ge \kappa \leq 1$ follows from the positivity of $\kappa$ and the positivity of the extended potential operator. If $f$ is a weakly $E$-subharmonic function with $0 \leq f \leq 1$, then the function $1-f$ is a nonnegative weak supersolution to the equation $\Ew(g,\cdot) = \kappa$. Theorem~\ref{theorem:maximal principle potential operator} implies $1-f \geq \Ge \kappa$. Furthermore, Theorem~\ref{theorem:maximal principle potential operator} shows that $\Ge \kappa$ is a weak solution to the equation $\Ew(g)  = \kappa$, and so $1 - \Ge \kappa$ is weakly $E$-harmonic. This finishes the proof.
\end{proof}

\begin{proof}[Proof of Theorem~\ref{theorem:uniqueness of bounded weak solutions}]
 (i) $\Rightarrow$ (ii): Let $f$ be a nonnegative essentially bounded  weakly \mbox{$E$-subharmonic} function. Without loss of generality, we assume that $f \leq 1$. The previous lemma implies  $0 \leq f \leq 1- \Ge \kappa \ = 0$ and we arrive at (ii).
 
 (ii) $\Rightarrow$ (i): According to the previous lemma, the function $1 - \Ge \kappa$ is a nonnegative essentially bounded weakly $E$-subharmonic function. By (ii) it vanishes and we obtain (i). 
 
 (i) $\Rightarrow$ (iii): Let $f$ be an essentially bounded weakly $E$-harmonic function. Without loss of generality, we assume $\|f\|_\infty \leq 1$. In this case, both $1 + f$ and $1-f$ are nonnegative solutions to the equation $\Ew(g,\cdot) = \kappa$. Hence, Theorem~\ref{theorem:maximal principle potential operator} implies
 $$1- f \geq \Ge \kappa = 1 \text { and } 1 + f \geq \Ge \kappa = 1.$$
 This shows $f = 0$ and proves (iii).
 
 (iii) $\Rightarrow$ (iv): According to Corollary~\ref{corollary:potential operator and weak solutions},  the function $\Ge \ell$ is a weak solution to the equation $\Ew(g,\cdot) = \ell$. By the monotonicity of $\Ge$ the condition $\Ge |\ell| \in L^\infty(m)$ implies that $\Ge \ell$ is essentially bounded. It remains to prove uniqueness. To this end, let $f \in L^\infty(m)$  be weak solution to the equation $\Ew(g,\cdot)  = \ell$. By definition, there exists a special $E$-nest $\cN$ such that $f \in D(E)_{w,\, \cN}$ and for all $\psi \in D(E)_{\cN}$ the equation
 $$\Ew_\cN (f,\psi) = \ell(\psi)$$
 holds. The assumption $\Ge |\ell| \in L^\infty(m)$ implies $\Ge \ell_+, \Ge \ell_- \in L^\infty(m)$. According to Proposition~\ref{proposition:nh vs nh'}, the nests $\cN$, $\cN_{\Ge \ell_+}$ and $\cN_{\Ge \ell_-}$ are equivalent. With this observation  Theorem~\ref{theorem:maximal principle potential operator} yields  $\Ge \ell_+, \Ge \ell_-  \in D(E)_{w,\, \cN}$  and that for all $\psi \in D(E)_\cN$ the following identity holds:
 $$\Ew_\cN(f - \Ge \ell,\psi)  = \Ew_\cN(f,\psi) - \Ew_\cN(\Ge \ell,\psi) = \ell(\psi) - \ell(\psi) = 0.$$
 In other words,  $f - \Ge \ell$ is a bounded  weakly $E$-harmonic function. By (iii) it must vanish and we arrive at (iv). 
 
 (iv) $\Rightarrow$ (i): According to the previous lemma, the function $1 - \Ge \kappa$ is an essentially bounded weakly $E$-harmonic function. Since $0$ is another essentially bounded weakly $E$-harmonic function, (iv) implies $1 - \Ge \kappa = 0.$ This proves (i). 
\end{proof}

 \begin{corollary} \label{corollary:conservation implies e=er}
  Let $E$ be a transient energy form of full support. If $\Ge \kappa = 1$, then $E = \Er$.
 \end{corollary}
 \begin{proof}
  Assume that $E \neq \Er$. We show the existence of an essentially bounded weakly $E$-harmonic function, which contradicts $\Ge \kappa = 1$ by the previous theorem.
  
  Since essentially bounded functions are dense in the domains of energy forms, there exists $h \in D(\Er) \setminus D(E)$ with $\|h\|_\infty \leq 1$. The equality $\Ge \kappa = 1$ implies  $\Ekm \neq 0$ so that $\Er$ is transient. We let $h_0$ be the Hilbert space projection of $h$ onto $D(E)$, i.e., the unique element in $D(E)$ with 
  $$\Er(h - h_0) = \inf \{\Er(h-\psi) \mid \psi \in D(E)\}.$$

  The function $h_r := h-h_0$ satisfies $\Er(h_r,\psi) = 0$ for all $\psi \in D(E)$.  As seen in Remark~\ref{remark:weak extension of er}, this equation implies that $h_r$ is weakly $E$-harmonic. Therefore, it remains to prove its boundedness. We choose a special $E$-nest $\cN$ and a net $(h_i)$ in $D(E)_\cN$ that converges to $h_0$ in the form topology of $E$ and we let 
  $$\psi_i := h - \left((h-h_i) \wedge 1\right) \vee (-1).$$
  We prove $\psi_i \in D(E)$ and the convergence $\psi_i \to h_0$ in the form topology, from which the inequality $-1 \leq h_r \leq 1$ follows. For each $i$, there exists a set $N_i \in \cN$ such that $h_i$ vanishes outside of $N_i$. Since $\|h\|_\infty \leq 1$, the function $\psi_i$ also vanishes outside of $N_i$. In other words, we have $\psi_i \in D(\Er)_\cN$. Furthermore, $\cN$ is special and so there exists $\varphi_i \in D(E)_\cN$ with $1_{N_i} \leq \varphi$. We obtain 
  $$|\psi_i| \leq \|\psi_i\|_\infty \varphi_i.$$
  Since $\Er$ is a Silverstein extension of $E$, the domain $D(E)$ is an order ideal in $D(\Er)$. We infer $\psi_i \in D(E)$. The contraction properties of $E$ and the definition of $h_0$ and $\psi_i$ yield
  $$E(h-h_i) \geq E(h-\psi_i) = E(h_0 - \psi_i) + E(h - h_0).$$
  This shows the $E$-convergence of $\psi_i$ to $h_0$ and the transience of $E$ implies the convergence of $\psi_i$ to $h_0$ in $L^0(m)$. This finishes the proof.
 \end{proof}

 \begin{remark}
  \begin{itemize}
   \item We shall see in Subsection~\ref{section:lp resolvents} that the corollary is an extension of the well known fact that stochastically complete Dirichlet forms are Silverstein unique.
   \item  The main difficulty in the proof of the previous corollary was to show that the Hilbert space projection $P:D(\Er) \to D(E)$ maps $D(\Er)\cap L^\infty(m)$ to $D(E)\cap L^\infty(m)$. Indeed, we proved that for all $f \in D(\Er) \cap L^\infty(m)$ the inequality
   $$\|(I-P) f \|_\infty \leq \|f\|_\infty$$
   holds. This inequality is true for any transient Silverstein extension of $E$ with the corresponding projection. Our proof is inspired by the proof of the Royden decomposition theorem for jump-forms on discrete spaces, see e.g. \cite[Proposition~5.1]{KLSW}.
  \end{itemize}
 \end{remark}

 \subsection{$L^p$-Resolvents of Dirichlet forms and weak solutions} \label{section:lp resolvents}

In this subsection we apply the theory of weak solutions to resolvents of  Dirichlet forms. We show that the extended potential operator yields an extension of the resolvents to all $L^p$-spaces and that it provides minimal $L^p$-solutions to the weak Laplace equation of the perturbed form.  At the end of this section, we elaborate that the conservation property that was discussed in the previous section is an extension of the concept of stochastic completeness.

Let $\E$ be a Dirichlet form. Recall that for $\alpha> 0$ we define the transient energy form $\E_\alpha$ by
$$\E_{\alpha}(f):= \begin{cases}  \E(f) + \alpha \int_X f^2\, \D m &\text{ if } f \in L^2(m) \\ \infty &\text{ else} \end{cases}.$$
Its domain satisfies $D(\E_\alpha) = D(\E) = D(\Ee) \cap L^2(m)$, see Proposition~\ref{propostiont:domain extended dirichlet form}. We abuse notation and denote its potential operator by $G_\alpha$ (which we also used for the $L^2$-resolvent of $\E$). Any function $f \in L^2(m)$ induces a linear functional  $\ell_f\in L_c(D(\E_\alpha))$ via
 $$\ell_f: D(\E_{\alpha}) \to \IR,\, g \mapsto \int_U fg\, {\rm d} m $$
and we write $G_{\alpha}f$ instead of $G_{\alpha} \ell_f$.  In this sense, $G_\alpha$ is an extension of the $L^2$-resolvent of $\E$ to $L_c(D(\E_\alpha))$ and our abuse of notation is justified. In the subsequent discussion we need the following lemma.

\begin{lemma}\label{lemma:main part and killing of dirichlet form}
 Let $\E$ be a Dirichlet form and let $\alpha > 0$.
 \begin{itemize}
  \item[(a)] If $\cN$ is a special $\E_\alpha$-nest, then   for all $N \in \cN$ we have $m(N) < \infty$, and the inclusion $D(\E_\alpha)_\cN \subseteq L^1(m)$ holds.
  \item[(b)] The main part of the energy form $\E_\alpha$ is given by $\Ee^{(M)}$ and the domain of its killing part satisfies   $D((\E_\alpha)^{(k)}) = D(\Ee^{(k)}) \cap L^2(m)$, on which it acts by
 $$f \mapsto (\E_\alpha)^{(k)}(f) = \Ee^{(k)}(f) + \alpha \int_X f^2 \D m.$$
 In particular, the reflected energy form of $\E_\alpha$ satisfies $(\E^{\rm ref})_\alpha = (\E_\alpha)^{\rm ref}.$
 \end{itemize}
\end{lemma}
 \begin{proof}
  (a): Let $\cN$ be a special $\E_\alpha$-nest and let $N\in \cN$. Since $\cN$ is special, there exists $\varphi \in D(\E_{\alpha})_\cN$ with $\varphi  \geq 1_N$. This implies 
  $$\alpha m(N) \leq \alpha \int_X \varphi^2 \D m \leq  \E_{\alpha} (\varphi) < \infty.$$
   Any function in $L^2(m)$ that vanishes outside a set of finite measure belongs to $L^1(m)$. Therefore, $D(\E_\alpha)_\cN \subseteq L^1(m).$
  
  (b): {\em Claim 1:} Any special $\E_\alpha$-nest is a special $\Ee$-nest.  
  
  {\em Proof of  Claim~1.} Let $\cN$ be a special $\E_\alpha$-nest. The form $\Ee$ is the $L^0(m)$-closure of $\E$ and so $D(\E_\alpha) = D(\E)$ is dense in $D(\Ee)$ with respect to the $\Ee$-form topology. Furthermore, $D(\E_\alpha)_\cN$  is dense in $D(\E_\alpha)$ with respect to the $\E_\alpha$-form topology and, on $D(\E_\alpha)$, the $\Ee$-form topology is weaker than the $\E_\alpha$-form topology. Therefore, $\cN$ is an $\Ee$-nest. That it is special with respect to $\Ee$ follows from  the inclusion $D(\E_\alpha)_\cN \subseteq D(\Ee)_\cN$. \qedc

   Let $\cN$ be a special $\E_\alpha$-nest, which is also a special $\Ee$-nest by our claim. Since the sets in $\cN$ have finite measure, we obtain
  \begin{align*}J:= \{\varphi \in D(\Ee)_\cN \mid 0\leq \varphi \leq 1\} &= \{\varphi \in D(\Ee)_\cN \cap L^2(m) \mid 0\leq \varphi \leq 1\}\\
    &= \{\varphi \in D(\E_\alpha)_\cN   \mid 0\leq \varphi \leq 1\}.
  \end{align*}
   Hence, Corollary~\ref{corollary:alternative formula for eg} implies that for all $f \in L^0(m)$, we have 
   $$\Ee^{(M)}(f) = \sup_{\varphi \in J} \E_{{\rm e},\varphi}(f)$$
   and 
   $$(\E_\alpha)^{(M)}(f) = \sup_{\varphi \in J} \E_{\alpha, \varphi}(f).$$

   {\em Claim~2:}  For all $\varphi \in J$ we have  $\E_{\alpha,\varphi} = \E_{{\rm e},\varphi}$.
   
   {\em Proof of Claim~2.} Let $\varphi \in J$. The definition of $\E_{\alpha,\varphi}$ and $\E_{{\rm e},\varphi}$,  and Lemma~\ref{lemma:closability of ep} yield
   \begin{align*}
    D(\E_{\alpha,\varphi}) \cap L^\infty(m) &= \{ f \in L^\infty(m) \mid \varphi f \in D(\E_\alpha)\} \\
   &= \{ f \in L^\infty(m) \mid \varphi f \in D(\Ee)\cap L^2(m)\}\\
   &= \{ f \in L^\infty(m) \mid \varphi f \in D(\Ee) \}\\
   &= D(\E_{{\rm e},\varphi}) \cap L^\infty(m).
   \end{align*}
  Here, we used that  $\varphi$ vanishes outside a set of finite measure. With this at hand, the equality of $\E_{\alpha,\varphi}$ and $\E_{{\rm e},\varphi}$ on  $D(\E_{\alpha,\varphi}) \cap L^\infty(m) = D(\E_{{\rm e},\varphi}) \cap L^\infty(m)$ follows from the definition of $\E_{\alpha,\varphi}$ and $\E_{{\rm e},\varphi}$. Since bounded functions are dense in the domain of an energy form, this finishes the proof of Claim~2. \qedc.
   
   Claim~2 and the previous considerations imply $\Ee^{(M)} = (\E_\alpha)^{(M)}$. The statement on $(\E_\alpha)^{(k)}$ is straightforward from this equality and the definition of the killing part. The 'in particular' part follows from the definition of the reflected  Dirichlet form and the reflected energy form, cf. Definition~\ref{definition:reflected form} and Definition~\ref{definition:reflected dirichlet form}.
 \end{proof} 

 \begin{remark}
   The definition of the main part of $\E_\alpha$ is ambiguous as $\E_\alpha$ can be considered as a Dirichlet form and as an energy form. This is why we formulated the previous lemma as it is. It shows that the main part of $\E_\alpha$ when considered as an energy form is given by $\Ee^{(M)}$, while the main part of $\E_\alpha$ when considered as a Dirichlet form is $\E^{(M)}$, the restriction of  $\Ee^{(M)}$ to $L^2(m)$, cf.   Definition~\ref{definition:reflected dirichlet form}. For the reflected forms no such ambiguity arises. 
 \end{remark}

We say that a Dirichlet form $\E$ has full support if $\E_1$ has full support. This is obviously equivalent to requiring that for some $\alpha>0$ the form $\E_\alpha$  has full support, or to requiring that the extended form $\Ee$ has full support. A functional $\ell \in L_c(D(\E_\alpha))$ is said to satisfy $\ell \leq 1$ if for each  nonnegative $\psi \in D(\E_\alpha) \cap L^1(m)$ the inequality  $\ell(\psi) \leq \int_X \psi \D m$ holds.
\begin{proposition} \label{proposition:properties resolvent}
  Let $\E$ be a Dirichlet form of full support.  
 \begin{itemize}
  \item[(a)] For all $\alpha > 0$ the operator  $\alpha G_{\alpha}$ is Markovian, i.e., for any functional $\ell \in L_c(D(\E_\alpha))$ the inequality $0 \leq \ell \leq 1$ implies $ 0\leq \alpha G_\alpha \ell \leq 1$. 
  \item[(b)] The family $(G_\alpha)_{\alpha> 0}$ satisfies the resolvent equation, i.e., for all $\ell \in L_c(D(\E_1))$ and all $\alpha,\beta > 0$ the following identity holds:
  $$G_\alpha \ell - G_\beta \ell = (\beta - \alpha) G_\alpha G_\beta \ell.$$
  \item[(c)] The restriction of $G_{\alpha}$ to $L^2(m)$ is self-adjoint and for all $f\in L^2(m)$ it satisfies 
  $$\|\alpha G_{\alpha} f\|_2 \leq \|f\|_2.$$
 \end{itemize}
\end{proposition}
\begin{proof}
(c): Since the $L_2$-restriction of $G_\alpha$ coincides with the 'usual' $L_2$-resolvent, this follows from standard Dirichlet form theory, see e.g. \cite{FOT}

(a): Let $\ell \in L_c(D(\E_\alpha))$ with $0 \leq \ell \leq 1$. The nonnegativity of $\alpha G_\alpha$ follows from the corresponding statements for the potential operator, see Proposition~\ref{proposition:potential operator is positive}. We  use Theorem~\ref{theorem:maximal principle potential operator} to prove the inequality $ \alpha G_\alpha \ell \leq 1$. Let $\cN$ be a special $\E_\alpha$-nest. Then $1 \in D(E_\alpha)_{w,\, \cN}$.  The statement of Lemma~\ref{lemma:main part and killing of dirichlet form} on the killing part of $\E_\alpha$ shows that for all nonnegative $\psi\in D(E)_\cN$ we have
$$\mathcal{E}_{\alpha,\, \cN}^{(w)}(1,\psi) = \left(\mathcal{E}_{\alpha}\right)^{(k)}(1,\psi) \geq \alpha \int_X \psi \D m \geq \alpha \ell(\psi),$$
where we used $\ell  \leq 1$ for the last inequality. According to Theorem~\ref{theorem:maximal principle potential operator}, this implies $1 \geq \Ge_\alpha (\alpha \ell) = \alpha G_\alpha \ell$.

(b): For $\psi \in D(\E)$ and $\ell \in L_c(D(\E_1))$ we compute
\begin{align*}
 (\beta - \alpha) \E_\alpha (G_\alpha G_\beta \ell,\psi) &= (\beta - \alpha) \as{G_\beta \ell,\psi} \\
 &= \E_\beta(G_\beta \ell,\psi) - \E_\alpha(G_\beta \ell,\psi) \\
 &= \ell(\psi) - \E_\alpha(G_\beta \ell,\psi)\\
 &= \E_\alpha(G_\alpha \ell -  G_\beta \ell,\psi).
\end{align*}
This shows the desired equality and finishes the proof.
\end{proof}

\begin{remark}
 The previous proposition is certainly well known but we could not find a proper reference for assertions (a) and (b) in this generality.  
\end{remark}

 We now extend the resolvents to all $L^p$-spaces with $1 \leq p \leq \infty$. For $f \in L^0(m)$ we define $D(\ell_f) := \{g \in L^0(m) \mid \int_X |fg| \D m  <\infty\}$ and let
 $$\ell_f:D(\ell_f) \to \IR,\, g \mapsto \int_{X} fg \D m.$$ 
 As in the $L^2$-case, we identify $f \in L^0(m)$ with the linear functional $\ell_f$.  In particular, we say $f \in D(\Ge_\alpha)$ if $\ell_f \in D(\Ge_\alpha)$ and write  $\Ge_{\alpha}f$ instead of $\Ge_\alpha \ell_f$. Our main theorem on $L^p$-resolvents of Dirichlet forms is the following.

 \begin{theorem}\label{theorem:lp resolvents}
  Let $\E$ be a Dirichlet form of full support and let $\alpha > 0$. 
  \begin{itemize}
   \item[(a)] For all $1 \leq p \leq \infty$ we have $L^p(m) \subseteq D(\Ge_{\alpha})$. 
   
   \item[(b)]  The restriction of $\Ge_{\alpha}$ to $L^p(m)$ is a linear operator and for all $f\in L^p(m)$ it satisfies
  $$\|\alpha \Ge_{\alpha} f\|_p \leq \|f\|_p.$$
  \item[(c)] For all $f \in L^p(m)$ we have $\Ge_{\alpha} f \in D(E)_{w}$. More precisely, the following is true.
  \begin{itemize}
   \item[(c1)] If $1\leq p < \infty$, then for all $f \in L^p(m)$ there exists a special $\E_{\alpha}$-nest $\cN$ such that $\Ge_{\alpha}f \in D(\E_\alpha)_{w,\, \cN}$ and  for all  $\psi \in D(\E_{\alpha})_\cN$ it satisfies
   $$\E^{(w)}_{\alpha,\, \cN}(\Ge_{\alpha}f,\psi) = \int_X f\psi \D m.$$
   \item[(c2)] For all $f \in L^\infty(m)$ and all special $\E_{\alpha}$-nests $\cN$ we have $\Ge_{\alpha}f \in D(\E_{\alpha})_{w,\, \cN}$ and  for all $ \psi \in D(\E_{\alpha})_\cN$ it satisfies
   $$\E^{(w)}_{\alpha,\, \cN}(\Ge_{\alpha}f,\psi) = \int_X f\psi \D m. $$
   \item[(d)] For all $1 \leq p \leq \infty$ and all nonnegative $f \in L^p(m)$ we have
   $$\Ge_\alpha f = \sup \{G_\alpha g \mid g \in L^2(m) \text{ and } 0 \leq g \leq f\}.$$
  \end{itemize}
  \end{itemize}
 \end{theorem}
 \begin{proof} Once we prove assertions (a) and (b) the assertions (c1) and (d) are consequences of  Corollary~\ref{corollary:potential operator and weak solutions}  and Theorem~\ref{theorem:computation of resolvent for positive functionals}, respectively. Assertion (c2) also follows from  Corollary~\ref{corollary:potential operator and weak solutions} with the additional observation that $\Ge_\alpha |f| \in L^\infty(m)$ yields that every special $\E_\alpha$-nest $\cN$ is equivalent to the nest $\cN_{\Ge_\alpha|f|}$, see Proposition~\ref{proposition:nh vs nh'}.      
 
  $p=2$: All the claims are a consequence of Proposition~\ref{proposition:properties resolvent} and the fact that $\Ge_\alpha$ is an extension of $G_\alpha$.

  $p = \infty$:  We start with the following claim.

 {\em Claim:} Let $\cN$ be a special $\E_{\alpha}$-nest. For all $f \in L^\infty(m)$ we have $D(\E_{\alpha})_\cN \subseteq D(\ell_f)$ and the restriction $R^\cN \ell_f$ is a regular functional on $D(\E_{\alpha})_\cN$.

  {\em Proof of the claim.}  Let $N\in \cN$. For any $g \in D(\E_{\alpha})_N$, we obtain 
  $$\int_{X}|fg|\D m = \int_{N} |fg| \D m \leq  \left(\int_{N} f^2 \D m\right)^{1/2} \E_{\alpha}(g)^{1/2} \leq m(N)^{1/2} \|f\|_\infty \E_{\alpha}(g)^{1/2}. $$
  According to Lemma~\ref{lemma:main part and killing of dirichlet form}, we have $m(N)< \infty$. Therefore, the above inequality shows $D(\E_{\alpha})_N \subseteq D(\ell_f)$.  The regularity of the functional $\ell_f$ is obvious. \qedc
  
  Let $\cN$ be a special $\E_{\alpha}$-nest and let $f \in L^\infty(m)$.  By the claim we have $D(\E_{\alpha})_\cN \subseteq D(\ell_f)$.  We prove $|R^\cN\ell_f| \in D(G_\alpha	^\cN)$. According to Proposition~\ref{proposition:characterization of domain for positive functionals}, this is equivalent to the boundedness of the net $(G_{\alpha}^NR^N |R^\cN \ell_f|)$. The Markov property of $\alpha G^N_{\alpha}$ implies
  $$G_{\alpha}^NR^N |R^\cN \ell_f| = G_{\alpha}^NR^N\ell_{|f|} = G_\alpha^N |f||_N \leq \alpha^{-1} \|f\|_\infty.$$
  We obtain $\ell_f \in D(G^\cN_{\alpha})$ and 
  $$\|\Ge_{\alpha} f\|_\infty \leq \|\Ge_{\alpha}|f|\|_\infty \leq \alpha^{-1}\|f\|_\infty.$$
  This shows (a) and (b) in the case when $p=\infty$.

 $p=1$:  For $f \in L^1(m) \cap L^2(m)$ and $n \in \IN$, we let $\psi_n := 1_{\{|f| > 1/n\}} \in L^2(m)$. The properties of the restriction of $G_\alpha$ to $L^2(m)$ yield  
  \begin{align*}
   \int_X | G_{\alpha} f| \D m &\leq \int_X G_{\alpha} |f| \D m\\
   &= \lim_{n \to \infty}  \int_X  G_{\alpha} |f| \psi_n \D m\\
   &= \lim_{n \to \infty}  \int_X |f| G_{\alpha}  \psi_n \D m\\
   &\leq \alpha^{-1} \int_X |f| \D m.
  \end{align*}
  Since $L^1(m) \cap L^2(m)$ is dense in $L^1(m)$, this computation shows that $G_{\alpha}$ can be uniquely extended to a linear operator $H^1_{\alpha}$ on $L^1(m)$. For all $f \in L^1(m)$ it satisfies
  $$\|\alpha H^1_\alpha f\|_1 \leq \|f\|_1.$$
  The  continuity of $H^1_\alpha$ implies that for any $f \in L^1(m)$ we have 
  $$H_\alpha^1 |f| = \lim_{n\to \infty} H^1_\alpha (|f| \psi_n)  =  \lim_{n\to \infty} G_{\alpha} (|f|\psi_n).$$
  Moreover, the positivity of $G_{\alpha}$ shows $H^1_\alpha |f| \geq  G_{\alpha} (|f|\psi_n)$. An application of Theorem~\ref{theorem:computation of resolvent for positive functionals} yields $|f| \in D(\Ge_{\alpha})$ and 
  $$\Ge_{\alpha} |f| = \lim_{n\to \infty} G_{\alpha} (|f|\psi_n) = H^1_\alpha |f|.$$
  This proves all the claims in the case  $p=1$. 
  
  $1 < p < \infty$: With assertion (b) for the $p=1$ and the $p=\infty$ case at hand, the Riesz-Thorin interpolation theorem for positive operators, see e.g. \cite[Section~4.2]{Haa},  shows that for all $f \in L^p(m) \cap L^2(m)$ the following inequality holds:
  $$\|\alpha G_{\alpha} f\|_p \leq \|f\|_p. $$
  In particular, $G_{\alpha}$ can be uniquely extended to a linear operator $H^p_\alpha$ on $L^p(m)$ that satisfies the inequality in assertion (b). An argumentation as in the $L^1$-case shows  $L^p(m)\subseteq D(\Ge_{\alpha})$ and that for all $f \in L^p(m)$ we have
  $$H^p_\alpha f = \Ge_{\alpha} f.$$
   Consequently, we obtain (a) and (b). This finishes the proof.
 \end{proof}

 \begin{corollary}
  Let $\E$ be a Dirichlet form of full support, let $\alpha> 0$ and let  $1\leq p \leq \infty$. For all $f \in L^p(m)$ the function $\Ge_{\alpha} f$ is a weak solution to the equation
  $$\begin{cases}
     \E^{(w)}_{\alpha}(g,\cdot) = \ell_f\\
     g \in L^p(m)
    \end{cases}. 
 $$
  Moreover, if $f \geq 0$ and $h\in L^p(m)$ is a nonnegative weak supersolution to the above equation, then $h \geq \Ge_{\alpha} f$.
 \end{corollary}
 \begin{proof}
  This follows from the previous theorem and Theorem~\ref{theorem:maximal principle potential operator}.
 \end{proof}
\begin{remark}
  \begin{itemize}
   \item As mentioned previously, it is well known that resolvents of Dirichlet forms can be continuously extended to all $L^p$-spaces. For $1 \leq p < \infty$, the space $L^2(m) \cap L^p(m)$ is dense in $L^p(m)$. Therefore, these continuous extensions are necessarily unique. In particular, our definition of the extension of the resolvent to $L^p(m)$ coincides with the classical one. For $p = \infty$, this coincidence can not be inferred directly from the continuity of the extension; it follows from assertion~(d) in the previous theorem, the compatibility of $\Ge_\alpha$ with monotone limits.
   
   \item The novelty of the previous theorem lies in the observation that for all $1\leq p  \leq \infty$ the $L^p$-resolvent provides weak solutions to the Laplace equation with respect to $\E_\alpha$. This is  known for manifolds, see \cite{Gri}, and graphs, see \cite{KL}, where one has a satisfactory notion of distributions, but even for regular Dirichlet forms we could not find such a result in the literature. The reason behind this gap  is that for $1 \leq p < \infty$ the image of the $L^p$-resolvent  need not  be contained in the local space with respect to the nest of relatively compact open sets. It is the flexibility in the chosen nest that allowed us to prove assertion (a) and (c1).  Due to (c2), the situation is a bit better for the resolvent on $L^\infty$; it is known for strongly local regular Dirichlet forms that the $L^\infty$-resolvent provides minimal weak solutions to the Laplace equation, see \cite{Stu1}. Nevertheless, for general regular Dirichlet forms also the $L^\infty$-statement seems to be new.  
  \end{itemize}
  \end{remark} 
In the previous subsection we characterized uniqueness of bounded weak solutions to the Laplace equation in terms of a conservation property for the extended potential operator. We now interpret this  property for Dirichlet forms. To this end, we first need to identify the killing functional of $\E_\alpha$, cf. Subsection~\ref{subsection:uniqueness of bounded weak solutions}. Recall that for $\alpha > 0$, $\ell_\alpha$ is the linear functional with domain $D(\ell_\alpha) = L^1(m)$, on which it acts by
$$\ell_\alpha(\psi) = \alpha \int_X \psi \D m.$$
\begin{lemma}
 Let $\E$ be a Dirichlet form of full support. For $\alpha > 0$,  let $\kappa_\alpha$ be the killing functional of $\E_\alpha$ and let $k$ be the killing functional of $\Ee$.  The domain of $\kappa_\alpha$ satisfies 
 $$D(\kappa_\alpha) \subseteq D(k) \cap  L^1(m) = D(k) \cap D(\ell_\alpha),$$
 on which it acts by
 $$\psi \mapsto \kappa_\alpha(\psi) =  k(\psi) + \ell_\alpha(\psi) = k(\psi) + \alpha \int_{X} \psi \D m.$$
 Moreover, we have $\kappa_\alpha,k, \ell_\alpha \in  D(\Ge_\alpha)$ and 
 $$\Ge_\alpha \kappa_\alpha = \alpha \Ge_\alpha 1 + \Ge_\alpha k \leq 1.$$
\end{lemma}
\begin{proof}
 It follows from Lemma~\ref{lemma:main part and killing of dirichlet form} that the domain of the killing part of $\E_\alpha$ satisfies 
 $$D((\E_\alpha)^{(k)})  = D(\Ee^{(k)})\cap L^2(m),$$
 on which it acts by
 $$(\E_\alpha)^{(k)}(f) = \Ee^{(k)}(f) + \alpha \int_X f^2 \D m.$$
 For $\psi \in  D(\kappa_\alpha)$ there exists $\varphi \in D(\E_\alpha) = D(\Ee) \cap L^2(m)$ with $1_{\{|\psi|>0\}} \leq \varphi \leq 1$. This implies $\psi \in D(k)$ and 
 $$\alpha \int_X |\psi| \D m  = \alpha \int_X |\psi \varphi| \D m \leq \E_\alpha(\varphi)^{1/2} \E_\alpha(\psi)^{1/2} < \infty.$$
 Furthermore, by the definition of $k$ and $\ell_\alpha$ we have
 \begin{align*}
  \kappa_\alpha(\psi) = (\E_\alpha)^{(k)}(\psi,\varphi)  = \Ee^{(k)}(\psi,\varphi) + \alpha \int_X \psi \varphi \D m =  k(\psi) + \alpha \int_X \psi  \D m = k(\psi) + \ell_\alpha(\psi).
 \end{align*}
 The discussion in Subsection~\ref{subsection:uniqueness of bounded weak solutions} shows $\kappa_\alpha \in D(\Ge_\alpha)$  and  $\Ge_\alpha \kappa_\alpha \leq 1$. Furthermore, what we have already proven yields   $\kappa_\alpha  \geq k, \ell_\alpha$ on $D(\kappa_\alpha)$. According to Proposition~\ref{proposition:characterization of domain for positive functionals} and the positivity of the potential operator, this implies $k, \ell_\alpha \in D(\Ge_\alpha)$. The equation
 $$\Ge_\alpha \kappa_\alpha = \alpha \Ge_\alpha 1 + \Ge_\alpha k$$
 follows from the linearity of the potential operator, cf. Remark~\ref{remark:extended potential operator}. This finishes the proof.  
\end{proof}

With the previous lemma at hand, Theorem~\ref{theorem:uniqueness of bounded weak solutions} reads as follows.
\begin{theorem} \label{theorem:characterization stochastic completeness}
 Let $\E$ be a Dirichlet form of full support. The following assertions are equivalent.
 \begin{itemize}
  \item[(i)] For one/all $\alpha > 0$ the identity $1 = \alpha \Ge_\alpha 1 + \Ge_\alpha k$ holds.
  \item[(ii)] For one/all $\alpha > 0$  any nonnegative essentially bounded weakly $\E_\alpha$-subharmonic function equals zero.
  \item[(iii)] For one/all $\alpha > 0$  any essentially bounded weakly $\E_\alpha$-harmonic function equals zero. 
  \item[(iv)] For one/all $\alpha > 0$ and for all $f \in L^\infty(m)$ the equation $\E^{(w)}_\alpha(g,\cdot) = \ell_f$ has a unique essentially bounded weak solution. 
 \end{itemize}
\end{theorem} 
 \begin{proof}
  The equivalence of (i) - (iv) for one $\alpha >0$ follows from Theorem~\ref{theorem:uniqueness of bounded weak solutions} and Theorem~\ref{theorem:lp resolvents}. Now, suppose that  the assertions (i) - (iv) hold for $\alpha> 0$.  In particular, we have $1 = \alpha \Ge_\alpha 1  + \Ge_\alpha k$. The resolvent equation of Proposition~\ref{proposition:properties resolvent} extends to $(\Ge_\alpha)$. Therefore, we obtain for $\beta > 0$  
  \begin{align*}
   \Ge_\beta k &= (\alpha - \beta) \Ge_\beta \Ge_\alpha k  + \Ge_\alpha k\\
   &=  (\alpha - \beta) \Ge_\beta (1 - \alpha \Ge_\alpha 1)  + 1 - \alpha \Ge_\alpha 1\\
   &= (\alpha - \beta) \Ge_\beta 1 - \alpha \Ge_\beta 1 + \alpha \Ge_\alpha 1  + 1 - \alpha \Ge_\alpha 1\\
   &= 1 - \beta \Ge_\beta 1. 
  \end{align*}
 This shows that assertion (i) holds for $\beta$ and finishes the proof. 
 \end{proof}
 \begin{corollary}
  Let $\E$ be a Dirichlet form of full support.  If for some $\alpha> 0$ the identity $1 = \alpha \Ge_\alpha 1 + \Ge_\alpha k$ holds, then $\E = \E^{\rm \, ref}$. Moreover, if $1 = \alpha \Ge_\alpha 1$ for some $\alpha> 0$, then $\E$ is Silverstein unique. 
 \end{corollary}
 \begin{proof}
 It follows from Corollary~\ref{corollary:conservation implies e=er} that $1 = \alpha \Ge_\alpha 1 + \Ge_\alpha k$ implies $\E_\alpha = (\E_\alpha)^{\rm \, ref}.$  Moreover, the equality $(\E_\alpha)^{\rm \, ref} = (\E^{\rm \, ref})_\alpha$ is a consequence of Lemma~\ref{lemma:main part and killing of dirichlet form}.   Since $D(\E)=D(\E_\alpha)$ and $D(\E^{\rm \, ref}) = D((\E^{\rm \, ref})_\alpha)$, we infer $\E = \E^{\rm \, ref}$. For the 'moreover' statement, we note that $1 = \alpha \Ge_\alpha 1$ yields
  $$1 = \alpha \Ge_\alpha 1 \leq 1 + \Ge_\alpha k  =  \alpha \Ge_\alpha 1 + \Ge_\alpha k \leq 1.$$
  Therefore, we have $\Ge_\alpha k = 0$. Since $\Ge_\alpha k$ is a weak solution to the equation $\E_{\alpha}^{(w)}(g,\cdot) = k$, this implies $k = 0$. Now, the statement follows from Theorem~\ref{corollary:maximality er dirichlet form} and the equality $\E =  \E^{\rm \, ref}$.
 \end{proof}

 \begin{remark} 
 \begin{itemize}
  \item A Dirichlet form is called {\em stochastically complete} or  {\em conservative}, if for all $\alpha > 0$ the identity $1 = \alpha \Ge_\alpha 1$ holds; otherwise it is called {\em stochastically incomplete}.  Dirichlet forms with nonvanishing killing part are automatically stochastically incomplete. In this sense, the condition $1 = \alpha \Ge_\alpha 1 + \Ge_\alpha k$ is a generalization of  stochastic completeness to the case when a killing is present. It is a well known meta theorem in Dirichlet form theory that stochastic completeness can be characterized by the uniqueness of bounded weak solutions to the Laplace equation, see \cite{Gri} for manifolds, \cite{Stu1} for strongly local regular Dirichlet forms and \cite{Woj} for graphs. For discrete regular Dirichlet forms with killing, a generalized conservation criterion, which is equivalent to  $1 = \alpha \Ge_\alpha 1 + \Ge_\alpha k$, was introduced in \cite{KL}. There, the authors named this property {\em stochastic completeness at infinity}. The novelty of our result lies in developing a theory of weak solutions that is flexible enough to prove the meta theorem on stochastic completeness for all Dirichlet forms (before it was only known for specific examples, not even for all regular Dirichlet forms) and at the same time to extend the main result of \cite{KL} to all Dirichlet forms. 
  
  \item The list of equivalent assertions in the previous theorem can be extended by uniqueness statements for bounded weak solutions to the heat equation and by a conservation property for the semigroup. Namely, if one extends the associated semigroup $(T_t)$ properly, then the identity $1 = \alpha \Ge_\alpha 1 + \Ge_\alpha k$ holds for all $\alpha > 0$ if and only if for all $t > 0$ the identity
  $$1 = T_t 1 + \int_0^t T_s k \D s $$
  holds. Here, we do not discuss weak solutions to the heat equation nor introduce extended semigroups. For details we refer the reader to $\cite{KL}$ for the graph case and to  \cite{MS} for the manifold case. For general Dirichlet forms this will be discussed elsewhere. 
  
  \item It is well known that stochastic completeness implies Silverstein uniqueness. The result is noted in \cite{Sil} and a detailed proof is given in \cite{Kuw}.
 \end{itemize}
 \end{remark}

 \subsection{Weakly superharmonic functions and recurrence} \label{subsection:recurrence}
 
In this subsection we prove a meta theorem on recurrence. In many situations recurrence is equivalent to the following assertions:
\begin{itemize}
 \item Every nonnegative 'superharmonic' function is constant.
 \item Every bounded 'superharmonic' function is constant.
 \item Every 'superharmonic' function of 'finite energy' is constant. 
\end{itemize}
The terms in quotation marks are defined depending on the concrete setup. For an energy  form $E$ it turns out that 'superharmonic' should be replaced by weakly $E$-superharmonic and that 'finite energy' should be read as a synonym for functions in $D(\Er)$.

 \begin{theorem} \label{theorem:finish} Let $E$ be an irreducible energy form and assume that $D(E)$ contains a nonconstant function. The following assertions are equivalent.
 \begin{itemize}
   \item[(i)] $E$ is recurrent.
   \item[(ii)] Every nonnegative weakly $E$-superharmonic function is constant.
   \item[(iii)] Every essentially bounded weakly $E$-superharmonic function is constant. 
   \item[(iv)] Every weakly $E$-superharmonic function in $D(\Er)$ is constant.
 \end{itemize}
 \end{theorem}
  \begin{proof} According to Corollary~\ref{corollary: irreducible forms have full support}, irreducible energy forms have full support. In particular, all theorems of this chapter can be applied to irreducible energy forms.
  
   (i) $\Rightarrow$ (ii): Let $h$ be nonnegative and weakly $E$-superharmonic. By Theorem~\ref{theorem:characterization excessive functions as superharmonic functions} $h$ is $E$-excessive. Since $E$ is irreducible and recurrent,  Theorem~\ref{theorem:recurrence in terms of constant excessive functions} implies that $h$ is constant.
   
   (ii) $\Rightarrow$ (iii):  Let $h$ be essentially bounded and weakly $E$-superharmonic. The function $\ow{h}:= h + \|h\|_\infty 1$ is weakly $E$-superharmonic an nonnegative. By (ii) it is constant.
   
   (iii) $\Rightarrow$ (i): Assume $E$ were transient. According to Theorem~\ref{theorem:recurrence in terms of constant excessive functions}, there exists a nonconstant $E$-superharmonic function $h \in D(E)$.  The transience of $E$ implies that $h$ is nonnegative, see Proposition~\ref{proposition: superharmonic functions are strictly positive}.  Corollary~\ref{corollary:concatenation superharmonic functions} implies that for all $n \in \IN$ the function $h_n: = h\wedge n$ is superharmonic. Since the $(h_n)$ are also bounded, (iii) implies that  they are constant. This contradicts the fact that $h$ is not constant and we arrive at (i). 
   
   (iv) $\Rightarrow$ (i): Assume $E$ were transient. According to Theorem~\ref{theorem:recurrence in terms of constant excessive functions}, there exists a nonconstant $E$-superharmonic function $h \in D(E) \subseteq D(\Er)$.  Hence, (iv) implies that $h$ is constant, a contradiction.
   
   (i) $\Rightarrow$ (iv): Let $h \in D(\Er)$ weakly $E$-superharmonic. Sine $E$ is recurrent, Theorem~\ref{theorem:recurrence and uniqueness} shows $\Ekm = 0$ and $E = \Eg$. Therefore,  $h$ is $E$-superharmonic. Recurrence  and Theorem~\ref{theorem:recurrence in terms of constant excessive functions} imply that it is constant. 
  \end{proof}

  \begin{remark}
   This theorem is a meta theorem for  Dirichlet forms, Markov chains and more general Markov processes. For manifolds, it is essentially due to Ahlfohrs, see \cite{Ahl} and the discussion in the survey \cite{Gri1}. For the case of Markov chains on graphs, see the book \cite{Woe} and references therein. 
  \end{remark}

  \chapter{Some loose ends}

Every longer research text leaves some loose ends and open problems. In this chapter we collect some of them.

\section{Chapter 1}

Theorem~\ref{theorem:characterization closedness hilbert space} characterizes closedness of a quadratic form $q$ in terms of completeness of the inner product space $(D(q)/\ker q,q)$. One of its assumptions was that the form topology $\tau_q$ is metrizable. 

\begin{question}
 Does Theorem~\ref{theorem:characterization closedness hilbert space} hold without the assumption that $\tau_q$ is metrizable or are there counterexamples?
\end{question}

We gave a short proof for the existence of the extended Dirichlet space. It was one goal of this thesis to use form methods only and our proof of Theorem~\ref{theorem: existence extended Dirichlet space} somehow violated this mantra. 

\begin{question}\label{question:form methods extended space}
 Is it possible to prove Theorem~\ref{theorem: existence extended Dirichlet space} using form methods only?
\end{question}
We have a partial answer to this question when $m$ is finite. It has the advantage that other than our proof of Theorem~\ref{theorem: existence extended Dirichlet space} it immediately extends to Markovian convex functionals.
\begin{proof}[Partial answer to Question~\ref{question:form methods extended space}]
 Let $\E$ be a Dirichlet form and let $m$ finite. Furthermore, let $(f_n)$ be a sequence in $D(\E)$ and let $f \in D(\E)$ with $f_n \overset{m}{\to} f$. For $N \in \IN$, we define $g_{n,N} :=   (f_n \wedge N) \vee(-N)$. Since $m$ is finite, the constant function $N$ belongs to $L^2(m)$.  Lebesgue's dominated convergence theorem, Lemma~\ref{lemma:Lebesgue's theorem}, implies  $g_{n,N} \to (f\wedge N)\vee (-N)$  in $L^2(m)$, as $n \to \infty$. Furthermore, we have $(f\wedge N)\vee (-N) \to f$  in $L^2(m)$, as $N \to \infty$. The $L^2$-lower semicontinuity of $\E$ and its contraction properties imply
 $$\E(f) \leq \liminf_{N\to \infty} \liminf_{n\to \infty} \E(g_{n,N}) \leq \liminf_{n\to \infty} \E(f_n).$$
 This proves the existence of the extended Dirichlet space.
\end{proof}
In \cite{Schmu2} it is shown that not only Markovian forms but also positivity preserving forms are closable on $L^0(m)$.  
\begin{question}
 Let $q$ be a closed quadratic on $L^2(m)$ such that for all $f \in L^2(m)$ the inequality $q(|f|) \leq q(f)$ holds. Can the idea of our proof of Theorem~\ref{theorem: existence extended Dirichlet space} be employed to show that $q$ is closable on $L^0(m)$?
\end{question}

Under certain conditions a von Neumann algebra has associated $L^2$-spaces and one associated $L^0$-space (the affiliated operators) that behave similarly as their commutative pendants. 

\begin{looseend}
 Our proof of the existence of the extended Dirichlet space applies to noncommutative Dirichlet forms in the sense of \cite{Cip}.
\end{looseend}

\section{Chapter 2}

We proved that the triviality of the kernel of an energy form $E$ implies that $(D(E),E)$ is a Hilbert space. The following is a slight extension of Conjecture~\ref{conjecture:hilbert space}.

\begin{question}
 Let $E$ be an energy form. Is $(D(E)/\ker E,E)$ always a Hilbert space?
\end{question}

This question seems to be related to the following. It is known in most examples.

\begin{question}
 Does each energy form admit a local Poincar\'e inequality? More precisely, let $E$ be an energy form and let $U \subseteq X$ be measurable with $m(U) < \infty$. Does there exist a strictly positive function $g \in L^1(U,m)$ such that $D(E) \subseteq L^1(U,g \cdot m)$ and
 $$\int_U |f-L(f)| g \D m \leq E(f)^{1/2} \text{ for all }f \in D(E),$$
 where $L(f) = \|g\|_1^{-1} \int_U fg \D m$?
\end{question}

\section{Chapter 3}

\begin{looseend}
 The maximality statement for $\Egm$ that we mentioned in Remark~\ref{remark:maximality er} can be proven with similar methods as the ones that we used for proving properties of $\Eg$ and $\Ekm$. The proofs of some details become a bit more difficult.
\end{looseend}

We proved that an energy form $E$ that possesses a recurrent Silverstein extension and has a nonvanishing killing part with $1 \in D(\Ekm)$ does not admit a maximal Silverstein extension. 

\begin{question}
 Does every energy form with nonvanishing killing part that is not Silverstein unique have a recurrent Silverstein extension? If not, does an energy form whose Silverstein extensions are all transient always possess a maximal Silverstein extension?
\end{question}

We constructed the reflected energy form, and hence the maximal Silverstein extension, in terms of the algebraic structure provided by $L^0(m)$ on the form domain. However, we also discussed that Silverstein extensions can be characterized in terms of the order structure of $L^0(m)$ on the form domain. 

\begin{question}
 Can the reflected energy form be constructed using the order structure only? What about the semigroup or the resolvent of the reflected Dirichlet form?
\end{question}

When $S$ is a densely defined symmetric operator on $L^2(m)$ whose associated quadratic form $\E_S$ is Markovian we constructed the maximal Element in ${\rm Ext}(S)_{{\rm Sil}}$, under various assumptions on $D(S)$, see Theorem~\ref{theorem:maximality er operator}.

\begin{question}
 Can the assumptions of Theorem~\ref{theorem:maximality er operator} be weakened? More precisely, is $S^{\rm ref}$ always the maximal element in ${\rm Ext}(S)_{{\rm Sil}}$ when there exists a special $\bar{\E}_{S, {\rm e}}$-nest $\cN$ such that $D(S) \subseteq D(E)_\cN$?
\end{question}
The mentioned condition can be seen as an abstract formulation that functions in $D(S)$ vanish at infinity.

\section{Chapter 4}

 Since for jump-type forms the weak form value is obtained by some form of improper integration, it is likely that the value of the weak form may depend on the chosen nest.
 
 \begin{question}
  Does there exist an energy form $E$, special $E$-nests $\cN,\cN'$, and functions $f \in D(E)_{w,\, \cN} \cap D(E)_{w,\, \cN'}$, $\psi \in D(E)_\cN \cap D(E)_{\cN'}$ such that 
  $$\Ew_\cN(f,\psi) \neq \Ew_{\cN'}(f,\psi)?$$
 \end{question}
 Next, we discuss a weaker version of the previous question. For excessive functions $h$ we proved several independence statements of the underlying nest, which basically followed from the monotonicity of the map $\varphi \mapsto E(\varphi h,\psi)$. A positive answer to the following question would make our theory of weak form extensions much more convenient to use. 
 \begin{question}
  Let $E$ be an energy form and let $\cN,\cN'$ be special $E$-nests. If $h \in L^0(m)$ is $E$-excessive, does $h \in D(E)_{w,\, \cN} \cap D(E)_{w,\, \cN'}$ imply 
  $$\Ew_\cN(f,\psi) = \Ew_{\cN'}(f,\psi) \text{ for all } \psi \in D(E)_\cN \cap D(E)_{\cN'}?$$
 \end{question}
 The theory that we develop in Section~\ref{section:Weakly superharmonic and excessive functions} allows us to give an affirmative answer when $\cN'  = \cN_h$ or when $\psi \in D(E)_{\cN_h} \cap D(E)_{\cN'_h}$. The same questions can be asked for the extended potential operator $G^\cN$.
 
 We mentioned in the introduction that we see our work on weak solutions as a possible starting point for geometric analysis of all energy forms. One main ingredient for this analysis are weak solutions, the other main ingredient is geometry. At a very rough level geometry is encoded by a metric on the underlying space. For Dirichlet forms, intrinsic metrics are exactly the metrics whose geometry is related to properties of the form.  They are characterized in terms of the Beurling-Deny formula, see \cite{FLW}. This is why we ask the following question.
 \begin{question}
  Let $\E$ be a regular Dirichlet form and let $\rho$ be a metric on the underlying space. Can $\rho$ being intrinsic with respect to $\E$ be characterized in terms of algebraic properties of $\E$ and $\rho$? Can it be characterized in terms of the forms $\E_\varphi$ or variants thereof?
 \end{question}
A partial result for this question is contained in \cite{AH} for local forms. However, this does not seem to apply in the nonlocal situation.

\begin{looseend}
For energy forms it is possible to study weak solutions to the heat equation, to introduce extended semigroups, to prove corresponding maximum principles and to prove meta theorems involving solutions to the heat equation.  The additional time derivate in the heat equation  complicates some arguments.
\end{looseend}

For introducing the weak form nests and local spaces were essential objects. They have analogues in the noncommutative world. For a measurable set $N$ the set $D(E)_N$ is a closed order ideal in the form domain $D(E)$ (and every closed order ideal is almost of this form when the form is sufficiently regular, see \cite{Sto}). If $N$ belongs to a special $E$-nest $\cN$, then there exists a function $\varphi \in D(E)_\cN$ with  $\varphi D(E)_N = D(E)_N \varphi =  D(E)_N.$ These are clearly algebraic properties that can be formulated  in  noncommutative spaces. Similarly, the local space and the other concepts that we introduced can be extended. 

\begin{looseend}
 In order to study weak solutions for energy forms on noncommutative spaces, the proofs of Chapter~4 need to be formulated purely in terms of algebra without reference to the underlying space. This is possible, our proofs can be modified accordingly. We chose to present proofs with reference to the underlying space since otherwise certain arguments would have looked artificial.  
\end{looseend}

  \begin{appendix}

\chapter{A Lemma on monotone nets}

\begin{lemma}\label{lemma:monotone nets}
 Let $I$ and $J$ be directed sets and for $i \in I, j \in J$ let $a_{ij} \in \IR$ be given. Assume that for each $j\in J$ the map $I \to \IR,\, i \mapsto a_{ij}$ is monotone increasing. Then
 $$\liminf_i \liminf_j a_{ij} \leq \liminf_j \liminf_i a_{ij}.$$
\end{lemma}
\begin{proof}
 Without loss of generality we can assume $\liminf_j \liminf_i a_{ij} < \infty.$ Fix an $\varepsilon > 0$. For each $j \in J$, we can  choose a $\varphi(j) \succ  j$ such that
 $$ \liminf_{i} a_{i\varphi(j)} \leq \liminf_j \liminf_i a_{ij} + \varepsilon.$$
 From the monotonicity of $a_{i\varphi(j)}$ in $i$ we infer 
 $$a_{i\varphi(j)} \leq \liminf_j \liminf_i a_{ij} + \varepsilon. $$
 Since the map $\varphi:J \to J$ is cofinal in $J$, we obtain
 $$\liminf_j a_{ij} \leq \liminf_j \liminf_i a_{ij} + \varepsilon$$
 and the statement follows. 
\end{proof}

\chapter{Sobolev spaces on Riemannian manifolds} \label{appendix:manifolds}

In this appendix we introduce several Sobolev spaces on Riemannian manifolds and recall their properties. For a background and precise definitions  we refer to \cite{Gri}.

Let $(M,g)$ be a smooth Riemannian manifold with associated volume measure ${\rm vol}_g$. We write $C_c^\infty(M)$ for the smooth compactly supported functions and equip them with the usual locally convex topology, cf. \cite[Section~4.1]{Gri}. The distributions on $C_c^\infty(M)$ are denoted by $\mathcal{D}'(M)$ and equipped with the weak-* topology. Similarly, we treat the space of smooth compactly supported vector fields $\vec{C}_c^\infty(M)$ and their distributions $\vec{\mathcal{D}}'(M)$. We let $\nabla:\mathcal{D}'(M) \to \vec{\mathcal{D}}'(M)$ the distributional gradient operator, ${\rm div}:\vec{\mathcal{D}}'(M) \to \mathcal{D}'(M)$ the distributional divergence operator and $\Delta = {\rm div} \circ \nabla$ the (negative definite) distributional Laplace-Beltrami operator. The $L^2$-Sobolev space of first order is
$$W^1(M) := \{f \in L^2({\rm vol}_g) \mid |\nabla f| \in L^2({\rm vol}_g)\},$$
equipped with the norm
$$\|\cdot\|_{W^1}:W^1(M) \to [0,\infty),\, f\mapsto \|f\|_{W^1} := \sqrt{\|f\|_2^2 + \||\nabla f|\|^2_2}.$$
Furthermore,  we denote the closure of $C_c^\infty(M)$ with respect to $\|\cdot\|_{W^1}$ by $W^1_0(M)$, denote the local Sobolev space of first order by
$$W^1_{\rm loc}(M) := \{f \in \Ltlm\mid |\nabla f| \in \Ltlm\},$$
and we let
$$W^1_c(M) := \{f \in W^1(M) \mid f \text{ has compact support}\}.$$
The statements of the following lemma are contained in  \cite[Lemma~5.5]{Gri} and \cite[Corollary~5.6]{Gri}.

\begin{lemma}\label{lemma:product in w0}
The inclusion $W^1_c(M) \subseteq W_0^1(M)$ holds. Moreover, for all $f \in W^1_{\rm loc}(M)$ and all $\varphi \in C_c^\infty(M)$ we have $f\varphi \in W_0^1(M)$.
\end{lemma}

\begin{lemma} \label{lemma:product rule}
Let   $f \in W^1_{\rm loc}(M)$ and let $\varphi \in W^1_c(M)$. Then
$$\nabla(\varphi f) = \varphi \nabla f + f \nabla \varphi.$$
\end{lemma}
\begin{proof}
By the previous lemma, there exists a sequence $(\varphi_n)$ in $C_c^\infty(M)$ that converges to $\varphi$ with respect to $\|\cdot\|_{W^1}$. In particular,  $f \varphi_n \to f \varphi$, $\varphi_n \nabla f \to \varphi \nabla f$ and $f \nabla \varphi_n \to f \nabla \varphi$, as $n\to \infty$, in the sense of distributions. According to \cite[Exercise~4.2]{Gri}, we have
$$\nabla(\varphi_n f) = \varphi_n \nabla f + f \nabla \varphi_n.$$
Since $\nabla$ is a continuous operator on all distributions, this implies the statement.  
\end{proof}
The following lemma is the assertion of \cite[Exercise~5.8]{Gri}. It is an easy consequence of \cite[Theorem~5.7]{Gri}.
 \begin{lemma}\label{lemma:contraction manifold}
  Let $\psi \in C^\infty(\IR)$ with $\psi(0) = 0$ and $\sup \{ |\psi'(t)| \mid t \in \IR\} < \infty$. Furthermore, let  $f \in W^1_{\rm loc}(M)$. Then $ \psi \circ f \in W^1_{\rm loc}(M)$ and 
  $$\nabla(\psi \circ f) = (\psi' \circ f) \nabla f.$$
 \end{lemma}
For an arbitrary manifold, there is the following local version of the Poincar\'{e} inequality.
\begin{lemma} \label{lemma:local poincare}
 For each $x \in M$, there exists a relatively compact open neighborhood $U_x$ of $x$ and a constant $C_x > 0$ such that for all $f\in L^2_{\rm loc}({\rm vol}_g)$ with $\int_{U_x} f \D{\rm vol}_g = 0$ and $|\nabla f| \in L^2({\rm vol}_g)$ the inequality
 $$\int_{U_x} |f|^2 {\rm d vol}_g \leq C_x \int_M |\nabla f|^2 \D{\rm vol}_g$$
 holds.
\end{lemma}
\begin{proof}
 For $x \in M$, we choose a chart $(U_x,\varphi_x)$ around $x$ such that $B_x:=\varphi_x(U_x)$ is an open Euclidean ball. For $f \in \Ltl$ with $|\nabla f| \in L^2({\rm vol}_g)$, we have $\bar{f} := f \circ \varphi_x^{-1} \in W^{1}(B_x)$. In these coordinates the occurring integrals satisfy
 $$\int_{U_x} |\nabla f|^2 {d \rm vol}_g = \int_{B_x} \as{A \nabla \bar{f}, \nabla \bar{f}} \sigma \D x$$
 and
 $$\int_{U_x} |f|^2 \D {\rm vol}_g = \int_{B_x} |\bar{f}|^2 \sigma \D x, $$
  where  $A$ is a uniformly positive definite matrix function and $\sigma$ is a uniformly positive and bounded function.  If $ \int_{U_x} f \D{\rm vol}_g = 0$, then also $\int_{B_x} \bar{f} \sigma \D x = 0$. We let $\bar{f}_{B_x} := |B_x|^{-1} \int_{B_x} \bar{f} \D x$. Using the properties of $A$ and $\sigma$ and the standard Poincar\'e inequality for functions in $W^{1}(B_x)$, see \cite[Section~5.8.1]{Eva}, we obtain constants $C_i > 0$ such that
  \begin{align*}
   \int_{B_x} |\bar{f}|^2 \sigma \D x &\leq \int_{B_x} |\bar{f} - \bar{f}_{B_x}|^2 \sigma \D x \\
   &\leq C_1 \int_{B_x} |\bar{f} - \bar{f}_{B_x}|^2  \D x \\
   &\leq C_2 \int_{B_x} |\nabla \bar{f}|^2\D x \\
   &\leq C_3 \int_{B_x} \as{A \nabla \bar{f}, \nabla \bar{f}} \sigma \D x.
  \end{align*}
  Note that for the first inequality, we used the minimality property of the expectation $\int_{B_x} \bar{f} \sigma \D x = 0$. This finishes the proof.
  \end{proof}

\chapter{An extension of normal contractions}

\begin{lemma} \label{lemma:extension of normal contractions}
 Let $A \subseteq \IR$ an let $C:A \to \IR$. Assume that $C$ satisfies  
 $$|C(x)| \leq |x| \text{ and } |C(x)-C(y)| \leq |x-y| \text{ for all }x,y \in A.$$
 Then $C$ can be extended to a normal contraction on $\IR^n$. 
\end{lemma}
\begin{proof} Without loss of generality we can assume $0 \in A$ for otherwise we could set $C(0) = 0$ without affecting the inequalities. We let
 $$\widetilde{C}(x) := \inf \left\{C(y) + \sum_{i=1}^n |x_i-y_i| \,\middle|\,  y \in A \right\}.$$
 For $\varepsilon > 0$ and $x \in \IR^n$, we choose $x^\varepsilon \in A$ such that
 $$\widetilde{C}(x) \geq C(x^\varepsilon) + \sum_{i=1}^n |x_i-x^\varepsilon_i| - \varepsilon.$$
  Let $x \in A$. By definition we have $\ow{C}(x) \leq C(x)$. Furthermore, for any $\varepsilon > 0$ we obtain
 $$C(x) - \ow{C}(x) \leq C(x) - C(x^\varepsilon)  -  \sum_{i=1}^n |x_i-x^\varepsilon_i| +  \varepsilon \leq \varepsilon. $$
 This shows that $\ow{C}$ extends $C$. Now we prove that $\ow{C}$ is a normal contraction. Since $0 \in A$ and $C$ extends $\ow{C}$, we have $\ow{C}(0) = 0$. Furthermore, for arbitrary $x,y \in \IR^n$ and $\varepsilon > 0$ we obtain
 $$\ow{C}(y) - \ow{C}(x) \leq C(x^\varepsilon) + \sum_{i=1}^n |y_i-x^\varepsilon_i|  - C(x^\varepsilon) - \sum_{i=1}^n |x_i-x^\varepsilon_i|  + \varepsilon \leq \sum_{i=1}^n |x_i-y_i| +  \varepsilon.$$
 This finishes the proof.
\end{proof}
\end{appendix}

  \backmatter

 \bibliographystyle{plain}
 
\bibliography{literatur}




\end{document}